\documentclass{amsart}
\usepackage{amsmath,amsfonts,amssymb,amsthm,amscd}
\usepackage{url}
\usepackage[dvipdfm]{graphicx}

\newcommand{\mc}[1]{{}}
\renewcommand{\emptyset}{\varnothing}
\renewcommand{\Im}{\operatorname{Im}}
\renewcommand{\Re}{\operatorname{Re}}
\renewcommand{\setminus}{-}

\renewcommand{\epsilon}{\varepsilon}
\newcommand{\Vol}{\operatorname{Vol}}

\newcommand{\Area}{\operatorname{Area}}
\newcommand{\area}{\operatorname{area}}
\newcommand{\diam}{\operatorname{diam}}

\newcommand{\card}{\operatorname{card}}

\newcommand{\Id}{\operatorname{Id}}

\newcommand{\GL}{\operatorname{GL}(2,{\mathbb R})}
\newcommand{\SL}{\operatorname{SL}(2,{\mathbb R})}
\newcommand{\PSL}{\operatorname{PSL}(2,{\mathbb R})}
\newcommand{\SLZ}{\operatorname{SL}(2,{\mathbb Z})}
\newcommand{\PSLZ}{\operatorname{PSL}(2,{\mathbb Z})}
\newcommand{\SO}{\operatorname{SO}(2,{\mathbb R})}
\newcommand{\PSO}{\operatorname{PSO}(2,{\mathbb R})}
\newcommand{\CP}{{\mathbb C}\!\operatorname{P}^1}
\newcommand{\ext}{{\operatorname{Ext}}}
\newcommand{\noz}{n}

\newcommand{\SVc}{c_{\mathcal{C}}}

\newcommand{\cf}{\phi}

\newcommand{\Dflat}{\Delta_{\text{\textit{flat}}}}
\newcommand{\Dhyp}{\Delta_{\mathit{T\hspace{-1.5pt}ei\hspace{-0.5pt}ch}}}
\newcommand{\nhyp}{\nabla_{\!\mathit{T\hspace{-1.5pt}ei\hspace{-0.5pt}ch}}}

\newcommand{\nuhyp}{\nu_{\mathit{hyp}}}

\newcommand{\gflat}{g_{\mathit{flat}}}
\newcommand{\ghyp}{g_{\mathit{hyp}}}

\newcommand{\gfe}{g_{\mathit{flat},\epsilon}}
\newcommand{\rhofe}{\rho_{\mathit{flat},\epsilon}}
\newcommand{\cfe}{\mathit{const}_{\mathit{flat},\epsilon}}

\newcommand{\ghd}{g_{\mathit{hyp},\delta}}
\newcommand{\rhohd}{\rho_{\mathit{hyp},\delta}}
\newcommand{\chd}{\mathit{const}_{\mathit{hyp},\delta}}

\newcommand{\dd}{\partial\overline{\partial}}
\newcommand{\dtdtbar}{\frac{\partial^2}{\partial t\,\partial\overline{t}}}

\newcommand\const{\mathit{const}}

\newcommand\lf{\ell_{\mathit{flat}}}
\newcommand\lh{\ell_{\mathit{hyp}}}

\newcommand{\Mod}{\operatorname{Mod}}

\newcommand{\Y}{\mathsf{Y}}

\newcommand{\Ypunc}{\mathring{Y}}
\newcommand{\Spunc}{\mathring{\!S}}

\newcommand{\geff}{g_{\mathit{eff}}}

\newcommand{\pihor}{\pi_{\mathit{hor}}}
\newcommand{\pivert}{\pi_{\mathit{vert}}}

\newcommand\N{\mathbb N}

\newcommand\Z{\mathbb Z}

\newcommand\R[1]{{\mathbb R}^{#1}}
\newcommand{\reals}{{\mathbb R}}

\newcommand\C[1]{{\mathbb C}^{#1}}

\newcommand\Hyp{{\mathbb H}^2}
\newcommand\T{{\mathbb T}^2}

\newcommand{\cC}{{\mathcal C}}
\newcommand{\cE}{{\mathcal E}}

\newcommand{\cH}{{\mathcal H}}

\newcommand{\cL}{{\mathcal L}}
\newcommand{\cM}{{\mathcal M}}
\newcommand{\cN}{{\mathcal N}}
\newcommand{\cO}{{\mathcal O}}

\newcommand{\cQ}{{\mathcal Q}}

\newcommand{\PcH}{\operatorname{\mathbb{P}}\!{\mathcal H}}
\newcommand{\PcQ}{\operatorname{\mathbb{P}}\hspace*{-2pt}{\mathcal Q}}

\newlength{\halfbls}

\numberwithin{equation}{section}

\newtheorem{Theorem}{Theorem}
\newtheorem{theorem}[Theorem]{Theorem}
\newtheorem*{NNTheorem}{Theorem}
\newtheorem*{Theorem1prime}{Theorem $\mathbf{1'}$}
\newtheorem*{BackgroundTheorem}{Background Theorem}

\newtheorem{prop}{Proposition}[section]
\newtheorem{proposition}[prop]{Proposition}
\newtheorem{Proposition}[prop]{Proposition}

\newtheorem{Lemma}{Lemma}[section]
\newtheorem{lemma}[Lemma]{Lemma}

\newtheorem*{NNLemma}{Lemma}

\newtheorem{cor}{Corollary}
\newtheorem{corollary}[cor]{Corollary}
\newtheorem{Corollary}[cor]{Corollary}
\newtheorem*{NNCorollary}{Corollary}

\newtheorem{Conjecture}{Conjecture}
\newtheorem{Problem}{Problem}

\newtheorem*{PolyakovFormula}{Theorem (Polyakov Formula)}
\newtheorem*{GreenFormula}{Green's Formula}

\theoremstyle{remark}

\newtheorem{Remark}{Remark}[section]
\newtheorem{remark}[Remark]{Remark}
\newtheorem*{NNRemark}{Remark}

\newtheorem{Example}{Example}[section]

\theoremstyle{definition}

\newtheorem{Definition}{Definition}
\newtheorem{definition}[Definition]{Definition}

\begin{document}

\title[Lyapunov exponents of the Teichm\"uller flow]
{Sum of Lyapunov exponents of the Hodge bundle with respect to the
Teichm\"uller geodesic flow}

\dedicatory{Symphony in $\Delta$-Flat}

\author{Alex Eskin}
\thanks{Research  of  the first author is partially supported  by
NSF grant.}
\address{
Department of Mathematics,
University of Chicago,
Chicago, Illinois 60637, USA\\
}
\email{eskin@math.uchicago.edu}

\author{Maxim Kontsevich}
\thanks{Research of the second and of the third authors is partially
supported by ANR grant GeoDyM}
\address{IHES,
le Bois Marie,
35, route de Chartres,
91440 Bures-sur-Yvette, FRANCE\\
}
\email{maxim@ihes.fr}

\author{Anton Zorich}
\thanks{Research of the third author is partially supported by IUF}
\address{
Institut de Math\'ematiques de Jussieu (Paris Rive Gauche),
Universit\'e Paris 7 and IUF,
B\^atiment Sophie~Germain, Case 7012,  75205 Paris Cedex 13, FRANCE
}
\email{zorich@math.jussieu.fr}

\subjclass[2000]{
Primary
30F30, 
32G15, 
32G20, 
57M50; 
Secondary
14D07, 
37D25  
}

\keywords{Teichm\"uller geodesic flow, moduli space of quadratic
differentials, Lyapunov exponent, Hodge norm, determinant of Laplacian,
flat structure, saddle connection}

   %
\maketitle

\tableofcontents

\section{Introduction}

\subsection{Moduli spaces of Abelian and quadratic differentials}

The  moduli space $\cH_g$ of pairs $(C,\omega)$ where $C$ is a smooth
complex  curve  of  genus $g$ and $\omega$ is an Abelian differential
(or,  in the other words, a holomorphic 1-form) is a total space of a
complex  $g$-dimensional  vector bundle over the moduli space $\cM_g$
of  curves  of  genus $g$. The moduli space $\cQ_g$ of of holomorphic
quadratic  differentials  is  a  complex  $(3g-3)$-dimensional vector
bundle   over  the  moduli  space  of  curves  $\cM_g$.  In  all  our
considerations  we  always  remove the zero sections from both spaces
$\cH_g$ and $\cQ_g$.

There  are  natural  actions  of $\C{\ast}$ on the spaces $\cH_g$ and
$\cQ_g$  by  multiplication of the corresponding Abelian or quadratic
differential  by  a nonzero complex number. We will also consider the
corresponding     projectivizations    $\PcH_g=\cH_g/\C{\ast}$    and
$\PcQ_g=\cQ_g/\C{\ast}$ of the spaces $\cH_g$ and $\cQ_g$.
\smallskip

\noindent
\textbf{Stratification.}
Each  of  these  two spaces is naturally stratified by the degrees of
zeroes  of  the  corresponding  Abelian  differential or by orders of
zeroes  of the corresponding quadratic differential. (We try to apply
the  word  ``degree'' for the zeroes of  {\it Abelian}  differentials
reserving  the  word  ``order''  for  the  zeroes of  {\it quadratic}
differentials.)  We  denote the strata by $\cH(m_1,\dots,m_\noz)$ and
$\cQ(d_1,\dots,d_\noz)$ correspondingly. Here $m_1+\dots+m_\noz=2g-2$
and    $d_1+\dots+d_\noz=4g-4$.   By   $\PcH(m_1,\dots,m_\noz)$   and
$\PcQ(d_1,\dots,d_\noz)$  we  denote  the  projectivizations  of  the
corresponding strata.
We   shall  also consider slightly more general strata of meromorphic
quadratic  differentials  with at most simple poles, for which we use
the  same  notation $\cQ(d_1,\dots,d_\noz)$ allowing to certain $d_j$
be equal to $-1$.

The   dimension of a stratum of Abelian differentials is expressed as
$$
\dim_{\C{}}\cH(m_1,\dots,m_\noz)=2g+\noz-1\,.
$$
The   dimension of a stratum of quadratic differentials which are not
global squares of an Abelian differentials is expressed as
$$
\dim_{\C{}}\cQ(d_1,\dots,d_\noz)=2g+\noz-2\\
$$
Note  that,  in  general,  the  strata do not have the structure of a
bundle over the moduli space $\cM_g$, in particular, it is clear from
the  formulae  above that some strata have dimension smaller then the
dimension of $\cM_g$.
\smallskip

\noindent\textbf{Period coordinates.}
Consider  a  small  neighborhood  $U(C_0,\omega_0)$   of  a ``point''
$(C_0,\omega_0)$    in    a    stratum   of   Abelian   differentials
$\cH(m_1,\dots,m_\noz)$. Any Abelian differential $\omega$ defines an
element  $[\omega]$  of the relative cohomology $H^1(C,\{\text{zeroes
of  }\omega\};\C{})$.  For  a  sufficiently  small  neighborhood of a
generic  ``point'' $(C_0,\omega_0)$ the resulting map from $U$ to the
relative  cohomology  is  a bijection, and one can use an appropriate
domain   in   the   relative   cohomology   $H^1(C,\{\text{zeroes  of
}\omega\};\C{})$    as    a   coordinate   chart   in   the   stratum
$\cH(m_1,\dots,m_\noz)$.

Chose  some  basis  of cycles in $H_1(S,\{P_1,\dots,P_\noz\};\Z)$. By
$Z_1,\dots,  Z_{2g+\noz-1}$  we  denote  the  corresponding  relative
periods   which   serve   as   local   coordinates   in  the  stratum
$\cH(m_1,\dots,m_\noz)$.       Similarly,       one      can      use
$(Z_1:Z_2:\dots:Z_{2g+\noz-1})$    as   projective   coordinates   in
$\PcH(m_1,\dots,m_\noz)$.

The  situation with the strata $\cQ(d_1,\dots,d_\noz)$ of meromorphic
quadratic  differentials  with  at  most  simple  poles, which do not
correspond  to global squares of Abelian differentials, is analogous.
We  first  pass  to  the canonical double cover $p:\hat S\to S$ where
$p^\ast  q=\hat\omega^2$  becomes  a  global  square  of  an  Abelian
differential  $\hat\omega$  and  then  use  the  subspace $H^1_-(\hat
S,\{\text{zeroes  of  }\hat\omega\};\C{})$  antiinvariant  under  the
natural involution to construct coordinate charts. Thus, we again use
a  certain subcollection of relative periods $Z_1, \dots, Z_k$ of the
Abelian  differential  $\hat\omega$  as  coordinates  in  the stratum
$\cQ(d_1,\dots,d_\noz)$.     Passing    to    the    projectivization
$\PcQ(d_1,\dots,d_\noz)$  we  use  projective  coordinates $(Z_1:Z_2:
\dots:                                                          Z_k)$
\smallskip

\subsection{Volume element and action of the linear group.}
\label{ss:Volume:element:and:action:of:the:linear:group}

The vector space
$$
H^1(S,\{\text{zeroes of }\omega\};\C{})
$$
considered  over  real  numbers  is  endowed  with  a natural integer
lattice,   namely   with   the   lattice   $H^1(S,\{\text{zeroes   of
}\omega\};\Z\oplus  i\Z)$.  Consider  a linear volume element in this
vector  space normalized in such way that a fundamental domain of the
lattice  has  area  one.  Since  relative  cohomology  serve as local
coordinates  in  the  stratum, the resulting volume element defines a
natural  measure  $\mu$ in the stratum $\cH(m_1,\dots,m_\noz)$. It is
easy  to  see that the measure $\mu$ does not depend on the choice of
local  coordinates  used  in  the construction, so the volume element
$\mu$ is defined canonically.

The  canonical volume element in a stratum $\cQ(d_1,\dots,d_\noz)$ of
meromorphic  quadratic  differentials  with  at  most simple poles is
defined analogously using the vector space
$$
H^1_-(S,\{\text{zeroes of }\hat\omega\};\C{})
$$
described above and the natural lattice inside it.
\smallskip

\noindent\textbf{Flat structure.}
A  quadratic  differential  $q$ with at most simple poles canonically
defines  a  flat  metric  $|q|$  with  conical  singularities  on the
underlying Riemann surface $C$.

If  the  quadratic  differential  is  a  global  square of an Abelian
differential, $q=\omega^2$, the linear holonomy of the flat metric is
trivial; if not, the holonomy representation in the group $\Z/2\Z$ is
nontrivial. We denote the resulting flat surface by $S=(C,\omega)$ or
$S=(C,q)$ correspondingly.

A  zero of order $d$ of the quadratic differential corresponds to a
conical point with the cone angle $\pi(d+2)$. In particular, a simple
pole corresponds to a conical point with the cone angle $\pi$. If the
quadratic differential is a global square of an Abelian differential,
$q=\omega^2$,   then  a  zero  of  \textit{degree}  $m$  of  $\omega$
corresponds to a conical point with the cone  angle  $2\pi(m+1)$.

When  $q=\omega^2$ the area of the surface $S$ in the associated flat
metric  is defined in terms of the corresponding Abelian differential
as
$$
\Area(S)=
\int_C |q|=
\cfrac{i}{2}\int_C \omega\wedge\bar\omega\,.
$$
When  the quadratic differential is not a global square of an Abelian
differential,  one  can express the flat area in terms of the Abelian
differential   on   the   canonical   double   cover   where  $p^\ast
q=\hat\omega^2$:
$$
\Area(S)=
\int_C |q|=
\cfrac{i}{4}\int_{\hat C} \hat\omega\wedge\overline{\hat\omega}\,.
$$

By  $\cH_1(m_1,\dots,m_\noz)$  we denote the real hypersurface in the
corresponding  stratum  defined by the equation $\Area(S)=1$. We call
this  hypersurface  by  the same word ``stratum'' taking care that it
does not provoke ambiguity.
Similarly   we   denote   by   $\cQ_1(d_1,\dots,d_\noz)$   the   real
hypersurface  in  the  corresponding  stratum defined by the equation
$\Area(S)=\const$.  Throughout  this paper we choose $const:=1$; note
that   some   other   papers,   say~\cite{Athreya:Eskin:Zorich},  use
alternative convention $\const:=\frac{1}{2}$.
\smallskip

\noindent\textbf{Group action.}
Let  $X_j=\Re(Z_j)$ and let $Y_j=\Im(Z_j)$. Let us rewrite the vector
of periods $(Z_1,\dots,Z_{2g+\noz-1})$ in two lines
$$
\begin{pmatrix}
X_1&X_2&\dots&X_{2g+\noz-1}\\
Y_1&Y_2&\dots&Y_{2g+\noz-1}
\end{pmatrix}
$$
The   group $\operatorname{GL}_+(2,\R{})$ of $2\!\times\! 2$-matrices
with  positive  determinant  acts  on the left on the above matrix of
periods as
$$
\begin{pmatrix}
g_{11}&g_{12}\\
g_{21}&g_{22}
\end{pmatrix}
\cdot
\begin{pmatrix}
X_1&X_2&\dots&X_{2g+\noz-1}\\
Y_1&Y_2&\dots&Y_{2g+\noz-1}
\end{pmatrix}
$$
Considering  the  lines  of  resulting  product  as  the real and the
imaginary  parts  of periods of a new Abelian differential, we define
an   action   of   $\operatorname{GL}_+(2,\R{})$   on   the   stratum
$\cH(m_1,\dots,m_\noz)$ in period coordinates. Thus, in the canonical
local   affine   coordinates,   this   action   is   the   action  of
$\operatorname{GL}_+(2,\R{})$ on the vector space
\begin{multline*}
H^1(C,\{\text{zeroes of }\omega\};\C{})
\simeq \C{} \otimes H^1(C,\{\text{zeroes of }\omega\};\R{})
\simeq
\\
\simeq\R{2} \otimes H^1(C,\{\text{zeroes of }\omega\};\R{})
\end{multline*}
through the first factor in the tensor product.

The  action of the linear group on the strata $\cQ(d_1,\dots,d_\noz)$
is    defined    completely   analogously   in   period   coordinates
$H^1_-(C,\{\text{zeroes  of }\hat\omega\};\C{})$. The only difference
is  that  now  we  have  the action of the group $\PSL$ since $p^\ast
q=\hat\omega^2=(-\hat\omega)^2$, and the subgroup $\{\Id,-\Id\}$ acts
trivially on the strata of quadratic differentials.

\begin{NNRemark}
One  should  not  confuse the trivial action of the element $-\Id$ on
quadratic  differentials  with  multiplication  by  $-1$:  the latter
corresponds   to   multiplication   of   the   Abelian   differential
$\hat\omega$   by   $i$,   and   is   represented   by   the   matrix
$\begin{pmatrix} 0&1\\-1&0\end{pmatrix}$.
\end{NNRemark}

From  this  description it is clear that the subgroup $\SL$ preserves
the  measure  $\mu$  and  the  function  $\Area$, and, thus, it keeps
invariant  the  ``unit  hyperboloids''  $\cH_1(m_1,\dots,m_\noz)$ and
$\cQ_1(d_1,\dots,d_\noz)$. Let
$$
a(S):=\Area(S)
$$
The measure $\mu$ in the stratum defines canonical measure
$$
\nu:=\frac{\mu}{da}
$$
on     the     ``unit     hyperboloid''     $\cH_1(m_1,\dots,m_\noz)$
(correspondingly    on    $\cQ_1(d_1,\dots,d_\noz)$).    It   follows
immediately  from  the  definition of the group action that the group
$\SL$ (correspondingly $\PSL$) preserves the measure $\nu$.

The     following    two    Theorems    proved    independently    by
H.~Masur~\cite{Masur}  and  by  W.~Veech~\cite{Veech} are fundamental
for the study of dynamics in the Teichm\"uller space.

\begin{NNTheorem}[H.~Masur; W.~Veech]
The  total volume of any stratum $\cH_1(m_1,\dots,m_\noz)$ of Abelian
differentials   and   of  any  stratum  $\cQ_1(d_1,\dots,d_\noz)$  of
meromorphic  quadratic  differentials  with at most simple poles with
respect to the measure $\nu$ is finite.
\end{NNTheorem}

Note that the strata might have up to three connected components. The
connected components of the strata were classified by the authors for
Abelian   differentials~\cite{Kontsevich:Zorich:connected:components}
and  by  E.~Lanneau~\cite{Lanneau}  for  the  strata  of  meromorphic
quadratic differentials with at most simple poles.

\begin{Remark}
\label{rm:values:of:volumes}
The  volumes  of  the  connected  components of the strata of Abelian
differentials    were    effectively   computed   by   A.~Eskin   and
A.~Okounkov~\cite{Eskin:Okounkov}.   The   volume  of  any  connected
component  of  any  stratum  of  Abelian  differentials  has the form
$r\cdot\pi^{2g}$, where $r$ is a rational number. The exact numerical
values  of the corresponding rational numbers are currently tabulated
up to genus ten (up to genus $60$ for some individual strata like the
principal one).
\end{Remark}

\begin{NNTheorem}[H.~Masur; W.~Veech]
The action of the one-parameter subgroup of $\SL$ (correspondingly of
$\PSL$) represented by the matrices
$$
G_t=
\begin{pmatrix}
e^t&0\\
0&e^{-t}
\end{pmatrix}
$$
is  ergodic  with  respect  to  the measure $\nu$   on each connected
component   of  each  stratum  $\cH_1(m_1,\dots,m_\noz)$  of  Abelian
differentials  and  on  each  connected  component  of  each  stratum
$\cQ_1(d_1,\dots,d_\noz)$ of meromorphic quadratic differentials with
at most simple poles.
\end{NNTheorem}

The  projection  of trajectories of the corresponding group action to
the  moduli  space  of  curves  $\cM_g$  correspond  to Teichm\"uller
geodesics  in  the natural parametrization, so the corresponding flow
$G_t$  on  the  strata is called the ``Teichm\"uller geodesic flow''.
Notice,  however,  that  the Teichm\"uller metric is not a Riemannian
metric, but only a Finsler metric.

\subsection{Hodge bundle and Gauss--Manin connection}
\label{sec:Hodge:bundle}
A  complex  structure  on  the  Riemann surface $C$ underlying a flat
surface  $S$  of genus $g$ determines a complex $g$-dimensional space
of   holomorphic   1-forms   $\Omega(C)$   on   $C$,  and  the  Hodge
decomposition
$$
H^1(C;\C{}) = H^{1,0}(C)\oplus H^{0,1}(C) \simeq
\Omega(C)\oplus\bar \Omega(C)\ .
$$
The intersection form
\begin{equation}
\label{eq:hodge:hermitian:form}
\langle\omega_1,\omega_2\rangle:=\frac{i}{2}
\int_{C} \omega_1\wedge  \bar{\omega}_2\qquad\qquad\qquad
\end{equation}
is   positive-definite   on  $H^{1,0}(C)$  and  negative-definite  on
$H^{0,1}(C)$.

The projections      $H^{1,0}(C)\to      H^1(C;\R{})$,      acting as
$[\omega]\mapsto[\Re(\omega)]$ and $[\omega]\mapsto[\Im(\omega)]$ are
isomorphisms  of  vector  spaces over $\R{}$.      The Hodge operator
$\ast: H^1(C;\R{})\to   H^1(C;\R{})$   acts  as  the inverse  of  the
first isomorphism composed with the second one. In other words, given
$v\in   H^1(C;\R{})$,   there   exists   a  unique  holomorphic  form
$\omega(v)$    such  that  $v=[\Re(\omega(v))]$; the dual $\ast v$ is
defined  as $[\Im(\omega)]$.

Define the \textit{Hodge norm} of $v\in H^1(C,\R{})$ as
$$
\|v\|^2=\langle\omega(v),\omega(v)\rangle
$$

Passing  from  an  individual  Riemann  surface  to the moduli  stack
$\cM_g$  of                Riemann  surfaces,  we get vector  bundles
$H^1_{\C{}}=H^{1,0}\oplus  H^{0,1}$,  and $H^1_{\R{}}$ over   $\cM_g$
with   fibers   $H^1(C,\C{})=H^{1,0}(C)\oplus  H^{0,1}(C)$,       and
$H^1(C,\R{})$  correspondingly over $C\in\cM_g$.    The vector bundle
$H^{1,0}$  is  called  the  \textit{Hodge  bundle}.  When the context
excludes  any possible ambiguity we also refer to each of the bundles
$H^1_{\C{}}$ and to $H^1_{\R{}}$ as \textit{Hodge bundle}.
\medskip

Using    integer    lattices   $H^1(C,\Z{}\oplus   i\Z{})$        and
$H^1(C,\Z{})$   in   the  fibers  of  these  vector  bundles  we  can
canonically  identify  fibers  over  nearby  Riemann  surfaces.  This
identification  is  called the \textit{Gauss--Manin}  connection. The
Hodge   norm   \textit{is  not}  preserved  by  the     Gauss---Manin
connection  and  the  splitting    $H^1_{\C{}}=H^{1,0}\oplus H^{0,1}$
\textit{is   not}   covariantly    constant   with  respect  to  this
connection.

\subsection{Lyapunov exponents}
\label{sec:Lyapunov:exponents}
Informally,  the  Lyapunov  exponents of a vector bundle endowed with
a     connection      can be viewed as logarithms of mean eigenvalues
of monodromy    of    the    vector    bundle    along a    flow   on
the base.

In  the case of the Hodge bundle, we take a fiber of $H^1_{\R{}}$ and
pull  it  along a Teichm\"uller geodesic on the moduli space. We wait
till  the  geodesic  winds a lot and comes close to the initial point
and  then  compute the resulting monodromy matrix $A(t)$. Finally, we
compute  logarithms of eigenvalues of $A^T\!A$, and normalize them by
twice the length $t$ of the geodesic. By the Oseledets multiplicative
ergodic  theorem,  for  almost  all choices of initial data (starting
point,  starting  direction) the resulting $2g$ real numbers converge
as  $t \to \infty$, to limits which do not depend on the initial data
within   an   ergodic   component   of   the   flow.   These   limits
$\lambda_1\ge\dots\ge\lambda_{2g}$  are  called the \textit{Lyapunov
exponents}  of  the  Hodge  bundle  along  the Teichm\"uller flow.

The  matrix  $A(t)$ preserves the intersection form on cohomology, so
it  is  symplectic.  This implies that Lyapunov spectrum of the Hodge
bundle is symmetric with respect to the sign interchange,
$
\lambda_j=-\lambda_{2g-j+1}
$.
Moreover,  from  elementary  geometric  arguments it follows that one
always  has  $\lambda_1=1$. Thus, the Lyapunov spectrum is defined by
the remaining nonnegative Lyapunov exponents
$$
\lambda_2\ge\dots\ge\lambda_g\ .
$$

Given  a    vector bundle endowed with a norm and a connection we can
construct  other  natural  vector  bundles  endowed with a norm and a
connection:  it  is  sufficient  to apply elementary linear-algebraic
constructions  (direct  sums,  exterior  products, etc.) The Lyapunov
exponents  of  these  new  bundles might be expressed in terms of the
Lyapunov  exponents  of  the  initial vector bundle. For example, the
Lyapunov  spectrum  of  a  $k$th  exterior power of a vector   bundle
(where  $k$ is not bigger than a dimension of a fiber) is represented
by all possible sums
$$
\lambda_{j_1}+\dots+\lambda_{j_k}\qquad\text{where } j_1<j_2<\dots<j_k
$$
of $k$-tuples of Lyapunov exponents of the initial vector bundle.

\subsection{Regular invariant suborbifolds}
\label{sec:Regular:invariant:submanifolds}

For a subset $\cM_1 \subset \cH_1(m_1 \dots, m_\noz)$ we write
\begin{displaymath}
\reals \cM_1 = \{ (M, t \omega) \;|\; (M,\omega) \in \cM_1, \quad t \in
\reals \} \subset \cH(m_1 \dots, m_\noz)\,.
\end{displaymath}
Let $a(S):=\Area(S)$.

\begin{Conjecture}
\label{conj:main}
Let  $\cH(m_1  \dots, m_\noz)$ be a stratum of Abelian differentials.
Let  $\nu_1$  be  an  ergodic  $\SL$-invariant probability measure on
$\cH_1(m_1 \dots, m_\noz)$.
Then
\begin{itemize}
\item[{\rm (i)}]
The  support of $\nu_1$ is an immersed
suborbifold $\cM_1$ of $\cH_1(m_1,\dots,
m_\noz)$.   In cohomological       local       coordinates
$H^1(S,\{\text{\textit{zeroes}}\}\,;\,\C{})$,  the suborbifold $\cM =
\reals \cM_1$ of $\cH(m_1 \dots, m_\noz)$ is represented by a complex
affine   subspace,  such  that  the  associated  linear  subspace  is
invariant under complex conjugation.
\item[{\rm (ii)}]
Let  $\mu$  be  the measure on $\cM$ such that $d\mu = d\nu_1 \, da$.
Then  $\mu$  is \textit{ affine}, i.e. it is an affine linear measure
in         the         cohomological         local        coordinates
$H^1(S,\{\text{\textit{zeroes}}\}\,;\,\C{})$.
\end{itemize}
\end{Conjecture}

We  say that a suborbifold $\cM_1$, for which there exists a measure
$\nu_1$  such  that the pair $(\cM_1, \nu_1)$ satisfies (i) and (ii),
is an \textit{invariant suborbifold}.

\begin{Conjecture}
\label{conj:closures}
The closure of any $\SL$-orbit is an invariant
suborbifold. For any invariant suborbifold,
the set of self-intersections is itself a finite union of
affine invariant suborbifolds of lower dimension.
\end{Conjecture}

These  conjectures  have been proved by C.~McMullen in genus $2$, see
\cite{McMullen:genus2}.  They  are  also known in a few other special
cases,  see  \cite{Eskin:Marklof:Morris}  and \cite{Calta:Wortman}. A
proof  of  Conjecture  1  has been recently announced by A.~Eskin and
M.~Mirzakhani  \cite{Eskin:Mirzakhani};  a  proof of Conjecture 2 has
been  recently  announced by A.~Eskin, M.~Mirzakhani and A.~Mohammadi
\cite{Eskin:Mirzakhani:Mohammadi}.

\begin{definition}
\label{def:invariant:regular}
An  invariant  suborbifold is \textit{ regular} if in addition to (i)
and (ii) it satisfies the following technical condition:
\begin{itemize}
\item[{\rm  (iii)}]  \textit{For  $K  >  0$  and  $\epsilon  > 0$ let
  $\cM_1(K,  \epsilon)  \subset  \cM_1$  denote  the  set of surfaces
  which  contain  two  non-parallel cylinders $C_1$, $C_2$, such that
  for  $i=1,2$, $\Mod(C_i) > K$ and $w(C_i) < \epsilon$. An invariant
  suborbifold  is called \textit{ regular} if there exists a $K > 0$,
  such that
\begin{equation}
\label{eq:condition:III}
\lim_{\epsilon \to 0} \frac{\nu_1(\cM_1(K,\epsilon))}{\epsilon^2} = 0.
\end{equation}
}
\end{itemize}
\end{definition}

All  known  examples  of  invariant  suborbifolds are regular, and we
believe this is always the case. (After completion of
  work on this paper, it was proved by A.~Avila, C.~Matheus Santos and
  J.~C.~Yoccoz that indeed all $\SL$-invariant measures are regular,
  see \cite{Avila:Matheus:Yoccoz:regular}.)
In the rest of the paper we consider
only regular invariant suborbifolds. (However, the condition (iii) is
used only in section~\ref{sec:cutoff}.)

\begin{NNRemark}
In view of Conjecture~\ref{conj:main}, in this paper we consider only
density  measures;  moreover,  densities  always correspond to volume
forms  on appropriate suborbifolds. Depending on a context we use one
of  the three related structures mostly referring to any of them just
as   a   ``measure''.   Also,  if  $\cM_1$  is  a  regular  invariant
suborbifold,  we  often  write  $c_{\mathit{area}}(\cM_1)$ instead of
$c_{\mathit{area}}(\nu_1)$,    where   the   Siegel--Veech   constant
$c_{\mathit{area}}$ is defined in \S\ref{sec:subsec:Siegel:Veech}.
Throughout   this   paper   we   denote  by  $d\nu_1$  the  invariant
\textit{probability}   density  measure  and  by  $d\nu$  any  finite
invariant density measure on a regular invariant suborbifold $\cM_1$.
\end{NNRemark}

\begin{NNRemark}
We  say that a subset $\cM_1$ of a stratum of quadratic differentials
is  a  regular  invariant  suborbifold  if under the canonical double
cover  construction it corresponds to a regular invariant suborbifold
of a stratum of Abelian differentials. See section~\ref{sec:sums} for
details.
\end{NNRemark}

\subsection{Siegel--Veech constants}
\label{sec:subsec:Siegel:Veech}

Let  $S$  be  a  flat surface in some stratum of Abelian or quadratic
differentials.  Together  with every closed regular geodesic $\gamma$
on $S$ we have a bunch of parallel closed regular geodesics filling a
maximal  cylinder $\mathit{cyl}$ having a conical singularity at each
of  the  two  boundary  components.  By  the  \textit{width} $w$ of a
cylinder  we  call  the  flat  length  of  each  of  the two boundary
components, and by the \textit{height} $h$ of a cylinder --- the flat
distance between the boundary components.

The  number of maximal cylinders filled with regular closed geodesics
of  bounded  length  $w(cyl)\le L$ is finite. Thus, for any $L>0$ the
following quantity is well-defined:

\begin{equation}
\label{eq:N:area}
N_{\mathit{area}}(S,L):=
\frac{1}{\Area(S)}
\sum_{\substack{
\mathit{cyl}\subset S\\
w(\mathit{cyl})<L}}
\Area(cyl)
\end{equation}

The  following  theorem  is a special case of a fundamental result of
W.~Veech,     \cite{Veech:Siegel}     considered    by    Y.~Vorobets
in~\cite{Vorobets}:
\begin{NNTheorem}[W.~Veech; Ya.~Vorobets]
Let   $\nu_1$  be  an  ergodic  $\SL$-invariant  probability  measure
(correspondingly  $\PSL$-invariant  probability measure) on a stratum
$\cH_1(m_1,\dots,m_\noz)$  of  Abelian differentials (correspondingly
on  a  stratum  $\cQ_1(d_1,\dots,d_\noz)$  of  meromorphic  quadratic
differentials  with  at  most  simple  poles)  of area one. Then, the
following  ratio  is constant (i.e. does not depend on the value of a
positive parameter $L$):
\begin{equation}
\label{eq:SV:constant:definition}
\frac{1}{\pi L^2}\int N_{\mathit{area}}(S,L)\,d\nu_1=
c_{\mathit{area}}(\nu_1)
\end{equation}
\end{NNTheorem}
This  formula  is  called  a \textit{Siegel---Veech formula}, and the
corresponding   constant  $c_{\mathit{area}}(\nu_1)$  is  called  the
\textit{Siegel--Veech constant}.

\begin{Conjecture}
For any regular $\SL$-invariant suborbifold $\cM_1$ in any stratum of
Abelian   differentials   the  corresponding  Siegel--Veech  constant
$\pi^2\cdot c_{\mathit{area}}(\cM_1)$ is a rational number.
\end{Conjecture}

By  Lemma~\ref{lm:SVconst:for:cover}  below  an affirmative answer to
this  conjecture  automatically  implies an affirmative answer to the
analogous  conjecture  for  invariant  suborbifolds  in the strata of
meromorphic quadratic differentials with at most simple poles.

Let  $\nu_1$  be an ergodic $\PSL$-invariant probability measure on a
stratum    $\cQ_1(d_1,\dots,d_\noz)$    of    meromorphic   quadratic
differentials  with  at most simple poles, which are \textit{not} the
global  squares  of  Abelian  differentials.  Passing  to a canonical
double  cover  $p:\hat  C\to  C$,  where  $p^\ast q$ becomes a global
square  of  an Abelian differential we get an induced $\SL$-invariant
probability    measure   $\hat\nu_1$   on   the   resulting   stratum
$\cH_1(m_1,\dots,m_k)$.   The  degrees  $m_j$  of  the  corresponding
Abelian      differential      $\hat\omega$      are     given     by
formula~\eqref{eq:d:to:m}                                          in
section~\ref{ss:Sum:of:the:Lyapunov:exponents}  below.  We shall need
the   following   relation   between   the   Siegel--Veech   constant
$c_{\mathit{area}}(\hat\nu_1)$  of  the induced invariant probability
measure   $\hat\nu_1$   in   terms   of  the  Siegel--Veech  constant
$c_{\mathit{area}}(\nu_1)$   of  the  initial  invariant  probability
measure $\nu_1$.

\begin{Lemma}
\label{lm:SVconst:for:cover}
Let  $\hat\nu_1$  be  an  $\SL$-invariant  probability  measure  on a
stratum   $\cH_1(m_1,\dots,m_k)$   induced  from  a  $\PSL$-invariant
probability  measure  on  a  stratum $\cQ_1(d_1,\dots,d_\noz)$ by the
canonical  double  cover construction. The Siegel--Veech constants of
the two measures are related as follows:
\begin{equation*}
   %
c_{\mathit{area}}(\hat\nu_1)=
2\,c_{\mathit{area}}(\nu_1)
\end{equation*}
\end{Lemma}
\begin{proof}
Consider  any  flat  surface  $S=(C,q)$ in the support of the measure
$\nu_1$.  The  linear  holonomy  of  the flat metric on $S$ along any
closed  flat geodesic is trivial. Thus, the waist curves of cylinders
on  $S$  are  lifted to closed flat geodesics on the canonical double
cover  $\hat  S$  of  the same length as downstairs. Hence, the total
area  $\Area(\widehat{cyl})$ swept by each family of parallel closed
geodesics  on  the  double cover $\hat S$ doubles with respect to the
corresponding area downstairs. Since $\Area\hat S=2\Area S$ we get
$$
N_{\mathit{area}}(\hat S,L)
  =
\sum_{\substack{
\widehat{cyl}\subset \hat S\\
w(\widehat{cyl})<L}}
\frac{\Area(\widehat{cyl})}{\Area(\hat S)}
  =
\sum_{\substack{
cyl\subset S\\
w(cyl)<L}}
\frac{\Area(cyl)}{\Area(S)}
  =
N_{\mathit{area}}( S,L)
$$

For  a flat surface $M$ denote by $M_{(1)}$ a proportionally rescaled
flat surface of area one. The definition of $N_{\mathit{area}}( M,L)$
immediately implies that for any $L>0$
$$
N_{\mathit{area}}(M_{(1)},L)=
N_{\mathit{area}}\left(M,\sqrt{\Area(M)} L\right)\,.
$$
Hence,

\begin{multline*}
c_{\mathit{area}}(\hat\nu_1):=
\frac{1}{\pi L^2}\int N_{\mathit{area}}(\hat S_{(1)},L)\,d\hat\nu_1
=
\frac{1}{\pi L^2}\int
  N_{\mathit{area}}\left(\hat S,\sqrt{\Area(\hat S)}L\right)\,d\hat\nu_1
=\\=
\frac{2}{\pi \left(\sqrt{2}L\right)^2}\int
  N_{\mathit{area}}\left(\hat S,\sqrt{2}L\right)\,d\hat\nu_1
=
\frac{2}{\pi R^2}\int
  N_{\mathit{area}}(S,R)\,d\nu_1
=
2\,c_{\mathit{area}}(\nu_1)\,,
\end{multline*}
where we used the notation $R:=\sqrt{2} L$.
\end{proof}

\section{Sum of Lyapunov exponents for $\SL$-invariant suborbifolds}
\label{sec:sums}

\subsection{Historical remarks}

There  are  no  general  methods  of evaluation of Lyapunov exponents
unless  the  base  is a homogeneous space or unless the vector bundle
has  real  $1$-dimensional  equivariant  subbundles. However, in some
cases  it  is  possible  to evaluate Lyapunov exponents approximately
through  computer  simulation  of the corresponding dynamical system.
Such  experiments  with  \textit{Rauzy--Veech  induction} (a discrete
model of the Teichm\"uller geodesic flow) performed by the authors in
1995--1996,   indicated   a   surprising   rationality  of  the  sums
$\lambda_1+\dots+\lambda_g$ of Lyapunov exponents of the Hodge bundle
with respect to Teichm\"uller flow on strata of Abelian and quadratic
differentials,  see~\cite{Kontsevich:Zorich}.  An explanation of this
phenomenon  was  given by M.~Kontsevich in~\cite{Kontsevich} and then
developed by G.~Forni~\cite{Forni:Deviation}.

It took us almost fifteen years to collect and assemble all necessary
ingredients  to  obtain  and justify an explicit formula for the sums
$\lambda_1+\dots+\lambda_g$.   In   particular,  to  obtain  explicit
numerical  values of these sums, one needs estimates
from the work of A.~Eskin and H.~Masur on the asymptotic of the
counting function of periodic orbits \cite{Eskin:Masur}
(developing Veech's seminal paper \cite{Veech:Siegel});
one needs to know the classification
of  connected  components  of  the  strata  (which  was  performed by
M.~Kontsevich    and    A.~Zorich~\cite{Kontsevich:Zorich}   and   by
E.~Lanneau~\cite{Lanneau});  one  needs  to  compute volumes of these
components (they are computed in the papers of A.~Eskin, A.~Okounkov,
and                           R.~Pandharipande~\cite{Eskin:Okounkov},
\cite{Eskin:Okounkov:Pandharipande});   one   also   has  to  know  a
description  of  the  principal  boundary  of  the  components of the
strata,  and  values  of  the  corresponding  Siegel--Veech constants
(obtained      by      A.~Eskin,      H.~Masur      and     A.~Zorich
in~\cite{Eskin:Masur:Zorich} and~\cite{Masur:Zorich}).

Several  important  subjects  related  to  the  study of the Lyapunov
spectrum remain beyond the scope of our consideration. We address the
reader  to  the original paper of G.~Forni~\cite{Forni:Deviation}, to
the     survey~\cite{Forni:Handbook}     and     to     the    recent
papers~\cite{Forni:new},    \cite{Trevino},   \cite{Aulicino:thesis},
\cite{Aulicino:affine}    for    the   very   important   issues   of
\textit{determinant  locus} and of \textit{nonuniform hyperbolicity}.
We address the reader to the paper~\cite{Avila:Viana} of A.~Avila and
M.~Viana  for  the  proof  of  \textit{simplicity of the spectrum} of
Lyapunov  exponents for connected components of the strata of Abelian
differentials.  For  invariant  suborbifolds of the strata of Abelian
differentials   in   genus   two   (see~\cite{Bainbridge:Euler:char},
\cite{Bainbridge:L:shaped})  and  for  certain  special Teichm\"uller
curves,   the   Lyapunov   exponents   are   computed   individually,
see~\cite{Bouw:Moeller},              \cite{Eskin:Kontsevich:Zorich},
\cite{Forni:Handbook},                   \cite{Forni:Matheus:Zorich},
\cite{Wright:JMD}, \cite{Wright:GAFA}.

\subsection{Sum of Lyapunov exponents}
\label{ss:Sum:of:the:Lyapunov:exponents}

Now we are ready to formulate the principal results of our paper.

\begin{Theorem}
\label{theorem:general:Abelian}
Let   $\cM_1$   be   any  closed  connected  regular  $\SL$-invariant
suborbifold  of  some  stratum  $\cH_1(m_1,\dots,m_\noz)$  of Abelian
differentials,   where  $m_1+\dots+m_\noz=2g-2$. The top $g$ Lyapunov
exponents  of the of the Hodge bundle $H^1_{\R{}}$ over $\cM_1$ along
the Teichm\"uller flow satisfy the following relation:
\begin{equation}
\label{eq:general:sum:of:exponents:for:Abelian}
\lambda_1 + \dots + \lambda_g
\ = \
\cfrac{1}{12}\cdot\sum_{i=1}^\noz\cfrac{m_i(m_i+2)}{m_i+1}
\ +\ \frac{\pi^2}{3}\cdot c_{\mathit{area}}(\cM_1)
\end{equation}
where   $c_{\textit{area}}(\cM_1)$   is  the  Siegel--Veech  constant
corresponding   to  the  regular  suborbifold  $\cM_1$.  The  leading
Lyapunov exponent $\lambda_1$ is equal to one.
\end{Theorem}

We        prove       Theorem~\ref{theorem:general:Abelian}       and
formula~\eqref{eq:general:sum:of:exponents:for:Abelian}  in  the very
end of section~\ref{sec:Outline:of:proofs}.

\begin{NNRemark}
For  all  known  regular $\SL$-invariant suborbifolds, in particular,
for   connected  components  of  the  strata  and  for  preimages  of
Teichm\"uller  curves, the sum of the Lyapunov exponents is rational.
However,  currently  we do not have a proof of rationality of the sum
of  the  Lyapunov  exponents for  \textit{any} regular $\SL$-invariant
suborbifold.
\end{NNRemark}

Let  us proceed with a consideration of sums of Lyapunov exponents in
the  case  of meromorphic quadratic differentials with at most simple
poles.  Let  $S$  be  a  flat  surface  of  genus  $g$  in a stratum
$\cQ(d_1,\dots,d_\noz)$    of    quadratic    differentials,    where
$d_1+\dots+d_\noz=4g-4$. Similarly to the case of Abelian differentials
we  have  the  Hodge bundle $H^1_{\R{}}$ over $\cQ(d_1,\dots,d_\noz)$
with  a  fiber  $H^1(S,\R{})$  over  a  ``point'' $S$. As before this
vector   bundle   is  endowed  with  the  Hodge  norm  and  with  the
Gauss--Manin    connection.   We   denote   the   Lyapunov   exponents
corresponding  to  the  action  of the Teichm\"uller geodesic flow on
this vector bundle by $\lambda_1^+\ge \dots \ge \lambda_g^+$.

Consider  a  canonical (possibly ramified) double cover $p:\hat{S}\to
S$  such  that  $p^\ast  q=(\hat\omega)^2$,  where $\hat\omega$ is an
Abelian  differential  on  the  Riemann surface $\hat S$. This double
cover  has ramification points at all zeroes of odd orders of $q$ and
at  all  simple  poles, and no other ramification points. It would be
convenient to introduce the following notation:
\begin{equation}
\label{eq:g:eff}
g_{\mathit{eff}}:=\hat{g}-g
\end{equation}

By  construction the double cover $\hat{S}$ is endowed with a natural
involution $\sigma:\hat{S}\to\hat{S}$ interchanging the two sheets of
the cover. We can decompose the vector space $H^1(\hat{S},\R{})$ into
a    direct    sum    of    subspaces    $H^1_+(\hat{S},\R{})$    and
$H^1_-(\hat{S},\R{})$   which   are   correspondingly  invariant  and
anti-invariant  with  respect to the induced involution $\sigma^\ast:
H^1(\hat{S},\R{})\to  H^1(\hat{S},\R{})$  on  cohomology.  Note  that
topology  of  the  ramified  cover $\hat{S}\to S$ is the same for all
flat  surfaces  in  the stratum $\cQ(d_1,\dots,d_\noz)$. Thus, we get
two  natural  vector  bundles  over  $\cQ(d_1,\dots,d_\noz)$ which we
denote  by  $H^1_+$  and  by  $H^1_-$.  By construction, these vector
bundles  are  equivariant with respect to the $\PSL$-action; they are
endowed with the Hodge norm and with the Gauss--Manin connection.

Clearly,  the  vector bundle $H^1_+$ is canonically isomorphic to the
initial  Hodge  bundle  $H^1_{\R{}}$:  it  corresponds  to cohomology
classes   pulled  back  from  $S$  to  $\hat{S}$  by  the  projection
$p:\hat{S}\to S$. Hence,
$$
\dim H^1_- = \dim H^1_-(\hat{S},\R{}) = 2g_{\mathit{eff}}
$$
We denote the top $g_{\mathit{eff}}$ Lyapunov exponents corresponding
to the action of the Teichm\"uller geodesic flow on the vector bundle
$H^1_-$ by $\lambda_1^-\ge \dots \ge \lambda_{g_{\mathit{eff}}}^-$.

\begin{Theorem}
\label{theorem:general:quadratic}
Consider  a  stratum $\cQ_1(d_1,\dots,d_\noz)$ in the moduli space of
quadratic   differentials   with   at   most   simple   poles,  where
$d_1+\dots+d_\noz=4g-4$.  Let $\cM_1$ be any regular $\PSL$-invariant
suborbifold of $\cQ_1(d_1,\dots,d_\noz)$.

a)  The  Lyapunov exponents $\lambda_1^+\ge \dots \ge \lambda_g^+$ of
the  invariant  subbundle  $H^1_+$  of  the Hodge bundle over $\cM_1$
along the Teichm\"uller flow satisfy the following relation:

\begin{equation}
\label{eq:general:sum:of:plus:exponents:for:quadratic}
\lambda^+_1 + \dots + \lambda^+_g
\ = \
\cfrac{1}{24}\,\sum_{j=1}^\noz \cfrac{d_j(d_j+4)}{d_j+2}
+\frac{\pi^2}{3}\cdot c_{\mathit{area}}(\cM_1)
\end{equation}
where   $c_{\textit{area}}(\cM_1)$   is  the  Siegel--Veech  constant
corresponding  to  the  suborbifold $\cM_1$. By convention the sum in
the                 left-hand                 side                 of
equation~\eqref{eq:general:sum:of:plus:exponents:for:quadratic}    is
defined to be equal to zero for $g=0$.

b)    The    Lyapunov    exponents    $\lambda_1^-\ge    \dots    \ge
\lambda_{g_{\mathit{eff}}}^-$ of the anti-invariant subbundle $H^1_-$
of the Hodge bundle over $\cM_1$ along the Teichm\"uller flow satisfy
the following relation:
\begin{equation}
\label{eq:general:index:of:exponents:for:quadratic}
\big(\lambda^-_1 + \dots + \lambda^-_{g_{\mathit{eff}}}\big)-
\big(\lambda^+_1+ \dots +\lambda^+_g\big)
\ = \
\cfrac{1}{4}\,\cdot\,\sum_{\substack{j \text{ such that}\\
d_j \text{ is odd}}}
\cfrac{1}{d_j+2}
\end{equation}
The leading Lyapunov exponent $\lambda^-_1$ is equal to one.
\end{Theorem}

We  prove  part  (a)  of  Theorem~\ref{theorem:general:quadratic} and
formula~\eqref{eq:general:sum:of:plus:exponents:for:quadratic} in the
very end of section~\ref{sec:Outline:of:proofs}.

\begin{proof}[Proof of part (b) of Theorem~\ref{theorem:general:quadratic}]
Recall   that  we  reserve  the  word  ``degree''  for  the zeroes of
\textit{Abelian} differentials and the word ``order''  for the zeroes
of \textit{quadratic} differentials.

Let  the  covering  flat  surface  $\hat{S}$  belong  to  the stratum
$\cH(m_1,\dots,m_k)$.  The resulting holomorphic form $\hat\omega$ on
$\hat{S}$ has zeroes of the following degrees:
\begin{align}
\notag
\text{A singularity } & \text{of order $d$ of $q$ on $S$}\\
\label{eq:d:to:m}
&\text{gives rise to }
\begin{cases}
\text{two zeroes of $\hat\omega$ of degree $m=d/2$ when $d$ is even}
\\
\text{single zero of $\hat\omega$ of degree $m=d+1$ when $d$ is odd}
\end{cases}
\end{align}

Thus, we get the following expression for the genus $\hat{g}$ of the
double cover $\hat{S}$:
\begin{equation}
\label{eq:g:hat}
\hat{g}=2g-1+\cfrac{1}{2}\,
(\text{Number of singularities of odd order})
\end{equation}
which follows from the relation below:
\begin{multline*}
4\hat{g}-4
\ =\
\sum_{\substack{
j \text{ such that}\\
d_j \text{ is odd}}} \
(2d_j+2)
\ +\
\sum_{\substack{
j \text{ such that}\\
d_j \text{ is even}}} \
(2d_j)
\ = \\
=\ 2\sum_{j=1}^\noz d_j
\ +\ 2\,(\text{Number of singularities of odd order})
\ =\ \\
=\ 2(4g-4)+2\,(\text{Number of singularities of odd order})
\end{multline*}

Applying           Theorem~\ref{theorem:general:Abelian}          and
equation~\eqref{eq:sum:for:Abelian:strata}     to    the    invariant
suborbifold  $\hat{\cM}\subset\cH(m_1,\dots,m_k)$  induced from $\cM$
we get
$$
\lambda_1 + \dots + \lambda_{\hat{g}} \ = \
\cfrac{1}{12}\cdot\sum_{i=1}^\noz\cfrac{m_i(m_i+2)}{m_i+1}
+\frac{\pi^2}{3}\cdot c_{\mathit{area}}(\hat{\cM})
$$
where  $\hat{g}$  is  the genus of $\hat S$, and $\lambda_1 \ge \dots
\ge \lambda_{\hat{g}}$ are the Lyapunov exponents of the Hodge bundle
$H^1(\hat{S};\R{})$ over $\hat{\cM}$.

Note  that  $H^1(\hat{S};\R{})$  decomposes  into  a  direct  sum  of
symplectically orthogonal subspaces:
$$
H^1(\hat{S};\R{})=H_+^1(\hat{S};\R{})\oplus H_-^1(\hat{S};\R{})
$$
Hence,
$$
(\lambda_1 + \dots + \lambda_{\hat{g}}) \ = \
(\lambda^-_1 + \dots + \lambda^-_{g_{\mathit{eff}}})\ +\
(\lambda^+_1+ \dots +\lambda^+_g)
$$
Moreover,     by     Lemma~\ref{lm:SVconst:for:cover}     we     have
$c_{\mathit{area}}(\hat{\cM})=2\,c_{\mathit{area}}(\cM_1)$,     which
implies the following relation:
\begin{multline}
\label{eq:tmp}
\big(\lambda^-_1 + \dots + \lambda^-_{g_{\mathit{eff}}}\big)+
\big(\lambda^+_1+ \dots +\lambda^+_g\big)
\ =\\ =\
\cfrac{1}{12}\cdot\sum_{i=1}^\noz\cfrac{m_i(m_i+2)}{m_i+1}
\ +\
2\frac{\pi^2}{3}\cdot c_{\mathit{area}}(\cM_1)
\end{multline}

The  degrees $m_i$ of zeroes of the Abelian differential $\hat\omega$
defining  the flat metric on $\hat{S}$ are calculated in terms of the
orders  $d_j$  of  zeroes  and  of  simple  poles  of  the  quadratic
differential    $q$    defining   the   flat   metric   on   $S$   by
formula~\eqref{eq:d:to:m}, which implies:
$$
\sum_{i=1}^\noz
\cfrac{m_i(m_i+2)}{m_i+1}
\ =
\sum_{\substack{
j \text{ such that}\\
d_j \text{ is odd}}}
\cfrac{(d_j+1)(d_j+3)}{d_j+2}
\ +\
2
\sum_{\substack{
j \text{ such that}\\
d_j \text{ is even}}}
\cfrac{(d_j/2)(d_j/2+2)}{d_j/2+1}
$$
Thus, we can rewrite relation~\eqref{eq:tmp} as follows:
\begin{multline*}
\big(\lambda^-_1 + \dots + \lambda^-_{g_{\mathit{eff}}}\big)+
\big(\lambda^+_1+ \dots +\lambda^+_g\big)
\ = \\
=\
\cfrac{1}{12}
\sum_{\substack{
j \text{ such that}\\
d_j \text{ is odd}}}
\cfrac{(d_j+1)(d_j+3)}{d_j+2}
\ +\
\cfrac{1}{12}
\sum_{\substack{
j \text{ such that}\\
d_j \text{ is even}}} \cfrac{d_j(d_j+4)}{d_j+2}
\ +\ 2\frac{\pi^2}{3}\cdot c_{\mathit{area}}(\cM_1)
\end{multline*}

Taking    the    difference    between   the   above   relation   and
relation~\eqref{eq:general:sum:of:plus:exponents:for:quadratic} taken
with      coefficient      $2$      we     obtain     the     desired
relation~\eqref{eq:general:index:of:exponents:for:quadratic}.
\end{proof}

\subsection{Genus zero and hyperelliptic loci}

Our results become even more explicit in a particular case of genus
zero, and in a closely related case of hyperelliptic loci.

\begin{Theorem}
\label{theorem:sum:for:CP1}
Consider a stratum $\cQ_1(d_1,\dots,d_\noz)$ in the moduli space of
quadratic differentials with at most simple poles on $\CP$, where
$d_1+\dots+d_\noz=-4$. Let $\cM_1$ be any regular $\PSL$-invariant
suborbifold of $\cQ_1(d_1,\dots,d_\noz)$. Let
$g_\mathit{eff}$ be the genus of the canonical double cover $\hat
S$ over a Riemann surface $S$ in $\cQ_1(d_1,\dots,d_\noz)$.

\begin{itemize}
\item[{\rm        (a)}]        The       Siegel--Veech       constant
$c_{\mathit{area}}(\cM_1)$  depends  only  on the ambient stratum and
equals
$$
c_{\mathit{area}}(\cM_1)=
-\cfrac{1}{8\pi^2}\,\sum_{j=1}^\noz \cfrac{d_j(d_j+4)}{d_j+2}
$$

\item[{\rm  (b)}]  The  Lyapunov  exponents $\lambda_1^-\ge \dots \ge
\lambda^-_{g_\mathit{eff}}$  of  the anti-invariant subbundle $H^1_-$
of the Hodge bundle over $\cM_1$ along the Teichm\"uller flow satisfy
the following relation:
\begin{equation}
\label{eq:sum:for:CP1}
\lambda_1^- + \dots + \lambda^-_{g_\mathit{eff}}
\ = \
\cfrac{1}{4}\,\cdot\,\sum_{\substack{j \text{ such that}\\
d_j \text{ is odd}}}
\cfrac{1}{d_j+2}
\end{equation}
\end{itemize}
\end{Theorem}

\begin{NNRemark}
Relation~\eqref{eq:sum:for:CP1}            was            conjectured
in~\cite{Kontsevich:Zorich}.
\end{NNRemark}

\begin{proof}
Apply
equations~\eqref{eq:general:sum:of:plus:exponents:for:quadratic}
and~\eqref{eq:general:index:of:exponents:for:quadratic} and note that
by   convention   the   sum  of  exponents  $\big(\lambda^+_1+  \dots
+\lambda^+_g\big)$  in  the  left-hand side is defined to be equal to
zero for $g=0$.
\end{proof}

The  square  of  any  holomorphic  1-form $\omega$ on a hyperelliptic
Riemann  surface  $S$  is  a  pullback  $(\omega)^2=p^\ast q$ of some
meromorphic  quadratic  differential  with  simple poles $q$ on $\CP$
where   the   projection   $p:S\to\CP$   is  the  quotient  over  the
hyperelliptic   involution.   The   relation   between   the  degrees
$m_1,\dots,m_k$    of    zeroes    of   $\omega$   and   the   orders
$d_1,\dots,d_\noz$   of   singularities  of  $q$  is  established  by
formula~\eqref{eq:d:to:m}.

Note,   that   a   pair   of   hyperelliptic   Abelian  differentials
$\omega_1,\omega_2$  in  the  same stratum $\cH(m_1,\dots,m_k)$ might
correspond to meromorphic quadratic differentials in different strata
on  $\CP$ depending on which zeroes are interchanged and which zeroes
are  invariant  under  the  hyperelliptic involution. Note also, that
hyperelliptic  loci  in  the  strata  of  Abelian  differentials  are
$\SL$-invariant,   and   that   the   orders   $d_1,\dots,d_\noz$  of
singularities  of the underlying quadratic differential do not change
under the action of $\SL$.

\begin{Corollary}
\label{cor:hyp:connected:comps}
Suppose  that  $\cM_1$  is a regular $\SL$-invariant suborbifold in a
hyperelliptic locus of some stratum $\cH_1(m_1,\dots,m_k)$ of Abelian
differentials in genus $g$. Denote by $(d_1,\dots,d_\noz)$ the orders
of singularities of the underlying quadratic differentials.

The  top  $g$  Lyapunov  exponents  of  the  Hodge bundle $H^1$ over
$\cM_1$ along the Teichm\"uller flow satisfy the following relation:
$$
\lambda_1 + \dots + \lambda_g
\ = \
\cfrac{1}{4}\,\cdot\,\sum_{\substack{j \text{ such that}\\
d_j \text{ is odd}}}
\cfrac{1}{d_j+2}\,,
$$
where, as usual, we associate the order $d_i=-1$ to simple poles.

In particular, for any regular $\SL$-invariant suborbifold $\cM_1$ in
a hyperelliptic connected component one has
$$
\begin{array}{ccccl}
1+\lambda_2 + \dots + \lambda_g &
= & \cfrac{g^2}{2g-1} & \text{for } & \cM_1\subseteq\cH^{\mathit{hyp}}_1(2g-2)
        \\[-\halfbls] \\
1+\lambda_2 + \dots + \lambda_g &
= & \cfrac{g+1}{2}    & \text{for } & \cM_1\subseteq\cH^{\mathit{hyp}}_1(g-1,g-1)\,.
\end{array}
$$
\end{Corollary}
\begin{proof}
The   first   statement   is   just  an  immediate  reformulation  of
Theorem~\ref{theorem:sum:for:CP1}.  To  prove  the  second part it is
sufficient   to   note  in  addition,  that  hyperelliptic  connected
components               $\cH^{\mathit{hyp}}(2g-2)$               and
$\cH^{\mathit{hyp}}(g-1,g-1)$   are  obtained  by  the  double  cover
construction  from  the strata of meromorphic quadratic differentials
$\cQ(2g-3,-1^{2g+1})$ and $\cQ(2g-2,-1^{2g+2})$ correspondingly.
\end{proof}

\begin{Corollary}
\label{cor:g2}
For  any  regular  $\SL$-invariant suborbifold $\cM_1$ in the stratum
$\cH_1(2)$  of  Abelian  differentials in genus two the Siegel--Veech
constant  $c_{area}(\cM_1)$  is equal to $10/(3\pi^2)$ and the second
Lyapunov exponent $\lambda_2$ is equal to $1/3$.

For  any  regular  $\SL$-invariant suborbifold $\cM_1$ in the stratum
$\cH_1(1,1)$  of Abelian differentials in genus two the Siegel--Veech
constant  $c_{area}(\cM_1)$ is equal to $15/(4\pi^2)$ and the second
Lyapunov exponent $\lambda_2$ is equal to $1/2$.
\end{Corollary}
\begin{proof}
Any  Riemann  surface of genus two is hyperelliptic. The moduli space
of  Abelian  differentials  in  genus $2$ has two strata $\cH(2)$ and
$\cH(1,1)$.  Both  strata  are  connected  and  coincide  with  their
hyperelliptic  components. The value of the Siegel--Veech constant is
now      given      by      Theorem~\ref{theorem:sum:for:CP1}     and
Lemma~\ref{lm:SVconst:for:cover}   and   the   values   of  the  sums
$\lambda_1+\lambda_2=1+\lambda_2$       are       calculated       in
Corollary~\ref{cor:hyp:connected:comps}.
\end{proof}

\begin{NNRemark}
The  values  of  the  second  Lyapunov  exponent  in  genus  $2$ were
conjectured  by  the  authors in 1997 (see~\cite{Kontsevich:Zorich}).
This        conjecture       was       recently       proved       by
M.~Bainbridge~\cite{Bainbridge:Euler:char},
\cite{Bainbridge:L:shaped}   where  he  used  the  classification  of
ergodic  $\SL$-invariant  measures  in  the  moduli  space of Abelian
differentials in genus due to C.~McMullen~\cite{McMullen:genus2}.
\end{NNRemark}

\begin{NNRemark}
Note  that  although  the  sum of the Lyapunov exponents is constant,
individual        Lyapunov       exponents       $\lambda^-_j(\cM_1)$
in~\eqref{eq:sum:for:CP1}  might  vary from one invariant suborbifold
of  a  given stratum in genus zero to another, or, equivalently, from
one invariant suborbifold in a fixed hyperelliptic locus to another.
\end{NNRemark}

We  formulate  analogous  statements  for the hyperelliptic connected
components  in the strata of meromorphic quadratic differentials with
at most simple poles.

\begin{Corollary}
\label{cor:Q:hyp:connected:comps}
For   any   regular   $\PSL$-invariant   suborbifold   $\cM_1$  in  a
hyperelliptic  connected  component  of  any  stratum  of meromorphic
quadratic  differentials  with  at  most  simple  poles,  the  sum of
nonnegative  Lyapunov  exponents  $\lambda^-_1 +\lambda^-_2 + \dots +
\lambda^-_{\geff}$ has the following value:
$$
\begin{array}{cll}
\frac{g+1}{2}+\frac{g+1}{2(2g-2k-1)(2k+3)} & \text{for} &
\cQ_1^{\mathit{hyp}}\big(2(g-k)-3,2(g-k)-3,2k+1,2k+1\big)
        \\[-\halfbls] \\
&&\text{ where } k \geq
-1, \ g \geq 1, \ g-k \geq 2, \geff=g+1
        \\[-\halfbls] \\
\frac{2g+1}{4}+\frac{1}{8(g-k)-4} &   \text{for} &
\cQ_1^{\mathit{hyp}}\big(2(g-k)-3,2(g-k)-3,4k+2\big),
        \\[-\halfbls] \\
&&\text{ where } k \geq 0, \
g \geq 1, \ g-k \geq 1, \ \geff=g
        \\[-\halfbls] \\
\frac{g}{2}  &   \text{for} &
\cQ_1^{\mathit{hyp}}\big(4(g-k)-6,4k+2\big),
        \\[-\halfbls] \\
&&\text{ where } k \geq 0, \ g \geq 2, \
g-k \geq 2, \ \geff=g-1
\,.
\end{array}
$$
\end{Corollary}

We   shall   need  the  following  general  Lemma  in  the  proof  of
Corollary~\ref{cor:Q:hyp:connected:comps}.

\begin{Lemma}
\label{lm:cover:factors}    Consider    a    meromorphic    quadratic
differential  $q$ with at most simple poles on a Riemann surface $C$.
We assume that $q$ is not a global square of an Abelian differential.
Suppose that for some finite (possibly  ramified) cover
$$
P:\tilde C\to C
$$
the  induced  quadratic  differential  $P^\ast  q$ on $\tilde C$ is a
global  square  of  an  Abelian  differential.  Then  the  cover  $P$
quotients through the canonical double cover $p:\hat C\to C$
$$
\begin{array}{rcl}
\tilde C\!\!\!&\xrightarrow{\ P\ }&\!\!\!C\\
&\searrow\quad\ \nearrow &\\
&\hat C&
\end{array}
$$
constructed in section~\eqref{ss:Sum:of:the:Lyapunov:exponents}.
\end{Lemma}
\begin{proof}
Let  us  puncture  $C$  at all zeroes of odd orders and at all simple
poles  of  $q$;  let  us  puncture  $\tilde  C$  and  $\hat C$ at all
preimages  of  punctures on $C$. If necessary, puncture $\tilde C$ at
all  remaining ramification points. The covers $P$ and $p$ restricted
to the resulting punctured surfaces become nonramified.

A  non ramified cover $f:X\to Z$ is defined by the image of the group
$f_\ast\pi_1(X)\subset\pi_1(Z)$.  A  cover  $f$  quotients  through a
cover  $g:Y\to  Z$  if  and only if $f_\ast\pi_1(X)$ is a subgroup of
$g_\ast\pi_1(Y)$.

Consider the flat metric defined by the quadratic differential $q$ on
$C$  punctured  at the conical singularities. Note that by definition
of  the  cover  $p:\hat  C\to  C$, the subgroup $p_\ast\pi_1(\hat C)$
coincides   with   the   kernel   of   the   corresponding   holonomy
representation $\pi_1(C)\to \Z/2\Z$.

The quadratic differential $P^\ast q$ induced on the covering surface
$\tilde  S$ by a finite cover $P:\tilde C\to C$ is a global square of
an  Abelian  differential  if and only if the holonomy of the induced
flat   metric   is   trivial,   or,  equivalently,  if  and  only  if
$P_\ast\pi_1\tilde C$ is in the kernel of the holonomy representation
$\pi_1(C)\to  \Z/2\Z$.  Thus,  the  Lemma  is  proved  for  punctured
surfaces.

It  remains  to  note  that  the ramification points of the canonical
double  cover $p:\hat C\to C$ are exactly those, where $q$ has zeroes
of  odd  degrees  and simple poles. Thus, the cover $P:\tilde C\to C$
necessarily  has  ramifications  of  even orders at all these points,
which completes the proof of the Lemma.
\end{proof}

\begin{proof}[Proof of Corollary~\ref{cor:Q:hyp:connected:comps}]
Let  $\tilde  S$  be a surface in a hyperelliptic connected component
$\cQ^{\mathit{hyp}}(m_1,\dots,m_k)$;  let  $S$ be the underlying flat
surface  in  the  corresponding  stratum  $\cQ(d_1,\dots,d_\noz)$  of
meromorphic  quadratic  differentials  with  at  most simple poles on
$\CP$.   Denote   by   $\widehat{\tilde  S}$  and  by  $\hat  S$  the
corresponding  flat  surfaces  obtained  by  the  canonical  ramified
covering         construction         described         in         in
section~\eqref{ss:Sum:of:the:Lyapunov:exponents}.

By Lemma~\ref{lm:cover:factors} the diagram
$$
\begin{CD}
\hat S  @.  \widehat{\tilde S}\\
@VVV         @VVV    \\
S @<f<< \tilde S
\end{CD}
$$
can be completed to a commutative diagram
\begin{equation}
\label{eq:covers}
\begin{CD}
\hat S  @<\hat f<<  \widehat{\tilde S}\\
@VVV         @VVV    \\
S @<f<< \tilde S\,.
\end{CD}
\end{equation}
By  construction  $\hat  f$  intertwines  the  natural involutions on
$\widehat{\tilde S}$ and on $\hat S$. Hence, we get an induced linear
map $\hat f^\ast: H^1_-(\hat S)\to H^1_-(\widehat{\tilde S})$.
Note  that  since  $S\simeq\CP$, one has $H^1_-(\hat S)=H^1(\hat S)$.
Note  also  that  a  holomorphic  differential $\hat f^\ast\omega$ in
$H^{1,0}(\widehat{\tilde  S})$  induced  from  a  nonzero holomorphic
differential $\omega\in H^{1,0}(\hat S)$ by the double cover $\hat f$
is  obviously  nonzero.  This  implies  that $\hat f^\ast: H^1_-(\hat
S)\to H^1_-(\widehat{\tilde S})$ is a monomorphism.

An  elementary  dimension  count  shows  that for the three series of
hyperelliptic             components             listed            in
Corollary~\ref{cor:Q:hyp:connected:comps},   the   effective   genera
associated  to the ``orienting'' double
covers $\hat S\to S$ and to $\widehat{\tilde S}\to
\tilde  S$  coincide.  Hence,  for these three series of
hyperelliptic components the map
$\hat  f^\ast$  is,  actually,  an isomorphism. This implies that the
Lyapunov                                                     spectrum
$\lambda^-_1>\lambda^-_2\ge\dots\ge\lambda^-_{\geff}$             for
$\cQ^{\mathit{hyp}}(m_1,\dots,m_k)$  coincides with the corresponding
spectrum for $\cQ(d_1,\dots,d_\noz)$.

The  remaining part of the proof is completely analogous to the proof
of  Corollary~\ref{cor:hyp:connected:comps}. The relation between the
orders of singularities of $\cQ^{\mathit{hyp}}(m_1,\dots,m_k)$ and of
the   underlying   stratum   $\cQ_1(d_1,\dots,d_\noz)$  is  described
in~\cite{Lanneau:hyperell}.
\end{proof}

Let  us  use  Corollary~\ref{cor:Q:hyp:connected:comps}  to study the
Lyapunov  exponents  of  the  vector  bundle  $H^1_-$  over invariant
suborbifolds  in the strata of holomorphic quadratic differentials in
small genera. We consider only those strata, $\cQ(d_1,\dots,d_\noz)$,
for  which  the quadratic differentials \textit{do not} correspond to
global squares of Abelian differentials.

Recall  that any holomorphic quadratic differential in genus one is a
global  square  of  an  Abelian  differential, so $\cQ(0)=\emptyset$.
Recall also, that in genus two the strata $\cQ(4)$ and $\cQ(3,1)$ are
empty,  see~\cite{Masur:Smillie}. The stratum $\cQ(2,2)$ in genus two
has  effective genus one, so $\lambda^-_1=1$ and there are no further
positive Lyapunov exponents of $H^1_-$.

\begin{Corollary}
\label{cor:Q:g2}
For  any  regular $\PSL$-invariant suborbifold $\cM_1$ in the stratum
$\cQ_1(2,1,1)$  of  holomorphic  quadratic differentials in genus two
the second Lyapunov exponent $\lambda^-_2$ is equal to $1/3$.

For  any  regular $\PSL$-invariant suborbifold $\cM_1$ in the stratum
$\cQ_1(1,1,1,1)$  of holomorphic quadratic differentials in genus two
the  sum  of Lyapunov exponents $\lambda^-_2+\lambda^-_3$ is equal to
$2/3$.
\end{Corollary}
\begin{proof}
Each  stratum coincides with its hyperelliptic connected component, so
we are in the situation of Corollary~\ref{cor:Q:hyp:connected:comps}.
Namely,
\begin{align*}
\cQ(2,1,1)&=\cQ^{\mathit{hyp}}\big(2(2-0)-3,2(2-0)-3,4\cdot 0+2\big)\\
\cQ(1,1,1,1)&=\cQ^{\mathit{hyp}}\big(2(2-0)-3,2(2-0)-3,2\cdot 0+1,2\cdot 0+1\big)\,.
\end{align*}
\end{proof}

In  analogy with Corollary~\ref{cor:hyp:connected:comps} we can study
the  sum  of  the  top  $\geff$ exponents $\lambda^-_i$ for a general
$\PSL$-invariant  suborbifold  in  a hyperelliptic locus of a general
stratum  of  meromorphic  quadratic differentials with at most simple
poles.  However,  in  the  most general situation we only get a lower
bound for this sum.

\begin{Corollary}
\label{cor:hyp:connected:comps:quadratic}
Suppose  that $\tilde\cM_1$ is a regular $\PSL$-invariant suborbifold
in  a  hyperelliptic  locus  of some stratum of meromorphic quadratic
differentials with at most simple poles. Denote by $\geff(\tilde\cM)$
the  effective genus of $\tilde\cM_1$ and by $(d_1,\dots,d_\noz)$ the
orders  of singularities of the underlying quadratic differentials in
the  associated  $\PSL$-invariant  suborbifold $\cM_1$ in the stratum
$\cQ_1(d_1,\dots,d_\noz)$ in genus $0$.

The  top  $\geff(\tilde\cM_1)$ Lyapunov exponents of the Hodge bundle
$H^1_-$  over  $\tilde\cM_1$ along the Teichm\"uller flow satisfy the
following relation:
\begin{equation}
\label{eq:sum:lambda:minus:hyp}
\lambda_1^-(\tilde\cM_1) + \dots + \lambda^-_{\geff(\tilde\cM)}(\tilde\cM_1)
\ \ge \
\cfrac{1}{4}\,\cdot\,\sum_{\substack{j \text{ such that}\\
d_j \text{ is odd}}}
\cfrac{1}{d_j+2}\,,
\end{equation}
where, as usual, we associate the order $d_i=-1$ to simple poles.

If
$$
2\geff(\tilde\cM_1)-2=\text{number of odd entries in }(d_1,\dots,d_\noz)\,,
$$
then the nonstrict inequality~\eqref{eq:sum:lambda:minus:hyp} becomes
an equality.
\end{Corollary}
\begin{proof}
For  a  general  ramified  double cover $\tilde S\to S\simeq\CP$ from
diagram~\eqref{eq:covers}  the effective genera $\geff(\tilde S)$ and
$\geff(S)$  associated  to the ``orienting'' double covers $\hat S\to
S$   and   $\widehat{\tilde  S}\to  \tilde  S$  might  be  different,
$\geff(\tilde  S)\ge \geff(S)$. However, as we have seen in the proof
of  Corollary~\ref{cor:Q:hyp:connected:comps},  the induced map $\hat
f^\ast:  H^1_-(\hat  S)\to  H^1_-(\widehat{\tilde  S})$  is  still  a
monomorphism,   and  $f^\ast$  is  an  isomorphism  if  and  only  if
$\geff(\tilde S)=\geff(S)$.

This implies that when we have a regular $\PSL$-invariant suborbifold
$\cM_1$ in  some  stratum  $\cQ_1(d_1,\dots,d_\noz)$  of  meromorphic
quadratic  differentials  with  at most simple poles on $\CP$, and an
induced  regular  $\PSL$-invariant  suborbifold  $\tilde\cM_1$ in the
associated    hyperelliptic   locus   of   the   associated   stratum
$\cQ_1(m_1,\dots,m_k)$,  the  Hodge  bundle  $H^1_-(\tilde\cM)$  over
$\tilde\cM$ contains a $\PSL$-invariant subbundle $f^\ast H^1_-(\cM)$
of  dimension  $2\geff(\cM)$  with  symmetric  spectrum  of  Lyapunov
exponents  along  the  Teichm\"uller  flow. Here by $f$ we denote the
natural  projection $f:\tilde\cM\to\cM$. Thus, the sum of nonnegative
Lyapunov exponents of the bundle $H^1_-(\tilde\cM_1)$ is greater than
or  equal  to  the  sum  of  nonnegative  Lyapunov  exponents  of the
subbundle  $f^\ast H^1_-(\cM)$. Since $f^\ast$ is a monomorphism, the
Lyapunov  spectrum  of  $f^\ast  H^1_-(\cM_1)$  and of $H^1_-(\cM_1)$
coincide,  and  the  latter  sum  is  equal to the sum of nonnegative
Lyapunov    exponents    of    $H_1^-(\cM_1)$,    which    is   given
by~\eqref{eq:sum:for:CP1}:
$$
\lambda^-_1(\cM_1) + \dots + \lambda^-_{\geff(\cM)}(\cM_1)
\ = \
\cfrac{1}{4}\,\cdot\,\sum_{\substack{j \text{ such that}\\
d_j \text{ is odd}}}
\cfrac{1}{d_j+2}\,.
$$

When         $\geff(\tilde\cM_1)=\geff(\cM_1)$         we         get
$H^1_-(\tilde\cM)=f^\ast      H^1_-(\cM)$     and     a     nonstrict
inequality~\eqref{eq:sum:lambda:minus:hyp}  becomes  an  equality. It
remains to apply~\eqref{eq:d:to:m} to compute the the effective genus
$\geff(\cM_1)$:
$$
2\geff(\cM)-2=2\geff(\cQ(d_1,\dots,d_\noz))-2=
\text{number of odd entries in }(d_1,\dots,d_\noz)
$$
which            completes           the           proof           of
Corollary~\ref{cor:hyp:connected:comps:quadratic}.
\end{proof}
\subsection{Positivity of several leading exponents}

\begin{Corollary}
\label{cor:g:ge:7}
For  any  regular  $\SL$-invariant  suborbifold  in in any stratum of
Abelian  differentials  in  genus  $g\ge  7$  the  Lyapunov exponents
$\lambda_2\ge\dots\ge\lambda_k$    are   strictly   positive,   where
$k=\left[\frac{(g-1)g}{6g-3}\right]+1$.

For  any regular $\SL$-invariant suborbifold in the principal stratum
$\cH_1(1\dots  1)$  of  Abelian  differentials  in genus $g\ge 5$ the
Lyapunov   exponents   $\lambda_2\ge\dots\ge\lambda_k$  are  strictly
positive, where $k=\left[\frac{g-1}{4}\right]+1$.
\end{Corollary}

Currently we do not have much information on how sharp the above
estimates are. The paper~\cite{Matheus:appendix} contains an explicit
computation showing that certain infinite family of arithmetic
Teichm\"uller curves related to cyclic covers studied
in~\cite{Matheus:Yoccoz} has approximately $g/3$ positive Lyapunov
exponents, where the genus $g$ of the corresponding square-tiled
surfaces tends to infinity. Another family of $\SL$-invariant
submanifolds (also related to cyclic covers) seem to have
approximately $g/4$ positive Lyapunov exponents, where the genus $g$
tends to infinity, see~\cite{Avila:Matheus:Yoccoz:progress}. Finally,
numerical experiments of C.~Matheus seem to indicate that for certain
rather special square-tiled surfaces constructed
in~\cite{Matheus:Yoccoz:Zmiaikou} the contribution of the
Siegel-Veech constant to the
formula~\eqref{eq:general:sum:of:exponents:for:Abelian} for the sum
of the Lyapunov exponents for the corresponding arithmetic
Teichm\"uller curve might be very small compared to the combinatorial
term.

\begin{proof}
Consider   the  formula~\eqref{eq:general:sum:of:exponents:for:Abelian}.
Since  $c_{area}>0$, and $1=\lambda_1>\lambda_2\ge\lambda_3\ge\dots$,
we    get    at    least    $k+1$    positive    Lyapunov   exponents
$\lambda_1,\dots,\lambda_k$ as soon as the expression
\begin{equation}
\label{eq:1:12}
\cfrac{1}{12}\ \sum_{i=1}^\noz \cfrac{m_i(m_i+2)}{m_i+1}
\end{equation}
is  greater  than  or  equal to $k$, where $k$ is a strictly positive
integer.  (Here  the  strict  inequality $\lambda_1>\lambda_2$ is the
result  of  Forni~\cite{Forni:Deviation}.) It remains to evaluate the
minimum  of  expression~\eqref{eq:1:12} over all partitions of $2g-2$
and notice that it is achieved on the ``smallest'' partition $(2g-2)$
composed    of   a   single   element.   For   this   partition   the
sum~\eqref{eq:1:12} equals
$$
\frac{1}{12}\left(2g-1-\frac{1}{2g-1}\right)=\cfrac{(g-1)g}{6g-3}
\,.
$$
This proves the first part of the statement.

The  consideration for the principal stratum is completely analogous,
except that this time the above sum equals $(g-1)/4$.
\end{proof}

\begin{Problem}
\label{pb:degenerate:abelian}  Are  there  any  examples  of  regular
$\SL$-invariant   suborbifolds  $\cM_1$  in  the  strata  of  Abelian
differentials  in  genera  $g\ge 2$ different from the two arithmetic
Teichm\"uller  curves  found by G.~Forni in~\cite{Forni:Handbook} and
by        G.~Forni        and        C.~Matheus~\cite{Forni:Matheus},
\cite{Forni:Matheus:Zorich}  with  a  completely  degenerate Lyapunov
spectrum $\lambda_2=\dots=\lambda_g=0$?
\end{Problem}

By  Corollary~\ref{cor:g:ge:7}  such  example  might  exist  only  in
certain  strata  in  genera  from  $3$  to  $6$.
After completion of work on this paper, it was proved
  by D.~Aulicino \cite{Aulicino:affine} that any such an example must
  be a Teichm\"uller curve. By  the  result  of
M.~M\"oller~\cite{Moeller:Shimura:curves},  Teichm\"uller curves with
such a property might exist only in several strata in genus five.

\begin{Corollary}
\label{cor:qd:positivity}
For  any  regular  $\PSL$-invariant  suborbifold  in  any  stratum of
holomorphic  quadratic  differentials  in genus $g\ge 7$ the Lyapunov
exponents   $\lambda^+_2\ge\dots\ge\lambda^+_k$   and   the  Lyapunov
exponents  $\lambda^-_2\ge\dots\ge\lambda^-_k$ are strictly positive,
where $k=\left[\frac{(g-1)g}{6g+3}\right]+1$.

For any regular $\PSL$-invariant suborbifold in the principal stratum
of holomorphic quadratic differentials in genus $g\ge 5$ the Lyapunov
exponents  $\lambda^+_2\ge\dots\ge\lambda^+_k$ are strictly positive,
where $k=\left[\frac{5(g-1)}{18}\right]+1$.

For any regular $\PSL$-invariant suborbifold in the principal stratum
of  holomorphic  quadratic  differentials in genus $g=2$ the Lyapunov
exponent   $\lambda^-_2$   is  strictly  positive.  For  any  regular
$\PSL$-invariant  suborbifold in the principal stratum of holomorphic
quadratic  differentials  in  genus  $g\ge  3$ the Lyapunov exponents
$\lambda^-_2\ge\dots\ge\lambda^-_l$   are  strictly  positive,  where
$l=\left[\frac{11(g-1)}{18}\right]+1$.
\end{Corollary}
\begin{proof}
This                   time                   we                  use
formulae~\eqref{eq:general:sum:of:plus:exponents:for:quadratic}
and~\eqref{eq:general:index:of:exponents:for:quadratic}.   Note  that
since   the   quadratic   differentials   under   consideration   are
\textit{holomorphic}, we have $d_j\ge 1$ for any $j$. Note also, that
it  follows  from  the  result  of  Forni~\cite{Forni:Deviation} that
$\lambda_1^->\lambda_2^-$    and    that   $\lambda_1^->\lambda_1^+$.
Finally, by elementary geometric reasons one has $\lambda_1^-=1$. For
genus  two  we use Corollary~\ref{cor:Q:g2}. The rest of the proof is
completely analogous to the proof of Corollary~\ref{cor:g:ge:7}.
\end{proof}

\begin{Problem}
\label{pr:degenerate:H:minus}
Are  there  any  examples  of  regular  $\PSL$-invariant suborbifolds
$\cM_1$  in  the  strata  of  meromorphic  quadratic differentials in
genera  $g_{\mathit{eff}}\ge  2$  different  from  the  Teichm\"uller
curves       of      square-tiled      cyclic      covers      listed
in~\cite{Forni:Matheus:Zorich}  having completely degenerate Lyapunov
spectrum   $\lambda^-_2=\dots=\lambda^-_{\geff}=0$   for  the  bundle
$H^1_-$?
\end{Problem}

Note  that  under  the  additional restriction that the corresponding
quadratic           differentials           are           holomorphic
Corollary~\ref{cor:qd:positivity}   limits   the  genus  of  possible
examples  for Problem~\ref{pr:degenerate:H:minus} to several possible
values only.

When the work on this paper was completed, C.~Matheus indicated to us
that  the formula~\eqref{eq:general:index:of:exponents:for:quadratic}
implies  a  strong restriction on the strata of meromorphic quadratic
differentials   which   might  \textit{a  priori}  contain  invariant
submanifolds with completely degenerate $\lambda^-$-spectrum. Namely,
since                    the                    $\lambda^+$-exponents
in~\eqref{eq:general:index:of:exponents:for:quadratic}            are
nonnegative,   the   $\lambda^-$-spectrum   may   not  be  completely
degenerate  as  soon  as  the  ambient  stratum $Q(d_1,\dots,d_\noz)$
satisfies
$$
\sum_{\substack{j \text{ such that}\\
d_j \text{ is odd}}}
\cfrac{1}{d_j+2}>4\,,
$$
say, when quadratic differentials contain at least four poles,
and the stratum is different from $\cQ(-1^4)$.

\begin{Problem}
\label{pr:degenerate:H:plus}
Are  there  any  examples  of  regular  $\PSL$-invariant suborbifolds
$\cM_1$  in  the  strata  of  meromorphic  quadratic differentials in
genera   $g\ge   2$   different  from  the  Teichm\"uller  curves  of
square-tiled   cyclic  covers  listed  in~\cite{Forni:Matheus:Zorich}
having       completely       degenerate       Lyapunov      spectrum
$\lambda^+_1=\dots=\lambda^+_g=0$ for the bundle $H^1_+$?
\end{Problem}

Note                                                             that
formula~\eqref{eq:general:sum:of:plus:exponents:for:quadratic}
implies   that   Problem~\ref{pr:degenerate:H:plus}  does  not  admit
solutions  for  the  $\PSL$-invariant  suborbifolds  in the strata of
\textit{holomorphic} quadratic differentials.

After  completion  of the work on this paper J.~Grivaux and P.~Hubert
found    a    geometric    reason    for   the   vanishing   of   all
$\lambda^+$-exponents  in  examples  from~\cite{Forni:Matheus:Zorich}
and  constructed  further  examples  of the same type with completely
degenerate                                      $\lambda^+$-spectrum,
see~\cite{Grivaux:Hubert:in:progress}.  We  do not know whether their
construction    covers    all    possible    situations    when   the
$\lambda^+$-spectrum is completely degenerate.

\subsection{Siegel--Veech  constants:  values  for  certain  invariant
suborbifolds}
\label{ss:introduction:Siegel:Veech:values}

We  compute  numerical  values of the Siegel--Veech constant for some
specific       regular      $\SL$-invariant      suborbifolds      in
section~\ref{sec:Evaluation:of:SV:constants}. We consider the largest
possible  and  the  smallest  possible  cases,  namely,  we  consider
connected  components  of  the  strata  and  Teichm\"uller  discs  of
arithmetic  Veech  surfaces.  In the current section we formulate the
corresponding    statements;    the    proofs    are   presented   in
section~\ref{sec:Evaluation:of:SV:constants}.

\subsubsection{Arithmetic Teichm\"uller discs}

Consider  a  connected  square-tiled  surface  $S$ in some stratum of
Abelian  or  quadratic  differentials. For every square-tiled surface
$S_i$  in  its  $\SLZ$-orbit (correspondingly $\PSLZ$-orbit) consider
the decomposition of $S_i$ into maximal cylinders $\mathit{cyl}_{ij}$
filled  with  closed  regular horizontal geodesics. For each cylinder
$\mathit{cyl}_{ij}$  let  $w_{ij}$ be the length of the corresponding
closed  horizontal  geodesic  and  let  $h_{ij}$ be the height of the
cylinder     $\mathit{cyl}_{ij}$.     Let     $\card(\SLZ\cdot    S)$
(correspondingly  $\card(\PSLZ\cdot  S)$)  be  the cardinality of the
orbit.

\begin{Theorem}
\label{theorem:SVconstant:for:square:tiled}
For   any   connected   square-tiled   surface   $S$   in  a  stratum
$\cH(m_1,\dots,m_\noz)$  of  Abelian differentials, the Siegel--Veech
constant $c_{\mathit{area}}(\cM_1)$ of the $\SL$-orbit $\cM_1$ of the
normalized   surface   $S_{(1)}\in\cH_1(m_1,\dots,m_\noz)$   has  the
following value:
\begin{equation}
\label{eq:SVconstant:for:square:tiled}
c_{\mathit{area}}(\cM_1)=
\cfrac{3}{\pi^2}\cdot
\cfrac{1}{\card(\SLZ\cdot S)}\
\sum_{S_i\in\SLZ\cdot S}\ \
\sum_{\substack{
\mathit{horizontal}\\
\mathit{cylinders\ cyl}_{ij}\\
such\ that\\S_i=\sqcup\mathit{cyl}_{ij}}}\
\cfrac{h_{ij}}{w_{ij}}\,,
\end{equation}
For  a square-tiled surface $S$ in a stratum of meromorphic quadratic
differentials  with  at  most  simple  poles the analogous formula is
obtained by replacing $\SLZ$ with $\PSLZ$.
\end{Theorem}

Theorem~\ref{theorem:SVconstant:for:square:tiled}     is     proved     in
section~\ref{sec:Evaluation:of:SV:constants}.

\begin{Corollary}
\label{cr:sum:for:arithmetic:T:disc}
a)  Let  $\cM_1$  be  an  arithmetic  Teichm\"uller disc defined by a
square-tiled   surface   $S_0$   of   genus   $g$   in  some  stratum
$\cH_1(m_1,\dots,m_\noz)$  of  Abelian  differentials.  The  top  $g$
Lyapunov  exponents  of  the  of  the Hodge bundle $H^1$ over $\cM_1$
along the Teichm\"uller flow satisfy the following relation:
\begin{multline}
\label{eq:sum:for:arithmetic:T:disc}
\lambda_1+\dots+\lambda_g
\ =\\= \
\cfrac{1}{12}\cdot\sum_{i=1}^\noz\cfrac{m_i(m_i+2)}{m_i+1}
\ +\
\cfrac{1}{\card(\SLZ\cdot S_0)}\
\sum_{S_i\in\SLZ\cdot S_0}\ \
\sum_{\substack{
\mathit{horizontal}\\
\mathit{cylinders\ cyl}_{ij}\\
such\ that\\S_i=\sqcup\mathit{cyl}_{ij}}}\
\cfrac{h_{ij}}{w_{ij}}\ .
\end{multline}

b)  Let  $\cM_1$  be  an  arithmetic  Teichm\"uller disc defined by a
square-tiled   surface   $S_0$   of   genus   $g$   in  some  stratum
$\cQ_1(d_1,\dots,d_\noz)$ of meromorphic quadratic differentials with
at  most  simple  poles. The top $g$ Lyapunov exponents of the of the
Hodge  bundle  $H^1_+$  over  $\cM_1$  along  the  Teichm\"uller flow
satisfy the following relation:
\begin{multline}
\label{eq:sum:for:arithmetic:T:disc:q}
\lambda^+_1+\dots+\lambda^+_g
\ =\\= \
\cfrac{1}{24}\cdot\sum_{i=1}^\noz\cfrac{d_i(d_i+4)}{d_i+2}
\ +\
\cfrac{1}{\card(\PSLZ\cdot S_0)}\
\sum_{S_i\in\PSLZ\cdot S_0}\ \
\sum_{\substack{
\mathit{horizontal}\\
\mathit{cylinders\ cyl}_{ij}\\
such\ that\\S_i=\sqcup\mathit{cyl}_{ij}}}\
\cfrac{h_{ij}}{w_{ij}}\ .
\end{multline}
\end{Corollary}
\begin{NNRemark}
Combining     equation~\eqref{eq:sum:for:arithmetic:T:disc:q}    from
statement      b)      of      the      Corollary      above     with
equation~\eqref{eq:general:index:of:exponents:for:quadratic}     from
Theorem~\ref{eq:general:sum:of:plus:exponents:for:quadratic}       we
immediately  obtain  a  formula for the sum of the Lyapunov exponents
$\lambda^-_1+\dots+\lambda^-_{\geff}$     of     the    corresponding
Teichm\"uller disc.
\end{NNRemark}

\begin{figure}[htb]
   %
   %
\includegraphics{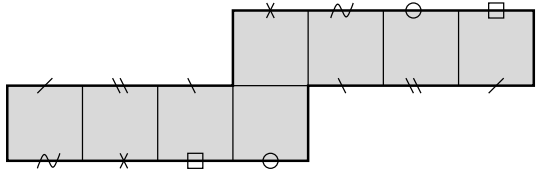}
\vspace{70pt}
\caption{
\label{fig:Eierlegende:Wollmilchsau}
Eierlegende Wollmilchsau}
\end{figure}

To  illustrate  how  the  above  statement  works,  let us consider a
concrete    example.    The   following   square-tiled   surface   is
$\SLZ$-invariant.  It belongs to the principal stratum $\cH(1,1,1,1)$
in genus $g=3$.

Hence,  the  sum  of  the  Lyapunov  exponents  for the corresponding
Teichm\"uller disc equals
$$
1+\lambda_2+\lambda_3\ =\
\cfrac{1}{12}\cdot\sum_{i=1}^4\cdot\cfrac{1(1+2)}{1+1}\ +\
\cfrac{1}{1}\left(\cfrac{1}{4}+\cfrac{1}{4}\right)\ =\
\cfrac{1}{2}+\cfrac{1}{2}\ =\ 1
$$
This  implies  that  $\lambda_2=\lambda_3=0$.  (This result was first
proved   by  G.~Forni  in~\cite{Forni:Handbook},  who  used  symmetry
arguments.   See  also  Problem~\ref{pb:degenerate:abelian}  and  the
discussion after it.)

\subsubsection{Connected components of the strata}

Let us come back to generic flat surfaces $S$ in the strata. Consider
a  maximal  cylinder  $\mathit{cyl}_1$  in a flat surface $S$. Such a
cylinder is filled with parallel closed regular geodesics. Denote one
of  these geodesics by $\gamma_1$. Sometimes it is possible to find a
regular  closed  geodesic  $\gamma_2$  on $S$ parallel to $\gamma_1$,
having  the  same  length  as  $\gamma_1$,  but living outside of the
cylinder  $\mathit{cyl}_1$. It is proved in~\cite{Eskin:Masur:Zorich}
that   for  almost  any  flat  surface  in  any  stratum  of  Abelian
differentials   this   implies   that  $\gamma_2$  is  homologous  to
$\gamma_1$.  Consider  a maximal cylinder $\mathit{cyl}_2$ containing
$\gamma_2$   filled   with   closed  regular  geodesics  parallel  to
$\gamma_2$.  Now  look  for  closed  regular  geodesics  parallel  to
$\gamma_1$ and to $\gamma_2$ and having the same length as $\gamma_1$
and   $\gamma_2$   but  located  outside  of  the  maximal  cylinders
$\mathit{cyl}_1$  and  $\mathit{cyl}_2$,  etc.  The resulting maximal
decomposition  of  the  surface is encoded by a \textit{configuration
$\cC$      of      homologous      closed      regular     geodesics}
(see~\cite{Eskin:Masur:Zorich} for details).

One  can consider a counting problem for any individual configuration
$\cC$.   Denote  by  $N_{\cC}(S,L)$  the  number  of  collections  of
homologous  saddle  connections  on $S$ of length at most $L$ forming
the given configuration $\cC$. By the general results of A.~Eskin and
H.~Masur~\cite{Eskin:Masur}    almost    all    flat    surfaces   in
$\cH_1^{comp}(m_1,\dots,m_\noz)$ share the same quadratic asymptotics
\begin{equation}
\label{eq:c:configuration:definition}
\lim_{L\to\infty} \cfrac{N_\cC(S,L)}{L^2}=c_\cC
\end{equation}
where the  \textit{Siegel---Veech constant} $c_\cC$ depends only on the
chosen connected component of the stratum.

\begin{NNTheorem}[Vorobets]
For  any  connected component of any stratum of Abelian differentials
the  Siegel--Veech  constants  $c_{\mathit{area}}$  and  $c_\cC$ are
related as follows:
\begin{equation}
\label{eq:c:area:in:terms:of:c:configurations:Abelian}
c_{\mathit{area}}=
\cfrac{1}{\dim_\C{}\cH(m_1,\dots,m_n)-1}\cdot\sum_{q=1}^{g-1} q\cdot
\sum_{\substack{\textit{Configurations}\ \cC\\
\textit{containing exactly}\\
q\textit{ cylinders}}}
c_\cC\,.
\end{equation}
\end{NNTheorem}

The  above  Theorem  is  proved  in~\cite{Vorobets}.  As an immediate
corollary  of  Theorem~\ref{theorem:general:Abelian}  and  the  above
theorem we get the following statement:

\begin{Theorem1prime}
   %
For  any  connected  component of any stratum $\cH(m_1,\dots,m_n)$ of
Abelian  differentials  the  sum  of  the  top $g$ Lyapunov exponents
induced  by  the  Teichm\"uller  flow  on  the  Hodge  vector  bundle
$H^1_{\R{}}$ satisfies the following relation:
\begin{multline}
\label{eq:sum:for:Abelian:strata}
\lambda_1+\dots+\lambda_g =\\
=\ \cfrac{1}{12}\cdot\sum_{i=1}^\noz\cfrac{m_i(m_i+2)}{m_i+1}
\ +\
\cfrac{\pi^2}{3\dim_\C{}\cH(m_1,\dots,m_n)-3}\cdot\sum_{q=1}^{g-1}
q\cdot
\sum_{\substack{\textit{Admissible}\\\textit{configurations}\ \cC\\
\textit{containing exactly}\\
q\textit{ cylinders}}}
c_\cC
\end{multline}
where  $c_\cC$  are  the  Siegel--Veech constant of the corresponding
connected component of the stratum $\cH(m_1,\dots,m_n)$.
\end{Theorem1prime}

The Siegel--Veech       constants       $c_\cC$       were       computed
in~\cite{Eskin:Masur:Zorich}.  Here  we  present  an  outline  of the
corresponding formulae.

A  ``configuration''  $\cC$  can  be viewed as a combinatorial way to
represent  a  flat  surface  as  a collection of $q$ flat surfaces of
smaller  genera joined cyclically by narrow flat cylinders. Thus, the
configuration  represented  schematically  on  the  right  picture in
Figure~\ref{fig:dance:ugly:drawing}    is   admissible,   while   the
configuration on the left picture is not.

\begin{figure}[ht]
\includegraphics{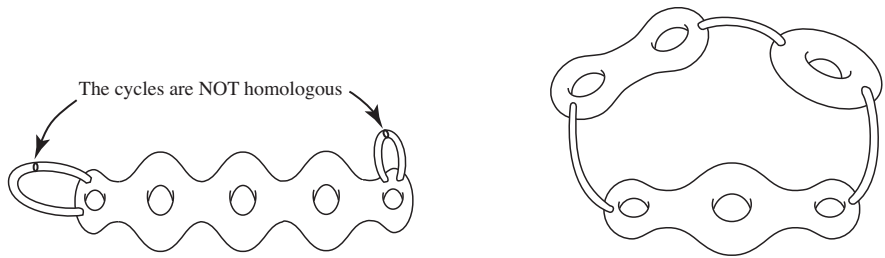}
\vspace{100bp}
\caption{
\label{fig:dance:ugly:drawing}
Topological pictures for admissible (on the right) and non-admissible
(on the left) configurations of cylinders. }
\end{figure}

Denote  by  $\cH^\epsilon_1(\cC)$  the subset of flat surfaces in the
stratum $\cH(m_1,\dots,m_\noz)$ having a maximal collection of narrow
cylinders  of width at most $\epsilon$ forming a configuration $\cC$.
Here ``maximal'' means that the narrow cylinders in the configuration
$\cC$ do not make part of a larger configuration $\cC'$.

Contracting the waist curves of the cylinders completely and removing
them  we  get a collection of disjoint closed flat surfaces of genera
$g_1,  \dots,  g_q$.  By  construction $g_1+\dots+g_q=g-1$. Denote by
$\cH^{comp}(\beta'_j)$  the  ambient  stratum  (more  precisely,  its
connected  component)  for  the  resulting  flat  surfaces. Denote by
$\cH^{comp}_1(\beta)$   the  ambient  stratum  (more  precisely,  its
connected    component)    for   the   initial   surface.   According
to~\cite{Eskin:Masur:Zorich} the Siegel--Veech constant $\SVc$ can be
expressed as

\begin{multline}
\label{eq:SV:constant:individual:configuration}
\SVc
\ =\
\lim_{\varepsilon\to 0} \cfrac{1}{\pi\varepsilon^2}\,
\cfrac{\Vol\cH^{\varepsilon}_1(\cC)}
{\Vol\cH^{comp}_1(d_1,\dots,d_\noz)}
\ =\\=\
(\text{explicit combinatorial factor})\cdot
\cfrac{\prod_{j=1}^k\Vol\cH_1(\beta'_k)}
{\Vol\cH^{comp}_1(\beta)}\ .
\end{multline}

Thus,  the Theorem above allows to compute the exact numerical values
of $c_{\textit{area}}$ for all connected components of all strata (at
least  in  small genera, where we know numerical values of volumes of
connected components of the strata). The resulting explicit numerical
values of the sums of Lyapunov exponents for all strata in low genera
are presented in Appendix~\ref{a:Values:of:exponents}.

By the results of A.~Eskin and A.~Okounkov~\cite{Eskin:Okounkov}, the
volume of any connected component of any stratum of Abelian
differentials   is   a   rational   multiple   of  $\pi^{2g}$.  Thus,
relations~\eqref{eq:sum:for:Abelian:strata}
and~\eqref{eq:SV:constant:individual:configuration} imply rationality
of  the  sum of Lyapunov exponents for any connected component of any
stratum of Abelian differentials.

\section{Outline of proofs}
\label{sec:Outline:of:proofs}

To  simplify  the  exposition of the proof, we have isolated its most
technical  fragments.  In  the  current  section  we present complete
proofs   of  all  statements  of  section~\ref{sec:sums},  which  are
however,                           based                           on
Theorems~\ref{theorem:main:local:formula}--\ref{theorem:int:Dflat:equals:SV:const}
stated   below.   These   Theorems   will  be  proved  separately  in
corresponding
sections~\ref{sec:analytic:Riemann:Roch:proof}\,--\,\ref{sec:cutoff}.

In  section~\ref{sec:Evaluation:of:SV:constants}  we describe in more
detail  the  Siegel--Veech  constant  $c_{area}$;  in  particular  we
explicitly  evaluate  it  for  arithmetic  Teichm\"uller discs, thus,
proving Theorem~\ref{theorem:SVconstant:for:square:tiled}.

In  Appendix~\ref{a:Values:of:exponents}  we present the exact values
of  the  sums  of  the Lyapunov exponents and conjectural approximate
values  of  individual Lyapunov exponents for connected components of
the   strata   of   Abelian   differentials   in   small  genera.  In
Appendix~\ref{a:pairs:of:permutations}   we  present  an  alternative
combinatorial   approach   to   square-tiled   surfaces  and  to  the
construction of the corresponding arithmetic Teichm\"uller curves. We
apply  it  to  discuss the non-varying phenomenon of their Siegel--Veech
constants in the strata of small genera.


\subsection{Teichm\"uller discs.}
\label{ss:Teichmuller:discs}

We                   have                   seen                   in
section~\ref{ss:Volume:element:and:action:of:the:linear:group}   that
each     ``unit     hyperboloid''     $\cH_1(m_1,\dots,m_\noz)$    and
$\cQ_1(d_1,\dots,d_\noz)$  is  foliated  by  the  orbits of the group
$\SL$  and  $\PSL$ correspondingly. Recall that the quotient of these
groups by the subgroups of rotations is canonically isomorphic to the
hyperbolic plane:
$$
\SL/\SO\simeq\PSL/\PSO\simeq\Hyp\,.
$$
Thus,     the    projectivizations    $\PcH(m_1,\dots,m_\noz)$    and
$\PcQ(d_1,\dots,d_\noz)$  are foliated by hyperbolic discs $\Hyp$. In
other words, every $\SL$-orbit in $\cH(m_1,\dots,m_\noz)$ descends to
a commutative diagram
$$
\begin{CD}
\SL  @>>>  \cH(m_1,\dots,m_\noz)\\
@VVV         @VVV    \\
\SL/\SO \simeq\Hyp @>>> \PcH(m_1,\dots,m_\noz)\,,
\end{CD}
$$
and  similarly,  every  $\PSL$-orbit  in  the  stratum  of  quadratic
differentials descends to a commutative diagram
$$
\begin{CD}
\PSL  @>>>  \cQ(m_1,\dots,m_\noz)\\
@VVV         @VVV    \\
\PSL/\PSO \simeq\Hyp @>>> \PcQ(m_1,\dots,m_\noz)\,.
\end{CD}
$$
The composition of each of the immersions
$$
\Hyp\subset\PcH(m_1,\dots,m_\noz)\quad\text{and}\quad
\Hyp\subset\PcQ(d_1,\dots,d_\noz)
$$
with the projections to the moduli space of curves $\cM_g$ defines an
immersion $\Hyp\subset\cM_g$. The latter immersion is an isometry for
the   hyperbolic   metric   of  curvature  $-1$  on  $\Hyp$  and  the
Teichm\"uller  metric  on $\cM_g$. The images of hyperbolic planes in
$\cM_g$  are  also  called Teichm\"uller discs. Following C.~McMullen
one  can  consider  them  as  ``\textit{complex geodesics}'' in the
Teichm\"uller  metric.  The  images of the diagonal subgroup in $\SL$
are  represented  by  geodesic  lines  in the hyperbolic plane; their
projections  to the Teichm\"uller discs in $\cM_g$ might be viewed as
geodesics in the Teichm\"uller metric.

It  would  be  convenient  to  consider  throughout  this  paper  the
hyperbolic  metric  of  constant curvature $-4$ on $\Hyp$. Under this
choice  of  the  curvature,  the  parameter  $t$ of the one-parameter
subgroup represented by the matrices
$$
G_t=
\begin{pmatrix}
e^t&0\\
0&e^{-t}
\end{pmatrix}
$$
corresponds  to  the natural parameter of geodesics on the hyperbolic
plane  $\Hyp$.  In  the standard coordinate $\zeta=x+iy$ on the upper
half-plane  model  of  the  hyperbolic  plane  $y>0$,  the  metric of
constant curvature $-4$ has the form
$$
\ghyp=\frac{|d\zeta|^2}{4\Im^2\zeta}=\frac{dx^2+dy^2}{4y^2}\,.
$$
The  Laplacian of this metric in coordinate $\zeta=x+iy$ has the form
$$
\Dhyp=
16\Im^2\zeta\,\frac{\partial^2}{\partial\zeta\partial\bar\zeta}
=4y^2\left(\frac{\partial^2}{\partial x^2}
+\frac{\partial^2}{\partial y^2}\right)
$$

In  the  Poincar\'e  model  of  the  hyperbolic  plane,  $|w|<1$, the
hyperbolic metric of constant curvature $-4$ has the form
$$
\ghyp=\frac{|dw|^2}{\left(1-|w|^2\right)^2}
$$
In  the  next  section  we  will  also  use  polar  coordinates  $w=r
e^{i\theta}$ in the Poincar\'e model of the hyperbolic plane. Here
\begin{equation}
\label{eq:hyp:polar:coordinates}
r=\tanh t\,,
\end{equation}
where  $t$ is the distance from the point to the origin in the metric
of   curvature  $-4$.  The  coordinates  $t,\theta$  will  be  called
hyperbolic polar coordinates.

\begin{figure}[htb]
%
\centering
\includegraphics{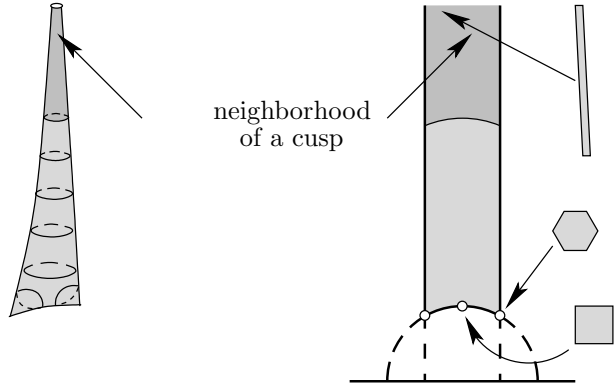}
\begin{picture}(0,0)(0,0)
\put(-50,-40){neighborhood}
\put(-50,-50){\ \ \ of a cusp}
\end{picture}
\vspace{140bp}
\caption{
\label{fig:space:of:flat:tori}
Space of flat tori
}
\end{figure}

\begin{Example}
\label{ex:M1}
The  moduli space $\cM_1$ of curves of genus one is isomorphic to the
projectivized  space  of  flat tori $\PcH(0)$; it is represented by a
single Teichm\"uller disc
\begin{equation}
\label{eq:modular:surface}
\begin{array}{rcl}
&\big\backslash\, \SL\, \big/&\\
[-\halfbls]
\SO\!\!\!\!\!&&\!\!\!\!\!\SLZ
\end{array}
\quad
=
\quad
\begin{array}{rl}
\Hyp\big/&\\
[-\halfbls]  &\!\!\!\!\!  \SLZ
\end{array}
\end{equation}
(see Figure~\ref{fig:space:of:flat:tori}).

Geometrically  one can interpret the local coordinate $\zeta$ on this
Teichm\"uller  disc  as  follows. Consider a pair $(C,\omega)$, where
$C$  is a Riemann surface of genus one, and $\omega$ is a holomorphic
one-form  on  it.  By  convention $C$ is endowed with a marked point.
Choose  the  shortest  flat  geodesic  $\gamma_1$ passing through the
marked  point  and  the  next  after  the  shortest, $\gamma_2$, also
passing  through  the  marked  point.  Under an appropriate choice of
orientation   of   the  geodesics  $\gamma_1$  and  $\gamma_2$,  they
represent   a   pair   of   independent   integer  cycles  such  that
$\gamma_1\circ  \gamma_2=1$.  Consider  the  corresponding periods of
$\omega$,
$$
A:=\int_{\gamma_1} \omega\qquad B:=\int_{\gamma_2} \omega\,.
$$
It  is  easy  to  see  that  the  canonical coordinate $\zeta$ on the
modular  surface~\eqref{eq:modular:surface}  can  be  represented  in
terms of the periods $A$ and $B$ as:
$$
\zeta=\frac{B}{A}\,.
$$

\end{Example}

\subsection{Lyapunov  exponents  and curvature  of  the  determinant
bundle.}
\label{sec:sum:as:mean:curvature}

The      following      observation      of     \mbox{M.~Kontsevich},
see~\cite{Kontsevich},  might  be considered as the starting point of
the  entire construction. Consider a flat surface $S$ in some stratum
$\cH_1(m_1,\dots,m_\noz)$  of  Abelian  differentials  and consider a
Teichm\"uller  disc  passing  through the projection of the ``point''
$S$      to      the      corresponding     projectivized     stratum
$\PcH(m_1,\dots,m_\noz)$.  Recall  that  any  Teichm\"uller  disc  is
endowed  with a canonical hyperbolic metric. Take a circle of a small
radius $\epsilon$ in the Teichm\"uller disc centered at $S$. Consider
a  Lagrangian  subspace  of  the fiber $H^1(S,\R{})$ of the the Hodge
bundle  over $S$ and a basis $v_1,\dots, v_g$ in it. Apply a parallel
transport  of  the  vectors  $v_1,\dots,  v_g$  to every point of the
circle.  The  vectors  do  not  change,  but  their  Hodge norm does.
Evaluate   an   average   of   the   logarithm   of  the  Hodge  norm
$\|v_1\wedge\dots\wedge   v_g\|_{g_\epsilon  r_\theta  S}$  over  the
circle  and  subtract the Hodge norm $\|v_1\wedge\dots\wedge v_g\|_S$
at the initial point.
The  starting observation in~\cite{Kontsevich} claims that the result
does  not depend on the choice of the basis $v_1,\dots, v_g$, and not
even  on  the  Lagrangian  subspace $L$ but only on the initial point
$S$. For the sake of completeness, we present the arguments here.

We start with a convenient  expression  for  the  Hodge  norm  of  a
polyvector  $v_1\wedge\dots\wedge v_g$ spanning a Lagrangian subspace
in $H^1(S,\R{})$. Note that the vector space $H^1(S,\R{})$ is endowed
with  a  canonical  integer  lattice  $H^1(S,\Z{})$,  which defines a
canonical  linear  volume element on $H^1(S,\R{})$: the volume of the
fundamental domain of the integer lattice with respect to this volume
element is equal to one. In other words,
we have a map
\begin{displaymath}
\Omega: \Lambda^{2g} H^1(S, \R{}) \to \R{}/\pm
\end{displaymath}
given by
\begin{displaymath}
\Omega(\lambda) = \lambda(c_1, \dots, c_{2g}),
\end{displaymath}
where $\lambda\in\Lambda^{2g} H^1(S, \R{})$, and
$\{c_1, \dots, c_{2g}\}$ is any $\Z$-basis for $H_1(S, \Z{})$.
This map naturally extends to a linear map:
$$
\Omega:\Lambda^{2g}H^1(S,\C{})\to \C{}/\pm\,.
$$
Let  $L=v_1\wedge\dots\wedge v_g$, where vectors $v_1,\dots,v_g$ span
a Lagrangian subspace in $H^1(S,\R{})$.
Let $\omega_1,\dots,\omega_g$ form a basis
in $H^{1,0}(S)$. We define
\begin{equation}
\label{eq:def:norm:L}
\|L\|^2:=
\frac{
|\Omega(v_1\wedge\dots\wedge v_g\wedge\omega_1\wedge\dots\wedge\omega_g)|\cdot
|\Omega(v_1\wedge\dots\wedge v_g\wedge\bar\omega_1\wedge\dots\wedge\bar\omega_g)|}
{|\Omega(\omega_1\wedge\dots\wedge\omega_g\wedge
\bar\omega_1\wedge\dots\wedge\bar\omega_g)|}\,.
\end{equation}
For  vectors $v_1,\dots,v_g$ spanning a Lagrangian subspace, the norm
defined    above    coincides    with    the   Hodge   norm   as   in
section~\ref{sec:Hodge:bundle}    and    is    thus    non-degenerate
(see~\cite{Hubert:Bourbaki} where this important issue is clarified).
Clearly,  this definition does not depend on a choice of the basis in
$H^{1,0}(S)$. Note that
$$
\Omega(\omega_1\wedge\dots\wedge\omega_g\wedge
\bar\omega_1\wedge\dots\wedge\bar\omega_g)=
\det\langle\omega_i,\omega_j\rangle\ ,
$$
where
\begin{equation}
\label{eq:w:i:w:j}
\langle\omega_i,\omega_j\rangle:=
\left(\begin{array}{ccc}
\langle\omega_1,\omega_1\rangle & \dots & \langle\omega_1,\omega_g\rangle\\
\dots&\dots&\dots\\
\langle\omega_g,\omega_1\rangle & \dots & \langle\omega_g,\omega_g\rangle
\end{array}
\right)
\end{equation}
is      the      matrix      of     pairwise     Hermitian     scalar
products~\eqref{eq:hodge:hermitian:form}  of elements of the basis in
$H^{1,0}(S)$.

\begin{Proposition}(\cite{Kontsevich})
\label{prop:Dhyp:L:equals:Lambda}
For any flat surface $S$, any $L = v_1 \wedge \dots
  v_g$, where the vectors $v_1, \dots v_g$ span a Lagrangian subspace
  of $H^1(S,   \R{})$,  and  for any basis $\{\omega_k\}$ of
local  holomorphic sections of the Hodge vector bundle $H^{1,0}$ over
the ambient stratum, the following identity holds:
$$
\Dhyp\log\|L\|=
-\cfrac{1}{2}\,\Dhyp\log|\det\langle\omega_i,\omega_j\rangle|
$$
where $\Dhyp$ is the hyperbolic Laplacian along the Teichm\"uller
disc.
\end{Proposition}
\begin{proof}
Applying       the       hyperbolic       Laplacian       to      the
expression~\eqref{eq:def:norm:L} we get
\begin{multline*}
\Dhyp\log\|L\|=
\cfrac{1}{2}\,\Dhyp\log\|L\|^2=
\cfrac{1}{2}\Big(\Dhyp\log
|\Omega(v_1\wedge\dots\wedge v_g\wedge\omega_1\wedge\dots\wedge\omega_g)|
\ +\\+\
\Dhyp\log
|\Omega(v_1\wedge\dots\wedge v_g\wedge\bar\omega_1\wedge\dots\wedge\bar\omega_g)|
\ -\
\Dhyp \log
|\det\langle\omega_i,\omega_j\rangle|\Big)
\end{multline*}
Note that $v_1,\dots,v_g$ do not change along the Teichm\"uller disc,
so the function
$
\Omega(v_1\wedge\dots\wedge v_g\wedge\omega_1\wedge\dots\wedge\omega_g)
$
is   a   holomorphic  function  of  the  deformation  parameter,  and
$
\Omega(v_1\wedge\dots\wedge v_g\wedge\bar\omega_1\wedge\dots\wedge\bar\omega_g)
$
is  an  antiholomorphic  one.  Hence both functions are harmonic. The
Lemma is proved.
\end{proof}

Denote
\begin{equation}
\label{eq:lambda}
\Lambda(S):=
-\cfrac{1}{4}\,\Dhyp\log|\det\langle\omega_i,\omega_j\rangle|\,,
\end{equation}
where  $\Dhyp$  is  the  hyperbolic Laplacian along the Teichm\"uller
disc in the metric of constant negative curvature $-4$.

\begin{NNRemark}
Note  that  one fourth of the hyperbolic Laplacian in curvature $-4$,
as   in   definition~\eqref{eq:lambda},   coincides  with  the  plain
hyperbolic Laplacian in curvature $-1$.
\end{NNRemark}

The  function  $\Lambda(S)$ is initially defined on the projectivized
strata    $\PcH(m_1,\dots,m_\noz)$    and   $\PcQ(d_1,\dots,d_\noz)$.
Sometimes it would be convenient to pull it back to the corresponding
strata  $\cH(m_1,\dots,m_\noz)$  and $\cQ(d_1,\dots,d_\noz)$ by means
of the natural projection. As we already mentioned, $\Lambda(S)$ does
not depend on a choice of a basis of Abelian differentials.

One  can  recognize  in $\Lambda(S)$ the curvature of the determinant
line   bundle  $\Lambda^g  H^{1,0}$.  This  relation  is  of  crucial
importance     for     us;     it     will     be     explored     in
sections~\ref{ss:Sum:of:exponents:for:Teichmuller:curve}--\ref{sec:Riemann:Roch}
and in section~\ref{ss:analytic:Riemann:Roch:statement}.

\begin{NNRemark}
The   function  $\Lambda(S)$  defined  by  equation~\eqref{eq:lambda}
coincides with the function
$$
\Phi_g(q,I_g)=\Lambda_1(q)+\dots+\Lambda_g(q)
$$
introduced  in  formula  (5.9) in~\cite{Forni:Deviation}; see also an
alternative                    geometric                   definition
in~\cite{Forni:Matheus:Zorich:second:fund:form}. In particular, it is
proved  in~\cite{Forni:Deviation}  that  $\Lambda(S)$  is  everywhere
nonnegative.  (A  similar  statement in terms of the curvature of the
determinant line bundle is familiar to algebraic geometers.)
\end{NNRemark}

The next argument follows G.~Forni~\cite{Forni:Deviation}; see also
the survey of R.~Krikorian~\cite{Krikorian}. In the original
paper of M.~Kontsevich~\cite{Kontsevich} an equivalent statement was
formulated for connected components of the strata; it was proved by
G.~Forni~\cite{Forni:Deviation} that it is valid for any regular
invariant suborbifold.

Following  G.~Forni  we  start  with a formula from harmonic analysis
(literally  corresponding  to  Lemma  3.1 in~\cite{Forni:Deviation}).
Consider  the  Poincar\'e  model  of  the  hyperbolic plane $\Hyp$ of
constant   curvature   $-4$;   let  $t,\theta$  be  hyperbolic  polar
coordinates~\eqref{eq:hyp:polar:coordinates}.  Denote by $D_t$ a disc
of  radius  $t$  in  the hyperbolic metric, and by $|D_t|$ denote its
area.

\begin{NNLemma}
For  any  smooth  function  $L$  on  the hyperbolic plane of constant
curvature $-4$ one has the following identity:
\begin{equation}
\label{eq:Green:formula}
\cfrac{1}{2\pi} \cfrac{1}{\partial t}\int_0^{2\pi} L(t,\theta)\,d\theta=
\cfrac{1}{2} \tanh(t) \cfrac{1}{|D_t|}\int_{D_t}\Dhyp L\, d\ghyp
\end{equation}
\end{NNLemma}

To  prove  the  key  Background  Theorem  below  we  need a couple of
preparatory statements.

\begin{NNLemma}[Forni]
For  any  flat surface $S$ in any stratum in any genus the derivative
of the Hodge norm admits the following uniform bound:
$$
\max_{\substack{
c\in H^1(S,\R{})\text{\rm such}\\
\text{\rm that }\|c\|=1}}
\left|\,\frac{d\log\|c\|}{dt}\right|\le 1
$$
   %
and the function $\Lambda(S)$ defined in~\eqref{eq:lambda} satisfies:
\begin{equation}
\label{eq:bound:for:Lambda}
|\Lambda(S)|\le g\,.
\end{equation}
\end{NNLemma}
\begin{proof}
The  statement  of the Lemma is an immediate corollary of variational
formulas  from  Lemma~$2.1'$ in~\cite{Forni:Deviation}; basically, it
is  proved  in Corollary 2.2 in~\cite{Forni:Deviation} (in a stronger
form).
\end{proof}

As  an  immediate  Corollary we obtain the following universal bound:

\begin{NNCorollary}
For any flat surface $S$ in any stratum in any genus, the logarithmic
derivative   of   the  induced  Hodge  norm  on  the  exterior  power
$\Lambda^g( H^1(S,\R{}))$ admits the following uniform bound:
\begin{equation}
\label{eq:bound:for:derivative:of:L}
\max_{\substack{
L\in \Lambda^g(H^1(S,\R{}))\\
L\neq 0}}
\left|\,\frac{d\log\|L\|}{dt}\right|\le 1\,.
\end{equation}
\end{NNCorollary}

Now  everything is ready to prove the Proposition below, which is the
starting point of the current work.

\begin{BackgroundTheorem}[M.~Kontsevich; G.~Forni]
Let   $\cM_1$   be   any  closed  connected  regular  $\SL$-invariant
suborbifold  of  some  stratum of Abelian differentials in genus $g$.
The  top   $g$  Lyapunov  exponents  of  the  Hodge bundle $H^1$ over
$\cM_1$  along the Teichm\"uller flow satisfy the following relation:
\begin{equation}
\label{eq:sum:of:exponents:equals:Lambda}
\lambda_1 + \dots + \lambda_g =
\int_{\cM_1}\Lambda(S)\,d\nu_1(S)\ .
\end{equation}

Let   $\cM_1$   be  any  closed  connected  regular  $\PSL$-invariant
suborbifold  of  some  stratum of meromorphic quadratic differentials
with   at  most  simple  poles  in  genus  $g$.  The top $g$ Lyapunov
exponents  of  the  Hodge  bundle  $H^1_+$  over  $\cM_1$  along  the
Teichm\"uller flow satisfy the following relation:
\begin{equation}
\label{eq:sum:of:exponents:equals:Lambda:q}
\lambda^+_1 + \dots + \lambda^+_g =
\int_{\cM_1}\Lambda(S)\,d\nu_1(S)\ .
\end{equation}
\end{BackgroundTheorem}
\begin{proof}
We  prove  the  first  part of the statement; the proof of the second
part is completely analogous.

Consider   the   bundle   ${\mathcal   Gr}_g(\cM_1)$   of  Lagrangian
Grassmannians  ${\mathcal  Gr}_g(\R{2g})$  associated  to  the  Hodge
vector  bundle $H^1_{\R{}}$ over $\cM_1$. A fiber of this bundle over
a  ``point''  $S\in\cM_1$ can be naturally identified with the set of
of Lagrangian subspaces of $H^1(S,\R{})$.

Note  also that the sum of the top $k$ Lyapunov exponents of a vector
bundle  is  equal to the top Lyapunov exponent of its $k$-th exterior
power. Denote by $d\sigma_S$ the normalized Haar measure in the fiber
of  the  Lagrangian  Grassmannian bundle over a point $S\in\cM_1$. By
the      Oseledets     multiplicative     ergodic     theorem     for
$(\nu_1\times\sigma)$-almost all pairs $(S,L)$ where $S\in\cM_1$, and
$L\in{\mathcal Gr}_g\left(H^1(S,\R{})\right)$ one has
$$
\lambda_1+\dots+\lambda_g=
\lim_{T\to+\infty}\frac{1}{T}
\log\|L(g_t S)\|\,.
$$
(Here  we  use  the  simple  fact  that for $\nu_1$-almost every flat
surface     $\sigma$-almost    every    Lagrangian    subspace    is
Oseledets-generic.)

Using the identity
$$
\log\|L(g_t S)\|=\int_0^T\frac{d}{dt}\log\|L(g_t S)\|\,dt
$$
we  average  the right hand side of the above formula along the total
space  of  the Grassmanian bundle obtaining the first equality below.
Then  we  apply  an  extra  averaging over the circle, and, using the
uniform bound~\eqref{eq:bound:for:derivative:of:L} we interchange the
limit with the integral over the circle. Thus, we establish a further
equality with the expression in the second line below. We apply Green
formula~\eqref{eq:Green:formula}  to  the  inner  expression  in  the
second  line thus establishing an equality with the expression in the
third line. Then we apply Proposition~\ref{prop:Dhyp:L:equals:Lambda}
to  pass  to  the  expression  in  line  four  below.  We pass to the
expression  in line five applying definition~\eqref{eq:lambda}. (Note
that  the  fraction  $\cfrac{\tanh(t)}{2|D_t|}$  in  line  four  gets
transformed to $\cfrac{\tanh(t)}{|D_t|}$ in line five; the factor $2$
from  the  denominator  of  the  first  fraction  is  incorporated in
$\Lambda(S)$.)  Finally,  to pass to the left-hand side expression in
the bottom line, we use the uniform bound~\eqref{eq:bound:for:Lambda}
to  change  the  order  of  integration. The very last equality is an
elementary  property  of  $\tanh(t)$.  As  a  result  we  obtain  the
following sequence of equalities:
\begin{multline*}
\lambda_1+\dots+\lambda_g=
\int_{{\mathcal Gr}_g(\cM_1)}
\lim_{T\to+\infty}\frac{1}{T}
\int_0^T
\frac{d}{dt}\log\|L(g_t S)\|\,dt\,d\nu_1\,d\sigma_S
\ =\\=\
\int_{{\mathcal Gr}_g(\cM_1)}
\lim_{T\to+\infty}\frac{1}{T}
\int_0^T\frac{1}{2\pi}\int_0^{2\pi}
\frac{d}{dt}\log\|L(g_t r_\theta S)\|d\theta\,\,dt\,d\nu_1\,d\sigma_S
\ =\\=\
\int_{{\mathcal Gr}_g(\cM_1)}
\lim_{T\to+\infty}\frac{1}{T}
\int_0^T
\frac{\tanh(t)}{2|D_t|}
\int_{D_t}
\Dhyp\log\|L(g_t r_\theta S)\|d\ghyp\,\,dt\,d\nu_1\,d\sigma_S
\ =\\=\
\int_{\cM_1}
\lim_{T\to+\infty}\frac{1}{T}
\int_0^T
\frac{\tanh(t)}{2|D_t|}
\int_{D_t}
-\cfrac{1}{2}\,\Dhyp\log|\det\langle\omega_i,\omega_j\rangle|
d\ghyp\,\,dt\,d\nu_1
\ =\\=\
\int_{\cM_1}
\lim_{T\to+\infty}\frac{1}{T}
\int_0^T
\frac{\tanh(t)}{|D_t|}
\int_{D_t}
\Lambda(S)\,
d\ghyp\,\,dt\,d\nu_1
\ =\\=\
\int_{\cM_1}\Lambda(S)\,d\nu_1\cdot
\left(\lim_{T\to+\infty}\frac{1}{T}
\int_0^T \tanh(t)\,dt\right)
\ =\
\int_{\cM_1}\Lambda(S)\,d\nu_1(S)
\end{multline*}
The Proposition is proved.
\end{proof}

This  result  was developed by G.~Forni in~\cite{Forni:Deviation}. In
particular, he defined a collection of very interesting submanifolds,
called  \textit{determinant  locus}.  The  way  in  which the initial
invariant  suborbifold  $\cM_1$ intersects with the determinant locus
is   responsible   for  degeneration  of  the  spectrum  of  Lyapunov
exponents,     see~\cite{Forni:Deviation},     \cite{Forni:Handbook},
\cite{Forni:Matheus:Zorich},
\cite{Forni:Matheus:Zorich:second:fund:form}.      However,     these
beautiful  geometric results of G.~Forni are beyond the scope of this
paper, as well as further results of G.~Forni~\cite{Forni:Deviation},
and  of A.~Avila and M.~Viana~\cite{Avila:Viana} on simplicity of the
spectrum of Lyapunov exponents for connected components of the strata
of Abelian differentials.

\subsection{Sum of Lyapunov exponents for a Teichm\"uller curve}
\label{ss:Sum:of:exponents:for:Teichmuller:curve}
For   the   sake  of  completeness  we  consider  an  application  of
formula~\eqref{eq:sum:of:exponents:equals:Lambda}   to  Teichm\"uller
curves.

Let  $\cC$  be a smooth possibly non-compact complex algebraic curve.
We  recall  that a variation of {\it real} polarized Hodge structures
of  weight $1$  on $\cC$ is given by a real symplectic vector bundle
$\mathcal{E}_{\R{}}$  with  a flat connection $\nabla$ preserving the
symplectic  form,  such  that  every  fiber of $\mathcal E$ carries a
Hermitian  structure  compatible  with  the symplectic form, and such
that     the     corresponding     complex    Lagrangian    subbundle
$\mathcal{E}^{1,0}$   of  the  complexification  $\mathcal{E}_{\C{}}=
\mathcal{E}_{\R{}}\otimes\C{}$ is {\it holomorphic}. The variation is
called  {\it  tame}  if all eigenvalues of the monodromy around cusps
lie  on  the  unit  circle,  and the subbundle $\mathcal{E}^{1,0}$ is
meromorphic  at cusps. For example, the Hodge bundle of any algebraic
family  of  smooth compact curves over $\cC$ (or an orthogonal direct
summand of it) is a tame variation.

Similarly, a variation of {\it complex} polarized Hodge structures of
weight $1$  is given by a complex vector bundle $\mathcal{E}_{\C{}}$
of  rank  $p+q$ (where $p,q$ are nonnegative integers) endowed with a
flat  connection $\nabla$, by a covariantly constant pseudo-Hermitian
form   of   signature   $(p,q)$,   and  by  a  holomorphic  subbundle
$\mathcal{E}^{1,0}$  of  rank  $p$,  such that the restriction of the
form  to  it  is  strictly  positive.  The  condition  of tameness is
completely parallel to the real case.

Any  real  variation  of  rank  $2r$ gives a complex one of signature
$(r,r)$  by  the complexification. Conversely, one can associate with
any        complex       variation       $(\mathcal{E}_{\C{}},\nabla,
\mathcal{E}^{1,0})$  of  signature  $(p,q)$  a real variation of rank
\mbox{$2(p+q)$},  whose  underlying  local  system of real symplectic
vector spaces is obtained from $\mathcal{E}_{\C{}}$ by forgetting the
complex structure.

Let   us  assume  that  the  variation  of  complex  polarized  Hodge
structures of weight $1$ has a unipotent monodromy around cusps. Then
the   bundle   $\mathcal{E}^{1,0}$   admits   a  canonical  extension
$\overline{\mathcal{E}^{1,0}}$   to   the   natural  compactification
$\overline{\mathcal{C}}$.  It  can  be described as follows: consider
first     an     extension     $\overline{\mathcal{E}_{\C{}}}$     of
$\mathcal{E}_{\C{}}$    to   $\overline{\mathcal{C}}$   as   a   {\it
holomorphic} vector bundle in such a way that the connection $\nabla$
will  have  only first order poles at cusps, and the residue operator
at  any  cusp is nilpotent (it is called the {\it Deligne extension}).
Then     the    holomorphic    subbundle    $\mathcal{E}^{1,0}\subset
\mathcal{E}_\C{}$     extends     uniquely     as     a     subbundle
$\overline{\mathcal{E}^{1,0}}\subset  \overline{\mathcal{E}_\C{}}$ to
the cusps.

Let   $(\mathcal{E}_{\R{}},\nabla,\mathcal{E}^{1,0})$   be   a   tame
variation  of polarized real Hodge structures of rank $2r$ on a curve
$\cC$  with  {\it  negative} Euler characteristic. For example, $\cC$
could  be  an  unramified cover of a general arithmetic Teichm\"uller
curve,  and  $\mathcal  E$  could  be a subbundle of the Hodge bundle
which  is  simultaneously invariant under the Hodge star operator and
under the monodromy.

Using  the  canonical  complete  hyperbolic  metric  on $\cC$ one can
define  the  geodesic  flow  on  $\cC$ and the corresponding Lyapunov
exponents  $\lambda_1\ge  \dots \ge \lambda_{2r}$ for the flat bundle
$(\mathcal{E}_{\R{}},\nabla)$, satisfying the usual symmetry property
$\lambda_{2r+1-i}=-\lambda_i,\,i=1,\dots, r$.

The holomorphic vector bundle $\mathcal{E}^{1,0}$ carries a Hermitian
form, hence its top exterior power $\wedge^r(\mathcal{E}^{1,0})$ is a
holomorphic  line bundle also endowed with a Hermitian metric. Let us
denote  by $\Theta$ the curvature $(1,1)$-form on $\cC$ corresponding
to this metric. Then we have the following general result:
\begin{NNTheorem}
Under  the  above  assumptions,  the  sum  of  the  top  $r$ Lyapunov
exponents of $V$ with respect to the geodesic flow satisfies
\begin{equation}
\label{eq:sum:lambda:Teich:curve}
\lambda_1+\dots+\lambda_r=
\cfrac{\frac{i}{\pi}\int_C \Theta}{2G_\cC-2+s_\cC}\,,
\end{equation}
where  we  denote  by  $G_\cC$ --- the genus of $\cC$, and by $s_\cC$
---    the    number    of    hyperbolic    cusps    on   $\cC$.
\end{NNTheorem}
Note  that  the genus $G_\cC$ of the Teichm\"uller curve $\cC$ has no
relation to the genus $g$ of the flat surface $S$.

Formula~\eqref{eq:sum:lambda:Teich:curve}  was  first  formulated  by
M.~Kontsevich (in a slightly different form) in~\cite{Kontsevich} and
then proved rigorously by G.~Forni~\cite{Forni:Deviation}.

\begin{proof}
We  prove the above formula for $\mathcal{E}_{\R{}}:=H^1_{\R{}}$; the
proof in general situation is completely analogous.

By formula~\eqref{eq:sum:of:exponents:equals:Lambda} one has
$$
\lambda_1 + \dots + \lambda_g =
\int_{\cM_1}\Lambda(S)\,d\nu_1(S)=
\cfrac{1}{\operatorname{Area}(\cC)}
\int_{\cC}\Lambda(S)\,d\ghyp(S)\,,
$$
where  $\operatorname{Area}(\cC)=\cfrac{\pi}{2}\,(2G_\cC-2+s_\cC)$ is
the area of $\cC$ in the hyperbolic metric of curvature $-4$.

Let  $\zeta$  be  the  natural  complex  coordinate in the hyperbolic
plane; let $\partial=\partial/\partial\zeta$. The latter integral can
be expressed as
\begin{multline*}
\int_{\cC}\Lambda(S)\,d\ghyp(S)
\ =\
-\cfrac{1}{4}
\int_{\cC}\Dhyp\log|\det\langle\omega_i,\omega_j\rangle|\,d\ghyp(S)
\ =\\=\
-\cfrac{1}{4}
\int_{\cC}4\dd\log|\det\langle\omega_i,\omega_j\rangle|\,
\cfrac{i}{2}\,d\zeta\wedge d\bar\zeta
\ =\\=\
\cfrac{i}{2}
\int_{\cC}-2\dd\log|\det\langle\omega_i,\omega_j\rangle|^\frac{1}{2}\,
d\zeta\wedge d\bar\zeta
\ =\
\cfrac{i}{2}
\int_{\cC}\Theta(\Lambda^g H^{1,0})
\end{multline*}
where  $\Theta(\Lambda^g  H^{1,0})$  is  the  curvature  form  of the
determinant  line  bundle.  Dividing  the  latter  expression  by the
expression for the $\operatorname{Area}(\cC)$ found above we complete
the proof.
\end{proof}

Note  that  a  similar  result  holds  also for \textit{complex} tame
variations  of polarized Hodge structures. Namely, for a variation of
signature  $(p,q)$ one has $p+q$ Lyapunov exponents
$$
\lambda_1\ge \dots \ge \lambda_{p+q}\,.
$$
Let $r:=\min(p,q)$. Then, it is easy to verify that we again have the
symmetry  $\lambda_{p+q+1-i}=-\lambda_i,\,i=1,\dots,  p+q$,  and that
when     $p\neq     q$     we    have    an    additional    relation
$\lambda_{r+1}=\dots=\lambda_{p+q-r}=0$
(see~\cite{Forni:Matheus:Zorich:semisimplicity}).    The   collection
(with  multiplicities) $\{\lambda_1,\dots,\lambda_r\}$ will be called
the  \textit{non-negative}  part  of  the Lyapunov spectrum. We claim
that the sum of non-negative exponents $\lambda_1+\dots+\lambda_r$ is
again given by the formula~\eqref{eq:sum:lambda:Teich:curve}.

The  proof follows from the simple observation that one can pass from
a complex variation to a real one by taking the underlying real local
system.  Both  the  sum of non-negative exponents and the integral of
the curvature form are multiplied by two under this procedure.

The  denominator  in  the  above  formula is equal to minus the Euler
characteristic  of $\cC$, i.e. to the area of $\cC$ up to a universal
factor   $2\pi$.  The  numerator  also  admits  an  algebro-geometric
interpretation  for  variations  of  real Hodge structures arising as
direct  summands  of  Hodge bundles for algebraic families of curves.
Note  that the form $\frac{i}{2\pi}\Theta$ represents the first Chern
class  of  $\mathcal{E}^{1,0}$.  Let  us assume that the monodromy of
$(\mathcal{E},\nabla)$  around  any  cusp  is  unipotent (this can be
achieved  by passing to a finite unramified cover of $\cC$). Then one
has    the    following    identity   (see   e.g.   Proposition   3.4
in~\cite{Peters}):
$$
\frac{i}{\pi}\int_{\mathcal C} \Theta= 2\deg\overline{\mathcal{E}^{1,0}}\,.
$$
In  general, without the assumption on unipotency, we obtain that the
integral  above  is a rational number, which can be interpreted as an
orbifold  degree in the following way. Namely, consider an unramified
Galois  cover $\mathcal{C}'\to \mathcal{C}$ such that the pullback of
$(\mathcal{E},\nabla)$   has   a   unipotent   monodromy.   Then  the
compactified   curve   $\overline{\mathcal{C}}$   is  a  quotient  of
$\overline{\mathcal{C}'}$  by  a  finite  group  action, and hence is
endowed  with a natural orbifold structure. Moreover, the holomorphic
Hodge bundle on $\overline{\mathcal{C}'}$ will descend to an orbifold
bundle    on   $\overline{\mathcal{C}}$.   Then   the   integral   of
$\frac{i}{2\pi}\Theta$  over  $\mathcal  C$  is equal to the orbifold
degree of this bundle.

The  choice  of the orbifold structure on $\overline{\mathcal{C}}$ is
in  a  sense  arbitrary,  as we can choose the cover $\mathcal{C}'\to
\mathcal{C}$  in  different  ways. The resulting orbifold degree does
not  depend  on  this  choice.  The  corresponding  algebro-geometric
formula  for  the  denominator given as an orbifold degree, is due to
I.~Bouw  and  M.~M\"oller  in~\cite{Bouw:Moeller}.

In  the  next sections we compute the integral in the right-hand side
of~\eqref{eq:sum:of:exponents:equals:Lambda}, that is, we compute the
average  curvature  of  the determinant bundle. Our principal tool is
the             analytic             Riemann--Roch            Theorem
(Theorem~\ref{theorem:main:local:formula}  below)  combined  with the
study  of  the determinant of the Laplacian of a flat metric near the
boundary of the moduli space. The next section~\ref{sec:Riemann:Roch}
is used to motivate Theorem~\ref{theorem:main:local:formula}; readers
with  a  purely  analytic  background may wish to proceed directly to
section~\ref{sec:Determinant:of:Laplacian}.

\subsection{Riemann--Roch--Hirzebruch--Grothendieck Theorem}
\label{sec:Riemann:Roch}

Let  $\pi:C\to  B$  be a complex analytic family of smooth projective
algebraic    curves,    endowed   with   $n$   holomorphic   sections
$s_1,\dots,s_\noz$,  and  multiplicities  $m_i>0$. We assume that for
any   $x\in   B$   points  $s_i(x),\,i=1,\dots,\noz,$  in  the  fiber
$C_x:=\pi^{-1}(x)$     are     pairwise     distinct.    Denote    by
$D_i,\,i=1,\dots,\noz$  the  irreducible  divisor in $C$ given by the
image  of $s_i$. Moreover, we assume that a complex line bundle $\cL$
on $B$ is given, together with a holomorphic identification
$$
T^*_{C/B}\simeq
\pi^*\cL\otimes \mathcal{O}_C\left(\sum_i m_i D_i\right)\ .
$$
In  plain  terms  it  means  that any nonzero vector $l$ in the fiber
$\mathcal{L}_x$  of  $\cL$  at  $x\in B$ gives a holomorphic one form
$\alpha_l$  on  $C_x$  with  zeroes of multiplicities $m_i$ at points
$s_i(x)$.

Let  us  apply  the  standard Riemann--Roch--Hirzebruch--Grothendieck
theorem to the trivial line bundle $\mathcal{E}:=\mathcal{O}_C$:
$$
ch(R\pi_*(\mathcal{E}))=
\pi_*\left(ch(\mathcal{E})td(T_{C/B})\right)\in
H^{even}(B;\mathbb{Q})
$$
and  look  at  the term in $H^2(B;\mathbb{Q})$. The left-hand side is
equal to
$$
c_1(\mathcal{H})\ ,
$$
where  $\mathcal{H}$ is the holomorphic vector bundle on $B$ with the
fiber at $x\in B$ given by
$$
\mathcal{H}_x:=\Gamma(C_x,\Omega^1_{C_x})\ ,
$$
(that is the Hodge bundle $H^{1,0}$.) The reason is that the class of
$R\pi_*(\mathcal{O}_C)$ in the $K$-group of $B$ is represented by the
difference
$$
[R^0\pi_*(\mathcal{O}_C)]-[R^1\pi_*(\mathcal{O}_C)]=
[\mathcal{O}_B]-[\mathcal{H}^*]
$$
Let     us     compute     the     right-hand     side     in     the
Riemann--Roch--Hirzebruch--Grothendieck  formula. The Chern character
of $\mathcal{E}:=\mathcal{O}_C$ is
$$
ch(\mathcal{E})=1\in H^{even}(C;\mathbb{Q})\ .
$$
Therefore, the term in
$$
H^2(B;\mathbb{Q})
$$
is  the  direct  image of the term in $H^4(C;\mathbb{Q})$ of the Todd
class $T_{C/B}$, that is
$$
\frac{1}{12} \pi_*( c_1(T_{C/B})^2)\ .
$$

By our assumption, we have
$$
c_1(T_{C/B})=-\left(\pi^*c_1(\cL)+\sum_i m_i[D_i]\right)\ .
$$

First of all, we have
$$
\pi_*\left(\pi^*(c_1(\cL))\right)^2=\pi_*(1)\cdot c_1(\cL)^2=0
$$
because $\pi_*(1)=0$. Also, divisors $D_i$ and $D_j$ are disjoint for
$i\ne j$. Hence,
$$
\pi_*( c_1(T_{C/B})^2)=
2\sum_i m_i\pi_*(\pi^*c_1(\cL)\cdot[D_i])+
\sum_i m_i^2 \pi_*([D_i]\cdot [D_i])\ .
$$

Obviously,
$$
\pi_*(\pi^*c_1(\cL)\cdot[D_i])=c_1(\cL)\cdot \pi_*([D_i])=
c_1(\cL)\in H^2(B;\mathbb{Q})
$$
because $\pi_*([D_i])=1$.

Also,
$$
\pi_*([D_i]\cdot [D_i])=s_i^*(c_1(N_{D_i}))\,,
$$
where  $N_{D_i}$  is  the  normal  line  bundle  to  the $D_i$. If we
identify  $D_i$  with  the  base $B$ by map $s_i$, one can see easily
that
$$
s_i^*(c_1(N_{D_i}))=-\frac{1}{m_i+1} c_1(\cL)\in H^2(B;\mathbb{Q})
$$

The conclusion is that
$$
c_1(\mathcal{H})=const\cdot c_1(\cL)
$$
where the constant is given by
$$
const=\frac{1}{12}\sum_i\left(2m_i-\frac{m_i^2}{m_i+1}\right)=
\frac{1}{12}\sum_i \frac{m_i(m_i+2)}{m_i+1}
$$

The  line  bundle $\cL$ is endowed with a natural Hermitian norm, for
any $l \in \cL_x,\,\,x\in B$ we define
$$
|l|^2:=\int_{C_x}|\alpha_l|^2
$$
where $\alpha_l\in \Gamma(C_x,\Omega^1_{C_x})$ is the holomorphic one
form corresponding to $l$.

Hence, we have a canonical 2-form representing $c_1(\cL)$. Similarly,
the vector bundle $\mathcal {H}$ carries its own natural Hermitian metric
coming  form  Hodge  structure.  It  gives  another  canonical 2-form
representing   $c_1(\mathcal{H})$.   The analytic  Riemann--Roch  theorem
provides    an    explicit    formula    for    a   function,   whose
$\partial\overline{\partial}$  derivative  gives  the  correction. To
formulate the analytic Riemann--Roch theorem we need to introduce the
determinant of Laplace operator.

\subsection{Determinant of Laplace operator on a Riemann surface}
\label{sec:Determinant:of:Laplacian}
A good reference for this subsection is the book \cite{Soule}.

To  define a determinant $\det\Delta_g$ of the Laplace operator
on  a  Riemann surface $C$ endowed with a smooth Riemannian metric $g$
one defines the following \textit{spectral zeta function}:
$$
\zeta(s)=\sum_{\theta} \theta^{-s}
$$
where  the sum is taken over nonzero eigenvalues of $\Delta_g$.
This  sum  converges for $\Re(s)>1$. The function $\zeta(s)$ might be
analytically continued to $s=0$ and then one defines
$$
\log\det\Delta_g:=-\zeta'(0)
$$

The  analytic continuation can be obtained from the following formula
expressing $\zeta(s)$ in terms of the trace of the heat kernel,
\begin{displaymath}
\zeta(s) = \frac{1}{\Gamma(s)}\int_0^\infty t^{s-1} \operatorname{Tr}
\left(\exp ( t \Delta_g )\right) \, dt,
\end{displaymath}
and  the  well  known short-time asymptotics of the trace of the heat
kernel.

Let  $g_1$ and $g_2$ be two nonsingular metrics in the same conformal
class  on  a  closed  nonsingular  Riemann  surface $C$. Let the smooth
function  $2\cf$  be the logarithm of the conformal factor relating the
metrics $g_1$ and $g_2$:
$$
g_2=\exp(2\cf)\cdot g_1\,.
$$
The         theorem        below,        see~\cite{Polyakov:bosonic},
\cite{Polyakov:fermionic},   relates  the  determinants  of  the  two
Laplace operators:

\begin{PolyakovFormula}
\begin{multline}
\label{eq:Polyakov:formula}
\log\det\Delta_{g_2}-\log\det\Delta_{g_1}
\ =\\= \
\frac{1}{12\pi}\left(
\int_C\cf\,\Delta_{g_1}\cf\,d g_1
\ -\
2\int_C\cf\,K_{g_1}\,d g_1
\right)
\ +\
\bigg(\log\Area_{g_2}(C)-
\log\Area_{g_1}(C)\bigg)
\end{multline}
\end{PolyakovFormula}

\subsection{Determinant of Laplacian in the flat metric}
\label{sec:def:det:flat:metric}

Consider a flat surface $S$ of area one in some stratum of Abelian or
quadratic  differentials.  In a neighborhood of any nonsingular point
of  $S$  we  can  choose a \textit{flat coordinate} $z$ such that the
corresponding   quadratic   differential   $q$  (which  is  equal  to
$\omega^2$  when  we  work with an Abelian differential $\omega$) has
the form
$$
q=(dz)^2.
$$

A  conical  singularity $P$ of \textit{order} $d$ of $S$ has the cone
angle  $(d+2)\pi$.  One  can  choose  a  local  coordinate  $w$  in a
neighborhood  of $P$ such that the quadratic differential $q$ has the
form
\begin{equation}
\label{eq:def:local:coordinate:w}
q=w^{d}\, (dw)^2 \ .
\end{equation}
in  this  coordinate.  The corresponding flat metric $\gflat$ has the
form $|dz|^2$ in a neighborhood of a nonsingular point and
\begin{equation}
\label{eq:flat:metric:in:a:local:coordinate:w}
g_{\mathit{flat}}(w,\bar w)=
|w|^d\,|dw|^2 \ .
\end{equation}
in a neighborhood of a conical singularity.

Let  $\epsilon > 0$, and suppose that $\gflat$ is such, that the flat
distance   between   any   two  conical  singularities  is  at  least
$2\epsilon$.  We  define a smoothed flat metric $\gfe$ as follows. It
coincides    with    the   flat   metric   $|q|$   outside   of   the
$\epsilon$-neighborhood    of    conical    singularities.    In   an
$\epsilon$-neighborhood of a conical singularity it is represented as
$\gfe=\rhofe(|w|)\,|dw|^2$  where the local coordinate $w$ is defined
in~\eqref{eq:def:local:coordinate:w}.  We  choose  a  smooth function
$\rhofe(r)$ so that it satisfies the following conditions:
\begin{equation}
\label{eq:def:rho:flat:epsilon}
\rhofe(r)=\begin{cases}
r^d &r\ge\epsilon\\
\cfe&0\le r\le\epsilon'
\end{cases}
\ ,
\end{equation}
and  on  the interval $\epsilon'<r<\epsilon$ the function $\rhofe(r)$
is monotone and has monotone derivative.

\begin{figure}[htb]
\includegraphics{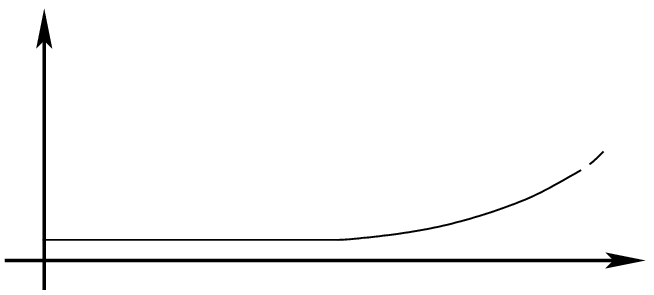}
\includegraphics{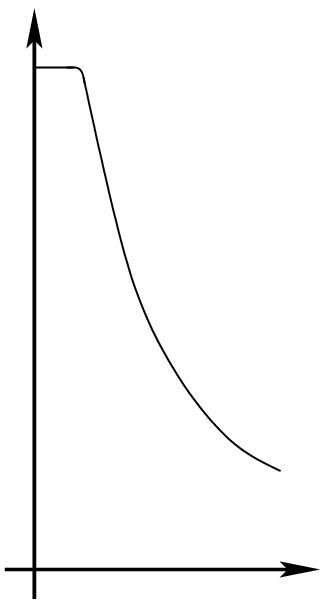}
\begin{picture}(0,0)(50,-70)
\put(-40,-165){\small $\epsilon'$}
\put(-35,-165){\small$\epsilon$}
\put(130,-165){\small $\epsilon'$}
\put(135,-165){\small$\epsilon$}
\end{picture}
\vspace{95pt}
\caption{
\label{fig:smoothed}
Function  $\rhofe(r)$  corresponding  to  a  zero  of  a  meromorphic
quadratic  differential  on  the left and to a simple pole --- on the
right. }
\end{figure}

It  is  convenient  for  us to obtain the function $\rhofe(r)$ in the
definition  of $\gfe$ from a continuous function which is constant on
the   interval   $[0,\epsilon]$   and   coincides   with   $r^d$  for
$r\ge\epsilon$.   This   continuous   function   is  not  smooth  for
$r=\epsilon$,  so we smooth out this ``corner'' in an arbitrary small
interval  $]\epsilon',\epsilon[$  by an appropriate convex or concave
function   depending   on   the   sign   of   the  integer  $d$,  see
Figure~\ref{fig:smoothed}.

Denote  by  $S$  a  flat  surface  of  area one defined by an Abelian
differential  or by a meromorphic quadratic differential with at most
simple  poles.  Denote  by  $S_0$ some fixed flat surface in the same
stratum.

\begin{Definition}
\label{def:relative:Laplace:flat}
We define the relative determinant of a Laplace operator as
\begin{equation}
\label{eq:def:relative:Laplace:flat}
\det\Dflat(S,S_0):=\lim_{\epsilon\to 0}
\frac{\det\Delta_{\mathit{flat},\epsilon}(S)}
{\det\Delta_{\mathit{flat},\epsilon}(S_0)}
\end{equation}
where  $\Delta_{\mathit{flat},\epsilon}$  is  the Laplace operator of
the metric $\gfe$.
\end{Definition}

Note  that  numerator and denominator in the above formula diverge as
$\epsilon  \to  0$.  However,  we  claim  that for sufficiently small
$\varepsilon$  the  ratio,  in  fact,  does  not  depend  neither  on
$\epsilon$  nor  the  exact  form  of the function $\rhofe$.  Indeed,
suppose $\epsilon_1 < \epsilon_2$. Then by the Polyakov formula,

\begin{multline*}
\log \frac{\det\Delta_{\mathit{flat},\epsilon_2}(S)}
{\det\Delta_{\mathit{flat},\epsilon_2}(S_0)}
-
\log \frac{\det\Delta_{\mathit{flat},\epsilon_1}(S)}
{\det\Delta_{\mathit{flat},\epsilon_1}(S_0)}
=\\=
\log \frac{\det\Delta_{\mathit{flat},\epsilon_2}(S)}
{\det\Delta_{\mathit{flat},\epsilon_1}(S)}
-
\log \frac{\det\Delta_{\mathit{flat},\epsilon_2}(S_0)}
{\det\Delta_{\mathit{flat},\epsilon_1}(S_0)}
\ =\\=\
\frac{1}{12\pi}\left(
\int_S\cf_S\,\Delta_{\mathit{flat},\epsilon_1}\cf_S
\,d g_{\mathit{flat},\epsilon_1}
\ -\
2\int_S\cf_S\,K_{\mathit{flat},\epsilon_1}\,d g_{\mathit{flat},\epsilon_1}
\right)
\ +\\+\
\bigg(\log\Area_{g_{\mathit{flat},\epsilon_2}}(C)-
\log\Area_{g_{\mathit{flat},\epsilon_1}}(C)\bigg)
\ -\\-\
\frac{1}{12\pi}\left(
\int_{S_0}\cf_{S_0}\,\Delta_{\mathit{flat},\epsilon_1}\cf_{S_0}\,d
g_{\mathit{flat},\epsilon_1}
\ -\
2\int_{S_0}\cf_{S_0}\,K_{\mathit{flat},\epsilon_1}\,
d g_{\mathit{flat},\epsilon_1}
\right)
\ -\\-\
\bigg(\log\Area_{g_{\mathit{flat},\epsilon_2}}(C_0)-
\log\Area_{g_{\mathit{flat},\epsilon_1}}(C_0)\bigg)\,.
\end{multline*}
Note    that    the    metrics   $g_{\mathit{flat},\epsilon_2}$   and
$g_{\mathit{flat},\epsilon_1}$     on     $C$    differ    only    on
$\epsilon_2$-neighborhoods  of conical points. Similarly, the metrics
$g_{\mathit{flat},\epsilon_2}$  and $g_{\mathit{flat},\epsilon_1}$ on
$C_0$ differ only on $\epsilon_2$-neighborhoods of conical points; in
particular the conformal factors are supported on this neighborhoods.
Since these neighborhoods are isometric by our construction the above
difference is equal to zero.

Thus, $\det\Dflat(S,S_0)$ is well-defined on the entire stratum.
\begin{remark}
\label{remark:S0}
It is clear from the definition that $\log \det\Dflat(S,S_0)$ depends on
the choice of $S_0$ only via an additive constant.
\end{remark}

\begin{NNRemark}
One can apply various approaches to regularize the determinant of the
Laplacian  of  a  flat  metric  with  conical singularities, see, for
example,   the  approach  of  A.~Kokotov  and  D.~Korotkin,  who  use
Friedrichs  extension  in~\cite{Kokotov:Korotkin}, or the approach of
A.~Kokotov~\cite{Kokotov:regularization}, who works with more general
metrics with conical singularities. All these various approaches lead
to essentially equivalent definitions, and to the same
  definition for the ``relative determinant'' $\det\Dflat(S,S_0)$.
\end{NNRemark}

\subsection{Analytic Riemann--Roch Theorem}
\label{ss:analytic:Riemann:Roch:statement}

The  Analytic Riemann--Roch Theorem was developed by numerous authors
in  different contexts. To give a very partial credit we would like to
cite         the        papers        of        A.~Belavin        and
V.~Knizhnik~\cite{Belavin:Knizhnik},      of     J.-M.~Bismut     and
J.-B.~Bost~\cite{Bismut:Bost}    of   J.-M.~Bismut,   H.~Gillet   and
C.~Soul\'e~\cite{Bismut:Gillet:Soule:1},
\cite{Bismut:Gillet:Soule:2},     \cite{Bismut:Gillet:Soule:3},    of
D.~Quillen~\cite{Quillen},        of        L.~Takhtadzhyan       and
P.~Zograf~\cite{Takhtadzhyan:Zograf}, and references in these papers.

The   results   obtained  in  the  recent  paper  of  A.~Kokotov  and
D.~Korotkin~\cite{Kokotov:Korotkin}    are    especially   close   to
Theorem~\ref{theorem:main:local:formula}                         (see
section~\ref{ss:Kokotov:Korotkin} below).

\begin{Theorem}
\label{theorem:main:local:formula}
For  any flat surface $S$ in any stratum $\cH_1(m_1,\dots,m_\noz)$ of
Abelian differentials the following formula holds:

\begin{equation}
\label{eq:main:local:formula:Abelian}
\Dhyp\log|\det\langle\omega_i,\omega_j\rangle|\,=\,
\Dhyp\log\det\Dflat(S,S_0)\,-\,
\cfrac{1}{3}\,\sum_{j=1}^\noz \cfrac{m_j(m_j+2)}{m_j+1}\,,
\end{equation}
where  $m_1+\dots+m_\noz=2g-2$. Here $\Dhyp$ is taken with respect to
the   canonical   hyperbolic   metric   of   curvature  $-4$  on  the
Teichm\"uller disc passing through $S$. (Note that the
  right-hand-side of (\ref{eq:main:local:formula:Abelian}) is
  independent of the choice of $S_0$ in view of Remark~\ref{remark:S0}.)

For  any flat surface $S$ in any stratum $\cQ_1(d_1,\dots,d_\noz)$ of
meromorphic  quadratic  differentials  with  at most simple poles the
following formula holds:
\begin{equation}
\label{eq:main:local:formula:quadratic}
\Dhyp\log|\det\langle\omega_i,\omega_j\rangle|\,=\,
\Dhyp\log\det\Dflat(S,S_0)\,-\,
\cfrac{1}{6}\,\sum_{j=1}^\noz \cfrac{d_j(d_j+4)}{d_j+2},
\end{equation}
where $d_1+\dots+d_\noz=4g-4$.
\end{Theorem}

Theorem~\ref{theorem:main:local:formula}       is      proved      in
section~\ref{sec:analytic:Riemann:Roch:proof}.

Consider two basic examples illustrating
Theorem~\ref{theorem:main:local:formula}.

\begin{Example}[\textit{Flat torus}]
\label{ex:torus}  Consider  the  canonical coordinate $\zeta=x+iy$ in
the     fundamental    domain,    $\Im\zeta>0$,    $|\zeta|\ge    1$,
$-1/2\le\Re\zeta\le  1/2$, of the upper half-plane parametrizing the
space   of   flat   tori.   This   coordinate   was   introduced   in
Example~\ref{ex:M1} in the end of section~\ref{ss:Teichmuller:discs}.

There are no conical singularities on a flat torus, so the definition
of  the  determinant  of Laplacian does not require a regularization.
For a torus of unit area, one has:
$$
\det \Dflat=4\Im(\zeta)\, |\eta(\zeta)|^4\,,
$$
where $\eta$ is the Dedekind $\eta$-function,
see,   for   example,  \cite[\S{4}]{Ray:Singer},
\cite{Osgood:Phillips:Sarnak},  page  205,  or
formula~(1.3) in~\cite{McIntyre:Takhtajan}. Since
  $\eta$ is holomorphic,
$$
\Dhyp\log \det \,\Dflat=\Dhyp\log|\Im\zeta|
=\Dhyp\log
y  = 4 y^2 \frac{\partial^2}{\partial y^2} \log y = -4.
$$

On  the  other  hand, as a holomorphic section $\omega(\zeta)$ we can
choose  the  Abelian  differential with periods $1$ and $\zeta$. Then
$\det\langle\omega_i,\omega_j\rangle|=\|\omega\|^2=\Area=\Im\zeta
=y$.
Thus, the equality~\eqref{eq:main:local:formula:Abelian}
holds.
In addition, we get
\begin{displaymath}
\Lambda(S) = - \frac{1}{4} \Dhyp \log \det \langle \omega_i, \omega_j
\rangle = -\frac{1}{4} \Dhyp \log y = 1.
\end{displaymath}
Thus, since $\nu_1$ is a probability measure, we get
\begin{equation}
\label{eq:torus:verification}
\int_{\cM_1} \Lambda(S) \, d\nu_1(S) = 1.
\end{equation}
In the torus case there is only one Lyapunov exponent, namely
$\lambda_1$, and we know from general arguments that $\lambda_1 =
1$. Therefore, (\ref{eq:torus:verification}) verifies explicitly the
key  formula~\eqref{eq:sum:of:exponents:equals:Lambda}.

\end{Example}



\begin{Example}[\textit{Flat sphere with four cone points}]
\label{ex:CP1:with:4:poles}
According        to        a        result       A.~Kokotov       and
D.~Korotkin~\cite{Kokotov:Korotkin:spere},  the  determinant  of  the
Laplacian  for  the  flat  metric defined by a quadratic differential
with  four  simple  poles and no other singularities on $\CP$ one has
the form
$$
\det\Delta^{|q|}
=const\cdot\frac{|\Im(A\bar B)|\cdot |\eta(B/A)|^2}{|A|}\,,
$$
where $A$ and $B$ are the periods of the covering torus (see the last
pages   of~\cite{Kokotov:Korotkin:spere}).   Here,   the  determinant
$\det\Delta^{|q|}$  of  Laplacian  corresponding  to  the flat metric
$|q|$    defined    in~\cite{Kokotov:Korotkin:spere}   differs   from
$\det\Delta^{|q|}(S,S_0)$  only  by  a  multiplicative constant. Note
that
$$
\Im(A\bar B)
=\Im\left(\frac{A\bar A \bar B}{\bar A}\right)=|A|^2\Im(\bar B/\bar A)
= -|A|^2\Im(B/A)
$$
Thus,
$$
\det\Delta^{|q|}=const\cdot|A|\cdot|\Im(B/A)|\cdot |\eta(B/A)|^2\,.
$$
One  should not be misguided by the fact that under the normalization
$A:=1$  one gets $\Dhyp|A|=0$ along a holomorphic deformation. Recall
that  in our setting we have to normalize the area of the flat sphere
to  one!  Doing  so  for  the  double-covering  torus  with $A=1$ and
$B=\zeta=x+iy$  we rescale $A$ to $A=1/\sqrt{y}$ and $B\sim\sqrt{y}$,
which implies that for the sphere of unit area we get
\begin{equation}
\label{eq:det:Dflat:4:CP1:with:4:poles}
\det\Delta^{|q|}=const\cdot y^{-1/2}\cdot y\cdot |\eta(B/A)|^2\,,
\end{equation}
so
$$
\Dhyp\log\det\Delta^{|q|}=\cfrac{1}{2}\,\log y\,.
$$
Comparing to the integral above, we get
$$
\int_{\cM_1}-\cfrac{1}{4}\,\Dhyp\log\det\Delta^{|q|}=\cfrac{1}{2}\,.
$$
On the other hand, for four simple poles one has
$$
\cfrac{1}{24}\,\sum_{j=1}^4 \cfrac{(-1)(-1+4)}{-1+2}=-\cfrac{1}{2}\,,
$$
and  integrating~\eqref{eq:main:local:formula:quadratic} we get zeros
on both sides, as expected.
\end{Example}

\subsection{Hyperbolic metric with cusps}
\label{sec:punctured:hyperbolic}

A  conformal  class  of  a  flat  metric  $|q|$  contains a canonical
hyperbolic  metric of any given constant curvature with cusps exactly
at  the  singularities  of  the  flat  metric.  (In  the  case,  when
$q=\omega^2$,   where  $\omega\in\cH(0)$  is  a  holomorphic  Abelian
differential  on  a  torus,  we  mark  a  point  on the torus.) In an
appropriate  holomorphic  coordinate  $\zeta$  in a neighborhood of a
conical  singularity  $P$ of such canonical hyperbolic metric $\ghyp$
of curvature $-1$ has the form
\begin{equation}
\label{eq:def:local:coordinate:zeta}
g_{\mathit{hyp}}(\zeta,\bar\zeta)=
\frac{|d\zeta|^2}{|\zeta|^2\log^2|\zeta|} \ .
\end{equation}

Similarly to the smoothed flat metric we define a smoothed hyperbolic
metric    $\ghd$.   It   coincides   with   the   hyperbolic   metric
$g_{\mathit{hyp}}$  outside  of a neighborhood of singularities. In a
small   neighborhood   of   a   singularity   it  is  represented  as
$\ghd=\rhohd(|\zeta|)\,|d\zeta|^2$ where the local coordinate $\zeta$
is  as  in~\eqref{eq:def:local:coordinate:zeta}.  We  choose a smooth
function $\rhohd(s)$ so that it satisfies the following conditions:
\begin{equation}
\label{eq:def:rho:hyp:delta}
\rhohd(s)=\begin{cases}
s^{-2}\log^{-2} s &s\ge\delta\\
\chd&0\le s\le\delta'\ ,
\end{cases}
\end{equation}
and  on  the  interval $\delta'<s<\delta$ the function $\rhohd(s)$ is
monotone and has monotone derivative. We can assume that $\delta'$ is
extremely  close  to  $\delta$  and that $\chd$ is extremely close to
$\delta^{-2}\log^{-2}(\delta)$, see Figure~\ref{fig:smoothed}.

Suppose that $S$ and $S_0$ are two surfaces in the same stratum.

\begin{Definition}[Jorgenson--Lundelius]
\label{def:relative:Laplace:hyperbolic}
Define  relative  determinant of a Laplace operator in the hyperbolic
metric as
$$
\det\Delta_{\ghyp}(S,S_0):=
\frac{\det\Delta_{\mathit{hyp},\delta}(S)}
{\det\Delta_{\mathit{hyp},\delta}(S_0)}\ ,
$$
where              $\Delta_{\mathit{hyp},\delta}(S)$              and
$\Delta_{\mathit{hyp},\delta}(S_0)$  are  Laplace  operators  of  the
metric $\ghd$ on $S$ and $S_0$ correspondingly.
\end{Definition}

As  in section section~\ref{sec:def:det:flat:metric}, we can see that
$\det\Delta_{\ghyp}(S,S_0)$  does  not depend either on $\delta$ or
on the exact choice of the function $\rhohd(s)$.
\medskip

\noindent
{\bf Strategy.}
Now  we  can  formulate  our  strategy  for the rest of the proof. By
formula~\eqref{eq:sum:of:exponents:equals:Lambda}, to compute the sum
of  the  Lyapunov  exponents  we  need  to  evaluate  the integral of
$\Dhyp\log|\det\langle\omega_i,\omega_j\rangle|$       over       the
corresponding   $\SL$-invariant   suborbifold.   Using  the  analytic
Riemann--Roch  Theorem  this  is  equivalent  to  evaluation  of  the
integral         of         $\Dhyp\log\det\Dflat(S,S_0)$,         see
equations~\eqref{eq:main:local:formula:Abelian}
and~\eqref{eq:main:local:formula:quadratic}.  Using  Polyakov formula
we           compare           $\log\det\Dflat(S,S_0)$           with
$\log\det\Delta_{\ghyp}(S,S_0)$  and  show  that  when the underlying
Riemann  surface  $S$  is  close to the boundary of the moduli space,
there is no much difference between them.

The  determinant  of  Laplacian in the \textit{hyperbolic} metric was
thoroughly   studied,   see,   for   example,  papers  of  B.~Osgood,
R.~Phillips,    and    P.~Sarnak~\cite{Osgood:Phillips:Sarnak},    of
S.~Wolpert~\cite{Wolpert},         of         J.~Jorgenson        and
R.~Lundelius~\cite{Jorgenson:Lundelius},     \cite{Lundelius}.     In
particular,   there   is  a  very  explicit  asymptotic  formula  for
$\log\det\Delta_{\ghyp}(S,S_0)$  due to S.~Wolpert~\cite{Wolpert} and
to  R.~Lundelius~\cite{Lundelius}. Using these formulas and performing
an  appropriate cutoff near the boundary, we evaluate the integral of
$\Dhyp\log\det\Dflat(S,S_0)$.

\subsection{Relating flat and hyperbolic Laplacians by means of
the Polyakov formula}
\label{sec:flat:to:hyperbolic}

Consider a function $f$ on a flat surface $S$ and a function $f_0$ on
a  fixed  flat  surface $S_0$ in the same stratum as $S$. Assume that
the functions are nonsingular outside of conical singularities of the
flat  metrics. By convention the surfaces belong to the same stratum.
We  assume  that  the  conical singularities are named, so there is a
canonical  bijection  between conical singularities of $S$ and $S_0$.
By   construction,   small  neighborhoods  of  corresponding  conical
singularities  are  isometric in the corresponding hyperbolic metrics
with   cusps   defined  by~\eqref{eq:def:local:coordinate:zeta}.  The
isometry is unique up to a rotation.

Suppose  that we can represent $f$ and $f_0$ in a neighborhood $\cO(R)$
of each cusp as
\begin{align*}
f(r,\theta)&=g(r)+h(r,\theta)\\
f_0(r,\theta)&=g(r)+h_0(r,\theta)
\end{align*}
where  $g(r)$ is rotationally symmetric and $h$ and $h_0$ are already
integrable  with  respect  to  the  hyperbolic metric of the cusp. We
define
\begin{multline}
\label{eq:relative:integral}
\left\langle\int_S f \, d\ghyp - \int_{S_0} f_0 \, d\ghyp
\right\rangle:=\\
=
\int_{S\setminus\cup\cO_j(R)} f \, d\ghyp -
\int_{S_0\setminus\cup\cO_j(R)} f_0 \, d\ghyp+
\sum_j\int_{\cO_j(R)} (f-f_0) \, d\ghyp
\end{multline}
Clearly, this definition does not depend on the cutoff parameter $R$.

Recall  that $\gflat$ and $\ghyp$ belong to the same conformal class.
Denote  by  $\cf$ (correspondingly $\cf_0$) the following function on
the surface $S$ (correspondingly $S_0$):
$$
\gflat=\exp(2\cf)\ghyp\ .
$$

\begin{Theorem}
\label{theorem:log:det:flat:minus:log:det:hyp}
For any pair $S,S_0$ of flat surfaces of the same area in any stratum
of  Abelian  differentials  or of meromorphic quadratic differentials
with at most simple poles one has
\begin{multline}
\label{eq:log:det:flat:minus:log:det:hyp}
\log\det\Dflat(S,S_0)-\log\det\Delta_{\ghyp}(S,S_0)
\ =\\=\
\frac{1}{12\pi}\left\langle\int_S \cf \, d\ghyp -
\int_{S_0} \cf_0 \, d\ghyp\right\rangle
\ -\\-\
\frac{1}{6} \sum_j\left(\log\left|\frac{dw}{d\zeta}(P_j,S)\right|-
\log\left|\frac{dw_0}{d\zeta}(P_j,S_0)\right|\right),
\end{multline}
where $\zeta$ is as in
  (\ref{eq:def:local:coordinate:zeta}), and $w$ and $w_0$ are as in
  (\ref{eq:def:local:coordinate:w}) for $S$ and $S_0$ respectively.
\end{Theorem}

Theorem~\ref{theorem:log:det:flat:minus:log:det:hyp}  is  a  corollary  of
Polyakov         formula;        it        is        proved        in
section~\ref{sec:Comparison:of:determinants}.

\subsection{Comparison  of relative determinants of Laplace operators
near the boundary of the moduli space}

In
section~\ref{sec:Comparison:of:determinants:end:asymptotic:behavior}
we            estimate            the           integral           in
formula~\eqref{eq:log:det:flat:minus:log:det:hyp}                from
Theorem~\ref{theorem:log:det:flat:minus:log:det:hyp}    and    prove   the
following statement.

\begin{Theorem}
\label{theorem:det:minus:det:is:O:ell:flat}
Consider  two  flat surfaces $S,S_0$ of area one in the same stratum.
Let  $\lf(S),\lf(S_0)$  be the lengths of shortest saddle connections
on   flat   surface   $S$  and  $S_0$  correspondingly.  Assume  that
$\lf(S_0)\ge l_0$. Then
\begin{multline}
\label{eq:det:minus:det:is:O:ell:flat}
\big|\log\det\Dflat(S,S_0)-\log\det\Delta_{\ghyp}(S,S_0)\big| \le \\
\le\const_1(g,\noz)\cdot|\log\ell_{\mathit{flat}}(S)|
+\const_0(g,\noz,l_0)
\end{multline}
with  constants $\const_0(g,\noz,l), \const_1(g,\noz)$ depending only
on the genus of $S$, on the number $\noz$ of conical singularities of
the flat metric on $S$ and on the bound $l_0$ for $\lf(S_0)$.
\end{Theorem}

In fact we prove a much more accurate statement in
Theorem~\ref{theorem:det:minus:det:upto:O(1)}, which gives the exact
difference between the flat and hyperbolic determinants up to an error
which is bounded in terms only of $S_0$, $g$ and $\noz$.  The optimal
constant $c_1(g,\noz)$ in (\ref{eq:det:minus:det:is:O:ell:flat}) can
also be deduced easily from
Theorem~\ref{theorem:det:minus:det:upto:O(1)}.

Establishing  a convention confining the choice of the auxiliary flat
surface  $S_0$  to  some  reasonable predefined compact subset of the
stratum  one  can  make $\const_0$ independent of $l_0$. For example,
the  subset of those $S_0$ for which $\lf(S_0)\ge 1/\sqrt{2g-2+n}$ is
nonempty   for   any  connected  component  of  any  stratum.  As  an
alternative  one  can impose a lower bound on the shortest hyperbolic
geodesic  on the Riemann surface underlying $S_0$ in terms of $g$ and
$\noz$.
\smallskip

We  prove  Theorem~\ref{theorem:det:minus:det:is:O:ell:flat}  applying the
following     scheme.     To     evaluate     the     integral     in
formula~\eqref{eq:log:det:flat:minus:log:det:hyp} we use a thick-thin
decomposition of the surface $S$ determined by the hyperbolic metric.
Then,      using      Theorem~\ref{theorem:Deligne:Mumford}     (Geometric
Compactification  Theorem) we obtain a desired estimate for the thick
part. We then use the maximum principle and some simple calculations
to obtain the desired estimates for the integral on the thin part.

\subsection{Determinant of Laplacian near the boundary of the moduli
space}
\label{ss:det:near:the:boundary}

Consider a holomorphic 1-form $\omega$ (or a meromorphic differential
$q$  with  at most simple poles) on a closed Riemann surface of genus
$g$.    Consider the   corresponding   flat   surface   $S=S(\omega)$
(correspondingly  $S(q)$). Assume that $\omega$ (correspondingly $q$)
is normalized in such way that the flat area of $S$ is equal to one.

Every  regular  closed geodesic on a flat surface belongs to a family
of  parallel  closed  geodesics  of equal length. Such family fills a
maximal cylinder with conical points of the metric on each of the two
boundary  components.  Denote by $h_j$ and $w_j$ a height and a width
correspondingly   of  such  a  maximal  cylinder.  (By  convention  a
``width''  of  a  cylinder is the length of its waist curve, which by
assumption   is   a  closed  geodesic  in  the  flat  metric.)  By  a
\textit{modulus} of the flat cylinder we call the ratio $h_j/w_j$.

\begin{Theorem}
\label{theorem:Dflat:near:the:boundary}
For  any  stratum  $\cH_1(m_1,\dots,m_\noz)$ of Abelian differentials
and   for   any   stratum  $\cQ_1(d_1,\dots,d_\noz)$  of  meromorphic
quadratic  differentials  with  at  most  simple  poles there exist a
constant  $M=M(g,n)\gg  1$ depending only on the genus $g$ and on the
number  $\noz$ of zeroes and simple poles, such that for any pair $S,
S_0$  of  flat surfaces of unit area in the corresponding stratum one
has
\begin{equation}
\label{eq:Dflat:near:the:boundary}
-\log\det\Dflat(S,S_0)= \frac{\pi}{3}\,
\sum_{
\substack{
\text{cylinders with}\\
h_r/w_r\ge M
}
}
\frac{h_r}{w_r}
+ O(\log\ell_{\mathit{flat}}(S))\,,
\end{equation}
where
$\lf(S)$  is the length of the shortest saddle connection on the flat
surfaces $S$ and $h_r, w_r$ denote heights and widths of maximal flat
cylinders of modulus at least $M$ on the flat surface $S$. Here
$$
\Big|O(\log\ell_{\mathit{flat}}(S))\Big|
\le C_1(g,\noz)\cdot|\log\ell_{\mathit{flat}}(S)|+C_0(g,\noz,S_0)
$$
with  $C_1(g,\noz), C_0(g,\noz,S_0)$ depending only on the genus $g$,
on  the  number  $\noz$  of  conical  singularities  of the base flat
surface $S_0$.
\end{Theorem}

\noindent\textbf{Choice of the constants.}
Similarly   to   the   way  suggested  in  the  discussion  following
Theorem~\ref{theorem:det:minus:det:is:O:ell:flat},      establishing     a
reasonable  convention  on  the  choice  of  $S_0$ one can get rid of
dependence   of   the  constants  on  the  base  surface  $S_0$.

\begin{NNRemark}
It is a well known fact, see e.g.  \cite[Proposition 3.3.7]{Hubbard:book}
that a  flat cylinder of sufficiently large modulus necessarily contains a
short  hyperbolic  geodesic for the underlying hyperbolic metric. The
number  of  short  hyperbolic  geodesics  on  a surface is bounded by
$3g-3+n$.  Thus, for sufficiently large $M$ depending only on $g$ and
$n$,         the         number         of         summands        in
expression~\eqref{eq:Dflat:near:the:boundary}  is  uniformly bounded.
\end{NNRemark}

\begin{Example}[\textit{Flat torus}]
In notations of Example~\ref{ex:torus} from the previous section, one
has  the following expression for the determinant of the Laplacian in
a flat metric on a torus of area one:
$$
\det\Dflat(\zeta)=4\Im(\zeta)\, |\eta(\zeta)|^4\,.
$$
see,   for   example,  \cite{Osgood:Phillips:Sarnak},  page  205,  or
formula~(1.3)  in~\cite{McIntyre:Takhtajan}.  Taking the logarithm of
the   above   formula  and  using  the  asymptotic  of  the  Dedekind
$\eta$-function for large values of $\Im\zeta$ we get
$$
\log\det\Dflat\sim 4\log|\eta(\zeta)| \sim
-\cfrac{\pi}{3}\,\Im\zeta\,,
\ \text{ when }\Im\zeta\to +\infty
\,.
$$
Note that $h/w$ does not depend on the rescaling
of the torus, and $h/w\sim\Im\zeta$. Thus, we
get the asymptotics promised by
relation~\eqref{eq:Dflat:near:the:boundary} of
Theorem~\ref{theorem:Dflat:near:the:boundary}.
\end{Example}

\begin{Example}[\textit{Flat sphere with four cone points}]
In  Example~\ref{ex:CP1:with:4:poles}  from  the  previous section we
considered the determinant of the Laplacian in a flat metric on $\CP$
defined  by  a quadratic differential with four simple poles and with
no      other      singularities;      see     the     last     pages
of~\cite{Kokotov:Korotkin:spere}        for       details.       This
expression~\eqref{eq:det:Dflat:4:CP1:with:4:poles}     implies    the
following  asymptotics  for  large  values  of  $\Im\zeta$  (we  keep
notations of Example~\ref{ex:CP1:with:4:poles}):
$$
\log\det\Delta^{|q|}\sim 2\log|\eta(\zeta)|\,,
\ \text{ when }\Im\zeta\to +\infty\,,
$$
which  is  one  half  of the torus case. Indeed, the height $h$ of the
single  flat  cylinder of the covering torus is twice bigger then the
height  of  the  single  flat cylinder on the underlying flat sphere,
while  the  width $w$ of the cylinder on the torus is the same as the
width  of  the  one  on  the  flat  sphere.  Thus,  we  again get the
asymptotics  promised  by relation~\eqref{eq:Dflat:near:the:boundary}
of Theorem~\ref{theorem:Dflat:near:the:boundary}.
\end{Example}

Theorem~\ref{theorem:Dflat:near:the:boundary}     is     proved    in
section~\ref{sec:det:near:the:boundary}.  Our  strategy  is to derive
the   result   from   an   analogous   estimate   by   Lundelius  and
Jorgenson---Lundelius for a hyperbolic metric punctured at the zeroes
of    $\omega$    (correspondingly,   $q$)   and   then   apply   the
estimate~\eqref{eq:Dflat:near:the:boundary}                      from
Theorem~\ref{theorem:Dflat:near:the:boundary}.

\subsection{The contribution of the  boundary of moduli space}
\label{sec:contrib:boundary}

A  regular  invariant  suborbifold  $\cM_1$  is never compact, so one
should  not  expect that the integral of $\Dhyp \log \det\Dflat$ over
$\cM_1$ would be zero. Indeed,

\begin{Theorem}
\label{theorem:int:Dflat:equals:SV:const}
Let  $\cM_1$  be  a regular invariant suborbifold of flat surfaces of
area  one in a stratum of Abelian differentials (correspondingly in a
stratum  of  quadratic  differentials with at most simple poles). Let
$\nu_1$     be    the    associated    probability    $\SL$-invariant
(correspondingly     $\PSL$-invariant)     density    measure.    Let
$c_{\mathit{area}}(\cM_1)   :=   c_{\mathit{area}}(\nu_1)$   be   the
corresponding Siegel--Veech constant. Then
\begin{equation}
\label{eq:int:Dflat:equals:SV:const}
   %
\int_{\cM_1} \Dhyp\log\Dflat(S,S_0)\;d\nu_1\ =\
-\frac{4}{3}\,\pi^2\cdot c_{\mathit{area}}(\cM_1)
\end{equation}
\end{Theorem}
Theorem~\ref{theorem:int:Dflat:equals:SV:const}    is    proved    in
section~\ref{sec:cutoff}.

\begin{Remark}
\label{rm:only:cylinders:joining:size:one}
   %
After integrating by parts, the left side of
(\ref{eq:int:Dflat:equals:SV:const}) can be written as
an integral over a neighborhood of the boundary of the moduli space,
which in view of
(\ref{eq:Dflat:near:the:boundary}) is dominated by a sum over all
cylinders of large modulus. Also the Siegel-Veech constant $c_{area}(\cM_1)$
measures the contribution of (certain kinds) of cylinders of large
modulus; this gives a heuristic explanation of
(\ref{eq:int:Dflat:equals:SV:const}). However, for the precise proof
of (\ref{eq:int:Dflat:equals:SV:const}) in \S\ref{sec:cutoff}
we need the assumptions of
\S\ref{sec:Regular:invariant:submanifolds}, (so we can e.g. justify the
integration by parts).
\end{Remark}

The  main  Theorems  now  become  elementary corollaries of the above
statements.

\begin{proof}[Proof of Theorem~\ref{theorem:general:Abelian}]
Suppose that $\cM_1$ is a regular suborbifold of a stratum of Abelian
differentials.                                                  Apply
equation~\eqref{eq:sum:of:exponents:equals:Lambda}               from
the Background Theorem        to express  the  sum  of  the  Lyapunov
exponents  as  the  integral  of  $\Lambda(S)$  defined by
relation~\eqref{eq:lambda}.                                    Use
equation~\eqref{eq:main:local:formula:Abelian}                   from
Theorem~\ref{theorem:main:local:formula}  to  rewrite the integral of
$\Dhyp\log|\det\langle\omega_i,\omega_j\rangle|$   in  terms  of  the
integral   of   $\Dhyp\log\det\Dflat(S,S_0)$.   Finally,   apply  the
relation~\eqref{eq:int:Dflat:equals:SV:const}                    from
Theorem~\ref{theorem:int:Dflat:equals:SV:const} to express the latter
integral in terms of the corresponding Siegel--Veech constant.
\end{proof}

\begin{proof}[Proof of part (a) of
Theorem~\ref{theorem:general:quadratic}]
The  proof  of part (a) of Theorem~\ref{theorem:general:quadratic} is
completely        analogous        to        the       proof       of
Theorem~\ref{theorem:general:Abelian}  with  the only difference that
one   uses   expression~\eqref{eq:main:local:formula:quadratic}  from
Theorem~\ref{theorem:main:local:formula}          instead          of
equation~\eqref{eq:main:local:formula:Abelian}.
\end{proof}

\section{Geometric Compactification Theorem}
\label{sec:Geometric:Compactification:Theorem}

In  section~\ref{ss:after:Rafi}, we present the results of K.~Rafi on
comparison  of  flat  and hyperbolic metrics near the boundary of the
moduli   space.   Using   the   notions   of   a   \textit{thick-thin
decomposition}  and  of  a  \textit{size} (in the sense of Rafi) of a
thick       part       we       formulate      and      prove      in
section~\ref{ss:Compactification:Theorem}    a    version    of   the
Deligne--Mumford--Grothendieck  Compactification Theorem in geometric
terms.  The proof is an elementary corollary of nontrivial results of
K.~Rafi.  The  Geometric  Compactification  Theorem  is an important
ingredient            of            the            proof           of
Theorem~\ref{theorem:Dflat:near:the:boundary}       postponed      to
section~\ref{sec:det:near:the:boundary}.


\subsection{Comparison of flat and hyperbolic geometry (after K.~Rafi)}
\label{ss:after:Rafi}

We  start  with  an  outline  of  results  of  K.~Rafi~\cite{Rafi} on
the comparison  of  flat  and hyperbolic metrics when the Riemann surface
underlying  the  flat  surface  $S$  is  close to the boundary of the
moduli space.

Throughout    section~\ref{sec:Geometric:Compactification:Theorem}   we
consider  a  larger class of flat metrics, namely, we consider a flat
metric  defined  by  a  meromorphic  quadratic differential $q$ which
might  have  poles  of any order. In particular, the flat area of the
surface  might  be  infinite.  Unless  it is stated explicitly, it is
irrelevant  whether or not the quadratic differential $q$ is a global
square of a meromorphic 1-form.

In   section~\ref{sec:Geometric:Compactification:Theorem}   we   mostly
consider  the  flat  surface $S$ and its subsurfaces $Y$ punctured at
all  singular  points  of  the flat metric (or, in other words at all
zeroes   and   poles   of  the  corresponding  meromorphic  quadratic
differential $q$). Sometimes, to stress that the surface is punctured
we denote it by $\Spunc$ and $\Ypunc$ correspondingly.

Following  K.~Rafi,  by  a  ``curve''  we  always  mean a non-trivial
non-peripheral   piecewise-smooth  simple  closed  curve.  Any  curve
$\alpha$  in $S$, has a geodesic representative in the flat metric.
This  representative  is unique except for the case when it is one of
the  continuous  family  of  closed  geodesics in a flat cylinder. We
denote  the flat length of the geodesic representative of $\alpha$ by
$l_{\mathit{flat}}[\alpha]$.

A  \textit{  saddle  connection}  is  a  geodesic segment in the flat
metric   joining  a  pair  of  conical  singularities  or  a  conical
singularity  to itself without any  singularities  in  its  interior.
A geodesic representative of any curve on $S$ is a closed broken line
composed from a finite number of saddle connections.

\begin{figure}[htb]
\includegraphics{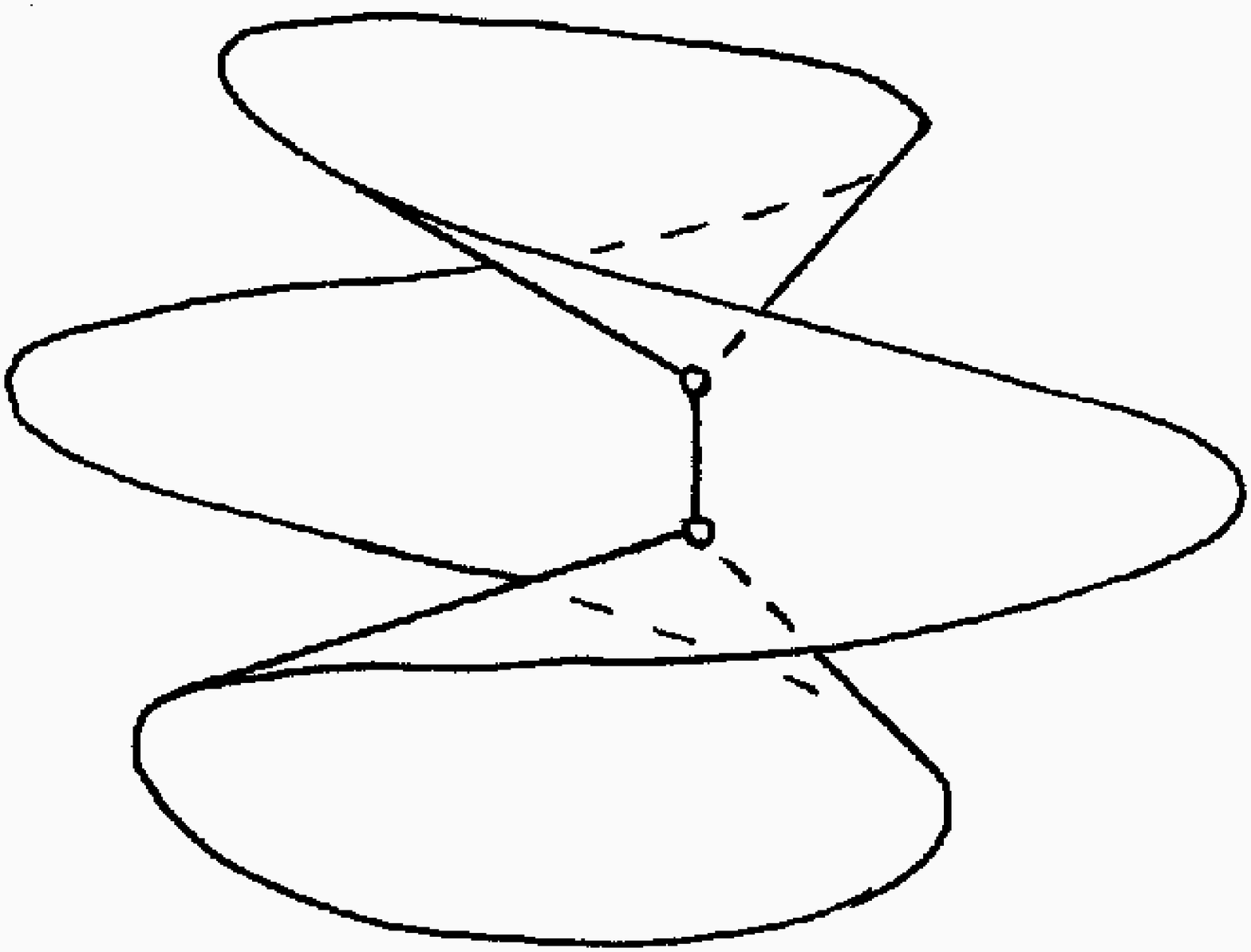}
\includegraphics{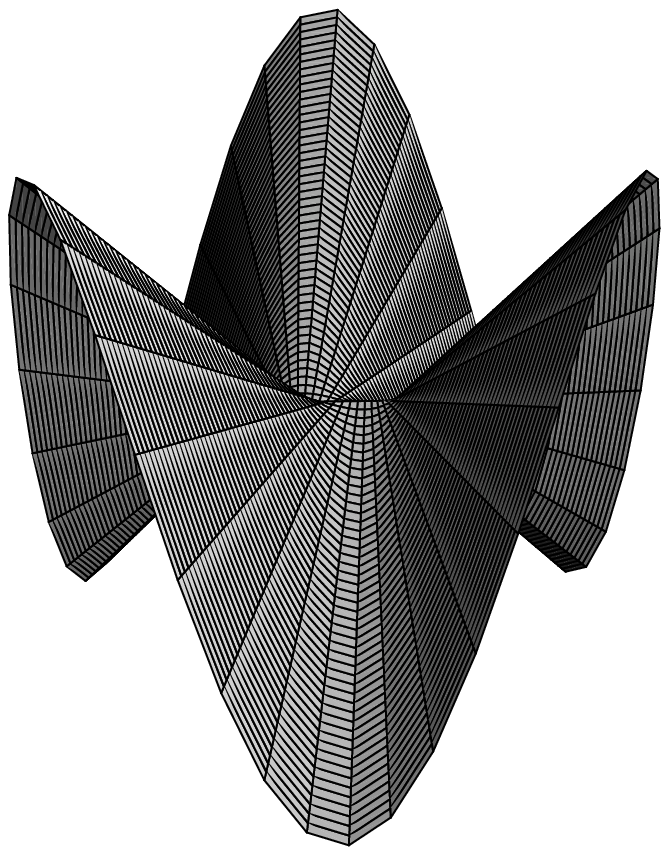}
\begin{picture}(0,0)(0,10)
\put(-94,-65){$\gamma$}
\put(-73,0){$\alpha$}
\end{picture}
\vspace{125pt}
\caption{
\label{fig:gamma:gamma}
As  a geodesic representative of a closed curve $\alpha$ encircling a
short saddle connection $\gamma$ we get a closed broken line composed
from two copies of $\gamma$. }
\end{figure}

Considering  the punctured flat surface $\Spunc$, formally we have to
speak   about  the \textit{infimum}  of  a  flat  length  over  essential
(non-peripheral)  curves  in  a free homotopy class of a given curve.
However,  even  in the case of the punctured flat surface $\Spunc$ it
is  convenient  to  consider  limiting  closed geodesic broken lines,
where  segments  of  the  broken  line are saddle connections joining
zeroes  and  simple poles of the quadratic differential. For example,
for  a  closed  curve  $\alpha$  encircling a short saddle connection
$\gamma$,   one  has  $l_{\mathit{flat}}[\alpha]=2|\gamma|$  and  the
corresponding  closed  broken  line  is  composed  from two copies of
$\gamma$,  see Figure~\ref{fig:gamma:gamma}. Following the discussion
in~\cite{Rafi:short:curves},  we can ignore this difficulty and treat
these special geodesics as we would treat any other geodesic.

Under  this  convention every curve $\alpha$ in the punctured surface
$\Spunc$  has  a geodesic representative in the flat metric, and this
representative  is  unique  except for the case when it is one of the
continuous family of closed geodesics in a flat cylinder.
We call such a representative a $q$-geodesic representative of $\gamma$.

Let  $\ghyp$ be the hyperbolic metric with cusps at all singularities
of  $S$  in  the conformal class of the flat metric on $S$. We define
$l_{\mathit{hyp}}[\alpha]$  to be the shortest hyperbolic length of a
curve  in  a  free  homotopy  class  of $\alpha$ on the corresponding
punctured surface $\Spunc$ or on its appropriate subsurface $\Ypunc$.

Let  $\delta\ll  1$  be a fixed constant; let $\Gamma(\delta)$ be the
set  of  simple  closed  geodesics of $\ghyp$ in $S$ whose hyperbolic
length  is  less  than or equal to $\delta$. A $\delta$\textit{-thick
component}  of $\ghyp$ is a connected component $Y$ of the complement
$S\setminus\Gamma(\delta)$.

Assume  that  $\delta$  is  sufficiently  small  (here the measure of
``sufficiently small'' depends only on the genus and on the number of
punctures  of  the  surface).
We now cut the surface $S$ along all the $q$-geodesic representatives of all
the short curves in
$\Gamma(\delta)$. More precisely, if $\gamma \in \Gamma(\delta)$ has a
unique $q$-geodesic representative, we cut along that representative;
otherwise $\gamma$ is represented by a
closed geodesic in a flat cylinder $F_\gamma$,
in which case we cut
along \textit{both} curves at the ends of $F_\gamma$ (and thus remove
the cylinder $F_\gamma$ from the surface). After this procedure the
surface $S$ breaks up into the following pieces:
\begin{itemize}
\item For each $\gamma \in \Gamma(\delta)$ whose $q$-geodesic
  representative is part of a continuous family of closed geodesics in
  a cylinder, we get the corresponding cylinder.
\item For each $\delta$-thick component $Y$ of $S\setminus
  \Gamma(\delta)$ we get a subsurface $\Y \subset S$ with boundaries
  which are geodesic in the flat metric defined by $q$.
Following  K.~Rafi,  we call such a flat surface with boundary $\Y$ a
\textit{$q$-representative}    of    $Y$    (see    an   example   at
Figure~\ref{fig:q:representatives}). Note that $\Y$ always has finite
area; in some particular cases it might degenerate to a graph. In that
case, we should think of $\Y$ as a ribbon graph (which, as all ribbon
graphs, uniquely defines a surface with boundary). With that caveat, we
can say that $\Y$ is in the same homotopy class as $Y$. We note that
$\Y$ is the smallest representative of the homotopy class of $Y$ with
$q$-geodesic boundaries.
\end{itemize}

\begin{figure}[htb]
\includegraphics{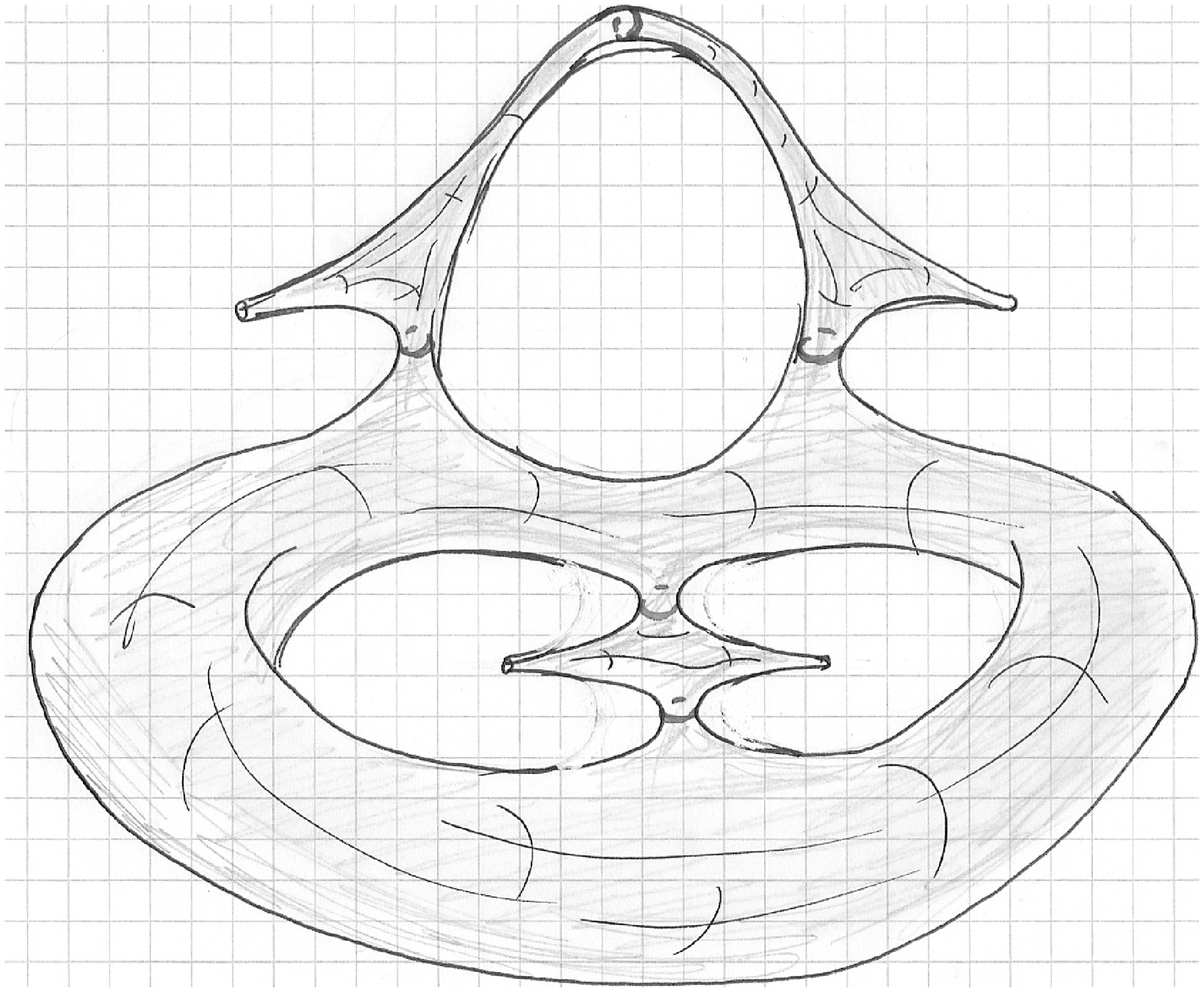}
\includegraphics{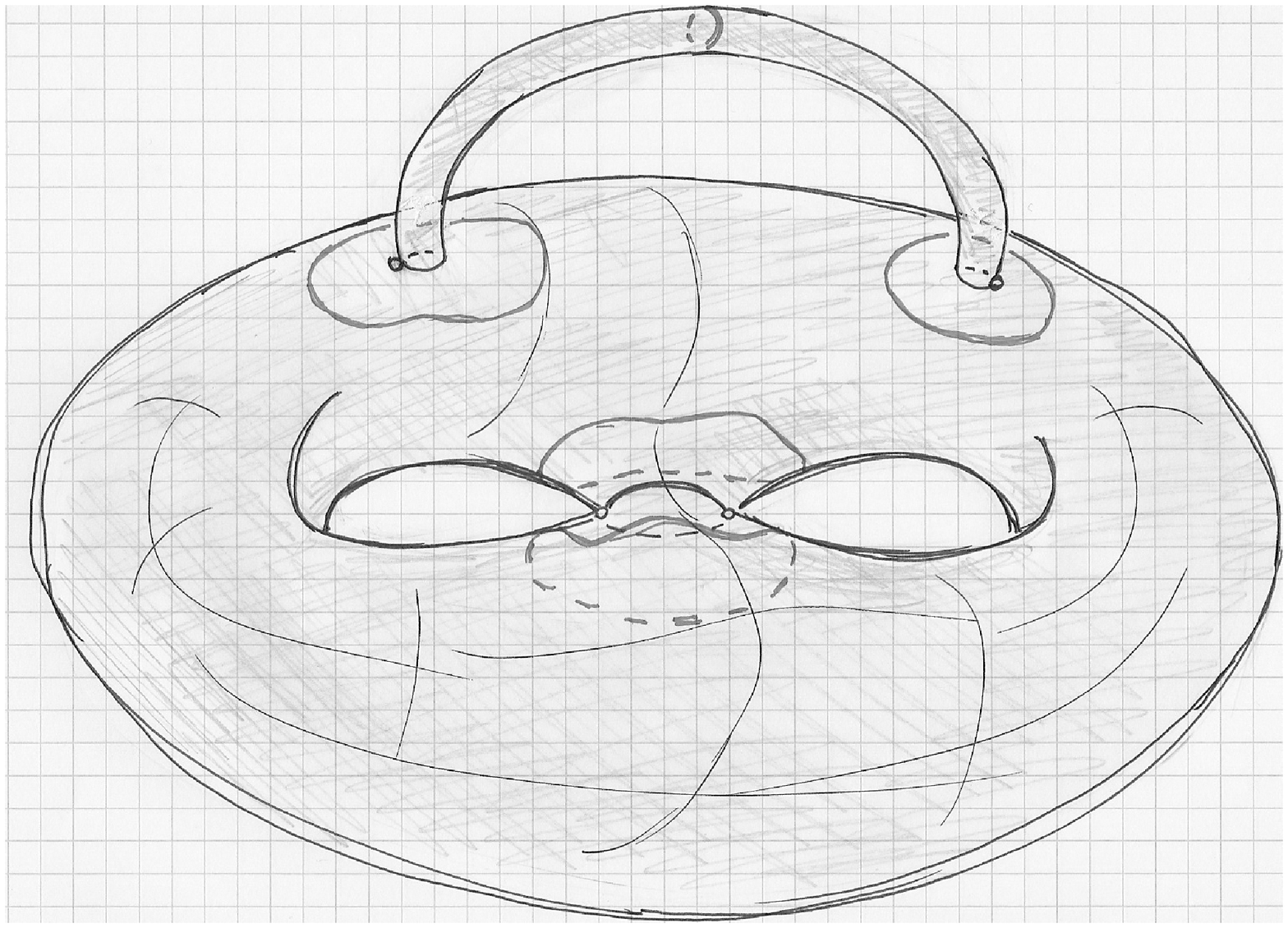}
\includegraphics{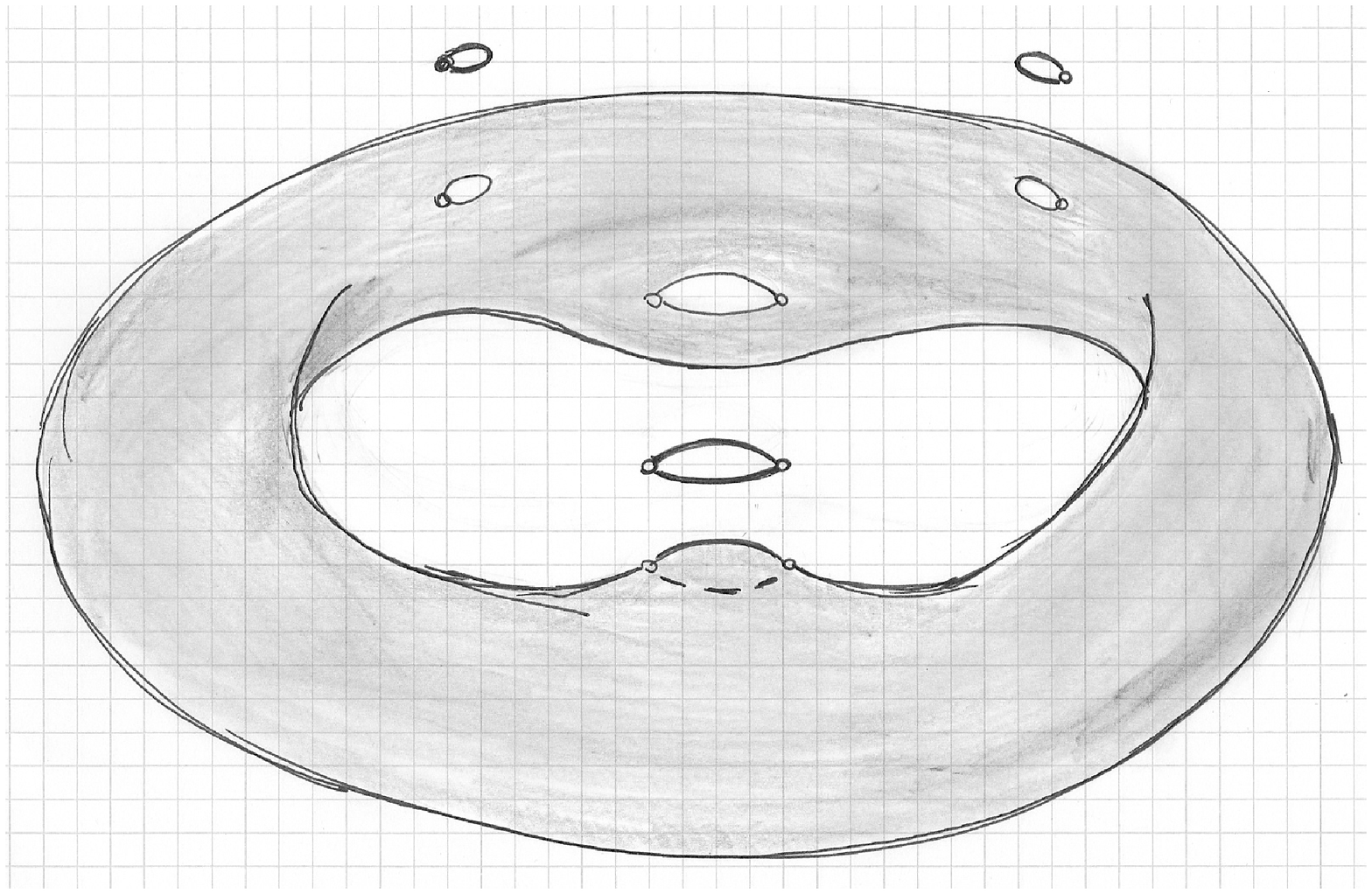}
\begin{picture}(0,0)(162,75) 
\put(-14,2){\small$\gamma_1$}
\put(32,17){\small$\gamma_2$}
\put(32,-14){\small$\gamma_3$}
\put(62,0){\small$\gamma_4$}
\put(75,0){\small$\gamma_5$}
\end{picture}
\begin{picture}(0,0)(20,75)
\put(-4,44){\small$Y_1$}
\put(-4,-33){\small$Y_2$}
\put(72,1){\small$Y_3$}
\put(42,1){\small$Y_4$}
\end{picture}
\begin{picture}(0,0)(-90,75)
\put(-10,44){\small$\Y_1$}
\put(-10,-33){\small$\Y_2$}
\put(62,1){\small$\Y_3$}
\put(37,-10){\small$\Y_4$}
\end{picture}
\vspace{150pt}
\caption{
\label{fig:q:representatives}
Schematic  picture  of  thick components of the underlying hyperbolic
metric  (on  the  left) and their $q$-representatives (on the right).
The   $q$-representative  $\Y_2$  degenerate  to  a  pair  of  saddle
connections.  Each  of  the  $q$-representatives  $\Y_3,\Y_4$  of the
corresponding  pair of pants degenerate to a single saddle connection
joining a zero to itself. }
\end{figure}


A  very expressive example of $q$-representatives is presented at the
very end of the original paper~\cite{Rafi} of K.~Rafi.

We  shall  also  need the notion of a \textit{curvature of a boundary
curve} of a subsurface $Y$ introduced in~\cite{Rafi}. Let $\gamma$ be
a    boundary    component    of   $\Y$.
The    \textit{curvature}
$\kappa_Y(\gamma)$  of  $\gamma$  in  the  flat metric on
$S$ is well
defined as a measure with atoms at the corners.

We  choose  the  sign  of  the  curvature  to  be  positive  when the
acceleration   vector   points   into  $\Y$.
If  a  curve  is  curved
non-negatively  (or  non-positively)  with  respect  to
$\Y$ at every
point,  we  say  that it is monotonically curved with respect to $Y$.
Let  $A$  be  an  annulus  in  $S$  with  boundaries  $\gamma_0$  and
$\gamma_1$.  Suppose  that  both  boundaries are monotonically curved
with  respect  to  $A$  and  that $\kappa_A(\gamma_0)\le 0$. Further,
suppose  that the boundaries are equidistant from each other, and the
interior  of $A$ contains no zeroes or poles. We call $A$ a primitive
annulus and write $\kappa_A:=-\kappa_A(\gamma_0)$. When $\kappa_A=0$,
$A$  is called a flat cylinder, in this case it is foliated by closed
Euclidean  geodesics  homotopic  to the
boundaries. Otherwise, $A$ is called an \textit{expanding
    annulus.} See~\cite{Minsky} for more details.

\begin{Definition}[K.~Rafi]
\label{def:size}  Define  the  \textit{flat  size} $\lambda(Y)$ of a
subsurface $Y$ different from a pair of pants to be the shortest flat
length of an essential (non-peripheral) curve in $\Y$.

When $Y$ is a pair of pants (that is, when $\Ypunc$ has genus $0$ and
$3$  boundary  components),  there are no essential curves in $Y$. In
this  case,  define the \textit{flat size} of $Y$ as the maximal flat
length of the three boundary components of $\Y$.
\end{Definition}

We will often use the notation $\lambda(\Y)$ to denote
$\lambda(Y)$.

\begin{NNTheorem}[K.~Rafi]
For  every  $\delta$-thick
component $Y$ of $S$
and  for  every  essential curve $\alpha$ in $Y$, the flat length  of
$\alpha$  is  equal to the size of $Y$ times the hyperbolic length of
$\alpha$  up  to  a multiplicative  constant
$C(g,n,\delta)$
depending  only  on $\delta$ and the topology of $S$:
$$
\frac{\lambda(Y)}{C(g,n,\delta)}\cdot
l_{\mathit{hyp}}[\alpha]\le
l_{\mathit{flat}}[\alpha]\le
C(g,n,\delta)\lambda(Y)\cdot l_{\mathit{hyp}}[\alpha]
$$
Also, the diameter of $\Y$ in the flat metric
is bounded by $C(g,n,\delta)   \lambda(Y)$.

\end{NNTheorem}

One possible heuristic explanation of this theorem is as follows (see
also Theorem~\ref{theorem:Deligne:Mumford} and
Remark~\ref{remark:deligne:mumford} below).  On compact subsets of the
moduli space the flat and hyperbolic metrics are comparable (by a
compactness argument), and so the theorem trivially holds.  Thus
assume that we have a sequence of surfaces $S_\tau = (C_\tau, q_\tau)$
tending to infinity in moduli space. By the Deligne-Mumford theorem,
we may assume that the Riemann surfaces $C_\tau$ tend to a noded
surface $C_\infty$. Then, the $\delta$-thick subsurfaces $Y_{\tau,j}$
of $C_\tau$ converge to the components of $C_{\infty,j}$ of
$C_\infty$.  We may also assume after passing to a subsequence that
the quadratic differentials $q_\tau$ tend to a (meromorphic) quadratic
differential on $C_\infty$. (If the original quadratic differentials
$q_\tau$ are holomorphic, the limit quadratic differential will be
holomorphic away from the nodes of $C_\infty$, but may develop poles
at the nodes.)  However, $q_\tau$ may tend to zero on some component
$C_{\infty,j}$ of $C_\infty$, i.e. it may be very small on the
subsurfaces $Y_{\tau,j}$. But, with the proper choice of rescaling
factors $\lambda_{\tau,j} \in \reals^+$, we can make sure that the
sequence of quadratic differentials $\lambda_{\tau,j} q_\tau$ tends to
a bounded and non-zero limit on $C_{\infty,j}$. This limit is a
meromorphic quadratic differential with poles, and number and the
degrees of the poles can be bounded in terms of only the topology. The
set of all such differentials is a finite dimensional vector space,
and so, as all such vector spaces, is projectively compact. Thus,
after the rescaling, the restriction to $Y_{\tau,j}$ is again in a
situation where the moduli space is compact, and thus (up to the
rescaling factor) the flat metric coming from $q$ is comparable to the
hyperbolic metric.

The strength of the above theorem of K.~Rafi (which is proved by
completely different methods) is to justify the above discussion, and
also to identify the rescaling factor $\lambda_{\tau,j}$ with
$\lambda(Y_{\tau,j})^{-2}$, where $\lambda(Y_{\tau,j})$ is the size of
$Y_{\tau,j}$ which can be detected by measuring the flat
lengths of saddle connections in the ($q$-geodesic representative of) the
$\delta$-thick subsurface $Y_{\tau,j}$.

We complete this section by the following elementary Lemma which will
be                               used                              in
section~\ref{sec:Comparison:of:determinants:end:asymptotic:behavior}.

\begin{Lemma}
\label{lemma:size:ge:lflat}
The size of any thick component of a flat surface $S$ is bounded from
below  by  the  length  $\lf(S)$ of the shortest saddle connection on
$S$:
$$
\lambda(Y)\ge \lf(S)\,.
$$
\end{Lemma}
\begin{proof}
We  consider  separately  the  situation when $Y$ is different from a
pair of pants, and when $Y$ is a pair of pants.

If  $Y$  is  not  a  pair of pants, $\lambda(Y)$ is the shortest flat
length  of  an essential (non-peripheral) curve $\gamma$ in $\Ypunc$.
This shortest length is realized by a flat geodesic representative of
$\gamma$,  that  is by a broken line composed from saddle connections
(possibly  a single saddle connection). This implies the statement of
the Lemma.

Implicitly  the  statement  for the pair of pants is contained in the
paper  of  K.~Rafi.  According to~\cite{Rafi} the size of any pair of
pants   is  strictly  positive.  Hence,  the  corresponding  boundary
component   has   a  geodesic  representative  composed  from  saddle
connections  and  the statement follows. A direct proof can be easily
obtained  from  the  explicit description of possible ``pairs of flat
pants'' in the next section.
\end{proof}

\subsection{Flat pairs of pants}
\label{ss:Flat:interpretation:of:the:boundary:surfaces}

In  this  section  we  describe the flat metric on $\CP$ defined by a
meromorphic  quadratic  differential  from  $\cQ(d_1,d_2,d_3)$, where
$d_1+d_2+d_3=-4$.   In  particular,  we  consider  the  size  of  the
corresponding flat surface.

\begin{figure}[htb]
\includegraphics{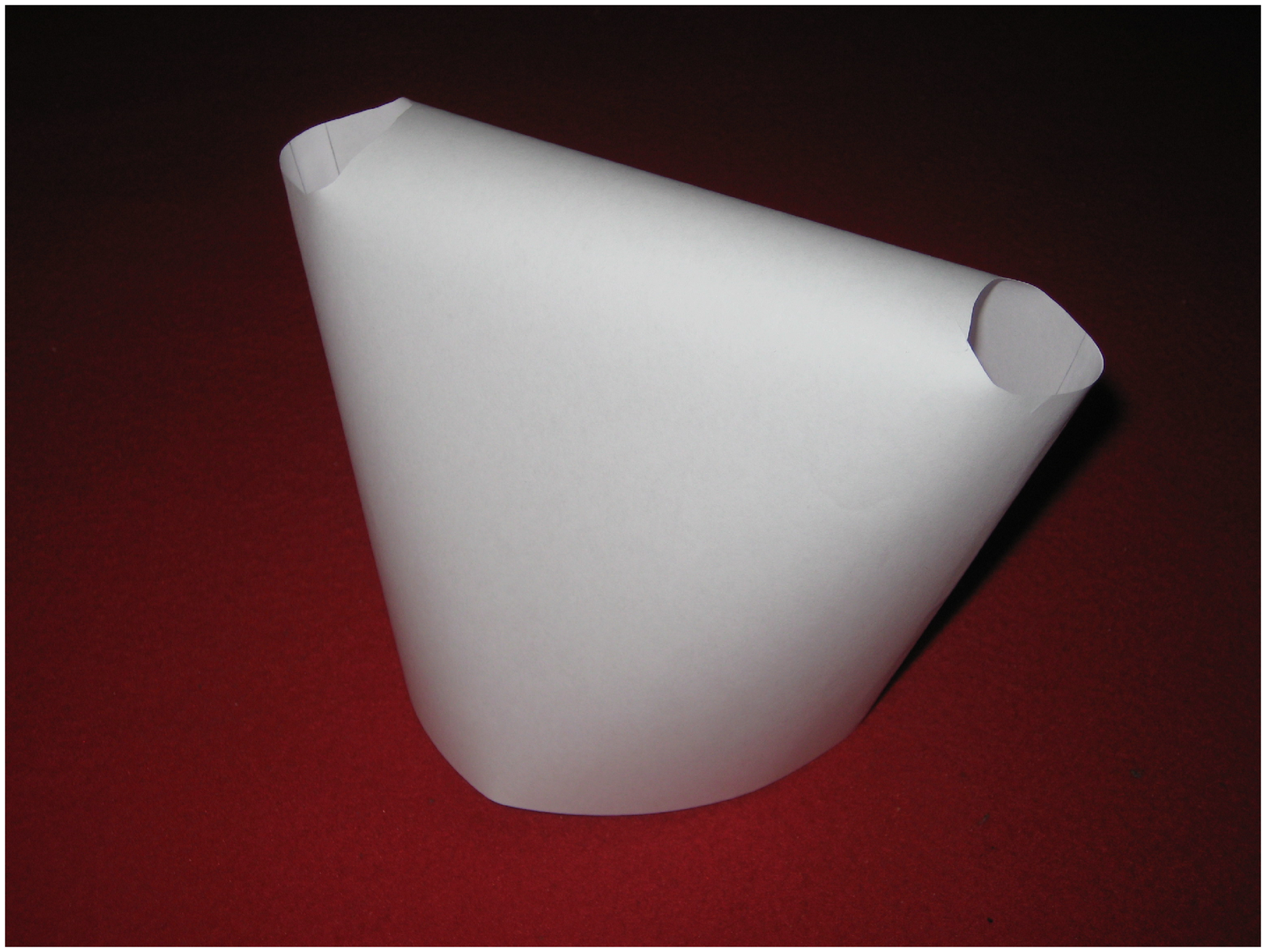}
\includegraphics{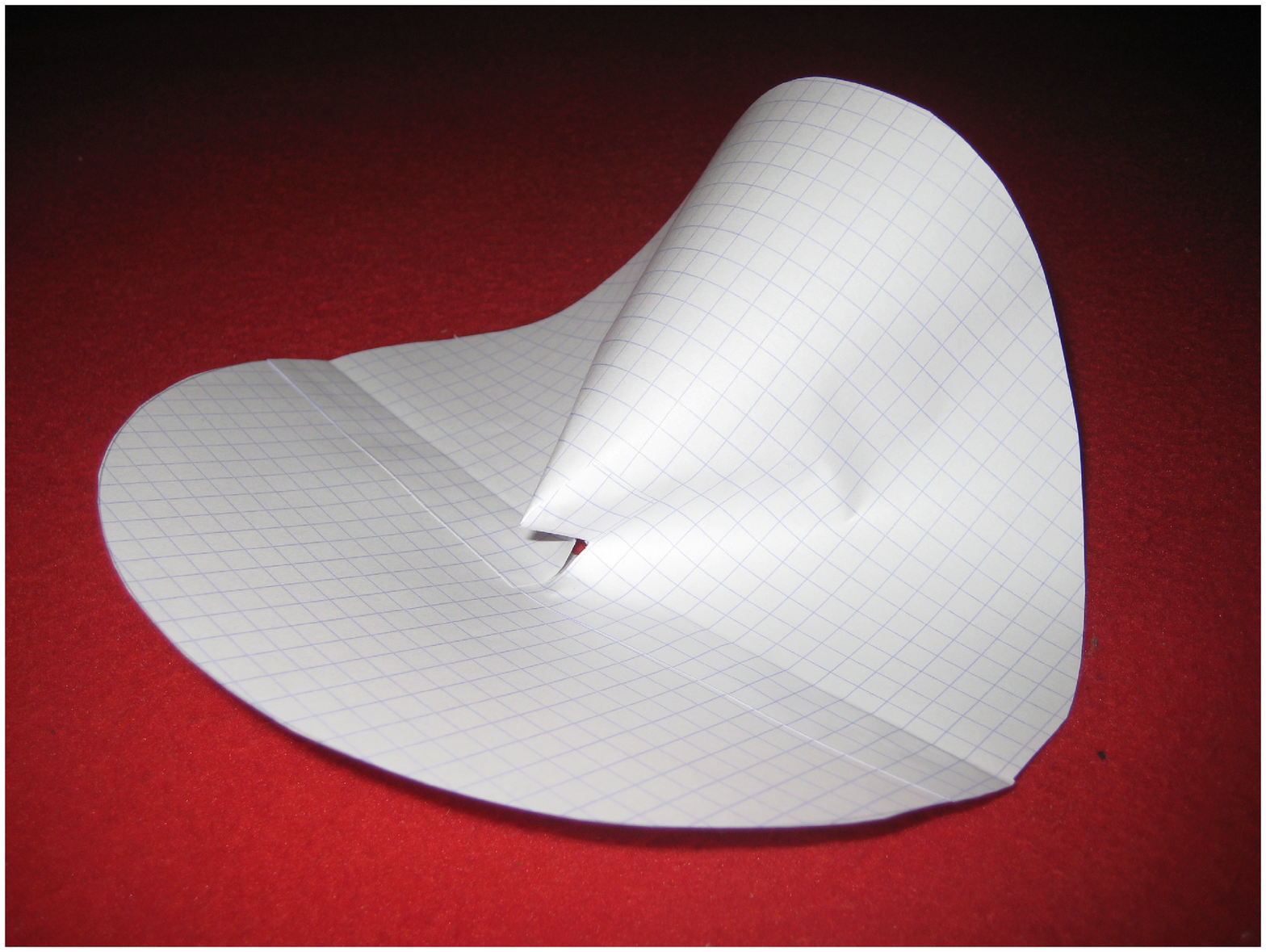}
\includegraphics{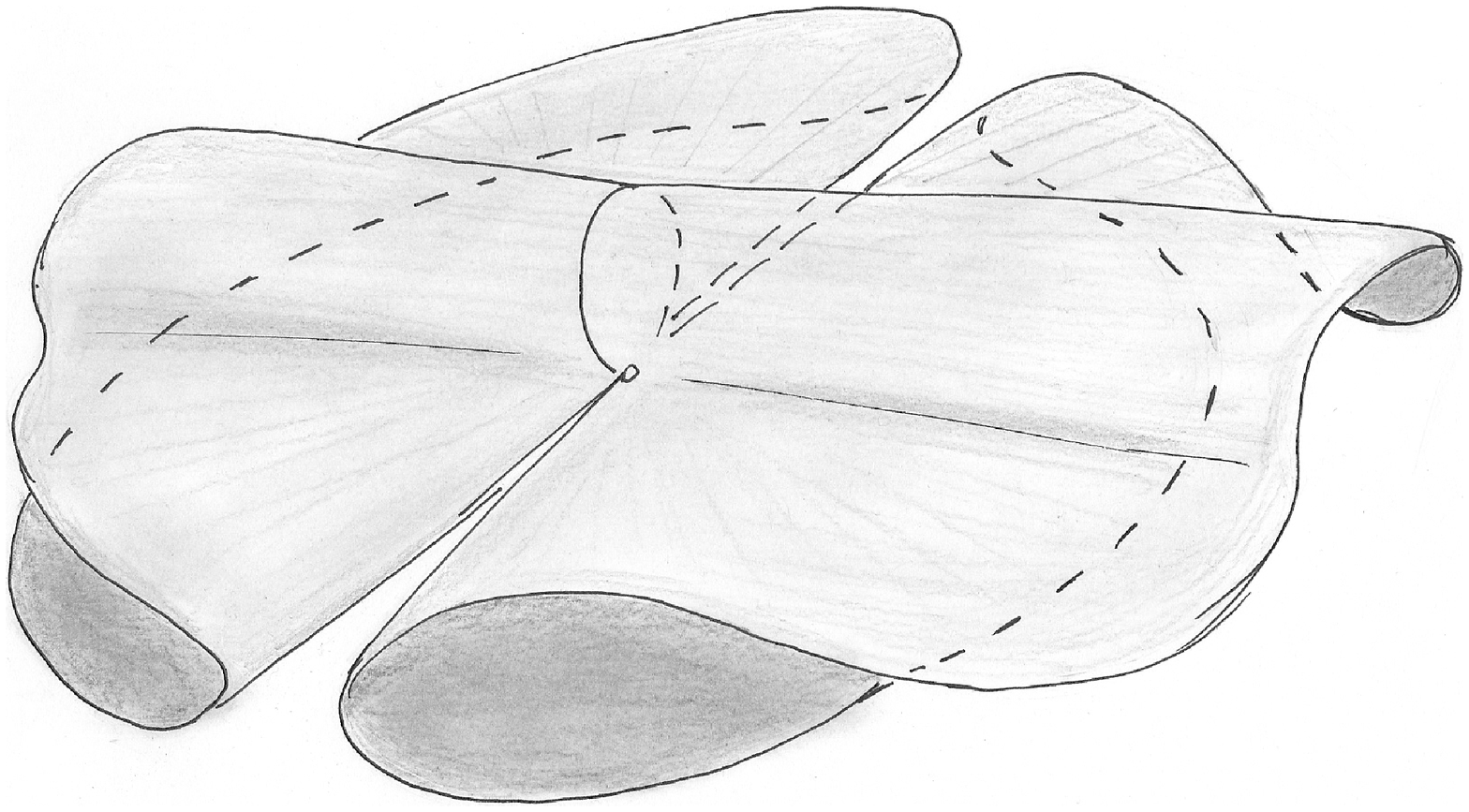}
\begin{picture}(0,0)(0,0)
\put(-140,-105){a.}
\put(-15,-105){b.}
\put(110,-105){c.}
\end{picture}
\vspace{100pt}
\caption{
\label{fig:size:of:pair:of:pants}
Different types of flat pairs of pants.}
\end{figure}

Consider the subcase, when among $d_1,d_2,d_3$ there are two entries,
say   $d_1,d_2$,   satisfying  the  inequality  $d_1,d_2\ge  -1$.  If
$d_1=d_2=-1$, then then $Y$ is metrically equivalent to the following
surface.  Take  a  flat cylinder and isometrically identify a pair of
symmetric  semi-circles  on  one  of  its  boundary  components,  see
Figure~\ref{fig:size:of:pair:of:pants}.a.  We get a saddle connection
joining  a  pair  of  simple  poles  as a boundary on one side of the
cylinder  and  an  ``open end'' on the other side. The size of $Y$ is
represented  by  the  flat length of the waist curve of the cylinder,
which  is  twice  longer  than  the  corresponding  saddle connection
joining the two simple poles.

If,  say,  $d_1\ge  0$,  and $d_2\ge -1$, the situation is completely
analogous  except  that  now  $Y$  is metrically equivalent to a flat
expanding  annulus  with a pair of singularities of degrees $d_1,d_2$
inside  it.  The  size  of  $Y$  is  twice  the  length of the saddle
connection        joining        these       singularities,       see
Figure~\ref{fig:size:of:pair:of:pants}.b.

Finally,  there  remains  the  case  when  there are two values, say,
$d_2,d_3$,  out  of three, satisfying the inequality $d_2,d_3\le -2$,
then  the  third  value,  $d_1$,  necessarily  satisfy the inequality
$d_1\ge  1$  (note  that $d_1$ cannot be equal to zero). In this case
$Y$  is  metrically equivalent to a pair of expanding annuli attached
to  a  common  saddle  connection  joining  a  zero of order $d_1$ to
itself, see Figure~\ref{fig:size:of:pair:of:pants}.c. The size of $Y$
coincides with the length of this saddle connection.

\subsection{Geometric Compactification Theorem}
\label{ss:Compactification:Theorem}

Recall,                        that                        throughout
section~\ref{sec:Geometric:Compactification:Theorem}   we   consider  a
wider  class  of  flat  metrics,  namely,  we  consider  flat metrics
corresponding to meromorphic quadratic differentials (and meromorphic
$1$-forms)  having poles of arbitrary order. We also deviate from the
usual  convention denoting by the same symbol $\cQ(d_1,\dots,d_\noz)$
strata  of  meromorphic  differentials  even  when they correspond to
``strata of global squares of $1$-differentials''.

Now    we    are    ready    to    formulate   a   version   of   the
Deligne--Mumford--Grothendieck  Compactification Theorem in geometric
terms. As remarked above, this theorem is implicit in the statement of
the theorem of K.~Rafi.

\begin{Theorem}
\label{theorem:Deligne:Mumford}
Consider  a  sequence of flat surfaces $S_\tau=(C_\tau,q_\tau)$ where
meromorphic  quadratic differentials $q_\tau$ stay in a fixed stratum
$\cQ(d_1,\dots,d_\noz)$. Suppose that the underlying Riemann surfaces
$C_\tau$  converge  to  a  stable  Riemann surface
$C_\infty$. Choose $\delta_0$ so that $\delta_0$ is
smaller then half the injectivity radius (in the hyperbolic metric)
of any desingularized irreducible component $C_{\infty,j}$
of $C_\infty$.
Let  $Y_{\tau,j}$  be  the  component corresponding to $C_{\infty,j}$
in a
$\delta_0$-thick-thin  decomposition  of  $C_\tau$;
let $\lambda(Y_{\tau,j})$ be
the size of a flat subsurface $(Y_{\tau,j}\,,\,q_\tau)$. Denote
$$
\tilde q_{\tau,j}:=
\cfrac{1}{\lambda(Y_{\tau,j})^2}\cdot q_\tau\,.
$$

There is a subsequence $S_{\tau'}=(C_{\tau'},q_{\tau'})$ and a
nontrivial meromorphic quadratic differential $\tilde q_{\infty,j}$ on
$C_{\infty,j}$ such that the $\tilde q_{\tau',j}$-representatives
$\tilde\Y_{\tau',j}$ of the corresponding thick components
$Y_{\tau',j}$ of the flat surfaces $(C_{\tau'},\tilde q_{\tau',j})$
converge to the $\tilde q_{\infty,j}$-representative
$\tilde\Y_{\infty,j}$ of the flat surface $(C_{\infty,j}\,,\,\tilde
q_{\infty,j})$. Furthermore, the conformal structures
  on $C_{\tau,j}$ converge to the conformal structure of
  $C_{\infty,j}$, and the quadratic differentials $\tilde{q}_{\tau,j}$
  converge to the limiting quadratic differential
  $\tilde{q}_{\infty,j}$ on compact subsets of $C_{\infty,j}$.

With the possible exception of the nodes of
  $C_{\infty,j}$ all zeroes and poles of $\tilde{q}_{\infty,j}$ are
  limits of zeroes and poles of the prelimit differentials
  $\tilde{q}_{\tau,j}$.  If all meromorphic quadratic differentials
$q_\tau$ are global squares of meromorphic 1-forms $\omega_\tau$, then
the limiting quadratic differential $\tilde q_{\infty,j}$ is also a
global square of a meromorphic 1-form $\tilde\omega_{\infty,j}$ on
$C_{\infty,j}$.
\end{Theorem}

\begin{NNRemark}
Completing  the  current paper we learned that analogous results were
simultaneously   and  independently  obtained  by  S.~Grushevsky  and
I.~Krichever    in~\cite{Grushevsky:Krichever}, by S.~Koch and
J.~Hubbard~\cite{Koch:Hubbard},    and   by
J.~Smillie~\cite{Smillie}.
\end{NNRemark}

We  start with the following Lemma which will be used in the proof of
Theorem~\ref{theorem:Deligne:Mumford}.

\begin{lemma}
\label{lemma:uniform:triangulation}
For  every  thick  component $Y$ of a thick-thin decomposition of $S$
the
$q$-geodesic representative $\Y$ can be triangulated by adding $C_1$
saddle connections $\gamma$, each satisfying the the flat length estimate:
\begin{equation}
\label{eq:flat:length:to:size}
  \cfrac{\lambda(Y)}{2}\le
|\gamma|\le C_2 \lambda(Y)\,,
\end{equation}
where   the   constants $C_1$ and $C_2$   depend   only   on   the
ambient  stratum $\cQ(d_1,\dots,d_\noz)$ of $S$.
\end{lemma}

\begin{proof}
We build this triangulation inductively. At each stage
  we have a partial triangulation of $\Y$. If some complementary
  region is not a triangle, it contains a saddle connection whose
  associated closed curve $\gamma'$  \mc{explain this}
is essential, i.e.
not homotopic to a
  boundary component of $\Y$. Let $\gamma$ be the shortest saddle
  connection with this property. Then the flat length of $\gamma'$,
  which is twice the flat length of $\gamma$ is
  bounded from below by the size $\lambda(Y)$ (by the definition of
  size). Also, the flat length of $\gamma$ is bounded above by the
diameter  $\diam_q(\Y)$  of  $\Y$  in the flat metric
defined   by   $q$.   By   the  Theorem  of  K.~Rafi  (see  Theorem~4
in~\cite{Rafi})
$$
\diam_q(\Y)\le \const\cdot\lambda(Y)\,
$$
and thus, (\ref{eq:flat:length:to:size}) holds. This process has to
terminate after finitely many steps (depending only on the stratum)
since the Euler characteristic is finite.  \mc{explain}
Thus the lemma holds.
\end{proof}

\begin{proof}[Proof of Theorem~\ref{theorem:Deligne:Mumford}]
For  each  component  of  the  stable  Riemann  surface  consider the
associated hyperbolic metric, and consider the length of the shortest
closed  geodesic in this metric.  Let  $L$  be the minimum of these lengths over all
components.  We  choose  $\delta$ in such way that $\delta\ll L$. For
each surface $S_\tau$ we consider a decomposition into $\delta$-thick
components as in section~\ref{ss:after:Rafi}.

Since  the  Riemann surfaces $C_\tau$ converge to $C_\infty$, we know
that,   for  sufficiently  large  $\tau$,  the topology  of  $Y_{\tau,j}$
coincides  with the topology of $C_{\infty,j}$ punctured at the points of
crossing   with  other  components  of  the  stable  Riemann  surface
$C_\infty$  and  at  the  points  of  self-intersection.  Hence,  for
sufficiently    large   $\tau$   the   $\tilde   q_{\tau,j}$-geodesic
representative  $\Y_{\tau,j}$  of  the  thick  component $Y_{\tau,j}$
might  have  only  finite  number  of combinatorial types of
triangulations as in Lemma~\ref{lemma:uniform:triangulation}.
Passing  to a subsequence we fix the combinatorial type of
the triangulation.

Such  a  triangulation  contains  a finite number of edges. Hence, by
Lemma~\ref{lemma:uniform:triangulation}  we may chose a subsequence
for which lengths of all sides of the triangulation of $\Y_{\tau',j}$
converge.    Note   that   by   continuity,   the   limiting   length
$\gamma_\infty$ of each side satisfies:
$$
|\gamma_\infty|\le \const\ .
$$
Hence, the limiting triangulation defines some flat structure sharing
with  $\Y_{\tau',j}$  the  combinatorial  geometry  of the triangulation.
Clearly,  the linear holonomy of the limiting flat metric is the same
as the the linear holonomy of the prelimiting flat metrics.

By construction, the underlying Riemann surface for the limiting flat
surface  is  $C_{\infty,j}$.  Thus,  to  complete  the  proof  it  is
sufficient  to  consider a meromorphic quadratic differential $\tilde
q_{\infty,j}$ representing the limiting flat structure.
Since $C_{\infty,j}$ and $C_{\tau,j}$ for large
$\tau$ have triangulations which are close, if we remove the
neighborhoods of the cusps of $C_{\infty,j}$, there is a
quasiconformal map with dilatation close to $1$ taking $C_{\tau,j}$
to $C_{\infty,j}$ which is close to the identity on compact sets.
This implies that $C_{\tau,j}$ converge to $C_{\infty,j}$ as Riemann
surfaces, and also that $\tilde{q}_{\tau,j}$ converge to
$\tilde{q}_{\infty,j}$.
\end{proof}

\begin{Remark}
\label{remark:deligne:mumford}
Note that the quadratic differentials $\tilde q_{\tau',j}$ defined in
the  statement of Theorem~\ref{theorem:Deligne:Mumford} might tend to zero
or   to   infinity   while   restricted  to  other  thick  components
$Y_{\tau',k}$,  where  $k\neq  j$.  To  get  a  well-defined limiting
quadratic  differentials  on  each  individual  component  one has to
rescale   the   quadratic   differentials   $q_{\tau'}$  individually
component by component. As an illustration the reader may consider an
example  at  the  very  end  of the paper~\cite{Rafi} of K.~Rafi.
\end{Remark}

\subsection{The $(\delta,\eta)$-thick-thin decomposition}
\label{sec:subsec:delta:eta:thick:thin}
Suppose $\delta >0$ in the choice of the thick-thin decomposition
is sufficiently small, and fix $\eta$ (depending only on the genus and
the number of punctures) so that $\delta \ll \eta \ll 1$.
In particular, we choose $\eta$ to be smaller than the Margulis
constant.
We  work  in  terms  of  a hyperbolic metric with cusps
$\ghyp(S)$. Consider an $(\eta,\delta)$-thick-thin decomposition
of  the surface  $S$.  Namely,  for each short closed
geodesic  $\gamma\in\Gamma(\delta)$  consider  the set of points in the surface
located  at  a bounded distance from $\gamma$. When the bound for the
distance  is  not  too large, we get a topological annulus. We choose
the  bounding distance to make the length of each of the two boundary
components  of  the  annulus equal to the chosen constant $\eta$. Let
$A_\gamma(\eta)$ denote this annulus. If we remove these annuli from
$S$, $S$ becomes disconnected; the connected components which we
denote by $Y_j(\eta)$ are subsets of the $\delta$-thick
components $Y_j$ defined in
\S\ref{sec:Geometric:Compactification:Theorem}. We have
\begin{displaymath}
S = \left( \bigcup_{j=1}^m Y_j(\eta) \right) \cup \left(
  \bigcup_{\gamma \in \Gamma(\delta)} A_\gamma(\eta) \right).
\end{displaymath}
We note that the \textit{$(\delta,\eta)$-thick} components $Y_j(\eta)$
and  the \textit{$(\delta,\eta)$-thin} components $A_\gamma(\delta)$
depend only on the hyperbolic metric on $S$, and not the quadratic
differential $q$.
\begin{figure}[htb]
\includegraphics{EKZ_5_2.ps}
\includegraphics{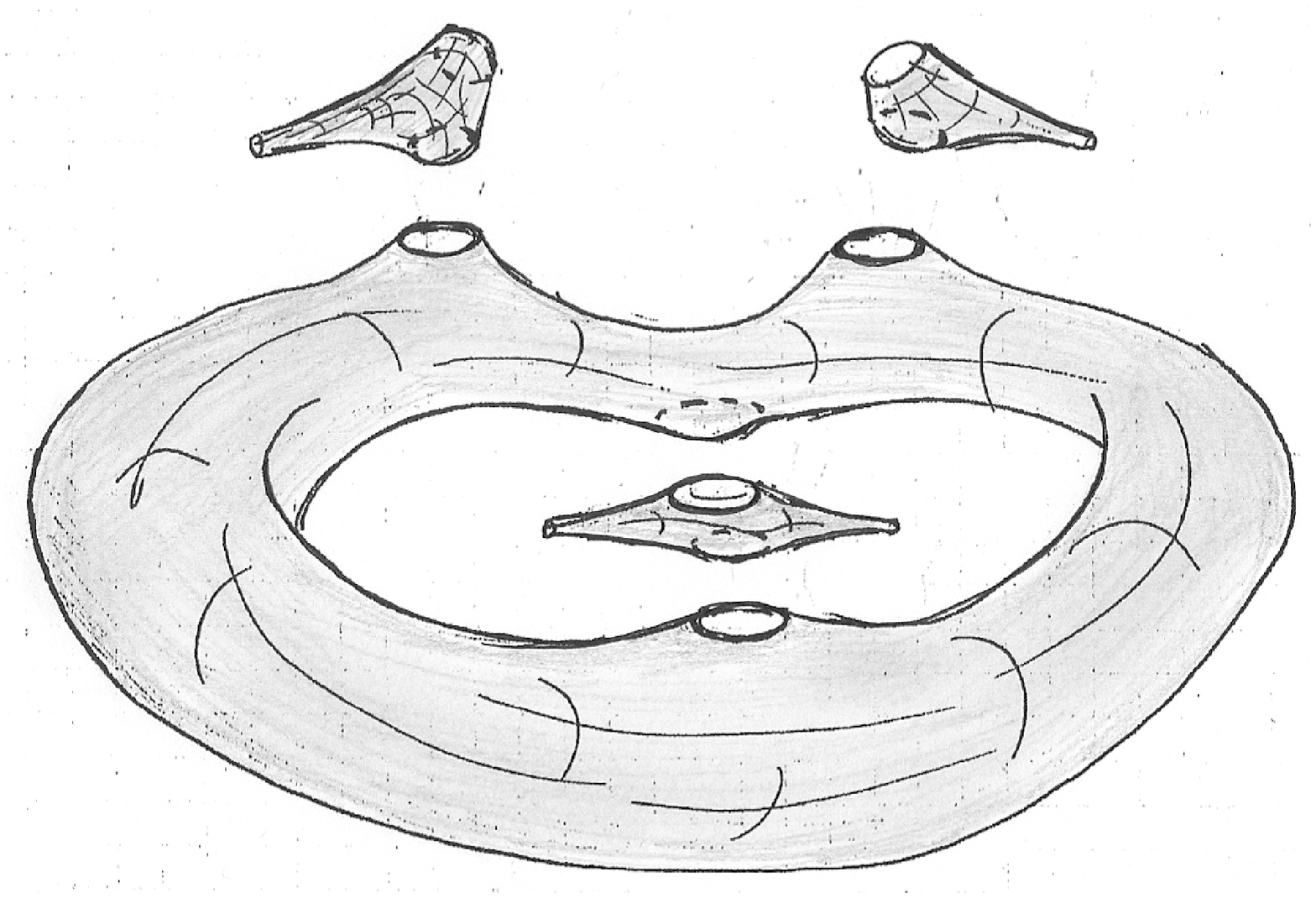}
\vspace{140pt}
\caption{
\label{fig:thick:thin:decomposition}
$(\delta,\eta)$-thick components in hyperbolic metric}
\end{figure}

\subsection{Uniform bounds for the conformal factor}
\label{sec:subsec:uniform:bound:conformal}
For $R > 0$ and a cusp $P$, let $\cO_P(R)$ denote the neighborhood
$\{ \zeta \mid |\zeta| < R \}$ where $\zeta$ is as in
(\ref{eq:def:local:coordinate:zeta}).
In this subsection we fix a constant $R$ (depending only on
$\delta$, $\eta$ and the stratum) such that for any hyperbolic surface
$S$ and each hyperbolic cusp $P$ of $S$, the neighborhood $\cO_P(R)$
 does not intersect any of the
$(\delta,\eta)$-thin components $A_\gamma(\eta)$, $\gamma \in
\Gamma(\delta)$, and also for distinct cusps $P$ and $Q$, the
neighborhoods $\cO_P(R)$ and
$\cO_Q(R)$ are disjoint.
\medskip

The following Proposition is a variant of the Theorem of K.~Rafi
stated in the beginning of \S\ref{sec:Geometric:Compactification:Theorem}.
\begin{proposition}
\label{prop:pointwise:conformal:factor}
Let $S=(C,q)$ be a flat surface, and let $Y$ be a $\delta$-thick component
of $S$. Let $Y(\eta)$ be the corresponding
$(\delta,\eta)$-thick component of $S$ (defined as in
\S\ref{sec:subsec:delta:eta:thick:thin}). For each $P \in Z(Y(\eta))$,
let $\cO_P(R)$ be the neighborhood of $P$ as defined in the beginning
of \S\ref{sec:subsec:uniform:bound:conformal}.
Then, there exists a constant $C'$
depending only on the stratum, $\delta$, $\eta$
and $R$ such that
for all $x \in Y(\eta) \setminus \bigcup_{P\in Z(Y)} \cO_P(R)$,
\begin{displaymath}
| \cf(q)(x) - \log \lambda(Y)| \le C',
\end{displaymath}
where $\cf(q)$ is the conformal factor of $q$
defined by $g_{flat}(q) = \exp (2 \phi(q)) g_{hyp}$.
\end{proposition}

\begin{proof}
The proof will be by contradiction. Suppose there is no
  $C'$ satisfying the conditions of the lemma. Then there exists
  sequence of triples $(x_\tau, Y_\tau, S_\tau)$ such that $Y_\tau
  \subset S_\tau$ is a $\delta$-thick subsurface, $x_\tau \in
  Y_\tau(\eta)$, and if we write $S_\tau = (C_\tau,q_\tau)$ then
\begin{equation}
\label{eq:sequence:phi:counterexamples}
| \cf(q_\tau)(x_\tau) - \log \lambda(Y_\tau)| \to \infty.
\end{equation}
After passing to a subsequence, we may assume that the flat surfaces
$S_\tau$ converge in the sense of
Theorem~\ref{theorem:Deligne:Mumford}. Let $C_{\infty}$, $\delta_0$,
$C_{\infty,j}$, $Y_{\tau,j}$, $\tilde{q}_{\tau,j}$ be as in the
statement of Theorem~\ref{theorem:Deligne:Mumford}. One technical
issue is that the constant $\delta_0$ (which depends on the sequence
$S_\tau$) might not coincide with the constant $\delta > 0$ which is
chosen in advance; in particular, we may have $\delta_0 < \delta$.

Since $Y_\tau$ is a thick component of $S_\tau$, for large enough
$\tau$ no boundary curve of
one of the $Y_{\tau,j}$ (which are all in $\Gamma(\delta_0)$) can cross
the interior of $Y_\tau$; therefore the subsurface $Y_\tau$ must be
contained in one of the $Y_{\tau,j}$ where $j = j(\tau)$; however
after passing to a subsequence we may assume that $j$ is fixed. Even
then, we might not have $Y_\tau = Y_{\tau,j}$ since all we know about
the boundary curves of $Y_\tau$ is that they have hyperbolic length at
most $\delta$, while by definition, the hyperbolic length of the boundary
curves of $Y_{\tau,j}$ tends to $0$ as $\tau \to \infty$. However, we
claim that
\begin{equation}
\label{eq:size:subsubsurface}
\limsup_{\tau \to \infty} |\log \lambda(Y_\tau) - \log
\lambda(Y_{\tau,j})| < \infty.
\end{equation}
Indeed, if (\ref{eq:size:subsubsurface}) failed, then (after passing
to a subsequence) by~\cite[Lemma~4.9]{Eskin:Mirzakhani:Rafi},
the subsurface $Y_{\tau,j}$ would contain a curve
$\gamma_{\tau,j}$
with the hyperbolic length of $\gamma_{\tau,j} \to 0$; this contradicts the
fact that $Y_{\tau,j} \to C_{\infty,j}$ where $C_{\infty,j}$ is
connected. Therefore
(\ref{eq:size:subsubsurface}) holds. It follows from
(\ref{eq:sequence:phi:counterexamples}) and
(\ref{eq:size:subsubsurface}) that
\begin{equation}
\label{eq:sequence:phi:newsize:counterexamples}
| \phi(q_\tau)(x_\tau) - \log \lambda(Y_{\tau,j})| \to \infty.
\end{equation}
As in Theorem~\ref{theorem:Deligne:Mumford}, let
\begin{displaymath}
\tilde{q}_{\tau,j} = \lambda(Y_{\tau,j})^{-2} q_\tau
\end{displaymath}
By Theorem~\ref{theorem:Deligne:Mumford}, we have $\tilde{q}_{\tau,j}
\to \tilde{q}_{\infty,j}$ on uniformly on compact subsets of
$C_{\infty,j}$
After passing to a subsequence, we have $x_\tau \to x_\infty \in
C_{\infty,j}$. Since $x_\tau$ stays away from $Z(Y_{\tau,j})$,
$x_\infty \not\in Z(Y_{\infty,j})$.
Also, since $x_\tau \in Y_{\tau,j}(\eta)$, $x_\infty$ is not one of the
nodes. Since $\tilde{q}_{\infty,j}$ is finite and does not vanish except at ponts of $Z(Y_{\infty,j})$ and the nodes, we see that $\tilde{q}_{\infty,j}(x_\infty) \ne 0$.

Recall  that  we  represent  the  conformal  factor  relating the flat and
hyperbolic  metrics as $\gflat(q_\tau)=\exp(2\cf(q_\tau))\ghyp$.
Therefore,
\begin{displaymath}
\cf(\tilde{q}_{\tau,j})(x_\tau) \to
\cf(\tilde{q}_{\infty,j})(x_\infty) \ne 0.
\end{displaymath}
Hence,
\begin{equation}
\label{eq:conf:factor:tildeq:converges}
\limsup_{\tau \to \infty}
|\phi(\tilde{q}_{\tau,j})(x_\tau) | < \infty.
\end{equation}
Recall that by definition $\tilde q_{\tau,j}:=\lambda(Y_{\tau,j})^{-2} q_\tau$.
Note
that  multiplying  $q_\tau$  by a constant factor $k$ we do not
change the hyperbolic metric, $\ghyp(kq_\tau)=\ghyp(q_\tau)$. Thus,
\begin{equation}
\label{eq:cf:formula}
\cf(\tilde{q}_{\tau,j})=\cf(q_\tau)- \log \lambda(Y_{\tau,j}).
\end{equation}
Now (\ref{eq:cf:formula}) and (\ref{eq:conf:factor:tildeq:converges})
contradict (\ref{eq:sequence:phi:newsize:counterexamples}).
\end{proof}

\section{Analytic Riemann--Roch Theorem}
\label{sec:analytic:Riemann:Roch:proof}

This    section    is    entirely    devoted    to    a    proof   of
Theorem~\ref{theorem:main:local:formula}.                          In
section~\ref{ss:proof:Fay} we present our original proof based on the
results of J.~Fay~\cite{Fay}.

Having  seen  a  draft  of  the  paper, D.~Korotkin indicated us that
Theorem~\ref{theorem:main:local:formula}   should   be  an  immediate
corollary     of     the     holomorphic     factorization    formula
from~\cite{Kokotov:Korotkin} combined with the homogeneity properties
of the tau-function established in~\cite{Korotkin:Zograf}. We present
a          corresponding         alternative         proof         in
section~\ref{ss:Kokotov:Korotkin}.

\subsection{Proof based on the results of J.~Fay}
\label{ss:proof:Fay}

Recall   the   setting  of  Theorem~\ref{theorem:main:local:formula}.
Consider  a  flat  surface  $S$  of  area one in a stratum of Abelian
differentials  or in a stratum of meromorphic quadratic differentials
with  at  most simple poles. In the current context we are interested
only  in the underlying flat metric, so we forget about the choice of
the  vertical  direction.  In other words, we do not distinguish flat
surfaces   corresponding   to   Abelian  differentials  $\omega$  and
$\exp(i\varphi)\omega$  (correspondingly  quadratic differentials $q$
and $\exp(2i\varphi)q$), where $\varphi$ is a constant real number. e
consider the flat surface W$S$ as a point of the quotient
$$
\cH_1(m_1,\dots,m_\noz)/\SO\simeq\
\PcH(m_1,\dots,m_\noz)
$$
or
$$
\cQ_1(d_1,\dots,d_\noz)/\SO\simeq\
\PcQ_1(d_1,\dots,d_\noz)
$$
correspondingly.

Consider   a   complex  one-parameter  family  of  local  holomorphic
deformations    $S(t)$    of    $S$    in    the    ambient   stratum
$\cH(m_1,\dots,m_\noz)$  or  $\cQ(d_1,\dots,d_\noz)$ correspondingly.
Denote  by  $z$ a flat coordinate on the initial flat surface, and by
$u$  denote  a flat coordinate on the deformed flat surface. The area
of the deformed flat surface $S(t)$ is not unit anymore. We denote by
$S^{(1)}(t)$ the flat surface of area one obtained from $S(t)$ by the
proportional  rescaling.  Smoothing the resulting flat metric of area
one  as it was described in section~\ref{sec:def:det:flat:metric}, we
get the smoothed flat metric $\rho_\epsilon(u,\bar u)|du|^2$.

By   $\omega_i$,   $i=1,\dots,g$,   we   denote   local  nonvanishing
holomorphic    sections   of   the   Hodge   bundle   $H^{1,0}$,   so
$\det^{\frac{1}{2}}\langle\omega_i,\omega_j\rangle$    is   a   local
holomorphic  section  of  the determinant bundle $\Lambda^g H^{1,0}$,
and $|\det\langle\omega_i,\omega_j\rangle|$ is the square of its norm
induced    by    Hermitian    metric~\eqref{eq:hodge:hermitian:form},
see~\eqref{eq:w:i:w:j}.

The  starting  point of the proof is the following reformulation of a
formula of J.~Fay.

\begin{Proposition}[after J.~Fay]
The following relation is valid
\begin{multline}
\label{eq:intermediate}
\partial_{\bar t}\partial_t\log\det\Delta_{\mathit{flat},\epsilon}\big(S^{(1)}(t)\big)-
\partial_{\bar t}\partial_t\log|\det\langle\omega_i,\omega_j\rangle|=\\
=\cfrac{1}{6\pi}\,
\int_C
\det\left(\begin{array}{c|c}
\partial_{\bar t}\partial_t\log \rho_\epsilon &
\partial_t\partial_{\bar u} \log \rho_\epsilon\\
[-\halfbls] &\\
\hline  & \\
[-\halfbls]
\partial_{\bar t}\partial_u \log \rho_\epsilon &
\partial_z\partial_{\bar z} \log \rho_\epsilon
\end{array}
\right)
dx\,dy\,.
\end{multline}
where  the  derivatives  of functions of the local coordinate $u$ are
evaluated at $t=0$.
\end{Proposition}
\begin{proof}
Actually,  formula~\eqref{eq:intermediate}  above  is  formula~(3.37)
in~\cite{Fay} adjusted to our notations.

A  vector  bundle $L_t$ in formula~(3.37) in~\cite{Fay} is trivial in
our  case.  This  means that the metric $h$ on it is also trivial and
equals  identically  one: $h=1$. The same is true for the determinant
$\det\langle\omega_i,\omega_j\rangle_{L_t}$     in     formula~(3.37)
in~\cite{Fay};  this  determinant  is identically equal to one in our
case.

A vector bundle $K_t\otimes L^\ast_t$ in formula~(3.37) in~\cite{Fay}
becomes in our context the vector bundle $H^{1,0}$ and a basis in the
fiber    of    this    vector   bundle   denoted   in~\cite{Fay}   by
$\{\omega^\ast_k\}$   becomes  a  basis  of  holomorphic  1-forms  in
$H^{1,0}(C(t))$,  denoted  in  our  notations  by $\omega_k(t)$ where
$C(t)$  is  a  Riemann  surface  underlying the deformed flat surface
$S(t)$.  Note  that each $\omega_k(t)$ considered as a section of the
holomorphic  vector  bundle  $H^{1,0}$ is holomorphic with respect to
the parameter of deformation $t$.

Note,   that   we   represent  the  metric  as  $\rho_\epsilon(u,\bar
u)\,|du|^2$ while in the original paper~\cite{Fay} the same metric is
written as $\rho^{-2}|du|^2$. This explains an extra factor of $4$ in
the  denominator  of  $1/(4\pi)$ in formula~\eqref{eq:Fay:3:37} below
with respect to the original formula~(3.37) in~\cite{Fay}.

Finally, using that
$$
\rho_\epsilon\partial(\rho_\epsilon^{-1})=-\partial\log \rho_\epsilon\,.
$$
we  can  rewrite  formula~(3.37)  in~\cite{Fay}  in  our notations as
\begin{multline}
\label{eq:Fay:3:37}
\partial_{\bar t}\partial_t\,
\log\left(\cfrac{\det\Delta_{\mathit{flat},\epsilon}
\big(S^{(1)}(t)\big)}
{|\det\langle\omega_i,\omega_j\rangle|}\right)
\ =\\ \ =
\frac{1}{4\pi}\int_{S}\Big(
(\partial_{\bar t}\partial_t\log \rho_\epsilon)\;
(\partial_{\bar z}\partial_z \log \rho_\epsilon) -
(\partial_t\partial_{\bar u}\log \rho_\epsilon)\;
(\partial_{\bar t}\partial_u\log \rho_\epsilon)-\\
-\frac{1}{3}(\partial_t\partial_{\bar t}\log \rho_\epsilon)\;
(\partial_z\partial_{\bar z}\log \rho_\epsilon)+
\cfrac{1}{3}(\partial_t\partial_{\bar u}\log \rho_\epsilon)\;
(\partial_{\bar t}\partial_u\log \rho_\epsilon)
\Big)\,dx\,dy\,.
\end{multline}
Simplifying     the     expression    in    the    right-hand    side
of~\eqref{eq:Fay:3:37},  we  can  rewrite  the  latter formula in the
form~\eqref{eq:intermediate}.
\end{proof}

\begin{Lemma}
\label{lm:Fay}
In  the  same  setting  as  above  the  following  formula  is  valid
\begin{multline}
\label{eq:with:limit}
\partial_{\bar t}\partial_t\,
\log\det\Dflat\big(S^{(1)}(t),S_0\big)
\ =\\=\
\partial_{\bar t}\partial_t\,
\log|\det\langle\omega_i,\omega_j\rangle|
\ +\
\cfrac{1}{6\pi}\,
\lim_{\epsilon\to+0}\,
\int_C
\det\left(\begin{array}{c|c}
\partial_{\bar t}\partial_t\log \rho_\epsilon &
\partial_t\partial_{\bar u} \log \rho_\epsilon\\
[-\halfbls] &\\
\hline  & \\
[-\halfbls]
\partial_{\bar t}\partial_u \log \rho_\epsilon &
\partial_z\partial_{\bar z} \log \rho_\epsilon
\end{array}
\right)
dx\,dy
\end{multline}
where  all  derivatives  of functions of the local coordinate $u$ are
evaluated at $t=0$.
\end{Lemma}
\begin{proof}
Combine           the          latter          equation          with
definition~\eqref{eq:def:relative:Laplace:flat}                    of
$\det\Dflat\big(S^{(1)}(t),S_0\big)$   and   pass  to  the  limit  as
$\epsilon\to+0$.
\end{proof}

Now  let  us  specify  the  holomorphic  1-parameter family $S(t)$ of
infinitesimal deformations of the flat surface $S=S(0)$.

When  the  flat surface $S$ is represented by an Abelian differential
$\omega$   in   a  stratum  $\cH(m_1,\dots,m_\noz)$  we  consider  an
infinitesimal   affine  line  $\gamma(t)$  defined  in  cohomological
coordinates
$$
(Z_1,\dots,Z_{2g+\noz-1})\in
H^1(S,\{\text{zeroes of }\omega\};\C{})
$$
by the parametric system of equation
\begin{align}
\label{eq:deformation:in:periods}
Z_j(t):=&a(t)Z_j(0)+b(t)\bar Z_j(0)\,,\quad\text{for }j=1,\dots,2g+\noz-1\,,\\
\notag
&\text{where }\ a(0)=1,\ b(0)=0,\ \text{ and }\ b'(0)\neq 0\,.
\end{align}

When  the  flat surface $S$ is represented by a meromorphic quadratic
differential $q$ in a stratum $\cQ(d_1,\dots,d_\noz)$, we consider an
infinitesimal   affine  line  $\gamma(t)$  defined  in  cohomological
coordinates
$$
(Z_1,\dots,Z_{2g+\noz-2})\in
H_-^1(S,\{\text{zeroes of }\hat\omega\};\C{})
$$
by an analogous parametric system of equations
\begin{align}
\label{eq:deformation:in:periods:quadratic}
Z_j(t):=&a(t)Z_j(0)+b(t)\bar Z_j(0)\,,\quad\text{for }
j=1,\dots,2g+\noz-2\,,\\
\notag
&\text{where }\ a(0)=1,\ b(0)=0,\ \text{ and }\ b'(0)\neq 0\,.
\end{align}

The      next      Proposition     evaluates     the     limit     in
equation~\eqref{eq:with:limit}         for         families        of
deformations~\eqref{eq:deformation:in:periods}
and~\eqref{eq:deformation:in:periods:quadratic}.

\begin{Proposition}
\label{pr:integrating:Fay}
In the same setting as above the following formulae hold.

For a family of deformations~\eqref{eq:deformation:in:periods} of the
initial  flat surface $S$ inside a stratum $\cH(m_1,\dots,m_\noz)$ of
Abelian differentials one has:
\begin{equation}
\label{eq:lim:for:Abelian}
\lim_{\epsilon\to+0}\int_C
\det\left(\begin{array}{c|c}
\partial_t\partial_{\bar t}\log \rho_\epsilon &
\partial_t\partial_{\bar u} \log \rho_\epsilon\\
[-\halfbls] &\\
\hline  & \\
[-\halfbls]
\partial_{\bar t}\partial_u \log \rho_\epsilon &
\partial_z\partial_{\bar z} \log \rho_\epsilon
\end{array}
\right)
dx\,dy =
\pi\cdot\sum_{j=1}^\noz \cfrac{m_j(m_j+2)}{2(m_j+1)}
\cdot|b'(0)|^2
\end{equation}

For                   a                   family                   of
deformations~\eqref{eq:deformation:in:periods:quadratic}    of    the
initial  flat surface $S$ inside a stratum $\cQ(d_1,\dots,d_\noz)$ of
meromorphic  quadratic  differentials  with  at most simple poles one
has:
\begin{equation}
\label{eq:lim:for:quadratic}
\lim_{\epsilon\to+0}
\int_C
\det\left(\begin{array}{c|c}
\partial_t\partial_{\bar t}\log \rho_\epsilon &
\partial_t\partial_{\bar u} \log \rho_\epsilon\\
[-\halfbls] &\\
\hline  & \\
[-\halfbls]
\partial_{\bar t}\partial_u \log \rho_\epsilon &
\partial_z\partial_{\bar z} \log \rho_\epsilon
\end{array}
\right)
dx\,dy =
\pi\cdot\sum_{j=1}^\noz \cfrac{d_j(d_j+4)}{4(d_j+2)}
\cdot|b'(0)|^2
\end{equation}
\end{Proposition}
\begin{proof}
We  are  going  to  show  that  the  integral  under consideration is
localized into small neighborhoods of conical singularities, and that
the  integral  over  any  such  neighborhood depends only on the cone
angle  at  the  singularity. In particular, it does not depend on the
holonomy  of the flat metric, so it does not distinguish flat metrics
corresponding  to holomorphic 1-forms and to quadratic differentials.
In   other   words   the   second   formula   in   the  statement  of
Proposition~\ref{pr:integrating:Fay}  is  valid  no  matter whether a
quadratic  differential  is  or  is not a global square of an Abelian
differential. The first formula, thus, becomes an immediate corollary
of  the  second one: if an Abelian differential has zeroes of degrees
$m_1,\dots,m_\noz$,  the quadratic differential $\omega^2$ has zeroes
of  orders  $2m_1,\dots,2m_\noz$. Applying the second formula to this
latter collection of singularities we obtain the first one.

We can represent a holomorphic deformation of the flat coordinate $z$
as follows:
$$
u(z,\bar z,t)=a(t)\,z+b(t)\,\bar z
$$
where  $t\in\C{}$  is a parameter of the deformation and coefficients
are   normalized   as   $a(0)=1$,   $b(0)=0$,   and   $b'(0)\neq  0$,
see~\eqref{eq:deformation:in:periods},
and~\eqref{eq:deformation:in:periods:quadratic}. We get
\begin{equation}
\label{eq:udu}
du\wedge d\bar u=(a\,dz+b\,d\bar z)\wedge(\bar a\,d\bar z+\bar b\,d z)=
(a\bar a-b\bar b)\,dz\wedge d\bar z
\end{equation}
Computing          the          derivatives          we          get:
\begin{align}
\label{eq:derivatives:u:bar:u}
\partial_t u&=a' z+b'\bar z &
\partial_t \bar u&=0\\
\partial_{\bar t} u&=0 &
\partial_{\bar t} \bar u&=\bar a' \bar z+\bar b' z\notag
\end{align}
where $a'=\partial_t\, a(t)$ and
$b'=\partial_t\, b(t)$.

It would be convenient to introduce the following notation: $G(t,\bar
t):=(a\bar a-b\bar b)^{-1}$. Computing the derivatives of $G$ we get:
\begin{align}
\partial_t\, G&= -G^2\cdot (a'\bar a-b'\bar b) &
\partial_t G\big\vert_{t=0} &=-a'(0)\notag\\
\label{eq:derivatives:G}
\partial_{\bar t}\, G&= -G^2\cdot (a\bar a'-b\bar b') &
\partial_{\bar t} G\big\vert_{\bar t=0} &=-\bar a'(0)\\
\partial_t\partial_{\bar t}\, G&= 2G^3\cdot (a'\bar a-b'\bar b)(a\bar a'-b\bar b')-
\notag
\\
& \qquad-G^2(a'\bar a'-b'\bar b')
&
\partial_t\partial_{\bar t} G\big\vert_{t=0} &=a'(0)\bar a'(0)+b'(0)\bar b'(0)\notag
\end{align}

Consider  a  neighborhood $\cO$ of a conical singularity $P$ of order
$d$ on the initial flat surface $S$. Recall that the local coordinate
$w$  in  $\cO$  is  defined  by the equation $(dz)^2=w^{d}\, (dw)^2$,
see~\eqref{eq:def:local:coordinate:w}. The smoothed metric $\gfe$ was
defined  in  $\cO$  as $\gfe=\rhofe(|w|)\,|dw|^2$, where the function
$\rhofe$  is  defined in equation~\eqref{eq:def:rho:flat:epsilon}. In
the   flat   coordinate   $z$   the  smoothed  metric  has  the  form
$\gfe=\rho_\epsilon(|z|)\,|dz|^2$,       where      the      function
$\rho_\epsilon(|z|)$ is defined by the equation
$$
\rhofe(|w|)\,|dw|^2=\rho_\epsilon(|z|)\,|dz|^2
$$
A simple calculation shows that
\begin{equation}
\label{eq:f:d}
\rho_\epsilon(r)=\begin{cases}
1&,\text{ when }r\ge\epsilon\\
\left(\frac{2}{d+2}\right)^2\cdot r^{-\frac{2d}{d+2}}
&,\text{ when }0<r\le\epsilon'\,.
\end{cases}
\end{equation}
Finally,  it  would  be  convenient  to  make  one more substitution,
representing the smoothed metric in $\cO$ as
$$
\rho_\epsilon(|z|)\,|dz|^2 =\exp\big(2\varphi_\epsilon(|z|^2)\big)\,|dz|^2\,.
$$
The above definition of $\varphi_\epsilon$ implies that
\begin{equation}
\label{eq:def:varphi}
\varphi_\epsilon(s)=\begin{cases}
0&,\text{ when }s\ge\epsilon^2\\
\log\left(\frac{2}{d+2}\right)-\frac{d}{2(d+2)}\log s
&,\text{ when }0<s\le(\epsilon')^2\,.
\end{cases}
\end{equation}
We  will  need below the following immediate implication of the above
expression:
\begin{equation}
\label{eq:phi:prime:s}
\varphi_\epsilon'(s)\cdot s=\begin{cases}
0&\text{ when }s\ge\epsilon^2\\
-\frac{d}{2(d+2)}&\text{ when }0<s\le(\epsilon')^2
\end{cases}
\end{equation}

Consider now a neighborhood $\cO$ of a conical singularity $P$ on the
deformed flat surface $S^{(1)}(t)$ with normalized metric. It follows
from~\eqref{eq:udu}   that   smoothed   metric  $\rho_\epsilon(u,\bar
u)|du|^2$ has the form
$$
\rho_\epsilon(u,\bar u) =\exp\big(2\varphi_\epsilon(u\bar u G)\big)
\cdot G(t,\bar t)
$$
in  such  neighborhood.  The second factor $G(t,\bar t)$ in the above
expression is responsible for the normalization
$$
\area\big(S^{(1)}(t)\big)=1
$$
of the total area of the deformed flat surface $S(t)$. Passing to the
logarithm we get
$$
\log \rho_\epsilon(u,\bar u)=
2\varphi_\epsilon(u\bar u G)+\log G\,.
$$
Now  everything  is  ready  to  compute  the  entries  of  the matrix
$$
\left(\begin{array}{c|c}
\partial_t\partial_{\bar t}\log \rho_\epsilon &
\partial_t\partial_{\bar u} \log \rho_\epsilon\\
[-\halfbls] &\\
\hline  & \\
[-\halfbls]
\partial_{\bar t}\partial_u \log \rho_\epsilon &
\partial_z\partial_{\bar z} \log \rho_\epsilon
\end{array}
\right).
$$


\textbf{Entry } $\left(\begin{array}{c|c}
\bullet&\phantom{\bullet}\\\hline&
\end{array}\right).$
\medskip

Evaluating the first derivative $\partial_{\bar t}\log \rho_\epsilon$
we get
$$
\partial_{\bar t}\log \rho_\epsilon=
\partial_{\bar t}\,\Big(2\varphi_\epsilon(u\bar u G) + \log G(t,\bar t)\Big)=
2\varphi_\epsilon'\cdot
\left(u\cfrac{\partial\bar u}{\partial\bar t}\,G+
u\bar u\,\cfrac{\partial G}{\partial\bar t}\right) +
\cfrac{\partial G}{\partial\bar t}\cdot\cfrac{1}{G}
$$

Passing to the second derivative we obtain
\begin{multline*}
\partial_t\partial_{\bar t}\log \rho_\epsilon=
2\varphi_\epsilon''\cdot
\left(\bar u\cfrac{\partial u}{\partial t}\,G+
u\bar u\,\cfrac{\partial G}{\partial t}\right)
\left(u\cfrac{\partial\bar u}{\partial\bar t}\;G+
u\bar u\,\cfrac{\partial G}{\partial\bar t}\right)+\\
+2\varphi_\epsilon'\cdot
\left(
\cfrac{\partial u}{\partial t}\,\cfrac{\partial\bar u}{\partial\bar t}\cdot
G+
\bar u\,\cfrac{\partial u}{\partial t}\,\cfrac{\partial G}{\partial\bar t}+
u\,\cfrac{\partial\bar u}{\partial\bar t}\,\cfrac{\partial G}{\partial t}+
u\bar u\, \cfrac{\partial^2 G}{\partial t\,\partial\bar t}
\right)
+\cfrac{\partial^2 G}{\partial t\partial\bar t}\,\cdot\cfrac{1}{G}-
\cfrac{\partial G}{\partial t}\,\cfrac{\partial G}{\partial\bar t}\cdot
  \cfrac{1}{G^2}
\end{multline*}

Applying   formulae~\eqref{eq:derivatives:u:bar:u}  we  evaluate  the
above expression at $t=0$ getting:
\begin{multline*}
2\varphi_\epsilon''\cdot
\left( (a' z+b'\bar z)\bar z\cdot G
  +z\bar z\,\cfrac{\partial G}{\partial t}\right)
\left( (\bar a'\bar z+\bar b' z)z\cdot G
  +z\bar z\,\cfrac{\partial G}{\partial\bar t}\right)+\\
+\;2\varphi_\epsilon'\cdot
\Bigg((a'z+b'\bar z)(\bar a'\bar z+\bar b' z)\cdot G+
(a'z+b'\bar z)\bar z\,\cfrac{\partial G}{\partial\bar t}\;+\\
+(\bar a'\bar z+\bar b' z) z\,\cfrac{\partial G}{\partial t}+
z\bar z\,\cfrac{\partial^2 G}{\partial t\,\partial\bar t}
\Bigg)\; +\\
+\;\cfrac{\partial^2 G}{\partial t\partial\bar t}\,\cdot\cfrac{1}{G}-
\cfrac{\partial G}{\partial t}\,\cfrac{\partial G}{\partial\bar t}\cdot
  \cfrac{1}{G^2}
\end{multline*}

Applying  formulae~\eqref{eq:derivatives:G} for derivatives of $G$ at
$t=0$ we can rewrite the latter expression as
\begin{multline*}
2\varphi_\epsilon''\cdot
\Big((a' z+b'\bar z)\bar z\cdot 1+z\bar z\cdot (-a')\Big)
\Big((\bar a'\bar z+\bar b' z)z\cdot 1+z\bar z\cdot (-\bar
a')\Big)+\\
+2\varphi_\epsilon'\cdot
\Big((a'z+b'\bar z)(\bar a'\bar z+\bar b' z)\cdot 1+
(a'z+b'\bar z)\bar z\,(-\bar a')+\\
+(\bar a'\bar z+\bar b' z) z(-a')+
z\bar z(a'\bar a'+b'\bar b')
\Big)+\\
+(a'\bar a'+b'\bar b')\cdot 1- (-a')(-\bar a')\cdot 1
\end{multline*}

Simplifying the latter expression we get
\begin{equation}
\label{eq:upper:left:corner}
\partial_t\partial_{\bar t}\log \rho_\epsilon=
b'\bar b'\Big(2\varphi_\epsilon''\cdot(z\bar z)^2+4\varphi'\cdot z\bar z+1\Big)
\end{equation}
\bigskip


\textbf{Entry } $\left(\begin{array}{c|c}
\phantom{\bullet}&\phantom{\bullet}\\\hline&\bullet
\end{array}\right).$
\medskip

For this entry of the determinant we have
$$
\partial_z\partial_{\bar z}\log \rho_\epsilon=
2\varphi_\epsilon''\cdot
\left(\bar u\cfrac{\partial u}{\partial z}+
u\cfrac{\partial\bar u}{\partial z}\right)
\left(\bar u\cfrac{\partial u}{\partial\bar z}+
u\cfrac{\partial\bar u}{\partial\bar z}\right)G^2
+2\varphi_\epsilon'\cdot
\left(
\cfrac{\partial u}{\partial\bar z}\,\cfrac{\partial\bar u}{\partial z}
+\cfrac{\partial u}{\partial z}\,\cfrac{\partial\bar u}{\partial\bar z}
\right)G
$$

Applying~\eqref{eq:derivatives:u:bar:u}  we  can  evaluate  the above
expression at $t=0$ which leads to
$$
\partial_z\partial_{\bar z}\log \rho_\epsilon=
2\varphi_\epsilon''\cdot z\bar z\cdot G^2(0)+2\varphi_\epsilon'\cdot 1\cdot G(0)=
2(\varphi_\epsilon''\cdot z\bar z+\varphi_\epsilon')
$$
\bigskip


\textbf{Product of diagonal terms} $\left(\begin{array}{c|c}
\bullet&\\\hline&\bullet
\end{array}\right).$
\medskip

Taking  into consideration~\eqref{eq:upper:left:corner} we obtain the
following value for the diagonal product in our determinant:
\begin{equation}
\label{eq:diagonal:product}
\partial_t\partial_{\bar t}\log \rho_\epsilon \cdot
\partial_z\partial_{\bar z}\log \rho_\epsilon=
2b'\bar b'\cdot
\Big(2\varphi_\epsilon''\cdot (z\bar z)^2+4\varphi_\epsilon'\cdot z\bar z+1\Big)
\Big(\varphi_\epsilon''\cdot z\bar z+\varphi_\epsilon'\Big)
\end{equation}
\bigskip


\textbf{Entry } $\left(\begin{array}{c|c}
\phantom{\bullet}&\bullet\\\hline&
\end{array}\right).$
\medskip

For the first derivative $\partial_{\bar u}\log \rho_\epsilon$ we get
$$
\partial_{\bar u}\log \rho_\epsilon=2\varphi_\epsilon'\cdot u\cdot G
$$

For the second derivative we obtain:
$$
\partial_t\partial_{\bar u}\log \rho_\epsilon=
2\varphi_\epsilon''\cdot
\left(\bar u\,\cfrac{\partial u}{\partial t}\,G+
u\bar u\,\cfrac{\partial G}{\partial t}\right)\cdot u\cdot G+
2\varphi_\epsilon'\cdot\left(\cfrac{\partial u}{\partial t}\,G+
u\,\cfrac{\partial G}{\partial t}\right)
$$

Evaluating     the     above     second     derivative    at    $t=0$
using~\eqref{eq:derivatives:u:bar:u}  and~\eqref{eq:derivatives:G} we
proceed as
\begin{multline}
\partial_t\partial_{\bar u}\log \rho_\epsilon=
2\varphi_\epsilon''\cdot\Big(
\bar z\cdot (a'z+b'\bar z)\cdot 1+
z\bar z\cdot(-a')\Big)\cdot z\cdot 1+\\
+2\varphi_\epsilon'\cdot\Big(
(a'z+b'\bar z)\cdot 1+z\cdot(-a')
\Big)=
2b'\bar z\cdot(\varphi_\epsilon''\cdot z\bar z+\varphi_\epsilon')
\label{eq:up:right}
\end{multline}
\bigskip


\textbf{Entry } $\left(\begin{array}{c|c}
&\phantom{\bullet}\\\hline\bullet&
\end{array}\right).$
\medskip

Analogously,  for the first derivative $\partial_u\log \rho_\epsilon$
we get
$$
\partial_u\log \rho_\epsilon=2\varphi_\epsilon'\cdot\bar u\cdot G
$$
and for the second derivative we obtain:
$$
\partial_{\bar t}\partial_u\log \rho_\epsilon=
2\varphi_\epsilon''\cdot
\left(u\,\cfrac{\partial\bar u}{\partial\bar t}\;G+
u\bar u\,\cfrac{\partial G}{\partial\bar t}\right)\cdot\bar u\cdot G+
2\varphi_\epsilon'\cdot\left(\cfrac{\partial\bar u}{\partial\bar t}\,G+
\bar u\,\cfrac{\partial G}{\partial\bar t}\right)
$$

Evaluating       the       above       expression       at      $t=0$
using~\eqref{eq:derivatives:u:bar:u}  and~\eqref{eq:derivatives:G} we
complete the calculation as
\begin{multline}
\partial_{\bar t}\partial_u\log \rho_\epsilon=
2\varphi_\epsilon''\cdot\Big(
z\cdot (\bar a'\bar z+\bar b' z)\cdot 1+
z\bar z\cdot(-\bar a')\Big)\cdot\bar z\cdot 1\;+\\
+\;2\varphi_\epsilon'\cdot\Big(
(\bar a'\bar z+\bar b' z)\cdot 1+\bar z\cdot(-\bar a')
\Big)=
2\bar b' z\cdot(\varphi_\epsilon''\cdot z\bar z+\varphi_\epsilon')
\label{eq:down:left}
\end{multline}
\bigskip

\textbf{Product of diagonal terms} $\left(\begin{array}{c|c}
&\bullet\\\hline\bullet&
\end{array}\right).$
\medskip

Combining~\eqref{eq:up:right}  and~\eqref{eq:down:left} we obtain the
following value for the anti diagonal product in our determinant:
\begin{equation}
\label{eq:anti:diagonal:product}
\partial_t\partial_{\bar u}\log \rho_\epsilon \cdot
\partial_{\bar t}\partial_u\log \rho_\epsilon=
4b'\bar b'\cdot z\bar z\cdot\big(\varphi_\epsilon''\cdot z\bar z+\varphi_\epsilon'\big)^2
\end{equation}
\bigskip

Finally,                        combining~\eqref{eq:diagonal:product}
and~\eqref{eq:anti:diagonal:product}  we  obtain the desired value of
the determinant:
\begin{multline}
\label{eq:det}
\det
\left(\begin{array}{c|c}
\partial_t\partial_{\bar t}\log \rho_\epsilon &
\partial_t\partial_{\bar u} \log \rho_\epsilon\\
[-\halfbls] &\\
\hline  & \\
[-\halfbls]
\partial_{\bar t}\partial_u \log \rho_\epsilon &
\partial_z\partial_{\bar z} \log \rho_\epsilon
\end{array}
\right)=\\
=2b'\bar b'\cdot\big(\varphi_\epsilon''\cdot z\bar z+\varphi_\epsilon'\big)
\cdot\Bigg(
\big(2\varphi_\epsilon''\cdot (z\bar z)^2+4\varphi_\epsilon'\cdot z\bar z+1\big)-
2z\bar z\cdot\big(\varphi_\epsilon''\cdot z\bar z+\varphi_\epsilon'\big)
\Bigg)=\\
=2b'\bar b'\cdot \Big( 2\varphi_\epsilon'\cdot\varphi_\epsilon''\cdot (z\bar
z)^2+2(\varphi_\epsilon')^2\cdot z\bar z+ \varphi_\epsilon''\cdot z\bar z+\varphi_\epsilon'\Big)
\end{multline}

Now  we  need to integrate the above expression over the flat surface
$S$.  First  note  that  outside  of  small  neighborhoods of conical
singularities,  the  smoothed  metric $\rho_\epsilon(z,\bar z)|dz|^2$
coincides  with the original flat metric, so for such values of $x,y$
we  have  $\rho_\epsilon=1$  and hence, for such values of $(x,y)$ we
have $\log \rho_\epsilon(x,y)=0$. This observation proves that

\begin{multline}
\label{eq:det:as:sum:over:Oj}
\int_S
\det
\left(\begin{array}{c|c}
\partial_t\partial_{\bar t}\log \rho_\epsilon &
\partial_t\partial_{\bar u} \log \rho_\epsilon\\
[-\halfbls] &\\
\hline  & \\
[-\halfbls]
\partial_{\bar t}\partial_u \log \rho_\epsilon &
\partial_z\partial_{\bar z} \log \rho_\epsilon
\end{array}
\right)\;dx\,dy\ =\\
=\sum_{j=1}^{\noz}
\int_{\cO_j(\epsilon)}
2b'\bar b'\cdot\Big( 2\varphi_\epsilon'\cdot\varphi_\epsilon''\cdot (z\bar z)^2+2(\varphi_\epsilon')^2\cdot
z\bar z+ \varphi_\epsilon''\cdot z\bar z+\varphi_\epsilon'\Big)\;dx\,dy\;,
\end{multline}
where  the  sum  is taken over all conical points $P_1,\dots,P_\noz$.
(We  did not introduce separate notations for flat coordinates in the
neighborhoods of different conical points).

Using  the  definition~\eqref{eq:def:varphi} of $\varphi_\epsilon(s)$
we can rewrite the expression which we integrate in terms of a single
variable $s=z\bar z=|z|^2$ as follows:
\begin{multline}
\label{eq:def:Phi:of:s}
2\varphi_\epsilon'\cdot\varphi_\epsilon''\cdot (z\bar z)^2+2(\varphi_\epsilon')^2\cdot
z\bar z+ \varphi_\epsilon''\cdot z\bar z+\varphi_\epsilon'=\\
= 2\varphi_\epsilon'(s)\cdot\varphi_\epsilon''(s)\cdot s^2+2(\varphi_\epsilon'(s))^2\cdot s+
\varphi_\epsilon''(s)\cdot s+\varphi_\epsilon'(s)=:\Phi_\epsilon(s)
\end{multline}

Recall  that  in  the  flat  coordinate $z$ a small neighborhood of a
conical   singularity  of  order  $d$  is  glued  from  $d+2$  metric
half-discs.   Taking  into  consideration  angular  symmetry  of  the
expression  which  we integrate and passing through polar coordinates
in  our  integral  we  can  reduce  integration  over  a $d+2$ metric
half-discs to integration over a segment:
\begin{multline}
\label{eq:disk:to:segment}
2b'\bar b'\int_{\cO_j(\epsilon)}
\Phi_\epsilon(z\bar z)\,dx\,dy\ =\ 2(d+2)b'\bar b'
\int_{\substack{|z|\le\epsilon\\ \Re(x)\ge 0}}
\Phi_\epsilon(r^2)\,r\,dr\,d\theta\ =\\
=\ (d+2)\pi\cdot b'\bar b'\int_0^{\epsilon^2}
\Phi_\epsilon(s)\,ds
\end{multline}

Finally,  observe  that  it  is  easy  to  find an antiderivative for
$\Phi_\epsilon(s)$, namely:
$$
\Phi_\epsilon(s)=\left(\big(\varphi'(s)\big)^2\cdot s^2+\varphi'(s)\cdot s\right)'
$$
which implies, that
$$
\int_0^{\epsilon^2}
\Phi_\epsilon(s)\,ds=
\big(\varphi_\epsilon'(s)\cdot s\big)^2\Big|^{\epsilon^2}_{+0}\ +\
\varphi_\epsilon'(s)\cdot s\Big|^{\epsilon^2}_{+0}
$$
Using  the properties~\eqref{eq:phi:prime:s} of $\varphi_\epsilon(s)$
we get
$$
\int_0^{\epsilon^2}
\Phi_\epsilon(s)\,ds=
\left(0-\left(-\frac{d}{2(d+2)}\right)^2\right)+
\left(0-\left(-\frac{d}{2(d+2)}\right)\right)=
   %
\frac{d(d+4)}{4(d+2)^2}
$$

Plug the value of the integral obtained in the right-hand side of the
above  formula in equation~\eqref{eq:disk:to:segment} and combine the
result             with~\eqref{eq:det:as:sum:over:Oj}             and
with~\eqref{eq:def:Phi:of:s}. The resulting expression coincides with
equation~\eqref{eq:lim:for:quadratic}    in    the    statement    of
Proposition~\ref{pr:integrating:Fay}.  As  we  have already indicated
above,  relation~\eqref{eq:lim:for:Abelian}  follows immediately from
equation~\eqref{eq:lim:for:quadratic}  and  from  the  fact  that the
integral is supported on small neighborhoods of the conical points of
the metric. Proposition~\ref{pr:integrating:Fay} is proved.
\end{proof}

\begin{Lemma}
\label{lm:Th5:in:t}
In the same setting as above the following formulae hold.

For a family of deformations~\eqref{eq:deformation:in:periods} of the
initial  flat  surface  $S$  in  a stratum $\cH(m_1,\dots,m_\noz)$ of
Abelian differentials one has:
\begin{multline}
\label{eq:Th5:A:in:t}
\partial_{\bar t}\partial_t\,
\log\det\Dflat(S,S_0)
\ =\\=\
\partial_{\bar t}\partial_t\,
\log|\det\langle\omega_i,\omega_j\rangle|
\ +\
\frac{1}{12}\cdot\sum_{j=1}^\noz \cfrac{m_j(m_j+2)}{(m_j+1)}
\cdot|b'(0)|^2
\end{multline}

For                   a                   family                   of
deformations~\eqref{eq:deformation:in:periods:quadratic}    of    the
initial  flat surface $S$ inside a stratum $\cQ(d_1,\dots,d_\noz)$ of
meromorphic  quadratic  differentials  with  at most simple poles one
has:
\begin{multline}
\label{eq:Th5:q:in:t}
\partial_{\bar t}\partial_t\,
\log\det\Dflat(S,S_0)
\ =\\=\
\partial_{\bar t}\partial_t\,
\log|\det\langle\omega_i,\omega_j\rangle|
\ +\
\frac{1}{24}\cdot\sum_{j=1}^\noz \cfrac{d_j(d_j+4)}{(d_j+2)}
\cdot|b'(0)|^2
\end{multline}
\end{Lemma}
\begin{proof}
Plugging                       expressions~\eqref{eq:lim:for:Abelian}
and~\eqref{eq:lim:for:quadratic}              obtained             in
Proposition~\ref{pr:integrating:Fay}                             into
formula~\eqref{eq:with:limit}  from  Lemma~\ref{lm:Fay}  we  get  the
relations~\eqref{eq:Th5:A:in:t} and~\eqref{eq:Th5:q:in:t} from above.
\end{proof}

Consider the natural projection
$$
p:\cH(m_1,\dots,m_\noz)\to \PcH(m_1,\dots,m_\noz)\,.
$$
Families       of      deformations~\eqref{eq:deformation:in:periods}
and~\eqref{eq:deformation:in:periods:quadratic}  are chosen in such a
way  that the resulting infinitesimal affine line $\gamma(t)$ defined
by    equation~\eqref{eq:deformation:in:periods}   in   the   stratum
$\cH(m_1,\dots,m_\noz)$              (correspondingly              by
equation~\eqref{eq:deformation:in:periods:quadratic}  in  the stratum
$\cQ(d_1,\dots,d_\noz)$)  projects  to the Teichm\"uller disc passing
through  $p(S)$.  We will show below that the projection map $p$ from
$\gamma(t)$  to  the  Teichm\"uller  disc  is  nondegenerate  in  the
neighborhood  of  $t=0$. Thus, we can induce the canonical hyperbolic
metric of curvature $-4$ to $\gamma(t)$.

\begin{Lemma}
\label{lm:from:t:to:Laplacian}
The   canonical   hyperbolic   metric   of   curvature  $-4$  on  the
Teichm\"uller   disc  induced  to  the  infinitesimal  complex  curve
$\gamma(t)$ under the projection $p$ has the form
$$
|b'(0)|^2\,|dt|^2
$$
at the point $t=0$.

In  particular,  the  Laplacian  of  the induced hyperbolic metric of
curvature $-4$ on $\gamma(t)$ satisfies the relation
\begin{equation}
\label{eq:laplacian:in:t}
|b'(0)|^2\cdot\cfrac{1}{4}\,\,\Dhyp\Big|_{t=0}=\dtdtbar\Big|_{t=0}
\end{equation}
at the point $t=0$.
\end{Lemma}
\begin{proof}
We  prove  the  Lemma  for a flat surface corresponding to an Abelian
differential;  for  a  flat  surface  corresponding  to a meromorphic
quadratic  differentials  with  at  most  simple  pole  the  proof is
completely analogous.

Choose  a  pair  of independent integer cycles $c_1,c_2\in H_1(C,\Z)$
such that $c_1\circ c_2=1$, and transport them to all surfaces $C(t)$
(we  assume  that  $\gamma(t)$  stays  in  a tiny neighborhood of the
initial  point,  so  we  would  not  have any ambiguity in doing so).
Consider the corresponding periods of $\omega(t)$,
$$
A(t):=\int_{c_1} \omega(t)\qquad B(t):=\int_{c_2} \omega(t)\,.
$$
By definition of the family of deformations we get
\begin{align*}
A(t)&=a(t)A+b(t)\bar A\\
B(t)&=a(t)B+b(t)\bar B\,,
\end{align*}
where  $A=A(0)$  and  $B=B(0)$  are  the corresponding periods of the
initial Abelian differential $\omega$. Define
$$
\zeta(t):=\frac{B(t)}{A(t)}=
\frac{a(t)B+b(t)\bar B}{a(t)A+b(t)\bar A}\,.
$$

At  the  first  glance  this  definition of the hyperbolic coordinate
$\zeta(t)$  depends  on the choice of a pair of cycles $c_1,c_2$, and
on  the  values  of  the periods of the initial Abelian differential.
However, it would be clear from the proof that the induced hyperbolic
metric  does  not  depend on this choice. Basically, the situation is
the  same  as  in  the  case of flat tori, see Example~\ref{ex:M1} in
section~\ref{ss:Teichmuller:discs}.

Consider  now  the  hyperbolic  half-plane  $\Hyp$  endowed  with the
canonical  metric  $\cfrac{|d\zeta|^2}{4|\Im\zeta|^2}$  of  curvature
$-4$. Let us compute the induced metric in the coordinate $t$.

Clearly
$$
\Im\zeta(0)=\Im\frac{B}{A}\,.
$$
Computing the derivative at $t=0$ we get
$$
\frac{\partial\zeta}{\partial t}\Big|_{t=0}=
b'(0)\,\frac{\bar B A-B\bar A}{A^2}\,.
$$
Thus
\begin{multline*}
\frac{\partial\zeta}{\partial t}
\frac{\partial\bar\zeta}{\partial\bar t}\Big|_{t=0}=
-b'(0)\bar b'(0)\,\left(\frac{\bar B A-B\bar A^{\phantom{!}}}{A\bar A}\right)^2=
4|b'(0)|^2\,\left(\Im\frac{B}{A}\right)^2
=\\=
4|b'(0)|^2\,\Im^2\zeta(0)\,.
\end{multline*}
Hence,  the  hyperbolic  metric has the following form in coordinates
$t$ at $t=0$
\begin{multline*}
\frac{|d\zeta|^2}{4(\Im\zeta)^2}\Bigg|_{\zeta(0)}=
\frac{\partial\zeta}{\partial t}
\frac{\partial\bar\zeta}{\partial\bar t}\,
\frac{|d t|^2}{4(\Im\zeta(t))^2}\Bigg|_{t=0}
=\\=
4|b'(0)|^2\,\Im^2\zeta(0)
\frac{|d t|^2}{4(\Im\zeta(0))^2}=
|b'(0)|^2\,|dt|^2
\end{multline*}
This  implies that the Laplacian of this metric at $t=0$ is expressed
as
$$
\Dhyp=\cfrac{4}{|b'(0)|^2}\,\,\dtdtbar\,.
$$
\end{proof}

\begin{proof}[Proof of Theorem~\ref{theorem:main:local:formula}]
Plug        the        expression~\eqref{eq:laplacian:in:t}       for
$\partial_t\bar\partial_t$                 obtained                in
Lemma~\ref{lm:from:t:to:Laplacian}                               into
formulae~\eqref{eq:Th5:A:in:t}  and~\eqref{eq:Th5:q:in:t} obtained in
Lemma~\ref{lm:Th5:in:t}.  Dividing  all  the  terms  of the resulting
equality  by  the  common  factor  $|b'(0)|$ (which is nonzero by the
definition  of  the  family of deformations) we obtain the relations
equivalent                to                the               desired
relations~\eqref{eq:main:local:formula:Abelian}
and~\eqref{eq:main:local:formula:quadratic}                        in
Theorem~\ref{theorem:main:local:formula}.
\end{proof}

\subsection{Alternative proof based on results of A.~Kokotov,
\mbox{D.~Korotkin} and P.~Zograf}
\label{ss:Kokotov:Korotkin}

By  assumption  the  initial  flat  surface  $S=S(0)$  has  area one.
However,    the    area    of    the    flat    surface   $S(t)$   in
family~\eqref{eq:deformation:in:periods}             or            in
family~\eqref{eq:deformation:in:periods:quadratic}   varies  in  $t$.
Define the function
$$
k(t,\bar t):=\Area\big(S(t)\big)^{-\frac{1}{2}}\,.
$$
We shall need the following technical Lemma concerning this function.

\begin{Lemma}
One  has  the  following  expression  for  the  partial derivative of
$k(t,\bar t)$ at $t=0$:
\begin{equation}
\label{eq:dd:k}
\frac{\partial^2
\log k(t,\bar t)}{\partial t\,\partial\overline{t}}\Big\vert_{t=0}
=\frac{1}{2}\,|b'(0)|^2
\end{equation}
\end{Lemma}
\begin{proof}
Relation~\eqref{eq:udu} implies that:
$$
\Area\big(S(t)\big)=
(a\bar a-b\bar b)\,,
$$
so  we  get  the following expression for the function $k(t,\bar t)$:
$$
k(t,\bar t)=(a\bar a-b\bar b)^{-\frac{1}{2}}\,.
$$
Computing  the  value  of  the  second  derivative  at  $t=0$, we get
$$
\frac{\partial^2 \log k(t,\bar t)}{\partial t\,\partial\overline{t}}\Bigg\vert_{t=0}
=\frac{1}{2}\,|b'(0)|^2\,,
$$
where  we  used  the conventions chosen above: $a(0)=1$ and $b(0)=0$.
\end{proof}

The  proof of Theorem~\ref{theorem:main:local:formula} can be derived
from   the   following   formula   due   to  A.~Kokotov,  D.~Korotkin
(formula~(1.10)      in~\cite{Kokotov:Korotkin}).      Denote      by
$\det\Delta^{|\omega|^2}$  the regularized determinant of the Laplace
operator  in the flat metric defined as in~\cite{Kokotov:Korotkin} by
a  holomorphic  form  $\omega\in\cH(m_1,\dots,m_\noz)$. It is defined
for      flat      surfaces      of      arbitrary      area.     For
$\omega\in\cH_1(m_1,\dots,m_\noz)$           the          determinant
$\det\Delta^{|\omega|^2}$                 differs                from
$\det\Dflat\left(S(\omega),S_0\right)$  by  a multiplicative constant
depending only on the choice of the base surface $S_0$.

\begin{NNTheorem}[A.~Kokotov, D.~Korotkin~\cite{Kokotov:Korotkin}]
For any flat surface in any stratum of Abelian differentials the
following formula of holomorphic factorization holds:
\begin{equation}
\label{eq:holomorphic:factorization}
\det\Delta^{|\omega|^2}=
\const\cdot\Area(C,\omega)\cdot \det(\Im B)\cdot|\tau(C,\omega)|^2\,,
\end{equation}
where  $B$  is  the  matrix  of $B$-periods and $\tau(C,\omega)$ is a
flat section  of  a holomorphic line  bundle  over  the  ambient  stratum
$\cH(m_1,\dots,m_\noz)$ of Abelian differentials.

Moreover     (see~\cite{Korotkin:Zograf}),     $\tau(S,\omega)$    is
homogeneous in $\omega$ of degree $p$, where
\begin{equation}
\label{eq:p:value}
p=\frac{1}{12}\sum_{i=1}^\noz\frac{m_i(m_i+2)}{m_i+1}\,.
\end{equation}
In   other  words,  for  any  nonzero  complex  number  $k$  one  has
\begin{equation}
\label{eq:homogenity}
\tau(C,k\omega)=k^p\tau(C,\omega)
\end{equation}
\end{NNTheorem}

\begin{Remark}
Note   that  the  ``Bergman  $\tau$-function''  $\tau(C,\omega)$  is,
actually,  a  flat section of a certain local system over the stratum
$\cH(m_1,\dots,m_n)$,  see  Definition  3 in~\cite{Kokotov:Korotkin}.
Such  a  section  is defined up to a constant factor. However, in the
calculation   below  the  $\tau$-function  is  present  only  in  the
expression  $\dtdtbar\log(|\tau(C,\omega)|^2)$  which does not depend
on  the  choice  of  the  particular  flat  section. (See section 3.1
in~\cite{Kokotov:Korotkin} for more details.)
\end{Remark}

\begin{proof}[Alternative proof of
Theorem~\ref{theorem:main:local:formula}]
Note  that  $\Im B(t)$ depends only on the underlying Riemann surface
$C(t)$;  in  particular,  rescaling $\omega(t)$ proportionally, we do
not change $\Im B(t)$.

Applying    formula~\eqref{eq:holomorphic:factorization}    to    the
\textit{normalized}   Abelian  differential  $k(t,\bar  t)\omega(t)$,
which defines a flat surface $S^{(1)}(t,\bar t)$ of unit area, we get
\begin{multline*}
\det\Dflat\left(S^{(1)}(t),S_0\right)=
\const\cdot 1\cdot \det\big(\Im B(t)\big)\cdot
\bigg|\tau\bigg(C(t),k(t,\bar t)\omega(t)\bigg)\bigg|^2
\ =\\=\
\const\cdot \det(\Im B(t))
\cdot k^{2p}(t,\bar t)\cdot|\tau(C,\omega(t))|^2\,,
\end{multline*}
where  we  used homogeneity~\eqref{eq:homogenity} of $\tau$ to get the
latter  expression.  Passing  to logarithms of the above expressions,
applying     $\dtdtbar$,     taking     into    consideration    that
$\tau\big(C(t),\omega(t)\big)$  is  a holomorphic function, and using
relations~\eqref{eq:dd:k} and~\eqref{eq:p:value} we get
\begin{multline}
\label{eq:temp}
\dtdtbar\log|\det\Dflat(S(t),S_0)|
=\\=
\dtdtbar\log|\det\Im B|
+\frac{1}{12}\sum_{i=1}^\noz\frac{m_i(m_i+2)}{m_i+1}\,|b'(0)|^2\,.
\end{multline}
It remains to note that
$$
|\det\langle\omega_i(t),\omega_j(t)\rangle|\,=\,
|\text{holomorphic function of }t|\cdot\Im B(t)\,.
$$
Thus,
$$
\dtdtbar\log|\det\langle\omega_i(t),\omega_j(t)\rangle|=
\dtdtbar\log|\det\Im B(t)|\,.
$$
Applying  the  latter  remark to expression~\eqref{eq:temp}, dividing
the  result by $|b'(0)|^2$ and recalling~\eqref{eq:laplacian:in:t} we
get
$$
\Dhyp\log|\det\langle\omega_i(t),\omega_j(t)\rangle|=
\Dhyp\log|\Dflat(S,S_0)|\,-\,
\frac{1}{3}\sum_{i=1}^\noz\frac{m_i(m_i+2)}{m_i+1}
$$
\end{proof}

The  proof for quadratic differentials is completely analogous. It is
based  on  the  following  statement  of  A.~Kokotov  and D.~Korotkin
(see\cite{Kokotov:Korotkin:MPI}):

\begin{NNTheorem}[A.~Kokotov, D.~Korotkin]
For  any  flat  surface  in  any  stratum  of  meromorphic  quadratic
differentials  with  at  most  simple  poles the following formula of
holomorphic factorization holds:
\begin{equation*}
\det\Delta^{|q|}=
\const\cdot\Area(C,q)\cdot \det(\Im B)\cdot|\tau(C,q)|^2\,,
\end{equation*}
where  $\tau(S,q)$  is  a holomorphic function in the ambient stratum
$\cQ(d_1,\dots,d_\noz)$ of quadratic differentials.

Moreover,  $\tau(S,q)$  is  homogeneous  in  $q$ of degree $p$, where
\begin{equation*}
p=\frac{1}{48}\sum_{i=1}^\noz\frac{d_i(d_i+4)}{d_i+2}\,.
\end{equation*}
In   other  words,  for  any  nonzero  complex  number  $k$  one  has
\begin{equation*}
\tau(C,kq)=k^p\tau(C,q)
\end{equation*}
\end{NNTheorem}

Note  the  only  difference  with  the  previous case. Multiplying an
Abelian  differential  by  a  factor  $k$  we  change the area of the
corresponding  flat  surface  by  a  factor  $|k|^2$.  Multiplying  a
quadratic  differential  by  a  factor  $k$ we change the area of the
corresponding flat surface by a factor $|k|$.

\section{Relating flat and hyperbolic Laplacians by means of
Polyakov formula}
\label{sec:Comparison:of:determinants}

In             this            section            we            prove
Theorem~\ref{theorem:log:det:flat:minus:log:det:hyp}.  Our  proof is based
on the Polyakov formula. We start by rewriting the Polyakov formula in a more
symmetric  form~\eqref{eq:Polyakov:formula:symmetric:form}.  Then  we
perform  the integration separately over complements to neighborhoods
of cusps and over neighborhoods of cusps. A neighborhood of each cusp
we    also    subdivide    into    several   domains   presented   at
Figure~\ref{fig:domains},  and  we perform the integration separately
for each domain.

\subsection{Polyakov formula revisited}

In   local  coordinates  $x,y$  the  Laplace  operator  of  a  metric
$\rho(x,y)\,(dx^2+dy^2)$ has the form
$$
\Delta_g=\rho^{-1}\left(\frac{\partial^2}{\partial x^2}+
\frac{\partial^2}{\partial y^2}\right)
$$
and the curvature $K_g$ of the metric is expressed as
$$
K_g=-\Delta_g\log\sqrt{\rho}\,.
$$

In  some  situations  it  would  be  convenient  to use the following
coordinate     version     of     the     Polyakov     formula    (see
section~\ref{sec:Determinant:of:Laplacian}).  Let  in some coordinate
domain $x,y$
\begin{align*}
g_1& =\rho_1\,(dx^2+dy^2)=\exp(2\cf_1)\,(dx^2+dy^2)\\
g_2&
=\rho_2\,(dx^2+dy^2)=\exp(2\cf_2)\,(dx^2+dy^2)\,.
\end{align*}
Then,     $g_2=\exp\left(2(\phi_2-\phi_1)\right)\cdot     g_1$,    so
$\phi=\phi_2-\phi_1$.  An  elementary  calculation  shows that in the
corresponding coordinate domain
\begin{equation}
\label{eq:Polyakov:formula:symmetric:form}
\int(\cf\,\Delta_{g_1}\cf -
2\cf\,K_{g_1})\,d g_1=
\int(\cf_2\Delta\cf_2-\cf_1\Delta\cf_1)+
(\cf_2\Delta\cf_1-\cf_1\Delta\cf_2)\, dxdy\ ,
\end{equation}
where
$\Delta=\cfrac{\partial^2}{\partial x^2}+
   \cfrac{\partial^2}{\partial y^2}$.

\subsection{Polyakov  Formula applied to smoothed flat and hyperbolic
metrics}

Let  $w$  be a coordinate in a neighborhood of a conical point on $S$
defined  by~\eqref{eq:flat:metric:in:a:local:coordinate:w}; let $w_0$
be  analogous  coordinate for $S_0$. By assumptions, the order $d$ of
the  corresponding conical singularity is the same for $S$ and $S_0$.
Then, we obtain
\begin{multline}
\label{eq:cf:minus:cf0:is:regular}
2\left(\cf(\zeta)-\cf_0(\zeta)\right)=
\log\left|w^d\left(\frac{dw}{d\zeta}\right)^2 \zeta^2\log^2\zeta\right|-
 \log\left|w_0^d\left(\frac{dw_0}{d\zeta}\right)^2
 \zeta^2\log^2\zeta\right|=\\
=d\log\left|\frac{w}{w_0}\right|+\text{regular function}=
\text{regular function}
\end{multline}
In       the       last       equality       we       used       that
$w=\frac{dw}{d\zeta}\vert_{\zeta=0}\cdot\zeta(1+O(|\zeta|))$      and
$w_0=\frac{dw_0}{d\zeta}\vert_{\zeta=0}\cdot\zeta(1+O(\zeta))$  where
both  derivatives  are  different  from  zero.  This  proves that the
right-hand-side                     expression                     in
formula~\eqref{eq:log:det:flat:minus:log:det:hyp}                  of
Theorem~\ref{theorem:log:det:flat:minus:log:det:hyp} is well-defined.

Applying  the Polyakov  formula to the metrics $\gfe=\exp(2\phi)\ghd$ on $S$,
then  to  the metrics  $\gfe=\exp(2\phi_0)\ghd$  on $S_0$, and taking the
difference we get the following relation:
\begin{multline}
\label{eq:Polyakov:minus:Polyakov0}
\log\det\Delta_{\gfe}(S,S_0)-\log\det\Delta_{\ghd}(S,S_0)=\\=
\cfrac{1}{12\pi}
\int_S \cf(\Delta_{\ghd}\cf-2K_{\ghd})\, d\ghd -
\cfrac{1}{12\pi}
\int_{S_0} \cf_0(\Delta_{\ghd}\cf_0-2K_{\ghd})\, d\ghd\,,
\end{multline}
where we took into account that
$$
\Area_{\gfe}(S)=\Area_{\gfe}(S_0)\quad\text{and}\quad
\Area_{\ghd}(S)=\Area_{\ghd}(S_0)\,.
$$

To  prove  Theorem~\ref{theorem:log:det:flat:minus:log:det:hyp} we need to
compute  the  limit of expression~\eqref{eq:Polyakov:minus:Polyakov0}
as  $\epsilon$  and  $\delta$  tend  to  zero.  Note  that  the  term
$\log\det\Delta_{\gfe}(S,S_0)$  does not depend on $\delta$, and that the
existence  of  a  limit  of  this term as $\epsilon$ tends to zero is
\textit{a      priori}      known.      Similarly,      the      term
$\log\det\Delta_{\ghd}(S,S_0)$  does  not  depend  on  $\epsilon$ and
the existence  of  a limit of this term as $\delta$ tends to zero is also
\textit{a  priori}  known. Hence, to \textit{evaluate} the difference
of  the corresponding limits we can make $\epsilon$ and $\delta$ tend
to  zero  in any particular way, which is convenient for us. From now
on let us assume that $0<\delta\ll\epsilon\ll R\ll 1$.

As  usual, we perform the integration over a surface in several steps
integrating separately over complements to $R$-neighborhoods of cusps
and over $R$-neighborhoods of cusps. An $R$-neighborhood of each cusp
we  also  subdivide into several domains. We proceed by computing the
integral              in              the             right-hand-side
of~\eqref{eq:Polyakov:minus:Polyakov0} domain by domain.

\subsubsection{Integration over complements of cusps}

In  this  domain  $\gfe=\gflat$, and $\ghd=\ghyp$. In coordinates $z$
and $\zeta$ we have
$$
\cf=\log\left|\frac{dz}{d\zeta}\right|+\cfrac{1}{2}\log\rho^{-1}\,,
$$
where  $\rho(\zeta,\bar\zeta)=|\zeta|^{-2}(\log|\zeta|)^{-2}$  is the
density of the hyperbolic metric. Also, in this domain
$$
\Delta_{\ghd}=\Delta_{\ghyp}=
4\rho^{-1}\frac{\partial^2}{\partial\zeta\partial\bar\zeta}
$$
Hence,
$$
\Delta_{\ghd}\cf=
4\rho^{-1}\frac{\partial^2}{\partial\zeta\partial\bar\zeta}
\Bigg( \frac{1}{2}\left(
\log\frac{dz}{d\zeta}+\log\frac{d\bar z}{d\bar\zeta}\right)
-\cfrac{1}{2}\log\rho\Bigg)
$$
Since  $\log\cfrac{dz}{d\zeta}$  is holomorphic and $\log\cfrac{d\bar
z}{d\bar\zeta}$  is  antiholomorphic they both are annihilated by the
Laplace operator. Thus, in this domain
$$
\Delta_{\ghd}\cf=-\cfrac{1}{2}\,\Delta_{\ghyp}\log\rho_{hyp}=
K_{\ghyp}=-1\,,
$$
and, hence, in this domain we get
$$
(\Delta_{\ghd}\cf-2K_{\ghd})=1\quad\text{and}\quad
(\Delta_{\ghd}\cf_0-2K_{\ghd})=1
$$
In    notations~\eqref{eq:relative:integral}    we    can   represent
integrals~\eqref{eq:Polyakov:minus:Polyakov0}     over    complements
$S\setminus\sqcup\cO_j(R)$  and  $S_0\setminus\sqcup\cO_j(R)$  to the
cusps as
\begin{equation}
\label{eq:int:over:complement:to:cusps}
\frac{1}{12\pi}\left\langle\int_S \cf \, d\ghyp -
\int_{S_0} \cf_0 \, d\ghyp\right\rangle\,-\,
\frac{1}{12\pi}\int_{\sqcup\cO_j(R)} (\cf-\cf_0) \, d\ghyp\,.
\end{equation}

\subsection{Integration over a neighborhood of a cusp}
The rest of \S\ref{sec:Comparison:of:determinants} consists of a very tedious calculation. We fix  a  pair  of
corresponding conical singularities $P_j$ on $S$ and on $S_0$, and we
consider   neighborhoods  $\cO(R)$  of  the  corresponding  cusps  in
hyperbolic   metrics  on  $S$  and  $S_0$.  These  neighborhoods  are
isometric,  where  isometry is defined up to a global rotation of the
cusp.  Using  such  an  isometry  we  identify  the two corresponding
neighborhoods on $S$ and on $S_0$. Clearly, $\ghyp$ and $\ghd$ coming
from  $S$  and  from $S_0$ coincide, while the holomorphic functions $w$,
and  $w_0$  defined  in  a  disc  $\cO(R)=\{\zeta  \text{ such that }
|\zeta|\le  R\}$  (and  hence  the  corresponding  flat  metrics  and
smoothed  flat  metrics)  differ.  Note,  however,  that the cusp was
chosen  exactly at the conical point, so $w(0)=w_0(0)=0$. Also, since
$\zeta,  w$  are holomorphic coordinates in a neighborhood of a point
$P_j$    of    a    regular    Riemann    surface    $S$,   one   has
$\cfrac{dw}{d\zeta}\Big\vert_{\zeta=0}\neq        0$.       Similarly
$\cfrac{dw_0}{d\zeta}\Big\vert_{\zeta=0}\neq 0$.

\begin{figure}
\includegraphics{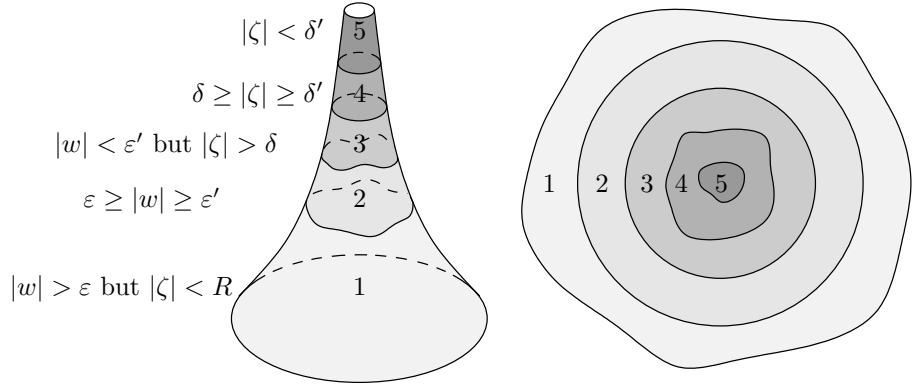}
\begin{picture}(0,0)(97,-53)
\put(-71,-170){$|w|>\epsilon \text{ but }|\zeta|<R$}
\put(-43,-137){$\epsilon\ge |w|\ge\epsilon'$}
\put(-54,-116){$|w|<\epsilon' \text{ but }|\zeta|>\delta$}
\put(-2,-97){$\delta\ge |\zeta|\ge \delta'$}
\put(16.5,-73){$|\zeta|<\delta'$}
\end{picture}
\begin{picture}(0,0)(-29.5,-53)
\put(-71,-170){1}
\put(-71,-137){2}
\put(-71,-116){3}
\put(-71,-97){4}
\put(-71,-73){5}
\end{picture}
\begin{picture}(0,0)(-40,-30)
\put(-13,-108){1}
\put(7,-108){2}
\put(24,-108){3}
\put(37,-108){4}
\put(52,-108){5}
\end{picture}
\vspace{150pt}
\caption{
\label{fig:domains}
Domains of integration}
\end{figure}

By  assumption $0<\delta\ll\epsilon\ll R\ll 1$. We subdivide the disc
$|\zeta|<R$     into     the     following    domains    (see    also
Figures~\ref{fig:domains})
\begin{enumerate}
\item $\epsilon<|w|$;
\item $\epsilon'\le |w|\le\epsilon$;
\item $|w|<\epsilon' \text{ but }\delta<|\zeta|$;
\item $\delta'\le |\zeta|\le \delta$;
\item $|\zeta|<\delta'$.
\end{enumerate}
and we perform integration over domains (1)--(5) in
parallel with integration over analogous domains defined in terms of
$w_0$.

\subsubsection{Integration over a cusp: the domain $|w|>\epsilon$}
\label{sss:w:ge:epsilon}

First  note  the  following elementary formula from calculus: for any
constant $C>0$
\begin{equation}
\label{eq:log:along:annulus:vanishes}
\int_{r\le|\zeta|\le Cr}
\cfrac{\log|\zeta|\,|d\zeta|^2}{|\zeta|^2\log^2|\zeta|}\
\to 0\quad\text{as }r\to +0
\end{equation}
In  other  words,  while  the  corresponding  integral  over  a  disc
diverges,  an  integral  over  a contracting annulus tends to zero as
soon as a modulus of the annulus remains bounded.

Now consider the smallest annulus
\begin{equation}
\label{eq:annulus:A}
A(\epsilon):=
\{\zeta\text{ such that }r(\epsilon)\le|\zeta|\le C(\epsilon)r(\epsilon)\}
\end{equation}
containing  both  curves $|w|=\epsilon$ and $|w_0|=\epsilon$. Clearly
$r(\epsilon)\to 0$ as $\epsilon\to +0$ and $C(\epsilon)$ is uniformly
bounded  by  some  constant  $C$ for all sufficiently small values of
$\epsilon$.

Now      let     us     compute     the     difference     of     the
integrals~\eqref{eq:Polyakov:minus:Polyakov0}    over    the   domain
$|w|>\epsilon$   and   integrals   over   the   corresponding  domain
$|w_0|>\epsilon$. Our computation mimics one in the previous section.
In  particular, our integrals~\eqref{eq:Polyakov:minus:Polyakov0} are
reduced to
$$
\frac{1}{12\pi}\left(\int_{\substack{|\zeta|<R\\|w|>\epsilon}} \cf \, d\ghyp -
\int_{\substack{|\zeta|<R\\|w_0|>\epsilon}} \cf_0 \, d\ghyp\right)\,,
$$
where
   %
\begin{equation}
\label{eq:kappa:equals:log}
\cf=(d+2)(\log|\zeta|)(1+o(1))
\qquad
\cf_0=(d+2)(\log|\zeta|)(1+o(1))
\end{equation}
in our domains. Decomposing the domains of integration we can proceed
as:
\begin{multline*}
\left(\int_{\substack{|\zeta|<R\\|w|>\epsilon}} \cf \, d\ghyp -
\int_{\substack{|\zeta|<R\\|w_0|>\epsilon}} \cf_0 \, d\ghyp\right)=\\
=\int_{\cO(R)}(\cf-\cf_0)\,d\ghyp-
\int_{|\zeta|\le r(\epsilon)}(\cf-\cf_0)\,d\ghyp+
\int_{\substack{|\zeta|>r(\epsilon)\\|w|\le\epsilon}}\cf\,d\ghyp-
\int_{\substack{|\zeta|>r(\epsilon)\\|w_0|\le\epsilon}}\cf_0\,d\ghyp
\end{multline*}

By~\eqref{eq:cf:minus:cf0:is:regular} the difference $(\cf-\cf_0)$ is
regular  in  a  neighborhood  of a cusp, so its integral over a small
disc  $\{|\zeta|\le  r(\epsilon)\}$ tends to zero as $\epsilon$ tends
to  zero.  We  can  bound  from  above absolute values of each of the
remaining  two  integrals  by  the integrals of $|\cf|$ and $|\cf_0|$
correspondingly     along     a    larger    domain    $A(\epsilon)$,
see~\eqref{eq:annulus:A}.                 Taking                 into
consideration~\eqref{eq:kappa:equals:log}
and~\eqref{eq:log:along:annulus:vanishes}  we conclude that these two
integrals also tend to zero.

Hence,  after passing to a limit $\epsilon\to 0$ integration over the
domains $w>\epsilon$ and $w_0>\epsilon$ compensates a missing term
$$
\frac{1}{12\pi}\int_{\sqcup\cO_j(R)} (\cf-\cf_0) \, d\ghyp
$$
in~\eqref{eq:int:over:complement:to:cusps}.

\subsubsection{Integration over a cusp: the domain $\epsilon'\le|w|\le\epsilon$}

In the annulus $\epsilon'\le|w|\le\epsilon$ we have
\begin{alignat*}{2}
     &\gfe =\rhofe(|w|)\,|dw|^2
     &=\exp(2\cf_2)\,|dw|^2\\
\ghd=&\ghyp=\cfrac{1}{|\zeta|^2\log^2|\zeta|}\left|\frac{d\zeta}{dw}\right|^2\,|dw|^2
     &=\exp(2\cf_1)|dw|^2
\end{alignat*}
where
\begin{align*}
\cf_2&=\cfrac{1}{2}\log\rhofe(|w|)\\
\cf_1&=-\log|\zeta|-\log|\log|\zeta||+\log\left|\frac{d\zeta}{dw}\right|
\end{align*}

An elementary calculation shows that
\begin{align}
\notag
\cf_2&=\cfrac{d}{2}\cdot\log|w|+\text{regular function of }w\text{ and }\bar w\\
\label{eq:kappa1:of:w}
\cf_1&=-\log|w|-\log|\log|w||+O\left(\frac{1}{\log|w|}\right)\\
\notag
\Delta\cf_1&=\cfrac{1}{|w|^2\log^2|w|}\left(1+O\left(\frac{1}{\log|w|}\right)\right)
\end{align}

Applying  formula~\eqref{eq:Polyakov:formula:symmetric:form}  to  the
first integral in~\eqref{eq:Polyakov:minus:Polyakov0} we obtain
\begin{multline}
\label{eq:4:terms}
\int_{\epsilon'\le|w|\le\epsilon} \cf(\Delta_{\ghd}\cf-2K_{\ghd})\,
d\ghd=\\
=\int_{\epsilon'\le|w|\le\epsilon}
(\cf_2\Delta\cf_2-\cf_1\Delta\cf_1+
\cf_2\Delta\cf_1-\cf_1\Delta\cf_2)\, dxdy\ ,
\end{multline}
An   expression  $\cf_2\Delta\cf_2$  in~\eqref{eq:4:terms}  does  not
depend  on  $\zeta$  or  $\bar\zeta$.  Hence  it  coincides  with the
corresponding expression for $w_0$, and the difference
$$
\int_{\epsilon'\le|w|\le\epsilon}
\cf_2\Delta\cf_2 \,dxdy\ -\
\int_{\epsilon'\le|w_0|\le\epsilon}
\cf_2\Delta\cf_2 \,dxdy\ =\ 0
$$
is  equal  to zero. This term produces no contribution to the difference of
integrals in~\eqref{eq:Polyakov:minus:Polyakov0}.

By  assumption  the  ratio  $\epsilon/\epsilon'$ is uniformly bounded
(and,  actually,  can  be  chosen  arbitrarily  close  to  one). Hence,
the estimates~\eqref{eq:kappa1:of:w} combined  with the
formula~\eqref{eq:log:along:annulus:vanishes} imply that
\begin{align*}
\int_{\epsilon'\le|w|\le\epsilon}
\cf_1\Delta\cf_1 \,dxdy \to 0\\
\int_{\epsilon'\le|w|\le\epsilon}
\cf_2\Delta\cf_1 \,dxdy \to 0\\
\end{align*}
as  $\epsilon$  tends to zero, and these two terms produce no contribution
to the difference of integrals in~\eqref{eq:Polyakov:minus:Polyakov0}
either.

Finally,
\begin{multline}
\label{eq:by:parts}
\int_{\epsilon'\le|w|\le\epsilon}
\cf_1\Delta\cf_2 \,dxdy\
=\ \int_0^{2\pi} d\theta\int_{\epsilon'}^{\epsilon}
\cf_1\cdot\left[
\frac{1}{r}\frac{\partial}{\partial r}\left(r\frac{\partial}{\partial
r}\right)+ \frac{1}{r^2}\frac{\partial^2}{\partial
\theta^2}\right]\cf_2(r)\ r dr\ =\\
=\
\int_0^{2\pi} d\theta\int_{\epsilon'}^{\epsilon}
\cf_1
\frac{\partial}{\partial r}\left(r\frac{\partial\cf_2}{\partial
r}\right)\, dr
\ =\
\int_0^{2\pi}d\theta
\left(
\cf_1\,r\frac{\partial\cf_2}{\partial
r}\Big\vert_{\epsilon'}^\epsilon-
\int_{\epsilon'}^{\epsilon}
\frac{\partial\cf_1}{\partial r}
r\frac{\partial\cf_2}{\partial r}\, dr\right)
\end{multline}

Recall  that  $\cf_2(r)=\cfrac{1}{2}\log\rhofe(r)$, where $\rhofe(r)$
is defined in~\eqref{eq:def:rho:flat:epsilon}. In particular,
\begin{align*}
r\cdot\cf'_2(r)\vert_{r=\epsilon}=\frac{d}{2}\\
r\cdot\cf'_2(r)\vert_{r=\epsilon'}=0
\end{align*}
and
\begin{multline*}
\int_0^{2\pi}\left(
\cf_1\,r\frac{\partial\cf_2}{\partial r}\Big\vert_{\epsilon'}^\epsilon
\right)\,d\theta=
\frac{d}{2}\int_0^{2\pi}\cf_1(r,\theta)\,d\theta\ =\\
=\
\frac{d}{2}\int_0^{2\pi}\left(-\log\epsilon-\log|\log\epsilon|+O\left(\frac{1}{\log\epsilon}\right)\right)\,d\theta\
=\\=\
-\pi d\left(\log\epsilon+\log|\log\epsilon|\right)\,+\,O\left(\frac{1}{\log\epsilon}\right)
\end{multline*}
where  we  used  expression~\eqref{eq:kappa1:of:w}  for $\cf_1$. Once
again,  the  first term in the above expression will be compensated by
an  identical  term  in the corresponding expression for $w_0$, while
the second term in both expressions tends to zero.

It remains to evaluate the difference of the integrals
\begin{equation}
\label{eq:d:phi1:r:d:phi2}
\int_0^{2\pi}d\theta
\int_{\epsilon'}^{\epsilon}
\frac{\partial\cf_1}{\partial r}
r\frac{\partial\cf_2}{\partial r}\, dr
\end{equation}
for $w$ and for $w_0$.

By the construction~\eqref{eq:def:rho:flat:epsilon} of $\rhofe(r)$, the maximum
of  the absolute  value  of  its derivative on the interval $\epsilon'\le
r\le     \epsilon$     is     attained    at    $r=\epsilon$    where
$\rhofe'(\epsilon)=d\epsilon^{d-1}$.   Also,   by   construction   of
$\rhofe(r)$,  the minimum of its value on the interval $\epsilon'\le r\le
\epsilon$ equals $\epsilon^d\cdot(1+o(1))$. Finally, by our choice of
$\epsilon'$  and  $\epsilon$  we  have  $\epsilon/\epsilon'\to  1$ as
$\epsilon\to 0$.
Hence
\begin{multline}
\label{eq:max:min}
\max_{\epsilon'\le r\le \epsilon}
\left|r\frac{\partial\cf_2}{\partial r}\right|=
\max_{\epsilon'\le r\le \epsilon}
\left|\frac{r}{2}\frac{\partial\log\rhofe}{\partial r}\right|\le
\frac{\epsilon}{2}\cdot
\frac
{\max_{\epsilon'\le r\le \epsilon}|\rhofe'(r)|}
{\min_{\epsilon'\le r\le \epsilon}\rhofe(r)}=
\\
=\frac{\epsilon}{2}\cdot\frac{d\epsilon^{d-1}}{\epsilon^d(1+o(1))}
=\frac{d}{2}+o(1)\quad\text{ as } \epsilon\to 0
\end{multline}

Now,
$$
\cf_1=
-\log|\zeta|-\log|\log|\zeta||-\log\left|\frac{dw}{d\zeta}\right|\ =\
\log\left|\frac{dw}{d\zeta}\zeta\right|
-\log|\log|\zeta||\,.
$$
Note that
$$
\frac{dw}{d\zeta}\zeta=w\cdot f_1(w)\qquad\zeta=w\cdot f_2(w)
$$
where  $f_1(w), f_2(w)$ are holomorphic functions different from zero
in a neighborhood of $w=0$. Hence
$$
\frac{\partial}{\partial r}
\left(\log\left|\frac{dw}{d\zeta}\zeta\right|\right)=
\frac{1}{r}+O(1)
\qquad
\frac{\partial}{\partial r}
\log|\log|\zeta||=\frac{1}{\log r+O(1)}
\left(\frac{1}{r}+O(1)\right)
$$
Taking  into  consideration  estimate~\eqref{eq:max:min} this implies
that the difference of the integrals~\eqref{eq:d:phi1:r:d:phi2} taken
for $S$ and $S_0$ is of order
$$
\int_{\epsilon'}^{\epsilon}
\frac{o(1)}{r\log r}\,
\, dr\ ,
$$
which  tends to zero as $\epsilon\to 0$ since $\epsilon'/\epsilon$ is
bounded (and, actually can be chosen to tend to $1$).

We    conclude    that    in    the    limit    the   difference   of
integrals~\eqref{eq:Polyakov:minus:Polyakov0}    over   the   domains
$\epsilon'\le|w|\le\epsilon$  and  $\epsilon'\le|w_0|\le\epsilon$  is
equal   to   zero;   in   particular,   it   produces  no  contribution  to
the formula~\eqref{eq:log:det:flat:minus:log:det:hyp}.

\subsubsection{Integration  over a cusp: the domain where $|w|<\epsilon'$
but $|\zeta|>\delta$}

The computation           of          the          difference          of
integrals~\eqref{eq:Polyakov:minus:Polyakov0}    over   the   domains
$\{|w|<\epsilon'\}\cap\{|\zeta|>\delta \}$                           and
$\{|w|<\epsilon'\}\cap\{|\zeta|>\delta \}$  is  analogous  to the one in
section~\ref{sss:w:ge:epsilon}.  In  particular the difference of the
integrals  for  these  domains tends to zero as $R\to 0$ and hence it
produces                 no                 contribution
to the
formula~\eqref{eq:log:det:flat:minus:log:det:hyp}.

\subsubsection{Integration over a cusp: the annulus
$\delta'\le|\zeta|\le\delta$}

In the annulus $\delta'\le|\zeta|\le\delta$ we have
\begin{alignat*}{2}
\gfe &=\cfe\,|dw|^2
     &=\exp(2\cf_2)\,|d\zeta|^2\\
\ghd &=\rhohd(|\zeta|)\,|d\zeta|^2
     &=\exp(2\cf_1)|d\zeta|^2
\end{alignat*}
where
\begin{align*}
\cf_2&=\cfrac{1}{2}\log\cfe+\log\left|\frac{dw}{d\zeta}\right|\\
\cf_1&=\cfrac{1}{2}\log\rhohd\ .
\end{align*}
In particular, $\Delta\cf_2=0$.

Applying  the formula~\eqref{eq:Polyakov:formula:symmetric:form}  to  the
first integral in~\eqref{eq:Polyakov:minus:Polyakov0} we obtain
\begin{multline}
\label{eq:4:terms:zeta}
\int_{\delta'\le|\zeta|\le\delta} \cf(\Delta_{\ghd}\cf-2K_{\ghd})\,
d\ghd=\\
=\int_{\delta'\le|\zeta|\le\delta}
(\cf_2\Delta\cf_2-\cf_1\Delta\cf_1+
\cf_2\Delta\cf_1-\cf_1\Delta\cf_2)\, |d\zeta|^2\ =\\
=\ \int_{\delta'\le|\zeta|\le\delta}
(-\cf_1\Delta\cf_1+
\cf_2\Delta\cf_1)\, |d\zeta|^2\ .
\end{multline}
The expression $\cf_1\Delta\cf_1$ in~\eqref{eq:4:terms:zeta} does not
depend   on  $w$  or  $\bar  w$.  Hence  it  is  annihilated  by  the
corresponding expression for $w_0$.

It  remains  to compute the integral of $\cf_2\Delta\cf_1$. Similarly
to the analogous computation~\eqref{eq:by:parts} we get
$$
\int_{\delta'\le|w|\le\delta}
\cf_2\Delta\cf_1 \,|d\zeta|^2
=
\int_0^{2\pi}d\theta
\left(
\cf_2\,r\frac{\partial\cf_1}{\partial
r}\Big\vert_{\delta'}^\delta-
\int_{\delta'}^{\delta}
\frac{\partial\cf_2}{\partial r}
r\frac{\partial\cf_1}{\partial r}\, dr\right)
$$

Note that
$$
\frac{\partial\cf_1}{\partial r}=
\cfrac{1}{2}\cdot\cfrac{\partial}{\partial r}\log\rhohd(r)\,.
$$

By  definition~\eqref{eq:def:rho:hyp:delta}  of  $\rhohd(r)$  we  get
$\cfrac{\partial\cf_1}{\partial r}\Big\vert_{r=\delta'}=0$ and
\begin{equation}
\label{eq:dphi1:dr}
\frac{\partial\cf_1}{\partial r}\Big\vert_{r=\delta}=
-\frac{\partial}{\partial r}(\log r+\log|\log r|)\Big\vert_{r=\delta}=
-\left(\frac{1}{\delta}+\frac{1}{\delta\log\delta}\right)\,.
\end{equation}
Hence
\begin{multline*}
\int_0^{2\pi}d\theta
\cf_2\cdot r\cdot\frac{\partial\cf_1}{\partial
r}\Big\vert_{\delta'}^\delta=\\
= -\int_0^{2\pi}
\left(\cfrac{1}{2}\,\log\cfe+\log\left|\frac{dw}{d\zeta}\Big\vert_{\zeta=\delta e^{i\theta}}\right|\right)
\left(1+\cfrac{1}{\log\delta}\right)d\theta=\\
=
-\pi\,\log\cfe\left(1+\cfrac{1}{\log\delta}\right)
-2\pi\log\left|\frac{dw}{d\zeta}\Big\vert_{\zeta=0}\right|
+o(1)
\end{multline*}

Evaluating  the  difference with the corresponding integral for $w_0$
and  passing  to  a  limit  as  $\delta\to+0$ we see that these terms
produces the following impact to~\eqref{eq:Polyakov:minus:Polyakov0}:
$$
\cfrac{1}{12\pi}\cdot
2\pi\left(
\log\left|\frac{dw_0}{d\zeta}\Big\vert_{\zeta=0}\right|-
\log\left|\frac{dw}{d\zeta}\Big\vert_{\zeta=0}\right|
\right)
$$

It remains to evaluate the integral
$$
\int_{\delta'}^{\delta}
\frac{\partial\cf_2}{\partial r}
r\frac{\partial\cf_1}{\partial r}\, dr
$$

Recall that
$$
\min_{\delta'\le r\le \delta}\rhohd(r)=\rhohd(\delta)=
\cfrac{1}{\delta^2\log^2\delta}\,,
$$
see the definition~(\ref{eq:def:rho:hyp:delta})
of the monotone function $\rhohd(r)$. Note also that
by definition $\rhohd(r)$ has monotone derivative
on the interval $[\delta',\delta]$ and
$\rhohd'(r)$ vanishes at $r=\delta'$.
Hence, the maximum of the absolute value of the logarithmic derivative
$$
\left|\frac{\partial\phi_1}{\partial r}\right|=
\frac{1}{2}\left|\frac{\partial}{\partial r}\log\rhohd(r)\right|=
\frac{1}{2}\left|\frac{\rhohd'(r)}{\rhohd(r)}\right|
$$
on the interval $[\delta',\delta]$ is attained at the endpoint $\delta$,
where its value is already evaluated in~\eqref{eq:dphi1:dr}.
Since the function $\cfrac{\partial\cf_2}{\partial r}$ is regular, we
conclude that the integral
$$
\left|\int_{\delta'}^{\delta}
\frac{\partial\cf_2}{\partial r}
r\frac{\partial\cf_1}{\partial r}\, dr
\right|\le
\int_{\delta'}^{\delta}
\left|\frac{\partial\cf_2}{\partial r}\right|
\delta\cdot
\left(\frac{1}{\delta}+\frac{1}{\delta\log\delta}\right)\, dr
$$
tends to zero as $\delta$ tends to zero.


\subsubsection{Integration over a cusp: the disc $|\zeta|<\delta'$}
\label{sss:interior:disc}

In the disc $|\zeta|<\delta'$ we have
\begin{alignat*}{2}
\gfe &=\cfe\,|dw|^2
     &=\exp(2\cf_2)\,|d\zeta|^2\\
\ghd &=\chd\,|d\zeta|^2
     &=\exp(2\cf_1)|d\zeta|^2
\end{alignat*}
where
\begin{align*}
\cf_2&=\cfrac{1}{2}\log\cfe+\log\left|\frac{dw}{d\zeta}\right|\\
\cf_1&=\cfrac{1}{2}\log\chd
\end{align*}

Hence $\Delta\cf_1=\Delta\cf_2=0$ and the integral
$$
\int_{|\zeta|\le\delta'}
(\cf_2\Delta\cf_2-\cf_1\Delta\cf_1+
\cf_2\Delta\cf_1-\cf_1\Delta\cf_2)\, |d\zeta|^2
$$
is identically equal to zero.

Applying   the  formula~\eqref{eq:Polyakov:formula:symmetric:form}    we
conclude  that  integrals  over  this region produce no contribution to the
difference of integrals in~\eqref{eq:Polyakov:minus:Polyakov0}.

Combining the relation~\eqref{eq:int:over:complement:to:cusps} with
the estimates from
sections~\ref{sss:w:ge:epsilon}--\ref{sss:interior:disc} we get
the formula~\eqref{eq:relative:integral}.
Theorem~\ref{theorem:log:det:flat:minus:log:det:hyp} is proved.
\qed

\section{Comparison  of  relative  determinants  of Laplace operators
near the boundary of the moduli space}
\label{sec:Comparison:of:determinants:end:asymptotic:behavior}

In   notations  of  Theorem~\ref{theorem:det:minus:det:is:O:ell:flat}
define the following function:
\begin{multline}
\label{eq:E}
E(S,S_0):=\\
\left\langle\int_{S} \cf \, d\ghyp -
\int_{S_0} \cf_0 \, d\ghyp\right\rangle-2\pi
\sum_j\left(\log\left|\frac{dw}{d\zeta}(P_j)\right|-
\log\left|\frac{dw_0}{d\zeta}(P_j)\right|\right)
\end{multline}
Here      $w$      denotes      the     coordinate     defined     by
equation~\eqref{eq:def:local:coordinate:w}  in  a  neighborhood  of a
conical  point  and  $\zeta$  is  a holomorphic coordinate defined by
equation~\eqref{eq:def:local:coordinate:zeta}  in the neighborhood of
the            same            conical            point.           By
equation~\eqref{eq:log:det:flat:minus:log:det:hyp}               from
Theorem~\ref{theorem:log:det:flat:minus:log:det:hyp} one has
$$
\log\det\Dflat(S,S_0)-\log\det\Delta_{\ghyp}(S,S_0)=
\frac{1}{12\pi}\,E(S,S_0)\,.
$$
In  this  section  we  estimate  the  value  of  $E(S,S_0)$ and prove
Theorem~\ref{theorem:det:minus:det:is:O:ell:flat}.

We  will  actually  prove a stronger statement, which is in some ways
best  possible.  Recall  the thick-thin decomposition for a quadratic
differential          which          was          defined          in
\S\ref{sec:Geometric:Compactification:Theorem}.

\begin{theorem}
\label{theorem:det:minus:det:upto:O(1)}
Let      $S      \in      \cQ_1(d_1,\dots,d_\noz)$,      $S_0     \in
\cQ_1(d_1,\dots,d_\noz)$  be flat surfaces, and let $Y_1, \dots, Y_m$
be  the  $\delta$-thick  components  of  $S$. Let $Z(Y_j)$ denote the
subset  of  zeroes  and  poles  which  is contained in $Y_j$ (so that
$\bigcup_{j=1}^m
\left\{\text{\rm orders of } Z(Y_j)\right\}
= \{ d_1, \dots, d_\noz\}$). Then,
\begin{displaymath}
\left| E(S,S_0) -
\sum_{j=1}^m \left( -2\pi \chi(Y_j)-\sum_{P\in Z(Y_j)}
\frac{4\pi}{d(P)+2} \right) \log \lambda(Y_j)\right|
\le C\,,
\end{displaymath}
where $\chi(Y_j)$ is the Euler characteristic of $Y_j$ (considered as
a  surface  with  boundary  which  is  punctured  at  all  points  of
$Z(Y_j)$),   $\lambda(Y_j)$   is   the  size  of  $Y_j$  (defined  in
\S\ref{sec:Geometric:Compactification:Theorem}), and $C$ depends only
on  $\delta$,
$\eta$, $R$,
on  the  stratum  $\cQ_1(d_1,  \dots, d_\noz)$, and on $S_0$.
\end{theorem}

For each stratum, we should consider
the positive parameters $\delta,\eta, R$,
and $S_0$ as fixed.
In       this       sense,       the       constant       $C$      in
Theorem~\ref{theorem:det:minus:det:upto:O(1)}  depends  only  on  the
stratum.

We note that Theorem~\ref{theorem:det:minus:det:upto:O(1)} immediately
implies Theorem~\ref{theorem:det:minus:det:is:O:ell:flat}, since by
Lemma~\ref{lemma:size:ge:lflat}, for any thick component $Y$ of $S$,
$\ell_{flat}(S) \le \lambda(Y)$ (and since $S$ is normalized to have
unit area, $\lambda(Y) = O(1)$).

\begin{Example}[\textit{Two merging zeroes}]
\label{ex:linear:in:lf}
Consider the following
one-parameter family
of flat surfaces. Take a flat surface
with  a  zero  $P$  of order $d$, and break this zero into two zeroes
$P_1,P_2$  of orders $d_1+d_2=d$ by a local surgery in a neighborhood
of   $P$,   see~\cite{Eskin:Masur:Zorich}  for  details.  Consider
a family
of   flat   surfaces   $S_\tau$  isometric  outside  of  a
neighborhood  of  $P_1,P_2$  such  that the saddle connection joining
$P_1$ with $P_2$ contracts.

\begin{figure}[hbt]
\includegraphics{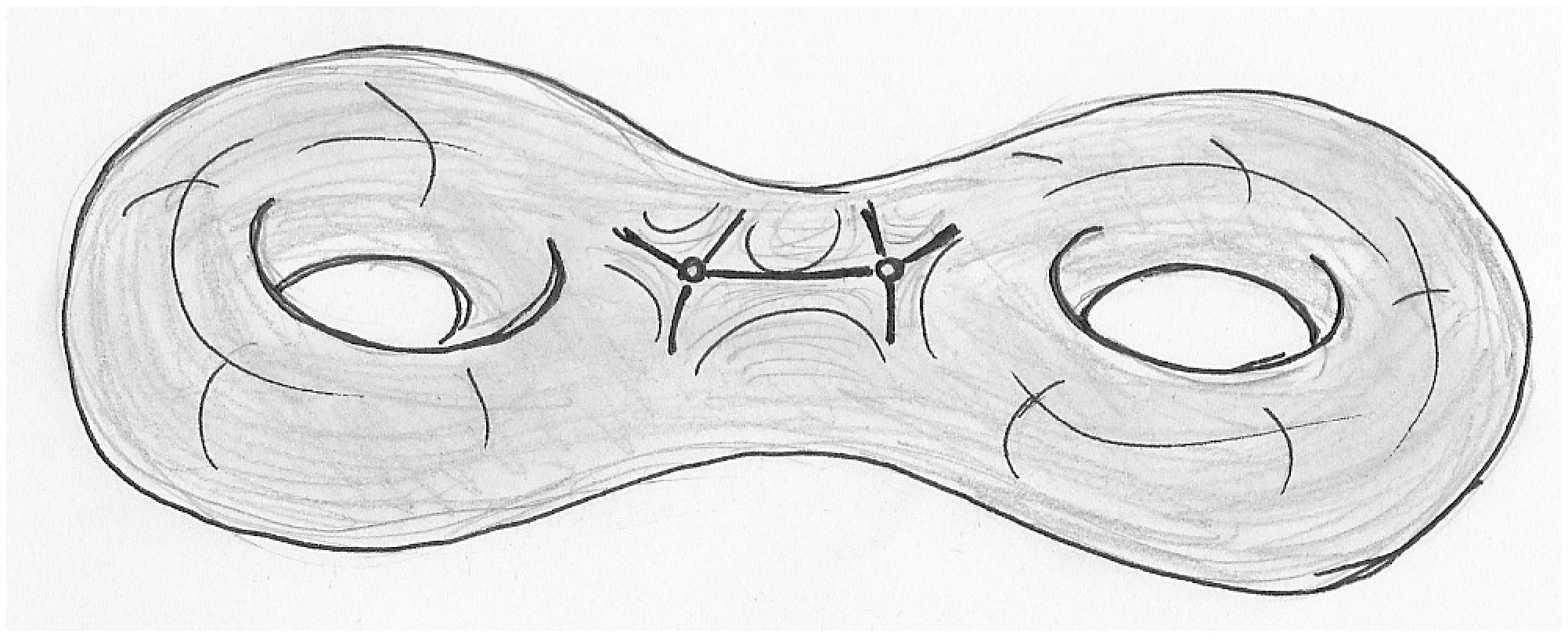}
\includegraphics{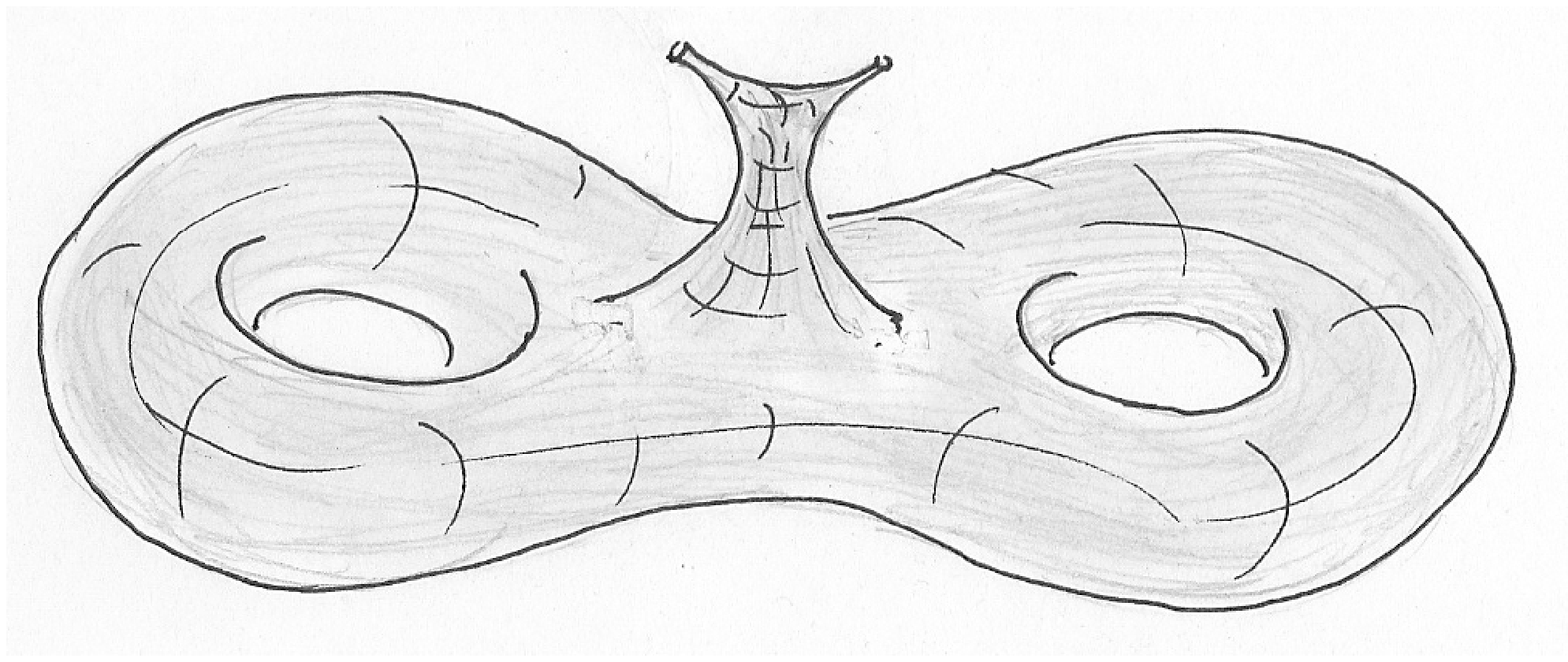}
\begin{picture}(0,0)(-25,70)
\put(-134,33){$P_1$}
\put(-98,33){$P_2$}
\put(45,67){$P_1$}
\put(82,67){$P_2$}
\end{picture}
\vspace{65pt}
\caption{
\label{fig:pinching:a:3:point:sphere}
A  simple  saddle  connection  in  the  flat  metric  produces in the
underlying  hyperbolic  metric  a  pair  of  pants  with  two  cusps.
Contracting  the  saddle connection we pinch the pair of pants out of
the main body of the surface. }
\end{figure}

For  the  underlying hyperbolic surface we get a ``bulb'' in the form
of  a pair of pants $Y_\tau$ growing out of our surface. This pair of
pants  has  cusps  at  the points $P_1,P_2$ and is separated from the
main  body of the surface by a short hyperbolic geodesic homotopic to
a          curve          encircling          $P_1,P_2$,          see
Figure~\ref{fig:pinching:a:3:point:sphere}.   Clearly,  the  size  of
$Y_\tau$      satisfies     $\lambda(Y_\tau)=2\lf(S_\tau)$,     where
$\lf(S_\tau)$  is  the  length of the short saddle connection joining
$P_1$  and  $P_2$.  The  size  of  the main body of the surface stays
bounded.  The Euler characteristic of the pair of pants $\Ypunc_\tau$
is  equal  to  minus  one. We assume that $S_\tau$ has no other short
saddle                      connections.                     Applying
Theorem~\ref{theorem:det:minus:det:upto:O(1)}, we get
\begin{multline*}
\log\det\Dflat(S,S_0)-\log\det\Delta_{\ghyp}(S,S_0)=
\frac{1}{12\pi}\,E(S,S_0)
=\\=
\frac{1}{6}\cdot
\left(1-\frac{2}{d_1+2}-\frac{2}{d_2+2}\right)\cdot \log\lf(S_\tau)
+ O(1)\,,
\end{multline*}
where the error term is bounded in terms only of the orders of the
singularities of $S_\tau$.
\end{Example}

\subsection{Admissible pairs of subsurfaces}
\label{sec:subsec:admissible:pairs}  Suppose  $Y  \subset S$ and $Y_0
\subset  S_0$  are  subsurfaces.  We  say  that the pair $(Y,Y_0)$ is
admissible  if  $Z(Y)  =  Z(Y_0)$ (i.e. the degrees of the zeroes and
poles  in $Y$ and $Y_0$ are the same). We now introduce the following
notation:  for  an admissible pair $(Y,Y_0)$, let (in the notation of
(\ref{eq:E})),
\begin{multline*}
E(Y, S; Y_0, S_0) = \left< \int_{Y} \phi \, dg_{hyp} - \int_{Y_0}
  \phi_0 \, dg_{hyp} \right>  \\ - 2\pi \sum_{P \in Z(Y)}
\left( \log \left|\frac{dw}{d\zeta}(P)\right| -
  \log\left|\frac{dw_0}{d\zeta}(P) \right| \right)\,.
\end{multline*}
This definition implies that
$$
E(Y, S; Y_0, S_0) = \int_{Y} \phi \, dg_{hyp} - \int_{Y_0}
  \phi_0 \, dg_{hyp}\,,\quad\text{when }Z(Y) = Z(Y_0)=\emptyset\,.
$$
If $Z(Y) = \emptyset$, we let
\begin{displaymath}
I(Y,S) = \int_Y \phi \, dg_{hyp}.
\end{displaymath}
Let
\begin{displaymath}
S = \left(\bigcup_{j=1}^m Y_j(\eta) \right) \cup \left(
  \bigcup_{\gamma \in \Gamma(\delta)} A_\gamma(\eta) \right)
\end{displaymath}
be  a  $(\delta,\eta)$-thick-thin decomposition of $S$ (as defined in
\S\ref{sec:subsec:delta:eta:thick:thin}).    We    now    choose    a
decomposition  $S_0$  into a sum $\bigcup_{j=1}^m Y_j'$ such that the
the  subsurfaces  $Y_j'$  have  pairwise disjoint interiors, and such
that  all  the  pairs  $(Y_j(\eta),  Y_j')$  are admissible. Then, it
follows immediately from the above definitions that
\begin{equation}
\label{eq:decomp:pieces}
E(S,S_0) = \sum_{j=1}^m E(Y_j(\eta),S; Y_j', S_0) + \sum_{\gamma \in
  \Gamma(\delta)} I(A_\gamma(\eta),S).
\end{equation}
Our  proof  of  Theorem~\ref{theorem:det:minus:det:upto:O(1)} will be
based  on (\ref{eq:decomp:pieces}). We will estimate the terms on the
right-hand-side   of   (\ref{eq:decomp:pieces})   in   the  following
subsections.

\subsection{Estimate for the thick part}
\label{sec:subsec:estimate:thick}

Recall   that  by  \textit{$\delta$-thick  components}  we  call  the
connected components of $S\setminus\sqcup_{\gamma_i\in\Gamma(\delta)}
\gamma_i$, where $\Gamma(\delta)$ is the collection of $\delta$-short
closed              hyperbolic             geodesics,             see
\S\ref{sec:Geometric:Compactification:Theorem}.

\begin{lemma}
\label{lemma:bound:E:S:S0:thick:part}   Suppose   that  $S,  S_0  \in
\cQ_1(d_1,  \dots,  d_\noz)$,  and  $(Y,Y_0)$  is an admissible pair,
where  $Y  \subset  S$,  $Y_0 \subset S_0$ and $Y$ is $\delta$-thick.
Then,
\begin{displaymath}
\left| E(Y(\eta), S; Y_0, S_0) - \left( \Area_{hyp}(Y(\eta))
-\sum_{P \in Z(Y)}
  \frac{4\pi}{d(P)+2} \right) \log \lambda(Y)\right| < C,
\end{displaymath}
where $C$ depends only on $\delta$, $\eta$,
$R$,
$Y_0$, $S_0$, where $R > 0$ be as defined in the beginning of
  \S\ref{sec:subsec:uniform:bound:conformal}.
\end{lemma}

\begin{proof}  In this proof we will say that a
  quantity is uniformly bounded if it is bounded only in terms of
  $\delta$, $\eta$, $Y_0$, $S_0$ and $R$.  Write $S = (C,q)$, and let
  $\tilde{q} = \lambda(Y)^{-2} q$. Then,
\begin{equation}
\label{eq:cf:tilde:q:cf:q}
\cf(\tilde{q})= \cf(q) - \log \lambda(Y).
\end{equation}
Let $P \in Y$ be a zero or a first-order pole of $q$. Let $\zeta$ be the
local coordinate near $P$ as in (\ref{eq:def:local:coordinate:zeta}), and
let $w$ be the local coordinate near $P$ as in
(\ref{eq:def:local:coordinate:w}). Let $\tilde{w}$ be the local
coordinate near $P$ as in (\ref{eq:def:local:coordinate:w}), for
$\tilde{q}$ instead of $q$. Then,
\begin{displaymath}
\tilde{w} = \lambda(Y)^{-2/(d(P)+2)} w,
\end{displaymath}
where $d(P)$ is the degree of $P$. Hence,
\begin{equation}
\label{eq:log:frac:tilde:w:log:frac:w}
\log \left| \frac{d\tilde{w}}{d\zeta}(P) \right| = - \frac{2}{d(P)+2}
\log \lambda(Y) +
\log \left| \frac{d w}{d\zeta}(P) \right|.
\end{equation}
Let
$\tilde{S} = (C,\tilde{q})$ and let $\tilde{Y} \subset \tilde{S}$ be
the corresponding subsurface (so the flat metric on
$\tilde{Y}$ is scaled to have size
$1$). Note that $\tilde{Y}$ and $Y$ have the same hyperbolic metric.
Then, by (\ref{eq:cf:tilde:q:cf:q}) and
(\ref{eq:log:frac:tilde:w:log:frac:w}),
\begin{multline*}
E(\tilde{Y}(\eta),\tilde{S}; Y_0, S_0) = \\
E(Y(\eta), S; Y_0, S_0) - \left( \Area_{hyp}(Y(\eta))
-  \sum_{P \in Z(Y)} \frac{4\pi}{d(P)+2} \right) \log \lambda(Y).
\end{multline*}
Thus, it is enough to show that $E(\tilde{Y}(\eta), \tilde{S}; Y_0,
S_0)$ is uniformly bounded. We may write
\begin{displaymath}
E(\tilde{Y}(\eta), \tilde{S}; Y_0, S_0) = H - 2\pi \sum_{P \in Z(Y)} J(P)\,,
\end{displaymath}
where
\begin{displaymath}
H = \left\langle \int_{Y(\eta)}\cf(\tilde{q}) \, dg_{hyp}  - \int_{Y_0}
  \phi_0 \, dg_{hyp} \right\rangle \\
\end{displaymath}
and
\begin{displaymath}
J(P) = \log \left|\frac{d\tilde{w}}{d\zeta}(P)\right| -
  \log\left|\frac{dw_0}{d\zeta}(P) \right|\,.
\end{displaymath}
Let $\cO(R) = \bigcup_{P \in Z(Y)} \cO_P(R)$.
We have
\begin{equation}
\label{eq:the:integral:I}
H = \int_{Y(\eta) \setminus \cO(R)} \cf(\tilde{q}) \, dg_{hyp} -
\int_{Y_0} \phi_0 \, dg_{hyp} + \sum_{P \in Z(Y)} \int_{\cO_P(R)}
(\cf(\tilde{q}) - \phi_0) \, dg_{hyp}\,.
\end{equation}
By Proposition~\ref{prop:pointwise:conformal:factor}, $\cf(\tilde{q})$ is
uniformly bounded (i.e. bounded depending only on $\delta$, $\eta$, $R$ and the stratum);
therefore, so is the first integral in (\ref{eq:the:integral:I})
Also, obviously the second integral in (\ref{eq:the:integral:I}) is
uniformly bounded,
since it is independent of $\tilde{q}$. To bound the third integral,
note that $\phi(\tilde{q}) - \phi_0$ is a harmonic function of the
coordinate $\zeta$ of (\ref{eq:def:local:coordinate:zeta}), and by
Proposition~\ref{prop:pointwise:conformal:factor},
$\phi(\tilde{q})-\phi_0$ is uniformly bounded on
$\partial \cO_P(R)$; then by the maximum principle,
$\phi(\tilde{q})-\phi_0$ is uniformly bounded
on all of $\cO_P(R)$. This shows that the third integral in
(\ref{eq:the:integral:I}) is uniformly bounded.

It remains to give a uniform bound for $J(P)$ for each $P$. We may write
\begin{equation}
\label{eq:tilde:q:tilde:w:d:f}
\tilde{q} = \tilde{w}^d (d\tilde{w})^2 = f(\zeta) \zeta^d (d\zeta)^2,
\end{equation}
where $\zeta$ is as in (\ref{eq:def:local:coordinate:zeta}) so $\zeta
= 0$ corresponds to the point $P$, $d$ is
the degree of $P$, and $f(\zeta)$ is some holomorphic function which
has no zeroes in $\cO_P(R)$.  Then,
we may write
\begin{displaymath}
\left(\frac{d\tilde{w}}{d\zeta}\right)^2 = f(\zeta) \left( \frac{\zeta -
      0}{\tilde{w} - 0} \right)^d,
\end{displaymath}
and taking the limit as $\zeta \to 0$ we get
\begin{displaymath}
\left(\frac{d\tilde{w}}{d\zeta}(0)\right)^2 = f(0) \left(
    \frac{d\zeta}{d\tilde{w}} (0) \right)^{d}.
\end{displaymath}
After taking logs, we get
\begin{displaymath}
\log \left| \frac{d\tilde{w}}{d\zeta}(0) \right| = \frac{1}{d+2} \log
|f(0)|.
\end{displaymath}
In view of (\ref{eq:tilde:q:tilde:w:d:f}) and of explicit formula (\ref{eq:def:local:coordinate:zeta})
for the hyperbolic metric in terms of $\zeta$, the conformal factor of
$\tilde{q}$
restricted to $\partial\cO_P(R)$
can be written as
\begin{equation}
\label{eq:cf:tilde:q:explicit:f}
\cf(\tilde{q}) = \frac{1}{2} \log \left| \frac{ f(\zeta)
    \zeta^d}{(|\zeta|^2\log^2|\zeta|)^{-1}} \right|
\quad\text{where }|\zeta|=R
\,.
\end{equation}
By Proposition~\ref{prop:pointwise:conformal:factor}, $\cf(\tilde{q})$
on $\partial \cO_P(R)$ is ``uniformly bounded'', i.e.
bounded by a constant depending only on
$\delta$, $\eta$, $R$ and the stratum; then by
(\ref{eq:cf:tilde:q:explicit:f}), $\log |f(\zeta)|$ is also
uniformly bounded on $\partial \cO_P(R)$.
Thus, by the maximum principle, $\log |f(0)|$ is uniformly
bounded.
\end{proof}

\subsection{Estimate for the thin part.}
\label{sec:subsec:estimates:thin:parts}

Let  $A_\gamma(\eta)$  be  a  thin  component  of the $(\delta,\eta)$
thick-thin   decomposition   of   a   flat   surface  $S=(C,q)$  (see
\S\ref{sec:subsec:delta:eta:thick:thin}),  corresponding to the curve
$\gamma \in \Gamma(\delta)$. Recall that each
short hyperbolic geodesic
$\gamma \in \Gamma(\delta)$
uniquely  determines  either  a flat cylinder or an expanding annulus
(see~\S\ref{sec:Geometric:Compactification:Theorem}      for      the
definitions). The short geodesic   $\gamma$   is  embedded  into  the
corresponding  maximal  flat  cylinder or expanding annulus and realizes
a generator of its fundamental group.
Let  $\lambda_+(A_\gamma)$ and $\lambda_-(A_\gamma)$ denote the sizes
of   the   $\delta$-thick   components
$Y_+,  Y_-\subset  S\setminus\Gamma(\delta)$  on  the  two  sides  of
$\gamma$.

\begin{lemma}
\label{lemma:flat:cylinder:equal:size}
Suppose  $\gamma$  is represented in the flat metric of $S$ by a flat
cylinder, of height $h$ and width $w$ (so that the flat length of the
$q$-geodesic represtative of $\gamma$ is $w$). Then,
\begin{equation}
\label{eq:lemma:flat:cylinder:equal:size}
|\log \lambda_+(A_\gamma) - \log w | \le C' \quad\text{ and } \quad
|\log \lambda_-(A_\gamma) - \log w | \le C',
\end{equation}
where $C'$ depends only on $\delta$, $\eta$ and the stratum.
\end{lemma}

It          is         important         to         note         that
Lemma~\ref{lemma:flat:cylinder:equal:size}   holds  only  because  we
consider the zeroes of the quadratic differential $q$ to be punctures
(cusps  in  the hyperbolic metric). Without this
assumption,   Lemma~\ref{lemma:flat:cylinder:equal:size}  fails,  and
part (a) of Lemma~\ref{lemma:modulus} below needs to be modified.

In   fact,   the   proof  is  contained  between  the  lines  of  the
paper~\cite{Rafi}  of K.~Rafi. However, since it is not stated in the
precise form which we need, we give a sketch of a proof below.

In    the   proofs   of   Lemmas~\ref{lemma:flat:cylinder:equal:size}
and~\ref{lemma:modulus}  the  constants $c_i$ will depend only on the
genus,  the number of punctures and the parameters $\delta, \eta$ and
$R$ of the thick-thin decomposition.

\begin{proof}[Proof of Lemma~\ref{lemma:flat:cylinder:equal:size}]
Suppose  that  the  perimeter $w$ of the cylinder is much bigger than
the  the  size $\lambda_\pm(A_\gamma)$ of the thick component $Y_\pm$
to  which it is adjacent. In order to glue a relatively wide cylinder
to  something  small  we  have  to fold the boundary of the cylinder.
However,  for  any  fixed  stratum, the complexity of this folding is
bounded   in   terms   of   the  genus  and  the  number  of  conical
singularities, which proves the inequalities
$$
\frac{w}{\lambda_\pm(A_\gamma)}\le c_1\,.
$$

\begin{figure}[htb]
\includegraphics{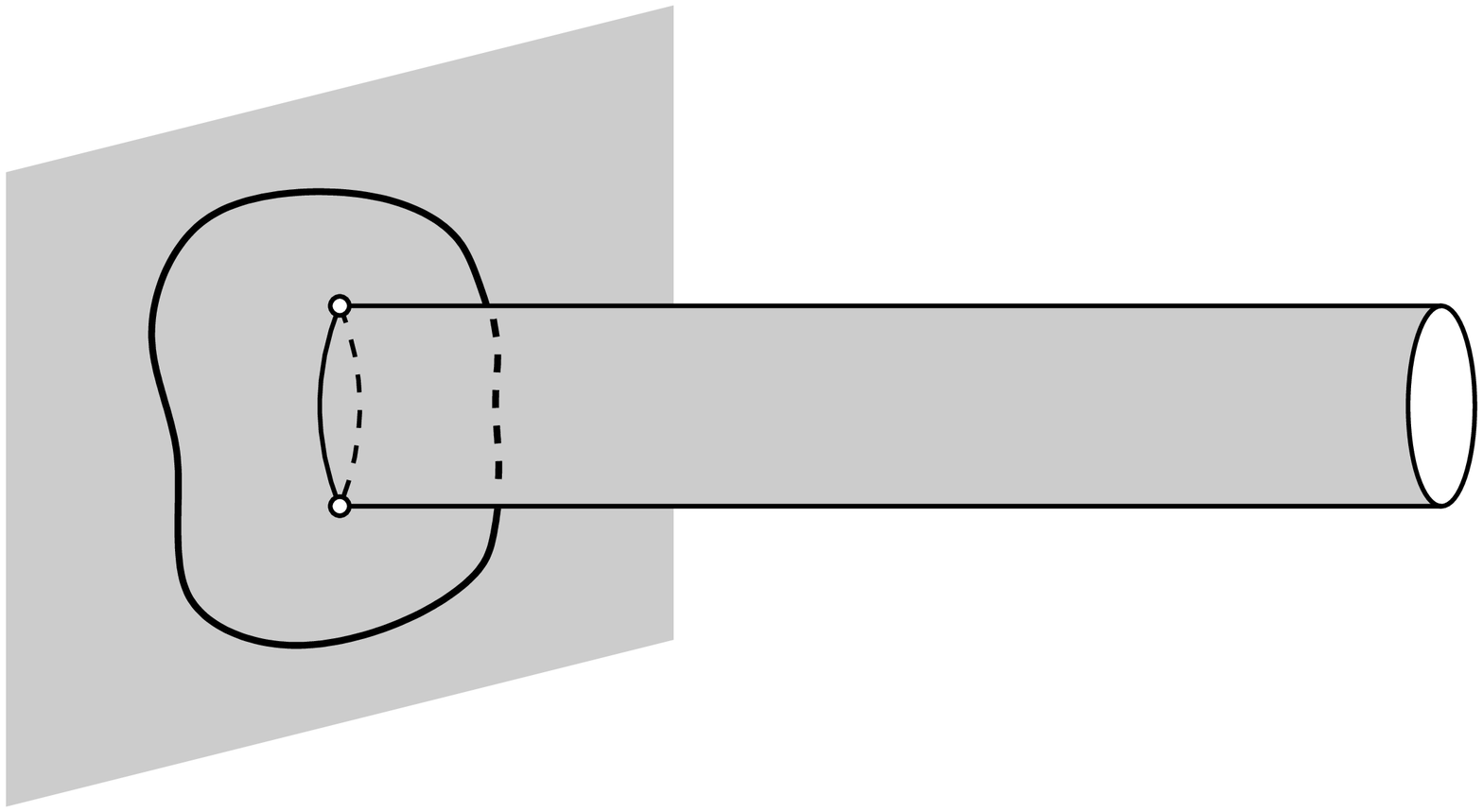}
\includegraphics{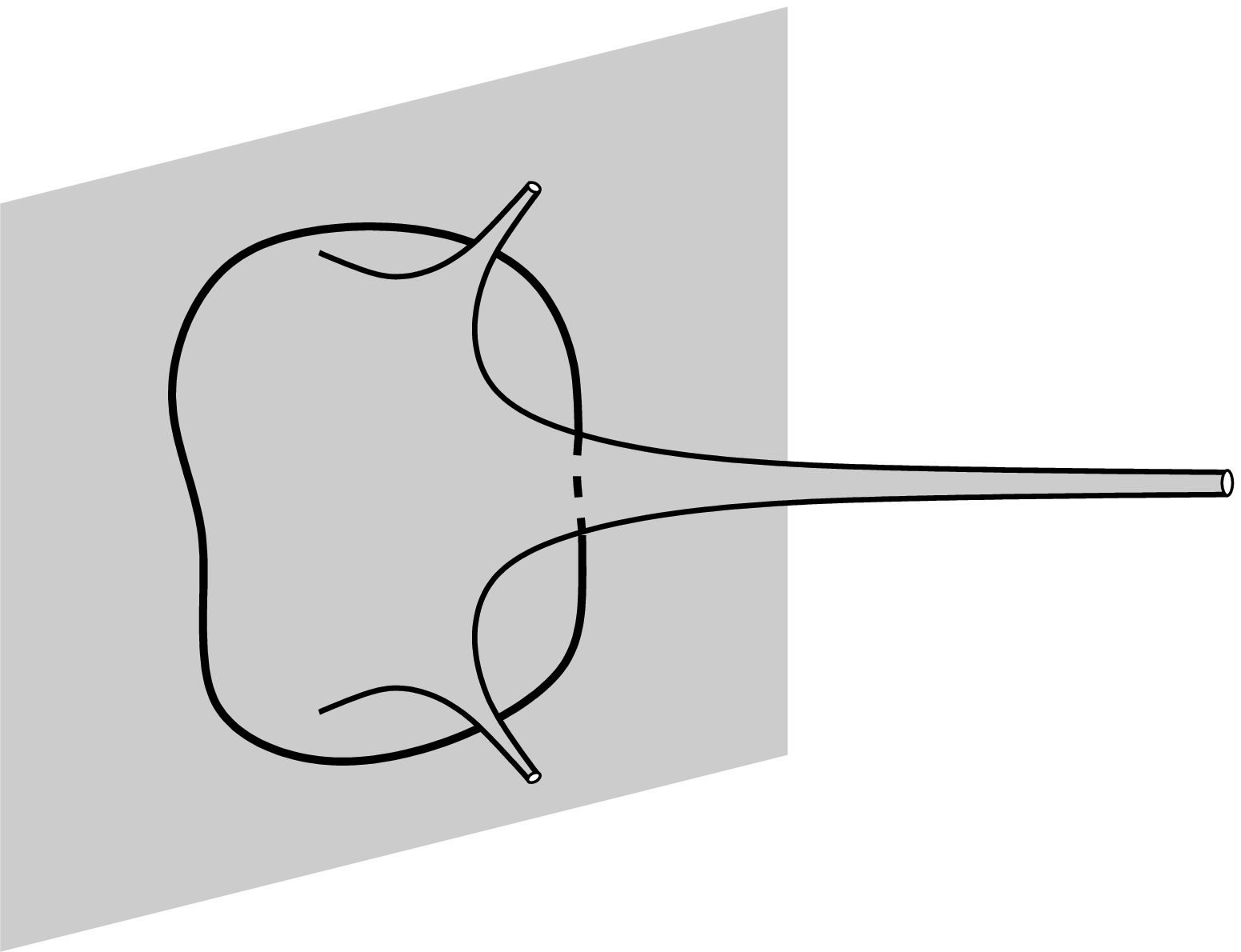}
\vspace{100pt}
\begin{picture}(0,0)(180,-115)
\put(24,-40){$\beta$}
\put(235,-40){$\beta$}
\end{picture}
\caption{
\label{fig:Lemma:7:2}
The simple closed curve $\beta$ separates all local geometry near the
boundary  of  the cylinder from the main body of the thick component.
The  left  picture  represents the flat metric, and the right picture
schematically represents the hyperbolic metric.
}
\end{figure}
Suppose  that  the perimeter $w$ of the cylinder is much smaller than
the  the  size $\lambda_\pm(A_\gamma)$ of the thick component $Y_\pm$
to which it is adjacent. Then all local geometry near the boundary of
the  cylinder  can  be  separated  from  the  main  body of the thick
component  by  a  simple  closed  curve  $\beta$  (non  necessarily a
geodesic)  such  that  the flat length of $\beta$ is much bigger than
$w$    but    much    smaller   than   $\lambda_\pm(A_\gamma)$,   see
Figure~\ref{fig:Lemma:7:2}.   Suppose  $\beta$  is  peripheral.  Then
either  $\beta$  is  homotopic  to  a  curve  in  the boundary of the
cylinder  $F_\gamma$  whose core curve is $\gamma$ or else $\beta$ is
homotopic    to    some    other    curve    in   the   boundary   of
$\lambda_{\pm}(A_\gamma$).  The  first possibility cannot occur since
the  boundary  of $F_\gamma$ has non-trivial topology (because of the
cone  points).  The  second possibility cannot occur since $\beta$ is
much smaller than the size of $\lambda_\pm(A_\gamma)$.

Thus $\beta$ must be non-peripheral. Then, by the definition of size,
$\lf(\beta)\ge  \lambda_\pm(A_\gamma)$  for  any nonperipheral curve.
Hence  our  assumption that the perimeter $w$ of the cylinder is much
smaller   than  the  the  size  $\lambda_\pm(A_\gamma)$  leads  to  a
contradiction, and we have proved that
$$
\frac{\lambda_\pm(A_\gamma)}{w}\le c_2\,.
$$
Lemma~\ref{lemma:flat:cylinder:equal:size} is proved.
\end{proof}

In  the statement and in the proof of Lemma~\ref{lemma:modulus} below
the  $c_i$  denote  constants  depending  only on the stratum, on the
parameters  $\delta$  and $\eta$ of the thick-thin decomposition, and
on  the  parameter  $R$  responsible  for neighborhoods of cusps. The
constants   $c_i$  are  different  from  those  used  in the proof of
Lemma~\ref{lemma:flat:cylinder:equal:size}.

\begin{lemma}
\label{lemma:modulus}
Suppose the constant $\delta$ defining the thick-thin decomposition is
sufficiently small (depending only on the genus and on the number of
punctures). Then,
for any $(\delta,\eta)$-thin component $A_\gamma(\eta)$ of a flat
surface $S$ the following holds:
\begin{itemize}
\item[{\rm (a)}]
Suppose $\gamma$ is represented in the flat metric by a flat
  cylinder of height $h$ and of width $w$ (so that
the flat length of the $q$-geodesic representative of $\gamma$ is
$w$). Then
\begin{displaymath}
\frac{\pi}{\ell_{hyp}(\gamma)} = \frac{h}{w} + O(1),
\end{displaymath}
where the implied constant is bounded only in terms of $\delta$,
$\eta$, $R$,and the stratum.
\item[{\rm (b)}] If $\gamma$ is represented in the flat metric
  by an expanding annulus, then
\begin{displaymath}
c_0  \le \ell_{hyp}(\gamma)
\left|\log \lambda_+(A_\gamma) - \log \lambda_-(A_\gamma)\right|
\le c_1
\end{displaymath}
where $c_0 > 0$ and $c_1> c_2$ depend only on $\delta$, $\eta$ and the
stratum.
\end{itemize}
In addition,
\begin{itemize}
\item[{\rm (c)}] There is a constant $M_0 > 0$ (depending only on
  $\delta$ and the stratum) such that any flat cylinder of modulus at
  least $M_0$ contains a hyperbolic geodesic of length at most
  $\delta$.
\end{itemize}
\end{lemma}

\begin{proof}
The        statement        (c)        is        classical,       see
e.g.~\cite[Proposition~3.3.7]{Hubbard:book}. The statement (b) is due
to                  Minsky~\cite[\S{4}]{Minsky},                  see
also~\cite[Theorem~3.1]{Rafi:fellow}.         (The         discussion
in~\cite{Rafi:fellow}  is  in  terms  of extremal lengths, but recall
that  for  very  short curves, the extremal length is
comparable to the hyperbolic length~\cite{maskit}).

The  statement  (a)  is  standard, but since we found it difficult to
extract it in the precise form we need from the literature, we give a
sketch   of   a   proof   below.   (Similar   results  can  be  found
in~\cite{Bers},  \cite[\S{6}]{Masur:metric},  \cite{Wolpert:metric}).

Let      $Y_\pm(\eta)$      be     as     in     the     proof     of
Lemma~\ref{lemma:flat:cylinder:equal:size},  and  let  $Y_\pm'(\eta,R)$
denote  $Y_\pm(\eta)$  with
$R$-neighborhoods   of   the  cusps  removed  (with  the  cuts  along
horocycles     around    cusps    of    hyperbolic    length    $R$),
see~\S\ref{sec:subsec:uniform:bound:conformal}.

Let  $\alpha_\pm$ denote the boundary curves of $A_\gamma(\eta)$. We do
not  know  the precise position of the $\alpha_\pm$
in the flat metric.
However, we claim that
\begin{equation}
\label{eq:d:flat:alpha:i:gamma:i}
\forall p \in \alpha_+\,,\ d_{flat}(p,\Sigma_+) \le c_2 w
\quad\text{and}\quad
\forall p \in \alpha_-\,,\ d_{flat}(p,\Sigma_-) \le c_2 w
\,,
\end{equation}
see Figure~\ref{fig:Lemma:7:3}.
We  prove  the  first estimate; the second one is proved analogously.

\begin{figure}[htb]
   %
   %
\includegraphics{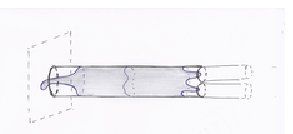}
\includegraphics{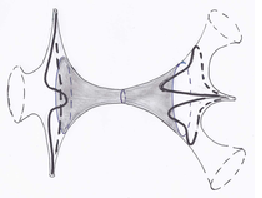}
\vspace{110pt}
\begin{picture}(0,0)(180,-110)
\put(80,-39){$\gamma$}
\put(50,-40){$\alpha_-$}
\put(110,-40){$\alpha_+$}
\put(55,-55){$A(\gamma)$}
\put(32,-70.5){$\Sigma_-$}
\put(125,-70.5){$\Sigma_+$}
\end{picture}
\begin{picture}(0,0)(2,-103)
\put(80,-44){$\gamma$}
\put(50,-30){$\alpha_-$}
\put(109,-30){$\alpha_+$}
\put(53,-55){$A(\gamma)$}
\put(41,-100){$\Sigma_-$}
\put(129,-86){$\Sigma_+$}
\end{picture}
   %
   %
\caption{
\label{fig:Lemma:7:3}
The  boundary  components  $\alpha_\pm$  of  the  hyperbolic cylinder
$A(\eta)$  of  large  modulus  (colored  in  grey)  stay  within flat
distance  of  order  $w$  from  the corresponding boundary components
$\Sigma_\pm$ of the flat cylinder $F_\gamma$ of perimeter $w$.
}
\end{figure}

Let us show that $Y_+'(\eta,R)$ has nonempty intersection
with the boundary component $\Sigma_+$ of the maximal flat cyliner $F_\gamma$.
First note that $Y_+'(\eta,R)$ cannot be completely contained in
the \textit{interior} of $F_\gamma$ for topological reasons.

By construction,
the boundary component $\alpha_+$ of $Y_+'(\eta,R)$ corresponding
to $\gamma$ is homotopic to the waist curve of the cylinder
$F_\gamma$. Each boundary component
$\Sigma_\pm$ of the maximal cylinder $F_\gamma$
passes through at least one conical singularity
of the flat metric, and this singularity defines a puncture.
Together these two observations imply that
$\alpha_+$
cannot be located completely outside of the part
$F_{\gamma,+}$
of the
flat cylinder
$F_\gamma$ bounded by $\gamma$ and $\Sigma_+$,
so $\alpha_+$ has nonempty intersection with $F_{\gamma,+}$.

Thus $Y_+'(\eta,R)$ has nonempty intersection with
$F_{\gamma,+}$ and is not contained in the interior
of $F_{\gamma,+}$. Hence, it intersects with the boundary
of $\partial F_{\gamma,+}=\gamma\sqcup\Sigma_+$.
Since $Y_+'(\eta,R)$ cannot intersect the boundary
component represented by the hyperbolic geodesic $\gamma$,
it should intersect the boundary component $\Sigma_+$.
Denote by $x_+$ a point in $Y_+'(\eta,R)\cap \Sigma_+$.

Suppose   $p   \in   \alpha_+$.  Since  the  hyperbolic  diameter  of
$Y_+'(\eta,R)$  is  bounded  by  a  constant $c_3$, there exists a path
$\lambda_{x_+,p}  \subset  Y_+'(\eta,R)$  connecting  $x_+$  to  $p$ of
hyperbolic      length      at      most     $c_3$.     But     then,
Lemma~\ref{lemma:flat:cylinder:equal:size}                        and
Proposition~\ref{prop:pointwise:conformal:factor} imply that
there exists a constant $c_2^+$ such that
the flat length of $\lambda_{x_+,p}$ is at most $c^+_2 w$.

Applying    a    similar   argument   to   $\Sigma_-$   and   letting
$c_2=\max(c_2^+,c_2^-)$            we            prove            the
estimate~\eqref{eq:d:flat:alpha:i:gamma:i}.

%
%

We note that as a consequence of (\ref{eq:d:flat:alpha:i:gamma:i}),
\begin{equation}
\label{eq:bound:area:flat}
\Area_{flat}(A_\gamma(\eta)) \le h w + c_4 w^2.
\end{equation}
Choose  any  $c_5  >  c_2$,  and  let  $A'$  denote the flat cylinder
obtained  by removing the $(c_5 w)$-neighborhood of the boundary from
$F_\gamma$.  Then,  by (\ref{eq:d:flat:alpha:i:gamma:i}), $A' \subset
A_\gamma(\eta)$.

Recall that the extremal length of a family of curves $\Gamma$ on a
surface $C$
endowed with a conformal structure
is defined to be
\begin{equation}
\label{eq:extremal:length:family}
\ext(\Gamma) = \sup_{\rho} \inf_{\gamma \in
  \Gamma}\frac{\ell_\rho(\gamma)^2}{\Area_\rho(C)}.
\end{equation}
The  supremum  in (\ref{eq:extremal:length:family}) is taken over all
the  metrics  in the conformal class of $C$. The extremal length is a
conformal  invariant,
and  the  modulus $M(A)$ of a topological annulus $A\subset C$ can be
expressed as
$$
M(A)=\frac{1}{\ext(\Gamma)}\,,
$$
where  the extremal length $\ext(\Gamma)$ is evaluated for the family
$\Gamma$ of curves $\gamma$ in $A$
given by the homotopy class of the generator of the fundamental group
of the annulus $A$.

Clearly, if $\Gamma_1 \subset \Gamma_2$ then $\ext(\Gamma_1) \ge
\ext(\Gamma_2)$. Then, since $A' \subset A_\gamma(\eta)$, we have
\begin{equation}
\label{eq:two:mods}
\Mod(A') \le \Mod(A_\gamma(\eta))\,.
\end{equation}
The cylinder $A'$ is flat, and so
\begin{equation}
\label{eq:left:mod}
\Mod(A') = \frac{h-2c_5 w}{w} = \frac{h}{w} - 2 c_5.
\end{equation}
Also by the explicit formula for the hyperbolic metric in a cylinder
(see \cite[pages 25-26 and
  page 72]{Hubbard:book} and also the
proof of Lemma~\ref{lemma:estimate:thin:part} below),
\begin{equation}
\label{eq:middle:mod}
\Mod(A_\gamma(\eta)) = \frac{\pi}{\lh(\gamma)} - c_6,
\end{equation}
where $c_6$ depends only on $\eta$.

It remains to bound $\Mod(A_\gamma(\eta))$ from above.
We now apply the definition (\ref{eq:extremal:length:family}) of
extremal length  to the family of curves $\Gamma''$ which consists of curves
homotopic to $\gamma$ and staying within
$A_\gamma(\eta)$. We get, by choosing the flat metric for $\rho$ and using
(\ref{eq:bound:area:flat}),
\begin{displaymath}
\frac{1}{\Mod(A_\gamma(\eta))} = \ext(\Gamma'') \ge
\frac{\lf(\gamma)^2}{\Area_{flat}(A_\gamma(\eta))} \ge
\frac{w^2}{hw + c_4 w^2}\,.
\end{displaymath}
Hence,
\begin{equation}
\label{eq:right:mod}
\Mod(A_\gamma(\eta)) \le \frac{h}{w} + c_4.
\end{equation}
Now part (a) of the lemma follows from (\ref{eq:two:mods}),
(\ref{eq:left:mod}), (\ref{eq:middle:mod}) and (\ref{eq:right:mod}).
\end{proof}

Let $\Gamma(\delta)$, $I(A_\gamma(\eta), S)$ be as defined in
\S\ref{sec:subsec:admissible:pairs}.

\begin{lemma}
\label{lemma:estimate:thin:part}
For any $\gamma \in \Gamma(\delta)$,
\begin{displaymath}
\left| I(A_\gamma(\eta),S) - \frac{1}{2} \Area_{hyp}(A_\gamma(\eta))(\log \lambda_+(A_\gamma) + \log
    \lambda_-(A_\gamma)) \right| < C,
\end{displaymath}
where $C$ depends only on $\delta$, $\eta$ and the stratum.
\end{lemma}

\begin{proof}
Choose  coordinates  in  which  $A_\gamma$ is represented by a
rectangle    $0\le    x\le   1\quad$;   $-h/2\le   y\le   h/2$,   see
Figure~\ref{fig:cylinder}.

\begin{figure}[htb]
\includegraphics{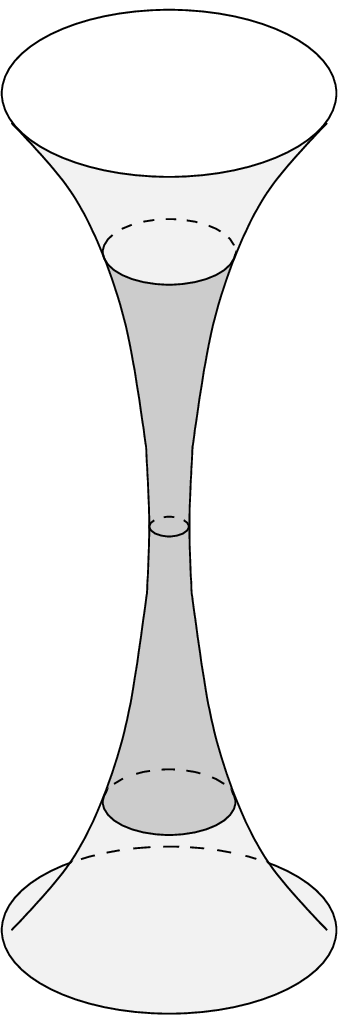}
\includegraphics{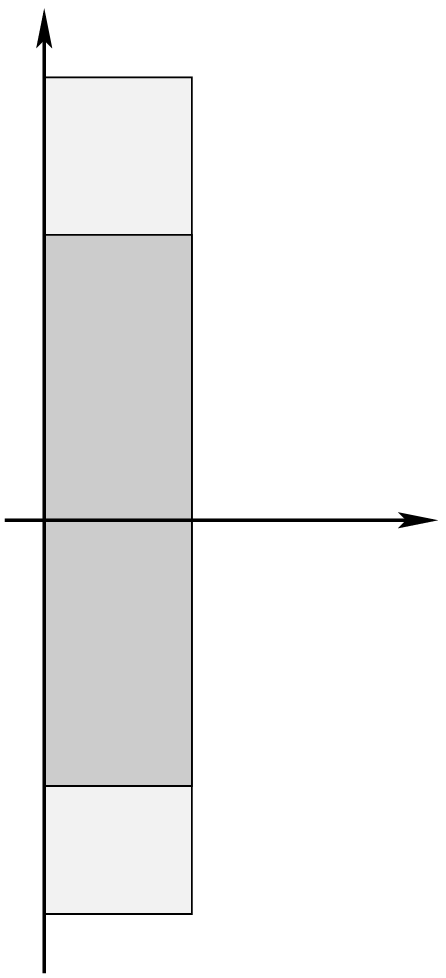}
  %
  %
\begin{picture}(0,0)(-25,70)
\put(19,5){\small$1$}
\put(33,-2){\small$x$}
\put(5,57){\small$y$}
\put(-5,45.5){\small$\frac{h}{2}$}
\put(-21,29.5){\small$\frac{h}{2}\!-\!y_0$}
\put(-10,-39.3){\small$-\frac{h}{2}$}
\put(-32,-26.7){\small$-\!\left(\!\frac{h}{2}\!-\!y_0\!\right)$}
\end{picture}
\vspace{120pt}
\caption{
\label{fig:cylinder}
Parametrization of a hyperbolic cylinder
}
\end{figure}
The hyperbolic metric on $A_\gamma$ is represented in
our coordinates as follows (see \cite[pages 25-26 and
  page 72]{Hubbard:book}):
\begin{equation}
\label{eq:canonical:hyperbolic:cylinder}
\ghyp=\frac{1}{\cos^2\left(\frac{\pi}{h}y\right)}
\left(\frac{\pi}{h}\right)^2
(dx^2+dy^2)\ .
\end{equation}
In  this hyperbolic metric the hyperbolic geodesic $\gamma$
representing  a  waist  curve  of the cylinder (the circle $y=0$) has
length  $l_{hyp}(\gamma)=\cfrac{\pi}{h}$. We assume that the modulus
of the cylinder is very large, so $l_{hyp}(\gamma)\ll 1$.

Cut  the  flat  cylinder at the vertical levels $\frac{h}{2}-y_0$ and
$-(\frac{h}{2}-y_0)$,  where  parameter $y_0(\eta)$ is chosen in such
a way that the hyperbolic length of the boundary curves is equal to
$\eta$.  As  usual, we assume that $l_{hyp}(\gamma)\ll \eta\ll 1$. It
is easy to see that
\begin{equation}
\label{eq:sin:pi:y0:over:h}
\cos\left(\frac{\pi}{h}\left(\frac{h}{2}-y_0\right)\right)=
\sin\left(\frac{\pi y_0}{h}\right)=\frac{\pi}{\eta h}=
 \frac{l_{\mathit{hyp}}(\gamma)}{\eta}\ll 1\,,
\end{equation}
so
\begin{equation}
\label{eq:formula:y0}
y_0=\frac{h}{\pi}\arcsin\frac{\pi}{\eta h}\approx \frac{1}{\eta}\ .
\end{equation}
Then, $A_\gamma(\eta)$ is represented in our coordinates by the
rectangle $0 \le x \le 1$, $-(h/2 - y_0) \le y \le (h/2-y_0)$.

The cylinder $A_\gamma(\eta)$ is subset of our surface $S$. As such,
it inherits a flat metric from the quadratic differential $q$ on
$S$. We may write
\begin{displaymath}
q = \psi(z) (dz)^2,
\end{displaymath}
where $z = x + iy$, and $\psi(z)$ is holomorphic. Note that $\psi$ has no
zeroes on $A_\gamma(\eta)$ (since zeroes of $\psi$ correspond to zeroes
of $q$ which will become cusps in our hyperbolic metric).
By (\ref{eq:canonical:hyperbolic:cylinder}),
the conformal factor $\phi(q)$ is given by:
\begin{equation}
\label{eq:phi:q:psi:cos:frac}
\phi(q) = \frac{1}{2} \log \left| \psi(x+iy)
\cos \left(\frac{\pi}{h} y\right)^2
\left(\frac{h}{\pi}\right)^2 \right|.
\end{equation}
Consider the values of $\cf(q)$ on the boundaries of
$A_\gamma(\eta)$, i.e on the segments $\alpha_+ \equiv [0,1] \times \{h/2 -y_0 \}$ and
$\alpha_- \equiv [0,1] \times \{-(h/2 - y_0)\}$. Let $\lambda_\pm$
be the
size of the thick component on the other side of $\alpha_\pm$ from
$A_\gamma(\eta)$.
Then, by Proposition~\ref{prop:pointwise:conformal:factor}, we have
\begin{equation}
\label{eq:cf:q:lambda:i}
|\cf(q) - \log \lambda_\pm| \le C,
\end{equation}
where $C$ is bounded in terms of $\delta$, $\eta$, $R$,
and the stratum. Then,
combining~\eqref{eq:sin:pi:y0:over:h}--\eqref{eq:cf:q:lambda:i}, we get
\begin{equation}
\label{eq:estimate:psi:on:Ii}
\left|\frac{1}{2} \log |\psi(z)| - \log \lambda_\pm \right| \le C'
\qquad\text{ on   $\alpha_\pm$,}
\end{equation}
where $C'$ is bounded in terms of $\eta$, $\delta$, $R$, and the stratum.

Let
\begin{equation}
\label{eq:def:fz:logpsi:z}
f(z) = \frac{1}{2} \log |\psi(z)| - \frac{\log \lambda_+ + \log \lambda_-}{2} -
\frac{(\log \lambda_+ - \log \lambda_-)y}{(h-2 y_0)}.
\end{equation}
Then, $f(z) = \frac{1}{2} \log |\psi(z)| -
\log \lambda_\pm$ on $\alpha_\pm$.
In view of (\ref{eq:estimate:psi:on:Ii}), we have $f(z) = O(1)$ on $\partial
A_\gamma(\eta)$. But $f$ is harmonic, and thus in view of the maximum
principle, $f(z) = O(1)$ (i.e. bounded in terms of $\delta$, $\eta$
$R$,
and the stratum) on all of $A_\gamma(\eta)$. Substituting
(\ref{eq:def:fz:logpsi:z}) into (\ref{eq:phi:q:psi:cos:frac}), we get
\begin{equation}
\label{eq:cfq:four:terms}
\cf(q) = \frac{\log \lambda_+ + \log \lambda_-}{2}  + \frac{(\log
  \lambda_+ - \log \lambda_-)y}{(h-2 y_0)}
+ \log \left|\cos \left(\frac{\pi}{h} y\right)
\frac{h}{\pi} \right| + f(z).
\end{equation}
We now multiply both sides by the hyperbolic metric (see
(\ref{eq:canonical:hyperbolic:cylinder}))
and integrate both sides over the rectangle $[0,1] \times [-(h/2
-y_0), (h/2 - y_0)]$. We get
\begin{equation}
\label{eq:IAS:I1:I2:I3:I4}
I(A_\gamma(\eta), S) = \frac{\log \lambda_+ + \log \lambda_-}{2}
\Area_{hyp}(A_\gamma(\eta)) + I_2 + I_3 + I_4\,,
\end{equation}
where  $I_2$,  $I_3$  and  $I_4$ are the contributions of the second,
third,                           and                           fourth
terms       in      (\ref{eq:cfq:four:terms}).      The      integral
$I_2$  vanishes  because  it  is  odd  under  the  map $y \to -y$. By
construction,  $|I_4| \le \sup |f(z)| \Area_{hyp}(A_\gamma(\eta))$ is
bounded  in  terms  of  $\delta$,  $\eta$,  $R$,  and the stratum. It
remains to bound $|I_3|$. We have

\begin{align*}
|I_3| & \le \int_0^1
\int_{-(h/2 - y_0)}^{(h/2-y_0)} \frac{1}{\cos^2(\frac{\pi y}{h})}
\left(\frac{\pi}{h}\right)^2 \left|\log \left(\cos (\frac{\pi y}{h})
\frac{h}{\pi} \right) \right| \, dx \, dy & \\
& =
2 \int_{0}^{(h/2-y_0)} \frac{1}{\cos^2(\frac{\pi y}{h})}
\left(\frac{\pi}{h}\right)^2 \left|\log \left(\cos (\frac{\pi y}{h})
\frac{h}{\pi} \right) \right| \, dy && \\
& =
2 \int_{\sin(\pi y_0/h)}^{1} \left(\frac{\pi}{h}\right) \frac{|\log
  ( h u/\pi)|}{u^2 \sqrt{1 -
    u^2}} \, du  \qquad\qquad\qquad \text{using $u =
      \cos(\pi y/h)$} \\
&
= 2\int_{\pi/(\eta h)}^{1} \left(\frac{\pi}{h}\right) \frac{|\log
  ( h u/\pi)|}{u^2 \sqrt{1 -
    u^2}} \, du  \qquad\qquad\qquad\qquad \text{using}~\eqref{eq:sin:pi:y0:over:h}
\\
& =
2 \int_{
\pi/(\eta h)
}^{1/\sqrt{2}} \left(\frac{\pi}{h}\right) \frac{|\log
  ( h u/\pi)|}{u^2 \sqrt{1 -
    u^2}} \, du + 2 \int_{1/\sqrt{2}}^{1} \left(\frac{\pi}{h}\right) \frac{|\log
  ( h u/\pi)|}{u^2 \sqrt{1 -
    u^2}} \, du \\
& =
2
(I_{3a} + I_{3b})\,.
\end{align*}
The integral $I_{3b}$ is bounded independently of $h\gg 1$
since it converges
and the integrand is bounded independently of $h$. Also,
\begin{align*}
I_{3a} & \le 2 \int_{
\pi/(\eta h)
}^{1/\sqrt{2}}
\left(\frac{\pi}{h}\right) \frac{|\log   ( h u/\pi)|}{u^2} \, du = \\
& = 2
\int_{
1/\eta
}^{\frac{h}{\pi \sqrt{2}}}
\frac{|\log v|}{v^2} \, dv && \text{using $v = h u/\pi$} \\
& \le 2 \int_{1/\eta}^{\infty}
\frac{|\log v|}{v^2} \, dv && \text{since the integral converges}.
\end{align*}
   %
We see that $I_{3a}$ is bounded depending only on $\eta$.
Thus, $|I_3|$ is bounded depending only $\eta$.
This completes the proof of the lemma.
\end{proof}

\subsection{Proof of Theorem~\ref{theorem:det:minus:det:upto:O(1)}}
The theorem follows almost immediately from (\ref{eq:decomp:pieces}),
Lemma~\ref{lemma:bound:E:S:S0:thick:part} and
Lemma~\ref{lemma:estimate:thin:part}. It remains only to note that for
any thick component $Y \subset S$,
\begin{displaymath}
\Area_{hyp}(Y(\eta)) + \frac{1}{2} \sum_{\gamma \in \partial Y}
\Area_{hyp}(A_\gamma(\eta)) = \Area_{hyp}(Y) = -2\pi \chi(Y)
\end{displaymath}
where the last equality follows from the Gauss-Bonnet theorem (since the geodesic curvature of
$\partial Y$ is $0$).  This
completes the proof of Theorem~\ref{theorem:det:minus:det:upto:O(1)}.

\section{Determinant  of  Laplacian  near  the boundary of the moduli
space}
\label{sec:det:near:the:boundary}

\subsection{Determinant of hyperbolic Laplacian near the boundary of
the moduli space}
\label{ss:Lundelius}

The  proof  of Theorem~\ref{theorem:Dflat:near:the:boundary} is based
on   the  following  result  of  R.~Lundelius,  see~\cite{Lundelius},
Theorem 1.2. This result generalizes an analogous statement proved by
S.~Wolpert in~\cite{Wolpert} for surfaces without cusps.

\begin{NNTheorem}[R.~Lundelius]
Let  $C_\tau$  be  a  family  of hyperbolic surfaces of finite volume
which tend to a stable Riemann surface $C_\infty$ as $\tau\to\infty$.
The surfaces are allowed to have cusps, but do not have boundary. Let
$C_0$  be  a  ``standard'' hyperbolic surface of the same topological
type as each $C_\tau$. Then
\begin{equation}
\label{eq:Lundelius}
-\log|\det\Delta_{\ghyp}(C_\tau,C_0)|
=\sum_k \frac{\pi^2}{3l_{\tau,k}}
+O(-\log\lh(C_\tau)) +O(1)
\end{equation}
as $\tau\to\infty$. Here $l_{\tau,k}$ are the lengths of the pinching
hyperbolic geodesics, and $\lh(C_\tau)$ is the length of the shortest
hyperbolic geodesic on $C_\tau$.
\end{NNTheorem}

\begin{Remark} The definition of relative determinant of the Laplacian
  in the hyperbolic metric used in \cite{Lundelius} differs from
  ours. However, it was shown to be equivalent by J.~Jorgenson and
  R.~Lundelius in \cite{Jorgenson:Lundelius}.
\end{Remark}

\begin{Remark}
Note  that  the original formula of R.~Lundelius contains a misprint:
the coefficient in the denominator of the leading term in Theorem 1.2
of~\cite{Lundelius}  is  erroneously indicated as ``$6$'' compared to
``$3$'' in formula~\eqref{eq:Lundelius} above. The missing factor $2$
is  lost  in  the  computation  in  section  3.3  ``Analysis  of  the
cylinder''  of~\cite{Lundelius}.  The  author  considers there a flat
cylinder  obtained  by  identifying  the vertical sides of the narrow
rectangle   $[0,l)\times  (2l,\pi-2l)$,  where  $0<l\ll  1$,  in  the
standard  coordinate  plane  and  describes the eigenfunctions of the
Laplacian on this flat cylinder with Dirichlet conditions as
$$
\sin\left(\frac{2\pi n u}{l}\right)\sin\left(\frac{2 \pi m v}{a}\right)
\qquad\text{and}\qquad
\cos\left(\frac{2\pi n u}{l}\right)\sin\left(\frac{2 \pi m v}{a}\right)
$$
while they should be written as
$$
\sin\left(\frac{2\pi n u}{l}\right)\sin\left(\frac{\pi m v}{a}\right)
\qquad\text{and}\qquad
\cos\left(\frac{2\pi n u}{l}\right)\sin\left(\frac{\pi m v}{a}\right)
$$
with  $n\in\N$  and $m\in\N\cup\{0\}$ (page 232 of~\cite{Lundelius}).
The  rest  of  the  computation  works, basically, in the same way as
in~\cite{Lundelius}  except  that  the  resulting asymptotics for the
determinant  of  Laplacian on this flat cylinder is get multiplied
by the factor $2$ producing:
$$
-\log|\det\Dflat|
\sim\frac{\pi^2}{3l}+O(\log l)\,.
$$
(The  original  paper  has ``$6$'' in the denominator of the fraction
above.)
\end{Remark}

\begin{Example}[\textit{A pair of homologous saddle connections}]
\label{ex:comparison:to:Kokotov}
Consider  the following one-parameter family of flat surfaces. Take a
pair  of  flat  surfaces  $S_1,  S_2$; make a short slit on each flat
surface;  open  up  the  slits  and  glue  the surfaces together, see
Figure~\ref{fig:pair:of:homologous:sc}.  Contracting continuously the
length $s$ of the slit we get a family of flat surfaces $S_\tau$.
\begin{figure}[hbt]
  %
  %
\includegraphics{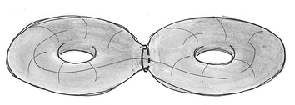}
\includegraphics{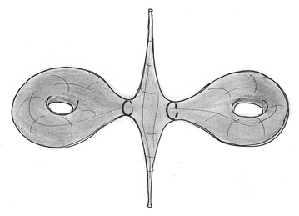}
\begin{picture}(0,0)(-25,70)
\put(-122,26){$P_1$}
\put(-122,-4){$P_2$}
\put(67,62){$P_1$}
\put(67,-44){$P_2$}
\end{picture}
   %
  %
\vspace{110pt}
\caption{
\label{fig:pair:of:homologous:sc}
A  pair  of homologous saddle connections in the flat metric produces
in  the underlying hyperbolic metric a thick component isometric to a
sphere with two cusps and with two boundary components represented by
short  hyperbolic  geodesics.  The stable curve obtained in the limit
has three irreducible components: the two Riemann surfaces underlying
$S_1$ and $S_2$ and a four-punctured sphere between them.
}
\end{figure}
For  the underlying hyperbolic surface we get three thick components:
two  obvious  ones,  but also a sphere separating the other two thick
components.  This  sphere  $Y_\tau$ has cusps at the endpoints points
$P_1,P_2$  of the slits and is separated from the rest of the surface
by a pair of short hyperbolic geodesic homotopic to curves encircling
$P_1,P_2$,  see  Figure~\ref{fig:pair:of:homologous:sc}. Clearly, the
size  of  $Y_\tau$  satisfies  $\lambda(Y_\tau)=2\lf(S_\tau)$,  where
$\lf(S_\tau)=s$ is the length of the slit. The sizes of the other two
thick components stay bounded. The Euler characteristic of the sphere
$\Ypunc_\tau$ with two holes and two punctures is equal to minus two.

Assuming  that  the slits which we made on the original flat surfaces
$S_1$,  $S_2$  are  not adjacent to conical singularities, the points
$P_1,  P_2$  on  the  compound flat surface $S_\tau$ have cone angles
$4\pi$  which  correspond  to  zeroes  of order $d=2$ in the sense of
quadratic                   differentials.                   Applying
Theorem~\ref{theorem:det:minus:det:upto:O(1)}, we get
\begin{multline}
\label{eq:det:minus:det:equals:1:6}
\log\det\Dflat(S,S_0)-\log\det\Delta_{\ghyp}(S,S_0)
=\\=
\frac{1}{6}\cdot
\left(2-\frac{2}{2+2}-\frac{2}{2+2}\right)\cdot \log\lf(S_\tau)
+ O(1)
=\frac{\log\lf(S_\tau)}{6}+ O(1)
\,,
\end{multline}
where the error term is bounded in terms only of the orders of the
singularities of $S_\tau$.

A   particular case of the above construction when the surfaces $S_1,
S_2$  belong  to  the principal stratum of Abelian differentials, was
recently    studied    by    A.~Kokotov    in   much   more   detail,
see~\cite{Kokotov:collapse:two:zeroes}. His result implies that
$$
\log\det\Dflat(S,S_0)= \frac{1}{2}\log\lf(S_\tau) + O(1)\,.
$$
We now compute the asymptotic of $\log\det\Delta_{\ghyp}(C_\tau,C_0)$
in        this       example       to       show       that       the
expression~\eqref{eq:det:minus:det:equals:1:6}  for the difference of
the flat and hyperbolic determinants matches the asymptotics obtained
by A.~Kokotov.

By Theorem of Lundelius, see~\eqref{eq:Lundelius},
$$
\log\det\Delta_{\ghyp}(C_\tau,C_0)
\sim -\sum_k \frac{\pi^2}{3l_{k}(S_\tau)}\,,
$$
where  summation is taken over all short hyperbolic geodesics. In our
case  we  have  two  short hyperbolic geodesics of approximately same
length $\lh(S_\tau)$, so we get
\begin{equation}
\label{eq:Lundelius:example}
\log\det\Delta_{\ghyp}(C_\tau,C_0)
\sim -2\cdot\frac{\pi^2}{3\lh(S_\tau)}\,.
\end{equation}
The length of a short hyperbolic geodesic is expressed in terms
of the modulus of the embodying maximal conformal annulus as
$$
\lh(S_\tau)=\frac{\pi}{\Mod_\tau}\,,
$$
see  (3.3.7) in~\cite{Hubbard:book}.

By      considering      the     Zhukovsky     function     $z\mapsto
\frac{1}{2}(z+\frac{1}{z})$  we  can  see that asymptotically, as the
size $\lf(S_\tau)$ of the slit tends to zero,
   %
   %
   %
$$
\frac{1}{\lh(S_\tau)}\sim -\frac{\log\lf(S_\tau)}{2\pi^2}\,.
$$
Plugging this into~\eqref{eq:Lundelius:example} we get
$$
\log\det\Delta_{\ghyp}(C_\tau,C_0)
\sim -2\cdot\frac{\pi^2}{3\lh(S_\tau)}
\sim -2\cdot\frac{\pi^2}{3}\,\frac{-\log\lf(S_\tau)}{2\pi^2}=
\frac{\log\lf(S_\tau)}{3}\,.
$$
Thus,
$$
\log\det\Dflat(S,S_0)-\log\det\Delta_{\ghyp}(S,S_0)\sim
\frac{\log\lf(S_\tau)}{6}\,,
$$
which matches~\eqref{eq:det:minus:det:equals:1:6}.
\end{Example}

It is immediate to recast the above Theorem of R.~Lundelius as a
uniform bound:
\begin{corollary}
\label{cor:lundelius}
Let  $C$, $C_0$ be two hyperbolic surfaces of finite volume and the
same topological type.
The surfaces are allowed to have cusps, but do not have boundary. Let
$\delta > 0$ (depending only on the genus $g$ and the number of
cusps $n$) be such that any two curves of hyperbolic length less than
$\delta$ are disjoint. Then, there exists $c_1 > 0$ (depending only
on $g$, $n$, $\delta$ and $C_0$) such that
\begin{equation}
\label{eq:cor:Lundelius}
\left| \log|\det\Delta_{\ghyp}(C,C_0)|
+ \frac{\pi^2}{3}\sum_{\gamma \in \Gamma(\delta)}
\frac{1}{\ell_{hyp}(\gamma)}\right| \le c_1 ( 1 +
|\log \lh(C)|)
\end{equation}
Here $\Gamma(\delta)$ is the set of closed geodesics of length at most
$\delta$ (so the cardinality of $\Gamma(\delta)$ is at most
$(3g-3+n)$), and $\lh(C)$ is the length of the shortest
hyperbolic geodesic on $C$.
\end{corollary}

\begin{proof}[Proof of Corollary~\ref{cor:lundelius}]
The proof is by contradiction. If such a constant $c_1$ did not exist,
then there would exists a sequence $C_\tau$ with fixed topology such
that
\begin{equation}
\label{eq:lundelius:seq:counterexamples}
\frac{1}{1 +
|\log \lh(C_\tau)|} \left| \log|\det\Delta_{\ghyp}(C_\tau,C_0)|
+ \frac{\pi^2}{3}\sum_{\gamma \in \Gamma(\delta)}
\frac{1}{\ell_{hyp}(\gamma)}\right| \to \infty
\end{equation}
The existence of the
Deligne-Mumford compactification implies that (after passing to a
subsequence) we  may assume that the sequence $C_\tau$ tends to a
stable Riemann surface $C_\infty$. Then, from (\ref{eq:Lundelius}) we
see that the left-hand-side of
(\ref{eq:lundelius:seq:counterexamples}) is bounded. This contradicts
(\ref{eq:lundelius:seq:counterexamples}).
\end{proof}

\subsection{Proof of Theorem~\ref{theorem:Dflat:near:the:boundary}}


We start
with the following preparatory Lemma.

\begin{Lemma}
\label{lemma:O:lhyp:O:lflat}
Consider  a  stratum $\cQ(d_1,\dots,d_\noz)$ of meromorphic quadratic
differentials  with  at most simple poles (the case of global squares
of  $1$-forms  is  not  excluded).  Let  $\lf(S)$  be the length of a
shortest saddle connection on a flat surface $S$; let $\lh(S)$ be the
length  of  the  shortest geodesic in the canonical hyperbolic metric
with cusps in the conformal class of $S$.

The following estimate is valid for any flat surface $S$ of unit area
in the stratum:
$$
|\log\lh(C)|=O(|\log\lf(S)|)
$$
where
$$
O(|\log\ell_{\mathit{flat}}(S)|)
\le 2|\log\lf(S)|+C(g,\noz)
$$
with  $C(g,\noz)$  depending only on a genus of $S$ and on the number
$\noz$ of zeroes and simple poles of the quadratic differential.
\end{Lemma}
\begin{proof}
It is straightforward to deduce this lemma from
Lemma~\ref{lemma:modulus}, but we find the following argument more
illuminating.
Recall that the extremal length of a curve $\gamma$ on a Reimann surface $C$
is defined to be:
\begin{displaymath}
\ext(\gamma) = \sup_{\rho} \inf_{\alpha \in [\gamma]}
\frac{\ell_\rho(\alpha)^2}{\Area_\rho(C)},
\end{displaymath}
where the inf is over the homotopy class $[\gamma]$ of $\gamma$, and
the sup is over all metrics in the conformal class of $C$. Letting
$\rho$ be the flat metric on $C$ we get
\begin{displaymath}
\ext(\gamma) \ge \ell_{flat}(\gamma)^2.
\end{displaymath}
It is a well known fact (see e.g. \cite{maskit}) that for sufficiently short
curves, the hyperbolic length is comparable to the extremal
length. Then, taking logs completes the proof of the lemma.
\end{proof}

\begin{proof}[Proof of Theorem~\ref{theorem:Dflat:near:the:boundary}]
We choose $\delta > 0$ so that Lemma~\ref{lemma:modulus}
holds, and also Corollary~\ref{cor:lundelius} holds.
Choose $M > M_0$ where $M_0$ is as in Lemma~\ref{lemma:modulus}
(c). As above, let $\Gamma(\delta)$ denote the simple closed curves of
hyperbolic length at most $\delta$.
Let $\Gamma'_M$ denote the simple closed curves which are represented
in the flat metric by a flat cylinder of modulus at least $M$. Then
the sum in (\ref{eq:cor:Lundelius}) is over $\gamma \in
\Gamma(\delta)$, while the sum in the expression
(\ref{eq:Dflat:near:the:boundary}) in the statement of
Theorem~\ref{theorem:Dflat:near:the:boundary} is over $\gamma \in
\Gamma'_M$. By Lemma~\ref{lemma:modulus} (c), $\Gamma'_M \subset
\Gamma(\delta)$.
Let
\begin{align*}
\cE(S) & \equiv \frac{\pi^2}{3} \sum_{\gamma \in \Gamma(\delta)}
\frac{1}{\lh(\gamma)} - \frac{\pi}{3} \sum_{\gamma \in \Gamma'_M}
\frac{h(\gamma)}{w(\gamma)} \\
 & =  \frac{\pi}{3} \sum_{\gamma \in \Gamma(\delta) \cap \Gamma'_M} \left(
 \frac{\pi}{\lh(\gamma)} - \frac{h(\gamma)}{w(\gamma)} \right) +
\frac{\pi^2}{3} \sum_{\gamma \in \Gamma(\delta)\setminus \Gamma'_M}
\frac{1}{\lh(\gamma)}.
\end{align*}
We claim that
\begin{equation}
\label{eq:bound:cE:error:term}
|\cE(S)| = O(| \log \ell_{flat}(S)|).
\end{equation}
Indeed, since the number of terms in both sums defining $\cE(S)$ is
bounded by $3g-3+n$, it is enough to bound each term separately. If
$\gamma \in \Gamma(\delta) \cap \Gamma'_M$ then by
Lemma~\ref{lemma:modulus} (a),
\begin{displaymath}
\left|\frac{\pi}{\lh(\gamma)} -
\frac{h(\gamma)}{w(\gamma)}\right| = O(1).
\end{displaymath}
Now suppose $\gamma \in
\Gamma(\delta)\setminus \Gamma'_M$. Since $\gamma \in \Gamma(\delta)$,
$\gamma$ is
represented in the flat metric by either a flat cylinder or an
expanding annulus. If the representative is a flat cylinder, then,
since $\gamma \not\in \Gamma'_M$, the
modulus of the cylinder can be at most $M$; this implies by
Lemma~\ref{lemma:modulus} (a) that $\frac{1}{\lh(\gamma)}$ is bounded in
terms of $M$, i.e. $\frac{\pi}{\lh(\gamma)} = O(1)$. If the
representative of $A_\gamma(\eta)$ is an expanding annulus, then by
Lemma~\ref{lemma:modulus} (b), and Lemma~\ref{lemma:size:ge:lflat},
\begin{displaymath}
\frac{\pi}{\lh(\gamma)} \approx \left|\log
\frac{\lambda_+(A_\gamma)}{\lambda_-(A_\gamma)} \right| \le \left|\log
\frac{O(1)}{\lf(S)} \right|.
\end{displaymath}
Therefore, in this case, $\frac{\pi}{\lh(\gamma)} = O(|\log
\lf(S)|)$. This concludes the proof of (\ref{eq:bound:cE:error:term}).

Now Theorem~\ref{theorem:Dflat:near:the:boundary} follows immediately
from Corollary~\ref{cor:lundelius}, Lemma~\ref{lemma:O:lhyp:O:lflat}
and (\ref{eq:bound:cE:error:term}).
\end{proof}

\section{Cutoff near the boundary of the moduli space}
\label{sec:cutoff}

In             this            section            we            prove
Theorem~\ref{theorem:int:Dflat:equals:SV:const}          establishing
relation~\eqref{eq:int:Dflat:equals:SV:const} between the integral of
$\Dhyp\log|\det\Dflat(S,S_0)|$  over  a regular invariant suborbifold
$\cM_1$  and  the  Siegel---Veech constant $c_{\mathit{area}}(\cM_1)$
corresponding to this suborbifold.

The  only  property  of $\log|\det\Dflat(S,S_0)|$ which we use in the
current   section   is,   basically,   reduced   to   the  asymptotic
formula~\eqref{eq:Dflat:near:the:boundary}                       from
Theorem~\ref{theorem:Dflat:near:the:boundary}.  This formula does not
distinguish  flat surfaces defined by Abelian differentials from flat
surfaces  defined by meromorphic quadratic differentials with at most
simple poles. Thus, in the current section it is irrelevant whether a
regular invariant suborbifold $\cM_1$ belongs to a stratum of Abelian
differentials  or to a stratum of meromorphic quadratic differentials
with at most simple poles.

Recall  that the Laplace operator associated to the hyperbolic metric
of   curvature   $-4$  on  Teichm\"uller  discs  is  defined  on  the
projectivized  strata  $\PcH(m_1,\dots,m_\noz)$;  it  acts  along the
leaves  of  the  corresponding foliation in $\PcH(m_1,\dots,m_\noz)$.
The relative determinant of the flat Laplacian $\det\Dflat(S,S_0)$ is
defined   for   flat   surfaces  $S$  of  area  one  in  the  stratum
$\cH_1(m_1,\dots,m_\noz)$.    Note,   that   $\det\Dflat(S,S_0)$   is
invariant under the action of $\SO$. Using the natural identification
$$
\PcH(m_1,\dots,m_\noz)\simeq \cH_1(m_1,\dots,m_\noz)/\SO
$$
we    may    consider    $\det\Dflat(S,S_0)$   as   a   function   on
$\PcH(m_1,\dots,m_\noz)$.

In practice, it would be convenient to pull back all the functions to
the stratum $\cH_1(m_1,\dots,m_\noz)$ and work there. Throughout this
section we consider only those functions on $\cH_1(m_1,\dots,m_\noz)$
which are $\SO$-invariant.

\subsection{Green's Formula and cutoff near the boundary}

We  start  by  recalling  Green's Formula adopted to our notations.

\begin{GreenFormula}
Suppose  that $f_1: \cM_1 \to \reals$ and $f_2: \cM_1 \to \reals$ are
continuous,     leafwise-smooth     along     Teichm\"uller    discs,
$\SO$-invariant,  and  at  least  one  of  the  functions has compact
support. Then,
\begin{equation}
\label{eq:stokes:teich}
\int_{\cM_1} f_1 (\Dhyp f_2)\, d\nu_1
  =
- \int_{\cM_1} (\nhyp f_1) \cdot
(\nhyp f_2)\, d\nu_1
  =
\int_{\cM_1} (\Dhyp f_1) f_2 \,d\nu_1\,.
\end{equation}
\end{GreenFormula}

Let  $C$  be  a  flat  cylinder.  We denote its modulus by $\Mod(C)$.
(Recall  that the modulus of a cylinder with closed horizontal curves
is  its  height  divided  by  its width.) We denote the length of the
waist  curve  (i.e.  of the closed trajectory) of the cylinder $C$ by
$w(C)$. For any point $S \in \cM_1$, let $Cyl_K(S)$ denote the set of
cylinders with modulus at least $K$.
We   shall   always   assume   that  $K$  is  large  enough, so  that
condition~\eqref{eq:condition:III} is satisfied.
We also assume that $K$ is sufficiently large so that
the core curves of all the cylinders
in    $Cyl_K(S)$    are    short    in    the    hyperbolic   metric,
see~\cite{Wolpert:metric}  or~\cite{Wolpert:metric:survey}. Thus, the
cylinders  in $Cyl_K(S)$ are disjoint, and their number is bounded by
$3g-3+\noz$.
Let
\begin{displaymath}
\ell_K(S) = \min_{C \in Cyl_K(S)} w(C).
\end{displaymath}
We set $\ell_K(S) = 1000$ if $Cyl_K(S)$ is empty.

As      in     Theorem~\ref{theorem:Dflat:near:the:boundary},     let
$\ell_{\mathit{flat}}(S)$  be  the  length  of  the  shortest  saddle
connection     in     the    flat    metric    on    $S$.    Clearly,
$\ell_{\mathit{flat}}(S) \le \ell_K(S)$.

\begin{lemma}
\label{lemma:sv:sc:estimate}
For    any   invariant   suborbifold   $\cM_1$,   we   have
\begin{equation}
\label{eq:sv:sc:estimate}
\nu_1\big( \{ S \in \cM_1 \;|\; \ell_{\mathit{flat}}(S) <
\epsilon \} \big) \le C \epsilon^2,
\end{equation}
where  $C$  depends  only  on  $\cM_1$.

In       particular, (after summing the geometric
  series), we see that       for       any       $\beta      <      2$,
$\big(\ell_{\mathit{flat}}(\cdot)\big)^{-\beta}   \in   L^1(\cM_1,
\nu_1)$.
\end{lemma}

\begin{proof} We use only the fact that $\nu_1$ is an
$\SL$-invariant
probability  measure (and not the manifold structure of $\cM_1$). Let
$N_s(S,L)$  denote  the number of saddle connections on $S$ of length
at  most  $L$.  By  the  Siegel---Veech  formula  applied  to  saddle
connections  \cite{Veech:Siegel},  \cite[Theorem~2.2]{Eskin:Masur} we
have for all $\epsilon > 0$,
\begin{displaymath}
\int_{\cM_1} N_s(S,\epsilon) \, d\nu_1(S) =
c_s(\cM_1) \cdot\pi \epsilon^2\,.
\end{displaymath}
Note  that  if $\ell_{\mathit{flat}}(S) < \epsilon$, $N_s(S,\epsilon)
\ge 1$. It follows that
\begin{displaymath}
\nu_1\big(\{ S \in \cM_1 \;|\; \ell_{\mathit{flat}}(S) < \epsilon \}\big)
\le \int_{\cM_1} N_s(S,\epsilon) \, d\nu_1(S) \le
c_s(\cM_1)\cdot\pi\epsilon^2\,.
\end{displaymath}
\end{proof}

Let  $\chi_\epsilon$  be the characteristic function of the set $\{ S
\in  \cM_1  \; \mid \; \ell_K(S) \ge \epsilon \}$. Pick a nonnegative
$\SO$-invariant  smooth  function  $\eta:  \SL  \to \reals$ such that
$\int_{\SL}  \eta(g)  \,  dg = 1$, and $\eta$ is supported on the set
$\{  g  \;|\;  1/2  < \|g\| < 2 \}$. Here $\|g\|$ is the operator
norm of $g$, viewed as a $2 \times 2$ matrix. Let
\begin{equation}
\label{eq:f:epsilon}
f_\epsilon(S):= \int_{\SL} \eta(g) \chi_\epsilon(g S) \, dg\ ,
\end{equation}
where  $dg$  is  the  Haar  measure  on  $\SL$.  Note  that since the
functions  $\eta$  is $\SO$-invariant, $f_\epsilon:\cM_1\to\reals$ is
also       $\SO$-invariant       and      thus      quotients      to
$f_\epsilon: \operatorname{\mathbb{P}}\!\cM \to\reals$.

\begin{lemma}
\label{lemma:properties:f2}
The  nonnegative  function  $f_\epsilon:  \cM_1  \to
\reals$ has the following properties:
\begin{itemize}
\item[{\rm (a)}] $f_\epsilon(S) = 0$ if $\ell_K(S) \le \epsilon/2$.
\item[{\rm (b)}] $f_\epsilon(S) = 1$ if $\ell_K(S) \ge 2\epsilon$.
\item[{\rm  (c)}] $f_\epsilon$ is leafwise-smooth along Teichm\"uller
discs,  and  $\nhyp f_\epsilon$ and $\Dhyp f_\epsilon$ are bounded on
$\cM_1$ by a uniform bound independent of $\epsilon$.
\end{itemize}
\end{lemma}
\begin{proof}
The properties (a) and (b) are clear from the definition. To see that
(c) holds, note that for $h(t) \in \SL$ we can rewrite
$$
f_\epsilon(hS) = \int_{\SL} \eta(g)\, \chi_\epsilon(ghS) \, dg =
\int_{\SL} \eta\big(g h^{-1}(t)\big)\, \chi_\epsilon(gS) \, dg
$$
and (c) follows since $\eta$ is smooth
and has compact support.
\end{proof}

\subsection{Restriction   to   cylinders  of  large  modulus  sharing
parallel core curves}
\label{ss:restriction:to:parallel:cylinders}

Let $\widetilde{Cyl}_K(S) \subseteq Cyl_K(S)$ denote those cylinders,
which are parallel to the cylinder whose waist curve is the shortest.
If  there  are  two  cylinders  in  $Cyl_K(S)$ with nonparallel waist
curves   of   the   same   shortest   length  $\ell_K(S)$  we  define
$\widetilde{Cyl}_K(S)$ to be empty.

We define
\begin{displaymath}
\psi^K(S) := \sum_{C \in Cyl_K(S)} (\Mod(C)-K)\,,
\end{displaymath}
and
\begin{displaymath}
\tilde{\psi}^K(S) := \sum_{C \in \widetilde{Cyl}_K(S)} (\Mod(C)-K)\,.
\end{displaymath}
By  convention,  a  sum  over  an empty set is defined to be equal to
zero.  Thus,  both  functions $\psi$ and $\tilde\psi$ are continuous,
piecewise smooth, and $\SO$-invariant on $\cM_1$.
Recall that it follows from our assumptions on $K$ that the waist
curves of the
cylinders  in $Cyl_K(S)$ are disjoint, and their number is bounded by
$3g-3+\noz$.  Since  the  area  of  any  cylinder  is at most $1$, it
follows that
\begin{equation}
\label{eq:disjoint:cyls}
\tilde{\psi}^K(S) \le \psi^K(S) \le
\frac{3g-3+\noz}{\big(\ell_{\mathit{flat}}(S)\big)^2}\ .
\end{equation}

\begin{lemma}
\label{lemma:can:neglect:nonparallel}
Let  $\cM_1$  be  a  regular  suborbifold,  and $f_\epsilon$ be as in
(\ref{eq:f:epsilon}). Then,
$$
\int_{\cM_1}
\Dhyp \log\det\Dflat(S,S_0) \,d\nu_1
= \cfrac{\pi}{3}\cdot
\lim_{\epsilon \to 0}
\int_{\cM_1}
\nhyp  \tilde{\psi}^K \cdot
 \nhyp f_\epsilon\,d\nu_1\ .
$$
\end{lemma}

\begin{proof}
By assumption $\cM_1$ is regular. Let $f := \log\det\Dflat(S,S_0)$.
Note  that  $f_\epsilon(S) \to 1$ as $\epsilon \to 0$. Then, by Green's
Formula~\eqref{eq:stokes:teich},
\begin{equation}
\label{eq:ders:on:f2}
\int_{\cM_1} \Dhyp f\, d\nu_1
   =
\lim_{\epsilon \to 0}
\int_{\cM_1} f_\epsilon\, \Dhyp f\,d\nu_1
   =
\lim_{\epsilon \to 0}
\int_{\cM_1} f\,\Dhyp f_\epsilon\,d\nu_1\ .
\end{equation}
Now,      by     equation~\eqref{eq:Dflat:near:the:boundary}     from
Theorem~\ref{theorem:Dflat:near:the:boundary} we have
\begin{equation}
\label{eq:log:error}
f(S) = -\cfrac{\pi}{3}\cdot\psi^K(S)
  + O\big(\log(\ell_{\mathit{flat}}(S))\big)\,,
\end{equation}
where   we  use  that  $K\cdot\card(Cyl_K(S))\le  (3g-3+n)K=O(1)$  is
dominated by $O\big(\log(\ell_{\mathit{flat}}(S))\big)$.

Note  that  by  Lemma~\ref{lemma:properties:f2}  the  function $\Dhyp
f_\epsilon$ is bounded and supported on the set
\begin{equation}
\label{eq:M:1:epsilon}
\cM_1^{\epsilon}=\{ S \; | \; \epsilon/2 < \ell_K(S) < 2 \epsilon \}\,.
\end{equation}
Since  $\lf(S)\le\ell_K(S)$, Lemma~\ref{lemma:sv:sc:estimate} implies
that  $\nu_1(\cM_1^\epsilon)=O(\epsilon^2)$.  Also,  it  follows from
Lemma~\ref{lemma:sv:sc:estimate},       that       the       function
$|\log\ell_{\mathit{flat}}|$  is  of  the  class  $L^1(\cM_1,\nu_1)$.
Then, by the dominated convergence theorem, we get:
\begin{displaymath}
\lim_{\epsilon \to 0} \int_{\cM_1} |\log
\ell_{\mathit{flat}}|\, \Dhyp f_\epsilon\,\,d\nu_1 = 0\,.
\end{displaymath}
Therefore,
\begin{equation}
\label{eq:intermediate:K}
\int_{\cM_1}
\Dhyp f\,d\nu_1
   =
-\cfrac{\pi}{3}\cdot\lim_{\epsilon \to 0}
\int_{\cM_1}
\psi^K\, \Dhyp f_\epsilon\,d\nu_1\,.
\end{equation}

Recall the definition of $\cM_1(K,\epsilon)$ from (\ref{eq:condition:III}).
Since $\cM_1$ is regular, there exists a function $R(\epsilon)$
(depending on $\cM_1$) with $R(\epsilon) \to \infty$ as $\epsilon \to
0$ such that
\begin{equation}
\label{eq:condition:III:R}
\lim_{\epsilon \to 0} \frac{\nu_1(\cM_1(K,\epsilon R(\epsilon)))}{\epsilon^2} = 0.
\end{equation}
For $S \in \cM_1^\epsilon$, we may write
\begin{displaymath}
\psi^K(S) = \psi_1^K(S) + \psi_2^K(S),
\end{displaymath}
where $\psi_2^K(S)$ is the contribution of all cylinders in
$Cyl_K(S)\setminus \widetilde{Cyl}_K(S)$
with waist curve of length at least
$\epsilon R(\epsilon)$, and
$\psi_1^K(S)$ is the contribution of the rest of the cylinders. Then,
\begin{equation}
\label{eq:disjoint:cyls:2}
\tilde{\psi}^K(S) \le \psi_1^K(S) \le \psi^K(S) \le
\frac{3g-3+\noz}{\big(\ell_{\mathit{flat}}(S)\big)^2}\ .
\end{equation}
Also, as in (\ref{eq:disjoint:cyls}), for $S \in \cM_1^\epsilon$ we have
\begin{displaymath}
\psi_2^K(S) \le \frac{3g-3+n}{\epsilon^2 R(\epsilon)^2}.
\end{displaymath}
By Lemma~\ref{lemma:properties:f2} (c), $|\Dhyp f_\epsilon|$ is bounded
by some constant $C(\cM_1)$ which does not depend on $\epsilon$.
Therefore,
\begin{displaymath}
\left|\int_{\cM_1}
\psi_2^K\, \Dhyp f_\epsilon\,d\nu_1 \right| \le
C(\cM_1)\cdot\frac{3g-3+n}{\epsilon^2 R(\epsilon)^2} \nu_1(\cM_1^\epsilon)\,.
  %
\end{displaymath}
Hence, since $R(\epsilon) \to \infty$ as $\epsilon \to 0$
and since $\nu_1(\cM_1^\epsilon)=O(\epsilon^2)$ by
Lemma~\ref{lemma:sv:sc:estimate}, we have
\begin{displaymath}
\lim_{\epsilon \to 0} \int_{\cM_1}
\psi_2^K\, \Dhyp f_\epsilon\,d\nu_1 = 0.
\end{displaymath}
By (\ref{eq:condition:III:R}), we have $\frac{1}{\epsilon^2} \nu_1(\{ S
\in  \cM_1^{\epsilon}  \;|\;  \psi_1^K(S) > \tilde{\psi}^K(S) \}) \to 0$ as
$\epsilon \to 0$.
By  (\ref{eq:disjoint:cyls:2}), we get
$\psi^K(S) = O(\epsilon^{-2})$ on
$\cM_1^{\epsilon}$. Thus,
\begin{multline*}
-\cfrac{\pi}{3}\cdot\lim_{\epsilon \to 0}
\int_{\cM_1}
\psi^K\, \Dhyp f_\epsilon\,d\nu_1
=
-\cfrac{\pi}{3}\cdot\lim_{\epsilon \to 0}
\int_{\cM_1}
\psi_1^K\, \Dhyp f_\epsilon\,d\nu_1
\ =\ \\
-\cfrac{\pi}{3}\cdot\lim_{\epsilon \to 0}
\int_{\cM_1}
\tilde{\psi}^K\, \Dhyp f_\epsilon\,d\nu_1
\ =\
\cfrac{\pi}{3}\cdot
\lim_{\epsilon \to 0}
\int_{\cM_1}
\nhyp  \tilde{\psi}^K \cdot \nhyp f_\epsilon\,d\nu_1\ .
\end{multline*}
For the last equality we applied Green's formula to $f_\epsilon$
and  $\tilde\psi^K$.  The  function  $\tilde\psi^K$  is continuous on
$\cM_1$  and  $\nhyp \tilde{\psi}^K$ is
\textit{piecewise}  continuous, which is sufficient for the validity
of Green's formula.
\end{proof}

Let  $Cyl(S,\epsilon,\epsilon/2)$  denote  the  cylinders  on $S$ for
which  the  length  of  the  core  curve  is between $\epsilon/2$ and
$\epsilon$.

\begin{lemma}
\label{lemma:tilde:N:approx}
Let
\begin{displaymath}
\tilde{N}^K_{\mathit{area}}(S,\epsilon,\epsilon/2) := \sum_{C \in \widetilde{Cyl}_K(S) \cap
  Cyl(S,\epsilon,\epsilon/2)} \Area(C)\,.
\end{displaymath}
Then,
\begin{displaymath}
c_{\mathit{area}}(\cM_1) =
\lim_{\epsilon \to 0} \frac{1}{\tfrac{3}{4}\pi \epsilon^2}
\int_{\cM_1} \tilde{N}^K_{\mathit{area}}(S,\epsilon,\epsilon/2) \, d\nu_1(S)\,.
\end{displaymath}
\end{lemma}

\begin{proof}
Write           $N_{\mathit{area}}(S,\epsilon,\epsilon/2)           =
N_{\mathit{area}}(S,\epsilon)  - N_{\mathit{area}}(S,\epsilon/2)$. By
Siegel---Veech   formula~\eqref{eq:SV:constant:definition},  for  any
$\epsilon > 0$,
\begin{equation}
\label{eq:c:mathit:area}
c_{\mathit{area}}(\cM_1) = \frac{1}{\tfrac{3}{4} \pi \epsilon^2}
\int_{\cM_1} N_{\mathit{area}}(S,\epsilon,\epsilon/2) \, d\nu_1(S).
\end{equation}
Let
\begin{displaymath}
N^{K}_{\mathit{area}}(S,\epsilon,\epsilon/2) :=
 \sum_{C \in Cyl_K(S) \cap
  Cyl(S,\epsilon,\epsilon/2)} \Area(C)\,.
\end{displaymath}

By~\cite[Theorem~5.1]{Eskin:Masur}, $\card Cyl(S,\epsilon,\epsilon/2)
=  O(\ell_{\mathit{flat}}(S)^{-\beta})$  for  any  $1  <  \beta < 2$.
Suppose  $C$  is  a cylinder in $Cyl(S,\epsilon,\epsilon/2) \setminus
Cyl_K(S)$.  Then, since $\Mod(C) \le K$, $\Area(C) \le K w(C)^2 \le K
\epsilon^2$. Thus,
\begin{equation}
\label{eq:N:NK}
N_{\mathit{area}}(S,\epsilon,\epsilon/2) -
N^K_{\mathit{area}}(S,\epsilon,\epsilon/2) \le K \epsilon^2 \ell_{\mathit{flat}}(S)^{-\beta}.
\end{equation}
Since  the  left hand side of (\ref{eq:N:NK}) is supported on $\{S\in
\cM_1\,|\,   \ell_{\mathit{flat}}(S)  \le  \epsilon  \}$,  and  since
$\ell_{\mathit{flat}}(   \cdot)^{-\beta}   \in   L^1(\cM_1,\nu_1)$   by
Lemma~\ref{lemma:sv:sc:estimate}, we have
\begin{displaymath}
\lim_{\epsilon \to 0} \frac{1}{\epsilon^2} \int_{\cM_1}
\big(N_{\mathit{area}}(S,\epsilon,\epsilon/2) -
N^K_{\mathit{area}}(S,\epsilon,\epsilon/2)\big) \, d\nu_1 = 0\,.
\end{displaymath}
Thus, in view of (\ref{eq:c:mathit:area}),
\begin{displaymath}
c_{\mathit{area}}(\cM_1) = \lim_{\epsilon \to 0} \frac{1}{\tfrac{3}{4}
  \pi \epsilon^2} \int_{\cM_1}
N^K_{\mathit{area}}(S,\epsilon,\epsilon/2) \, d\nu_1(S).
\end{displaymath}
By~\eqref{eq:condition:III}
$N^K_{\mathit{area}}(\cdot,\epsilon,\epsilon/2)$ and
$\tilde N^K_{\mathit{area}}(\cdot,\epsilon,\epsilon/2)$
might differ only on a set of measure $o(\epsilon^2)$.
Note also, that
$N^K_{\mathit{area}}(S,\epsilon,\epsilon/2)\le 3g-3+n$.
Hence,   we   may   replace
$N^K_{\mathit{area}}$  by  $\tilde{N}^K_{\mathit{area}}$ in the above
equation. Lemma~\ref{lemma:tilde:N:approx} is proved.
\end{proof}

Suppose   $P   >   1$.  Let  $\widetilde{Cyl}_{K,P}(S)  :=  \{  C  \in
\widetilde{Cyl}_{K}(S)  \;|\;  w(C)  < P \ell_K(S) \}$,
and let
\begin{displaymath}
\tilde{\psi}^{K,P}(S) := \sum_{C \in \widetilde{Cyl}_{K,P}(S)}
\left(\Mod(C)-K \right).
\end{displaymath}
Let
\begin{displaymath}
\tilde{N}^{K,P}_{\mathit{area}}(S,\epsilon,\epsilon/2) :=
\sum_{C \in \widetilde{Cyl}_{K,P}(S) \cap
  Cyl(S,\epsilon,\epsilon/2)} \Area(C).
\end{displaymath}

\begin{lemma}
\label{lemma:nearly:equal:widths}
For  all  $K$  sufficiently  large,  and  all  $P > 1$, the following
estimates hold:
\begin{multline}
\label{eq:psi:KP}
\Bigg|\int_{\cM_1}
\Dhyp \log\det\Dflat(S,S_0) \,d\nu_1
\ -\\-\
 \cfrac{\pi}{3}\cdot
\lim_{\epsilon \to 0}
\int_{\cM_1}
\nhyp  \tilde{\psi}^{K,P} \cdot \nhyp f_\epsilon\,d\nu_1\
\Bigg|
\le \frac{C(\cM_1)}{P^2}
\end{multline}
and
\begin{equation}
\label{eq:tilde:KP}
\left|c_{\mathit{area}}(\cM_1) - \lim_{\epsilon \to 0} \frac{1}{\tfrac{3}{4}
  \pi \epsilon^2} \int_{\cM_1}
\tilde{N}^{K,P}_{\mathit{area}}(S,\epsilon,\epsilon/2) \, d\nu_1(S)
\right| \le \frac{C(\cM_1)}{P^2}\,,
\end{equation}
where    the   constant   $C(\cM_1)$   depends   only   on   $\cM_1$.
\end{lemma}

\begin{proof}
If  $C  \in \widetilde{Cyl}_{K,P}(S)\setminus \widetilde{Cyl}_{K}(S)$
then  $w(C)  \ge  P  \ell_K(C)$, and hence $\Mod(C) \le \cfrac{1}{P^2
\big(\ell_K(S)\big)^2}$\,. The latter implies, that
\begin{equation}
\label{eq:tilde:psi:K:S}
\tilde{\psi}^K(S) - \tilde{\psi}^{K,P}(S)
  \le \frac{3g-3+\noz}{P^2\big(\ell_K(S)\big)^2}\,.
\end{equation}

Suppose  $C$  is  a  vertical
cylinder  on  a  surface  $S$ (so that the waist
  curve of $C$ is vertical).  Then  for $g \in
\SL$,  $gC$  is  a  cylinder on $gS$. Let $H(g) = \Mod(gC)$.
Then, we claim that
\begin{equation}
\label{eq:nhyp:h:h}
\nhyp H =
\begin{pmatrix}
0 \\
2 H
\end{pmatrix}
\relax .
\end{equation}
Indeed, we may write
\begin{equation}
\label{eq:rotation:dilatation:unipotent}
g =  \begin{pmatrix} \cos
\theta & \sin \theta \\ - \sin \theta & \cos \theta \end{pmatrix}
\begin{pmatrix}
y^{1/2} & 0 \\ 0 & y^{-1/2} \end{pmatrix}
\begin{pmatrix} 1 & x \\ 0 & 1 \end{pmatrix} =
r_\theta a_y u_x,
\end{equation}
in such a way that $u_x$ acts by Dehn twists on $C$. Then,
\begin{displaymath}
H(r_\theta a_y u_x) = H(a_y) = y \Mod(C).
\end{displaymath}
The decomposition~\eqref{eq:rotation:dilatation:unipotent}
was chosen in such way that $\zeta=x+iy$ provides a standard
coordinate in the hyperbolic upper half-plane parametrizing
the Teichm\"uller disc, see section~\ref{ss:Teichmuller:discs}.
For the associated hyperbolic metric of curvature -4 one has
\begin{displaymath}
\nhyp H = \begin{pmatrix}
2y \frac{\partial H}{\partial x} \vspace{0.05in} \\
2y \frac{\partial H}{\partial y}
\end{pmatrix}
 = \begin{pmatrix} 0 \\
2y \Mod(C)
\end{pmatrix}
= \begin{pmatrix}
0 \\
2 H
\end{pmatrix}
\end{displaymath}
This completes the proof of~\eqref{eq:nhyp:h:h}.

In general, the direction of the gradient of the function $H(g) =
\Mod(gC)$ depends on the cylinder $C$ (however we still have $\| \nhyp
H \| = 2 H$). This is the motivation for the restriction to parallel
cylinders in \S\ref{ss:restriction:to:parallel:cylinders}
and the ``regularity'' assumption in
\S\ref{sec:Regular:invariant:submanifolds}.

Now  in view of~\eqref{eq:tilde:psi:K:S}, and~\eqref{eq:nhyp:h:h}, we
have
\begin{displaymath}
0\le
\|
\nhyp \tilde{\psi}^K(S) - \nhyp \tilde{\psi}^{K,P}(S)
\|
\le \frac{2(3g-3+\noz)}{P^2\big(\ell_K(S)\big)^2}\,,
\end{displaymath}
for    all   $S$   where   $\nhyp   \tilde{\psi}^K(S)$   and   $\nhyp
\tilde{\psi}^{K,P}(S)$ are defined.

Note  that  by  Lemma~\ref{lemma:properties:f2}  the  function $\Dhyp
f_\epsilon$ is bounded and supported on the set $\cM_1^\epsilon$
defined in~\eqref{eq:M:1:epsilon}. On this set we can extend the latter
estimate as
\begin{displaymath}
\|
\nhyp \tilde{\psi}^K(S) - \nhyp \tilde{\psi}^{K,P}(S)
\|
\le \frac{2(3g-3+\noz)}{P^2\big(\ell_K(S)\big)^2}
\le \frac{2(3g-3+\noz)}{P^2\big(\epsilon/2\big)^2}\,.
\end{displaymath}
Finally,       note       that       since      $\lf(S)\le\ell_K(S)$,
Lemma~\ref{lemma:sv:sc:estimate}             implies             that
$\nu_1(\cM_1^\epsilon)=O(\epsilon^2)$.    By    property    (c)    of
Lemma~\ref{lemma:properties:f2},  $\| \nhyp  f_\epsilon \|$  is bounded by a
uniform      bound      independent      of      $\epsilon$.      The
estimate~\eqref{eq:psi:KP}          now          follows         from
Lemma~\ref{lemma:can:neglect:nonparallel}.

For the estimate (\ref{eq:tilde:KP}) note that if
$\tilde{N}^{K}_{\mathit{area}}(S,\epsilon,\epsilon/2) -
\tilde{N}^{K,P}_{\mathit{area}}(S,\epsilon,\epsilon/2) > 0$,
i.e. if there exists
$C \in \widetilde{Cyl}_{K}(S) \setminus \widetilde{Cyl}_{K,P}(S)$
with $\epsilon/2 < w(C) < \epsilon$, then
$\ell_{\mathit{flat}}(S) \le \ell_K(S) \le \cfrac{\epsilon}{P}$.
Now since
$\tilde{N}^{K}_{\mathit{area}}(S,\epsilon,\epsilon/2) -
\tilde{N}^{K,P}_{\mathit{area}}(S,\epsilon,\epsilon/2)
\le (3g-3+n)$,
\begin{multline*}
\int_{\cM_1} (\tilde{N}^{K}_{\mathit{area}}(S,\epsilon,\epsilon/2) -
\tilde{N}^{K,P}_{\mathit{area}}(S,\epsilon,\epsilon/2)) \, d\nu_1(S)
\le\\ \le
(3g-3+\noz)\cdot
  \nu_1\big(\{ S \;|\; \ell_{\mathit{flat}}(S)<\tfrac{\epsilon}{P} \}\big)
= O\left(\frac{{\epsilon}^2}{P^2}\right)\,,
\end{multline*}
where  we  have  used  Lemma~\ref{lemma:sv:sc:estimate}  for the last
estimate.   Now   the   estimate   (\ref{eq:tilde:KP})  follows  from
Lemma~\ref{lemma:tilde:N:approx}.
\end{proof}

\begin{Remark}
\label{rm:horizontal:versus:vertical}
Note that in the calculation in Lemma~\ref{lemma:nearly:equal:widths}
we confront a conflict of two conventions. One uses the \textit{upper
half-plane}  for  the Poincar\'e model of the hyperbolic plane, which
imposes the decomposition~\eqref{eq:rotation:dilatation:unipotent} of
$\SL$. The latter implies, that the holonomy vector associated to the
waist   curve   of   the   cylinder   $C$   should  be  expressed  as
$\begin{pmatrix}0\\w(C)\end{pmatrix}$, as if it was \textit{vertical}
and  not  traditionally  \textit{horizontal}.  A similar situation is
reproduced in the next section.
\end{Remark}

\subsection{The Determinant of the Laplacian and the Siegel---Veech constant}
\label{ss:det:of:L:and:SV}

\begin{lemma}
\label{lemma:last:lemma}  If $K/P^2$ is sufficiently large (depending
only on the genus), then
\begin{displaymath}
\lim_{\epsilon \to 0}
  \int_{\cM_1} \nhyp \tilde{\psi}^{K,P} \cdot
              \nhyp f_\epsilon\,d\nu_1
=
-(4\pi)\cdot
\lim_{\epsilon \to 0}
  \frac{1}{\tfrac{3}{4} \pi \epsilon^2}
  \int_{\cM_1} \tilde{N}^{K,P}_{\mathit{area}}(S,\epsilon,\epsilon/2)
    \, d\nu_1(S).
\end{displaymath}
\end{lemma}

\begin{proof}
Let  $Q = 2P$, where, by assumption, $P>1$. Note that the supports of
both  $\nhyp  f_\epsilon$  and  $\tilde{N}_{\mathit{area}}^{K,P}$ are
contained in the set
$$
\cM_1^{Q,\epsilon}  =  \{ S \in \cM_1 \;|\; \epsilon/Q < \ell_K(S) <
Q\epsilon \}\,.
$$
Note  also  that the support of $\tilde{N}_{\mathit{area}}^{K,P}$, is
contained           in           the          smaller          subset
$\tilde\cM_1^{Q,\epsilon}\subseteq\cM_1^{Q,\epsilon}$     of    those
surfaces,  for  which  all  cylinders  in $Cyl_K(S)$ having the waist
curve  of the shortest length $\ell_K(S)$ are parallel. Note that the
intersection   of   the   supports   of  $\nhyp  f_\epsilon$  and  of
$\tilde{\psi}^{K,P}$ is also contained in $\tilde\cM_1^{Q,\epsilon}$.

We    normalize    the    Haar    measure    $dg$    on    $\SL$   in
coordinates~\eqref{eq:rotation:dilatation:unipotent} as
$$
dg=\cfrac{1}{4y^2} \, dx\, dy \, d\theta=d\ghyp\,d\theta\,,
$$
where $\ghyp$ is the hyperbolic metric of curvature $-4$ on the upper
half-plane
$$
\Hyp\simeq\SL/\SO\,.
$$

We   choose   a   codimension   two   cross  section  $\tilde\cN$  of
$\tilde\cM_1^{Q,\epsilon}$  represented  by the surfaces $S_\epsilon$
for   which   $\ell_K(S_\epsilon)   =  \epsilon$  and  such  that  on
$S_\epsilon$ the cylinders in $\widetilde{Cyl}_{K,P}(S_\epsilon)$ are
\textit{horizontal}          in          the         sense         of
Remark~\ref{rm:horizontal:versus:vertical}     at    the    end    of
section~\ref{ss:restriction:to:parallel:cylinders}.  Then,  every  $S
\in \tilde\cM_1^{Q,\epsilon}$ can be represented as
\begin{equation}
\label{eq:coordinates}
S = r_{\!\theta}\, a_y S_\epsilon\,,
\end{equation}
where $y \in  [Q^{-2},Q^2]$, $S_\epsilon \in \tilde\cN$.

Recall that since the measure $d\nu_1$ is affine, it disintegrates as
\begin{displaymath}
d\nu_1 =  \frac{dy}{4y^2}\, d\theta\,d\beta',
\end{displaymath}
where $\beta'$ is a measure on $\tilde\cN$.

For $S_\epsilon \in \tilde\cN$, let
\begin{displaymath}
H(y,S_\epsilon) := \sum_{C \in \widetilde{Cyl}_{K,P}(S_\epsilon)}
 \Mod(a_y C)\,.
\end{displaymath}
Suppose that some cylinder $C$ belongs to the symmetric difference of
$\widetilde{Cyl}_{K,P}(S_\epsilon)$   and  $\widetilde{Cyl}_{K,P}(a_y
S_\epsilon)$ for some $y \in [Q^{-2}, Q^2]$. Then,
\begin{displaymath}
KQ^{-2} \le \Mod(C) \le KQ^2
\end{displaymath}
By  assumption  $KQ^{-2}$ is sufficiently large so that all cylinders
of  modulus  at  least $KQ^{-2}$ are disjoint. It follows that for $y
\in [Q^{-2},Q^2]$,
\begin{displaymath}
|\tilde{\psi}^{K,P}(a_y S_\epsilon) - H(y,S_\epsilon)| \le
(3g-3+n) KQ^2\,.
\end{displaymath}
By  the  same  argument  as in the proof of~\eqref{eq:nhyp:h:h}, this
implies
\begin{equation}
\label{eq:estimate:h}
\|
\nhyp \tilde{\psi}^{K,P}(a_y S_\epsilon) - \nhyp H(y,S_\epsilon)
\|
\le 2(3g-3+n) KQ^2
\end{equation}
We will eventually need to consider the integral
\begin{equation}
\label{eq:tmp:left:integral}
\int_{\cM_1} \nhyp \tilde{\psi}^{K,P} \cdot \nhyp f_\epsilon \,
  d\nu_1.
\end{equation}
 However, the integrand is supported on a set $\cM_1^{Q,\epsilon}$ satisfying $\nu_1(\cM_1^{Q,\epsilon}) \le
C(\cM_1) \epsilon^2$ and $\|\nhyp f_\epsilon\|$ is bounded independent of
$\epsilon$. Then,  the  contribution  of the right hand
side of~\eqref{eq:estimate:h} to (\ref{eq:tmp:left:integral})
will tend to $0$ as $\epsilon \to 0$.

Similarly, let
\begin{displaymath}
A(y,S_\epsilon) :=
\sum_{
  C\in\widetilde{Cyl}^{K,P}(S_\epsilon)
  \cap Cyl(a_y S_\epsilon,\epsilon,\epsilon/2)
} \Area(C).
\end{displaymath}
As above, if some cylinder $C$ belongs to the symmetric difference of
$\widetilde{Cyl}_{K,P}(S_\epsilon)$   and  $\widetilde{Cyl}_{K,P}(a_y
S_\epsilon)$ for some $y \in [Q^{-2}, Q^2]$, then,
\begin{displaymath}
\Area(C) = \big(w(C)\big)^2 \Mod(C) \le (Q \epsilon)^2 K Q^2 \le K Q^4
\epsilon^2.
\end{displaymath}
Thus,
\begin{equation}
\label{eq:estimate:A}
|\tilde{N}^{K,P}_{\mathit{area}}(a_y S_\epsilon, \epsilon,
\epsilon/2) - A(y,S_\epsilon)| \le
(3g-3+n) KQ^4 \epsilon^2.
\end{equation}
We will eventually need to consider the expression:
\begin{equation}
\label{eq:tmp:right:integral}
  \frac{1}{\tfrac{3}{4} \pi \epsilon^2}
  \int_{\cM_1} \tilde{N}^{K,P}_{\mathit{area}}(S,\epsilon,\epsilon/2)
    \, d\nu_1(S).
\end{equation}
Since the integrand is supported on a set $\cM_1^{Q,\epsilon}$ satisfying $\nu_1(\cM_1^{Q,\epsilon}) \le
C(\cM_1) \epsilon^2$,  the  contribution  of the right hand
side of~\eqref{eq:estimate:A} to (\ref{eq:tmp:right:integral})
will tend to $0$ as $\epsilon \to 0$.

We now claim that for any $S_\epsilon \in \tilde\cN$  we have
\begin{equation}
\label{eq:equal:regions}
\int_{1/Q^2}^{Q^2} \nhyp  H(y , S_\epsilon) \cdot \nhyp
f_\epsilon(a_y S_\epsilon) \, \frac{dy}{4y^2}
=
-4\pi\,\cdot\,
\frac{1}{\tfrac{3}{4} \pi \epsilon^2}
 \int_{1/Q^2}^{Q^2}
A(y,S_\epsilon)
\, \frac{dy}{4y^2}\,.
\end{equation}

Note  that  by definition the function $H(y,S_\epsilon)$ is linear in
$y$,     namely,     for     $y     \in    [1/Q^2,Q^2]$    we    have
$H(y,S_\epsilon)=y\cdot\Mod(C)$.    Also    by    construction,   for
$S_\epsilon \in \tilde\cN$, $f_\epsilon(a_{Q^2} S_\epsilon) = 0$, and
$f_\epsilon(a_{1/Q^2} S_\epsilon) = 1$.
Thus,
\begin{multline*}
\int_{1/Q^2}^{Q^2}
\nhyp H(y,  S_\epsilon) \cdot
\nhyp f_\epsilon( a_y  S_\epsilon) \, \frac{dy}{4y^2}
=
\int_{1/Q^2}^{Q^2}
\nabla H(y,  S_\epsilon)\, \nabla f_\epsilon( a_y  S_\epsilon) \, dy
= \\=
\sum_{C \in \widetilde{Cyl}_{K,P}(S_\epsilon)} \Mod(C) \int_{1/Q^2}^{Q^2}
\frac{\partial f_\epsilon(a_y  S_\epsilon)}{\partial   y} \, dy
\\
=
\sum_{C \in \widetilde{Cyl}_{K,P}(S_\epsilon)} \Mod(C)
(f_\epsilon(a_{Q^2} S_\epsilon) - f_\epsilon( a_{1/Q^2} S_\epsilon))
=
\sum_{C \in \widetilde{Cyl}_{K,P}(S_\epsilon)} \Mod(C)
\cdot(-1)\,.
\end{multline*}

Now,
\begin{displaymath}
A(y,S_\epsilon) = \sum_{C
  \in \widetilde{Cyl}_{K,P}(S_\epsilon)} \Area(C)\cdot
  \chi_{(\epsilon/2,\epsilon)}\big(y^{-1/2}\cdot w(C)\big)\,,
\end{displaymath}
where  the characteristic function $\chi_{(a,b)}(t)$ is $1$ if $a < t
< b$ and $0$ otherwise. By our choice of $Q$ and by the definition of
$\widetilde{Cyl}_{K,P}(S_\epsilon)$,     for     every     $C     \in
\widetilde{Cyl}_{K,P}(S_\epsilon)$, we have
$\left[\frac{w^2(C)}{4\epsilon^2},\frac{w^2(C)}{\epsilon^2}\right]
\subset [Q^{-2},Q^2]$. Then,
\begin{multline*}
\int_{1/Q^2}^{Q^2} A(y,S_\epsilon)\, \frac{dy}{4y^2}
=
\int_{1/Q^2}^{Q^2} \sum_{C\in \widetilde{Cyl}_{K,P}(S_\epsilon)}
   \Area(C) \chi_{(\epsilon/2,\epsilon)}\big(y^{-1/2} w(C)\big)\,
      \frac{dy}{4y^2}
= \\=
\sum_{C \in \widetilde{Cyl}_{K,P}(S_\epsilon)} \Area(C)
  \int_{w^2(C)/\epsilon^2}^{4w^2(C)/\epsilon^2}
    \frac{dy}{4y^2}
=
\cfrac{1}{4}\,\cdot
\frac{3}{4}\, \epsilon^2 \sum_{C \in \widetilde{Cyl}_{K,P}(S_\epsilon)}
\Area(C) \frac{1}{w(C)^2}
= \\=
\cfrac{1}{4}\,\cdot
\frac{3}{4}\, \epsilon^2 \sum_{C \in \widetilde{Cyl}_{K,P}(S_\epsilon)}
\Mod(C)\,.
\end{multline*}
This   completes   the  proof  of  (\ref{eq:equal:regions}).  We  now
integrate  (\ref{eq:equal:regions})  over $\tilde\cN$ with respect to
the measure  $d\beta'$,  and  over  $\theta$  from  $0$  to  $2\pi$,
use~\eqref{eq:estimate:h}  and~\eqref{eq:estimate:A},  and  take  the
limit  as  $\epsilon  \to  0$.  Since  $\nu_1(\cM_1^{Q,\epsilon}) \le
C(\cM_1) \epsilon^2$ and $\nhyp f_\epsilon$ is bounded independent of
$\epsilon$,  we  see  that  the  contributions  of  of  each  of  the
right-hand-sides  of~\eqref{eq:estimate:h}  and~\eqref{eq:estimate:A}
tend  to  $0$  as  $\epsilon  \to  0$.  Lemma~\ref{lemma:last:lemma}
follows.
\end{proof}

\begin{proof}[Proof of Theorem~\ref{theorem:int:Dflat:equals:SV:const}]
Choose  arbitrary  large  $P>1$  and  choose  $K'$ so large, that all
previous                       considerations                      in
sections~\ref{ss:restriction:to:parallel:cylinders} --
\ref{ss:det:of:L:and:SV}  work for $K=K'/Q^2=K'/(4P^2)$. Since $P$ is
arbitrary,   formula~\eqref{eq:int:Dflat:equals:SV:const}  and  thus,
Theorem~\ref{theorem:int:Dflat:equals:SV:const}      follow      from
Lemma~\ref{lemma:nearly:equal:widths}                             and
Lemma~\ref{lemma:last:lemma}.
Theorem~\ref{theorem:int:Dflat:equals:SV:const} is proved.
\end{proof}

\section{Evaluation of Siegel--Veech constants}
\label{sec:Evaluation:of:SV:constants}
It   follows   from   the   general   results   of   A.~Eskin   and
H.~Masur~\cite{Eskin:Masur}  that  almost  all  flat  surfaces in any
closed  connected  regular  $\SL$-invariant suborbifold $\cM_1$ share
the same quadratic asymptotics
\begin{equation}
\label{eq:c:area:definition}
\lim_{L\to\infty}
\cfrac{N_{\mathit{area}}(S,L)}{\pi L^2}=c_{\mathit{area}}(\cM_1)
\end{equation}
where  the  \textit{Siegel---Veech  constant} $c_{area}(\cM_1)$ depends
only    on    $\cM_1$    (see   also   more   specific   results   of
Ya.~Vorobets~\cite{Vorobets}).

In   section~\ref{sec:Arithmetic:Teichmuller:discs}  we recall  some
basic facts concerning arithmetic Teichm\"uller discs. The reader can
find     a     more     detailed     presentation     in   the  original
articles~\cite{Gutkin:Judge},             \cite{Eskin:Masur:Schmoll},
\cite{Hubert:Lelievre}, \cite{Zorich:square:tiled}.

Following      analogous     computations in~\cite{Veech:Teichmuller},
\cite{Veech:Siegel},
\cite{Eskin:Masur:Zorich}, \cite{Lelievre:Siegel}
  and~\cite{Eskin:Masur:Schmoll}  we compute
the Siegel--Veech constant $c_{area}$ for an arithmetic Teichm\"uller
surface                                                            in
section~\ref{Siegel:Veech:constants:for:square:tiled:surfaces}   thus
proving Theorem~\ref{theorem:SVconstant:for:square:tiled}.


\subsection{Arithmetic Teichm\"uller discs}
\label{sec:Arithmetic:Teichmuller:discs}

Consider  a  unit  square  representing  a  fundamental domain of the
integer  lattice  $\Z\oplus  \sqrt{-1}\cdot\Z$  in the complex plane.
Consider  a  flat  torus  $\T$ obtained by identification of opposite
sides  of  this unit square. A \textit{square-tiled surface} (also an
\textit{origami}) $S$ is a ramified cover
\begin{equation}
\label{eq:p:S:to:T}
S\xrightarrow{\ p\ }\T
\end{equation}
of finite degree $D$ over the torus such that all ramification points
project to the same point of the torus.

Clearly, $S\in\cH(m_1,\dots,m_\noz)$ where $m_1+1,\dots,m_\noz+1$ are
degrees of ramification points. By construction, the cohomology class
of  the  closed  1-form  $\omega=p^\ast  dz$ is integer: $[\omega]\in
H(S,\{\text{zeroes}\};\Z\oplus \sqrt{-1}\cdot\Z)$.

One  can  slightly  generalize the above construction admitting other
flat  tori  without singularities and with a single marked point as a
base  of  the  cover~\eqref{eq:p:S:to:T}.  The corresponding covering
flat  surface  $S$ is called an \textit{arithmetic Veech surface}. An
$\SL$-orbit  of  such  flat  surface  in the corresponding stratum is
called  an \textit{arithmetic Teichm\"uller disc}, and its projection
to  $\mathbb{P}\cH(m_1,\dots,m_\noz)$  (or  to the
moduli space of curves) is called an \textit{arithmetic Teichm\"uller curve}.

We  say that an arithmetic Veech surface is \textit{reduced} if the
cover~\eqref{eq:p:S:to:T}   does  not  factor  through  a  nontrivial
regular cover of a larger torus:
$$
\begin{array}{c}
S\xrightarrow{\quad p\quad}\T\\
\searrow\hspace*{12pt}\nearrow\\
\T
\end{array}
$$

Throughout  this  section we consider only reduced arithmetic Veech
surfaces  $S$.  Moreover, we always assume that the base torus of the
cover~\eqref{eq:p:S:to:T} has area one.

The  action of the group $\GL$ on an arithmetic Veech surface $S$ and
on   the   underlying   torus   $\T$   are   compatible:   having   a
cover~\eqref{eq:p:S:to:T}  we  get  a cover $gS\to g\T$ for any group
element  $g$;  moreover,  this new cover has the same topology as the
initial  one.  This implies, in particular, that if the base torus of
the  cover~\eqref{eq:p:S:to:T}  has  area  one,  than the $\SL$-orbit
$\SL\cdot S$ of an arithmetic Veech surface $S$ representing contains
at  least  one square-tiled surface. This also implies that the orbit
$\SL\cdot  S$  of  $S$  is a finite nonramified cover over the moduli
space $\cH_1(0)$ of flat tori with a marked point.

It  would  be  convenient  to  apply  extra  factorization  over $\pm
\operatorname{Id}\in\SLZ$  and  to  pass  to  $\PSL$ and $\PSLZ$. The
degree $N$ of the cover
$$
\Pi:\PSL\cdot S\to\cH_1(0)
$$
coincides   with   the   cardinality  of  the  $\PSLZ$-orbit  of  any
square-tiled surface $S_0$ in the orbit $\PSL\cdot S$,
$$
N=\deg(\Pi)=\card\,\PSLZ\cdot S_0
$$

Rescaling  every  flat  surface  in  the  orbit  $\PSL\cdot  S$  by a
homothety  with  a  factor  $1/\sqrt{D}$  we  can  identify the orbit
$\PSL\cdot S$ with a regular $\PSL$-invariant variety $\cM_1$ of flat
surfaces   of   area   one.  The  corresponding  Teichm\"uller  curve
$\mathbb{P}  \cM$  has  a  natural structure of a cover of degree $N$
over the \textit{modular curve} $\mathbb{P}\cH(0)$, where
$$
\mathbb{P}\cH(0)\simeq\PSO\backslash\PSL/\PSLZ\simeq\Hyp/\PSLZ\,.
$$
This  cover  might  have  ramification points over any (or over both)
orbifoldic points of the modular curve.

The  canonical density  measure $d\nu$ on $\cH_1(0)=\PSL/\PSLZ$ in
standard  normalization  disintegrates  to  the  hyperbolic area form
$d\nuhyp$ on $\mathbb{P}\cH(0)\simeq\Hyp/\PSLZ$. In particular,
$$
\nu\big(\cH_1(0)\big)=\frac{\pi^2}{3},\qquad
\nuhyp\big(\mathbb{P}\cH(0)\big)=\frac{\pi}{3}.
$$

Clearly, a flat torus of area one cannot have two short non-homologous
closed      geodesics.
Since $\cH_1^\epsilon(0)$   is   connected,   it   represents   the  single
\textit{cusp} of $\cH_1(0)$. It is easy to compute that
$$
\nu(\cH^\epsilon_1(0))=\pi\epsilon^2,\qquad
\nuhyp(\mathbb{P}\cH^\epsilon(0))=\epsilon^2.
$$

Since  any  arithmetic  Teichm\"uller  curve  $\mathbb{P} \cM$  is a (possibly
ramified)  cover  of finite order $N$ over the modular curve, $\mathbb{P} \cM$
is  a  Riemann  surface  of  finite  area $N\cdot\pi^2/3$ with cusps,
where the  \textit{cusps}  of  $\mathbb{P}
\cM$  are in a bijection
with   connected   components   $\cC_1,\dots,\cC_s$   of  the  subset
$\mathbb{P}\cM^\epsilon$.

Consider a very short (say, shorter than $\frac{\epsilon}{N}$) simple
closed     curve     $\gamma$     non-homotopic     to     zero    in
$\mathbb{P}\cH^\epsilon(0)$  (for  example,  a very short horocycle).
Consider  its  preimage $\Pi^{-1}\gamma$ in $\mathbb{P}\cM^\epsilon$.
By  construction  the  preimage  has  a  unique  connected  component
$\gamma_j$  in  each  cusp  $\cC_j$  of  $\mathbb{P}\cM^\epsilon$. We
define a \textit{width} $N_j$ \textit{of the cusp} $\cC_j$ as a ratio
of  lengths  of  $\gamma_j$  and  $\gamma$  measured in the canonical
hyperbolic    metric.    Note    that    the    connected   component
$\mathbb{P}\cM^\epsilon(\cC_j)$      of      $\mathbb{P}\cM^\epsilon$
representing  the  cusp  $\cC_j$  is a cover of degree $N_j$ over the
neighborhood  $\mathbb{P}\cH^{\epsilon/N_j}(0)$  of  the only cusp of
the modular curve.

Consider  a square-tiled surface $S_0$. Every nonsingular leaf of the
horizontal  foliation on $S_0$ is closed. Thus, $S_0$ decomposes into
a  finite number of maximal cylinders bounded by unions of horizontal
saddle  connections.  We  denote  the  length of the horizontal waist
curve  of the cylinder number $j$ by $w_j$ and the vertical height of
the  cylinder by $h_j$. We enumerate the cylinders in such a way that
$w_1\le  w_2\le\dots\le  w_k$,  where  $k$  is  the  total  number of
cylinders. Clearly all parameters $w_j, h_j$ are integer. The area of
the  cylinder  number  $j$  equals  $w_j h_j$. The area of the entire
square-tiled  surface  $S_0$  (which coincides with the number $D$ of
unit squares tiling it) is equal to the sum
$$
\area(S)=D=w_1 h_1+\dots+w_k h_k\,,
$$
where  $k$  is  the  total  number  of  cylinders.  We  enumerate the
cylinders in such a way that $w_1\le w_2\le\dots\le w_k$.

Consider a unipotent subgroup
$$
U=\left\{\begin{pmatrix} 1& n\\0 & 1\end{pmatrix}\ \vert\ n\in\Z\right\}
$$
of $\PSLZ$. Consider an orbit $U\cdot S_0$ of a square-tiled surface.
Any  flat  surface  in  this  orbit  is  also a square-tiled surface.
Moreover, it has the same number of maximal cylinders in its cylinder
decomposition,  and the cylinders have the same heights and widths as
the  ones  of  the initial square-tiled surface. (The only parameters
which  differ  for different elements of $U\cdot S_0$ are the integer
\textit{twists}  which  are  responsible  for  gluing  the  cylinders
together.)

The   proof   of   the   following   simple   Lemma   can   be  found
in~\cite{Hubert:Lelievre}.

\begin{Lemma}
Let  $S_0$  be  a  reduced  square-tiled  surface and let $Z(S_0) =
\PSLZ\cdot S_0$ be the set of square-tiled surfaces in its orbit. The
cusps    of   the   corresponding   arithmetic   Teichm\"uller   disc
$\cM_1=\PSL\cdot  S_0$  are  in  bijection  with  the  $U$-orbits  of
$Z(S_0)$,   and   the   widths  $N_j$  of  the  cusps  coincide  with
cardinalities of the corresponding $U$-orbits.
\begin{equation}
\label{eq:Z:equals:sum:Ui}
Z(S_0)=\sqcup_{i=1}^s U_i\qquad \card(U_i)=N_i\,,
\end{equation}
where $s$ is the total number of cusps.
\end{Lemma}

\subsection{Siegel--Veech constants for square-tiled surfaces}
\label{Siegel:Veech:constants:for:square:tiled:surfaces}

Consider  an  arithmetic  Veech  surface  $S$;  let $p:S\to\T$ be the
corresponding  torus  cover.  As usual we assume that the area of the
flat  torus in the base of the cover is equal to one. Let $\gamma$ be
a  closed  geodesic  on  $S$. Its projection $p(\gamma)$ to the torus
$\T$  is  also a closed geodesic. Let $\vec v\in\R{2}$ be a primitive
vector  of  the  lattice  associated to $\T$ representing this closed
geodesic   on   the   torus.   Applying   an   appropriate   rotation
$r_\theta\in\PSO$  to  $\vec v$ we can make it horizontal. Applying a
hyperbolic transformation
$$
g_t=\begin{pmatrix} e^{t} & 0\\ 0 & e^{-t}\end{pmatrix}
$$
with  a  sufficiently  large  \textit{negative}  $t$ to the resulting
horizontal  vector  we can make it very short. The corresponding flat
surface $g_t r_\theta\cdot S$ belongs to a neighborhood of one of the
cusps $\cC_j$ of the orbit $\PSL\cdot S$.

Note  that  a  direction  of  any  closed  geodesic (or of any saddle
connection)   on   a   square-tiled   surface  is  \textit{completely
periodic}:  any leaf of the foliation in the same direction is either
a  regular  closed  leaf  or is a saddle connection. Thus, any closed
geodesic  on  a square-tiled surface defines a cylinder decomposition
of it. Proportions of lengths of the waist curves of the cylinders or
of  heights of the cylinders as well as areas on the cylinders do not
change  under  the  action  of  the  group $\PSL$. In particular, any
closed  geodesic  on  a  square-tiled surface defines a \textit{rigid
configuration  of saddle connections}. We say that this configuration
has  type $\cC_j$ when the flat surface $g_t r_\theta\cdot S$ defined
as above belongs to a neighborhood of one of the cusps $\cC_j$.

Any  closed  geodesic  corresponds  to  a  unique  cusp  $\cC_j$,  so
$$
c_{area}=\sum_{i=1}^s c_{area}(\cC_i)\,.
$$
Here   $s$  denotes  the  total  number  of  cusps  of  $\mathbb{P} \cM$.  The
Siegel--Veech  constant  $c_{area}(\cC_i)$  corresponds  to  counting
total  areas  of  only  those  cylinders  of  bounded  length,  which
represent a given rigid configuration $\cC_i$ of saddle connections.

To  compute  the  Siegel--Veech  constant $c_{area}(\cC_i)$ we follow
analogous             computations
in~\cite{Eskin:Masur},
\cite{Eskin:Masur:Zorich},
\cite{Lelievre:Siegel}
     and     especially     a    computation
in~\cite{Eskin:Masur:Schmoll} which is the closest to our case.

Having  an arithmetic Veech surface $S\in\cM_1$ choose a cusp $\cC_i$
of  $\cM_1$.  Having  a configuration of closed geodesics of the type
$\cC_i$   choose   a   regular   closed  geodesic  $\gamma$  in  this
configuration  and consider the associated vector $\vec v(\gamma)$ as
above.  By  construction  $\vec v$ does not depend on the choice of a
representative  $\gamma$. Moreover, it can be explicitly evaluated as
follows. Consider the cylinder decomposition of square-tiled surfaces
in  the  orbit  $U$-orbit $U_i$ representing the cusp $\cC_i$. If the
representative $\gamma$ belongs to a cylinder number $j$, then
$$
\vec v(\gamma)=\frac{\vec\gamma}{w_j}
$$
where  $\vec{\gamma}$  is  a  plane  vector having the length and the
direction of $\gamma$.

Associating  to  every  configuration of parallel closed geodesics of
the  type  $\cC_i$ a vector $\vec v$ as above we construct a discrete
subset  $V_i(S)$  in  the  plane $\R{2}$. By construction the subset
changes  equivariantly  with  respect  to  the  group action: for any
$g\in\PSL$ we have $V_i(gS)=gV_i$.

Consider a Siegel--Veech transform which associates to a function $f$
with  compact support in $\R{2}$ a function $\hat f$ on $\cM$ defined
as
$$
\hat{f}(S)=\sum_{v\in V_i(S)} f(v)
$$

By   a   Theorem   of   Veech   (see~\cite{Veech:Siegel})   one   has
\begin{equation}
\label{eq:Siegel:Veech:formula}
\frac{1}{\nu(\cM_1)}\int_{\cM_1} \hat{f}(S) \,d\nu=
const\cdot\int_{\R{2}} f(x,y) \,dxdy
\,,
\end{equation}
where  the  constant  $const$  does  not  depend on the function $f$.

Hence,  to  compute the constant $const$ it is sufficient to evaluate
both  integrals  for  some convenient function $f$, for example for a
characteristic    function    $\chi_\epsilon(x,y)$    of    a    disc
$\{(x,y)\,\vert\,x^2+y^2\le\epsilon^2\}$   of  a  very  small  radius
$\epsilon$. In this particular case the integral on the right is just
the area $\pi\epsilon^2$ of the disc. Function $\hat\chi_\epsilon$ is
the  characteristic  function  of  those  component  of  the preimage
$\Pi^{-1}(\cH_1^\epsilon(0))$, which corresponds to the cusp $\cC_i$.
If the width of the corresponding cusp is $N_i$, than,
$$
\int_{\cM_1} \hat{f}(S) \,d\nu=
N_i\cdot\nu(\cH_1^\epsilon(0))=N_i\cdot \pi\epsilon^2
$$
Finally,   $\nu(\cM_1)=N\cdot\nu(\cH_1(0))=N\pi^2/3$.  Thus,  the
Siegel---Veech   formula~\eqref{eq:Siegel:Veech:formula}   applied   to
$\chi_\epsilon$ establishes the following relation:
$$
\frac{1}{N\pi^2/3}\cdot N_i\pi\epsilon^2=const\cdot\pi\epsilon^2
$$
which  implies  that  the constant in~\eqref{eq:Siegel:Veech:formula}
has the following value:
\begin{equation}
\label{eq:const}
const=\frac{3}{\pi^2}\frac{N_i}{N}\,.
\end{equation}

To   compute  $c_{area}(\cC_i)$  we  introduce  a  counting  function
$\chi_r(\vec{v},\cC_i):\R{2}\to\R{}$  with compact support defined as
follows:
$$
\chi_r(\vec{v},\cC_i):=
\begin{cases}
0      &\text{when }w_1\|\vec{v}\|>r \\
\cfrac{w_1 h_1}{D}&\text{when }w_2\|\vec{v}\|>r \ge w_1\|\vec{v}\| \\
\dots&\dots\\
\cfrac{1}{D}\,(w_1 h_1+\dots+w_j h_j)&\text{when }w_{j+1}\|\vec{v}\|>r \ge w_j\|\vec{v}\| \\
\dots&\dots\\
\cfrac{1}{D}\,(w_1 h_1+\dots+w_k h_k)&\text{when }r \ge w_k\|\vec{v}\|
\end{cases}
$$
Here  $k$  is total number of cylinders in the cylinder decomposition
corresponding  to the configuration $\cC_i$, and $D=\area(S)$ is the
number of unit squares used to tile the initial square-tiled surface.
As  always,  we  enumerate  the  cylinders in such a way that $w_1\le
w_2\le\dots\le w_k$.

By definition of $N_{area}(S,r;\cC_i)$ we have
$$
N_{area}(S,r;\cC_i)=
\sum_{v\in V_i(S)}
\chi_r(\vec{v},\cC_i)=
\hat\chi_r(S,\cC_i)\,.
$$

Note  that modifying a flat structure on a surface $S$ by a homothety
with  a  positive coefficient $\lambda$ is equivalent to changing the
scale.  Hence,  for  any  counting function $N(S,r)$ with a quadratic
asymptotics in $r$ we get
$$
N(\lambda\cdot S,r)=N\left(S,\cfrac{r}{\lambda}\right)\sim
\cfrac{1}{\lambda^2}\cdot N(S,r)
$$

By  definition the coefficient $c_{area}$ is defined as a coefficient
in  a  quadratic  asymptotics  of a counting function $N_{area}$ on a
surface of \textit{unit} area. Since arithmetic Veech surfaces in our
consideration  have  area  $D$ (the number of unit squares tiling the
initial  square-tiled  surface $S_0$), we need to normalize the limit
below by the area of $S$ in order to obtain $c_{area}$:
$$
c_{area}(\cC_i):=\area(S)\cdot\lim_{r\to\infty}\frac{N_{area}(S,r;\cC_i)}{\pi r^2}\ =\
D\cdot
\lim_{r\to\infty}\frac{1}{\pi r^2}\cdot \hat\chi_r(S,\cC_i)
\,.
$$

By  the  results  of  W.~Veech~\cite{Veech:Siegel}  for the case of a
Teichm\"uller  disc  of a Veech surface the constant above is one and
the same for all surfaces in the corresponding Teichm\"uller disc and
\begin{equation}
\label{eq:c:area:1}
c_{area}(\cC_i)=D\cdot
\lim_{r\to\infty}\frac{1}{\pi r^2}\cdot
\frac{1}{\nu(\cM_1)}\int_{\cM_1}\hat\chi_r(S,\cC_i)\,d\nu
\,.
\end{equation}

On       the      other      hand,      by      the      Siegel---Veech
formula~\eqref{eq:Siegel:Veech:formula} the above normalized integral
equals to
\begin{equation}
\label{eq:c:area:2}
\frac{1}{\nu(\cM_1)}\int_{\cM_1}\hat\chi_r(S,\cC_i)\,d\nu=
const\int_{\R{2}}\chi_r(v,\cC_i)\,dx dy\,,
\end{equation}
where  the  value  of  the  constant is obtained in~\eqref{eq:const}.

It remains to compute the integral
\begin{multline}
\label{eq:c:area:3}
\int_{\R{2}}\chi_r(\vec{v},\cC_i)\, dx dy\ =\ \pi r^2\cdot\cfrac{1}{D}
\left(\frac{h_1 w_1}{w_1^2}+\frac{h_2 w_2}{w_2^2}+\dots+\frac{h_k w_k}{w_k^2}\right)\ =\\ =\ \pi r^2\,\cfrac{1}{D}
\sum_{j=1}^k\
\cfrac{h_{j}}{w_{j}}\,,
\end{multline}
and to collect equations~\eqref{eq:const}--\eqref{eq:c:area:3} to get
$$
c_{area}(\cC_i)\ =D\cdot
\frac{3}{\pi^2}\frac{N_i}{N}\cdot
\frac{1}{D}\cdot
\sum_{j=1}^k
\cfrac{h_{j}}{w_{j}}\ =\
\frac{3}{\pi^2}\frac{1}{N}
\sum_{\substack{\text{surfaces}\\\text{in the}\\\text{orbit }U_i}}
\sum_{j=1}^k\
\cfrac{h_{j}}{w_{j}}
\,.
$$

Taking  a  sum  of  $c_{area}(\cC_i)$  over  all cusps $\cC_1, \dots,
\cC_s$   of   $\cM_1$   and   taking   into  consideration  that  the
$\PSLZ$-orbit  $Z(S)$  of the initial square-tiled surface decomposes
into a disjoint union of orbits $U_i$, see~\eqref{eq:Z:equals:sum:Ui}
we obtain the desired formula~\eqref{eq:SVconstant:for:square:tiled}:
\begin{multline*}
c_{area}=\sum_{i=1}^s c_{area}(\cC_i)\ =\
\frac{3}{\pi^2}\frac{1}{N}\cdot
\sum_{\text{cusps }\cC_i}\
\sum_{\substack{\text{surfaces}\\\text{in the}\\\text{orbit }U_i}}\
\sum_{j=1}^{k(i)}
\cfrac{h_{ij}}{w_{ij}}\ =\\
=\ \cfrac{3}{\pi^2}\cdot
\cfrac{1}{\card(\PSLZ\cdot S_0)}\
\sum_{S_i\in\PSLZ\cdot S_0}\ \
\sum_{\substack{
\mathit{horizontal}\\
\mathit{cylinders\ cyl}_{ij}\\
such\ that\\S_i=\sqcup\mathit{cyl}_{ij}}}\
\cfrac{h_{ij}}{w_{ij}}
\end{multline*}

Theorem~\ref{theorem:SVconstant:for:square:tiled} is proved.
\qed


\appendix

\section{Conjectural approximate values of individual Lyapunov
exponents in small genera}
\label{a:Values:of:exponents}

\begin{table}[hbt]
\small
$$
\begin{array}{|c|c||c|c||c||c|}

\hline
&&\multicolumn{4}{|c|}{}\\

\multicolumn{1}{|c|}{\text{Degrees}}&
\multicolumn{1}{|c||}{\text{Con-}}&
\multicolumn{4}{|c|}{\text{Lyapunov exponents}}\\

\multicolumn{1}{|c|}{\text{of }}&
\multicolumn{1}{|c||}{\text{nected}}&
\multicolumn{4}{|c|}{\text{}}\\

\cline{3-6}

\multicolumn{1}{|c|}{\text{zeros}}&
\multicolumn{1}{|c||}{\text{compo-}}&
\multicolumn{3}{|c||}{\text{}}&
\multicolumn{1}{|c|}{\text{}}\\

\multicolumn{1}{|c|}{}&
\multicolumn{1}{|c||}{\text{nent}}&
\multicolumn{3}{|c||}{\text{Experimental}}&
\multicolumn{1}{|c|}{\text{Exact}}\\

\multicolumn{1}{|c|}{\text{}}&
\multicolumn{1}{|c||}{\text{}}&
\multicolumn{3}{|c||}{\text{}}&
\multicolumn{1}{|c|}{\text{}}\\

\cline{3-6}

\multicolumn{1}{|c|}{\text{}}&
\multicolumn{1}{|c||}{\text{}}&
\multicolumn{1}{|c|}{\lambda_2}&
\multicolumn{1}{|c||}{\lambda_3}&
\multicolumn{1}{|c||}{\overset{g}{\underset{j=1}{\sum}} \lambda_j}&
\multicolumn{1}{|c|}{\overset{g}{\underset{j=1}{\sum}} \lambda_j}\\
[-\halfbls] &&&&&\\
\hline &&&&& \\ [-\halfbls]
(4)      & hyp   & 0.6156  & 0.1844 & 1.8000 & 9/5      \\
[-\halfbls] &&&&&\\
\hline &&&&& \\ [-\halfbls]
(4)      & odd & 0.4179  & 0.1821 & 1.6000 & 8/5      \\
[-\halfbls] &&&&&\\
\hline &&&& \\ [-\halfbls]
(1,3)    & -   & 0.5202  & 0.2298  & 1.7500 & 7/4      \\
[-\halfbls] &&&&&\\
\hline &&&&& \\ [-\halfbls]
(2,2)    & hyp   & 0.6883  & 0.3117 & 2.000 & 4/2      \\
[-\halfbls] &&&&&\\
\hline &&&&& \\ [-\halfbls]
(2,2)    & odd & 0.4218  & 0.2449 & 1.6667 & 5/3      \\
[-\halfbls] &&&&&\\
\hline &&&&& \\ [-\halfbls]
(1,1,2)  & -   & 0.5397  & 0.2936 & 1.8333 & 11/6     \\
[-\halfbls] &&&&&\\
\hline &&&&& \\ [-\halfbls]
(1,1,1,1)& -   & 0.5517  & 0.3411  & 1.8928 & 53/28    \\
[-\halfbls] &&&&&\\ \hline
\end{array}
$$
\end{table}

\clearpage   

\begin{table}
\small
$$
\begin{array}{|c|c||c|c|c||c||c|}

\hline
&&\multicolumn{5}{|c|}{}\\

\multicolumn{1}{|c|}{\text{Degrees}}&
\multicolumn{1}{|c||}{\text{Con-}}&
\multicolumn{5}{|c|}{\text{Lyapunov exponents}}\\

\multicolumn{1}{|c|}{\text{of }}&
\multicolumn{1}{|c||}{\text{nected}}&
\multicolumn{5}{|c|}{\text{}}\\

\cline{3-7}

\multicolumn{1}{|c|}{\text{zeros}}&
\multicolumn{1}{|c||}{\text{compo-}}&
\multicolumn{4}{|c||}{\text{}}&
\multicolumn{1}{|c|}{\text{}}\\

\multicolumn{1}{|c|}{}&
\multicolumn{1}{|c||}{\text{nent}}&
\multicolumn{4}{|c||}{\text{Experimental}}&
\multicolumn{1}{|c|}{\text{Exact}}\\

\multicolumn{1}{|c|}{\text{}}&
\multicolumn{1}{|c||}{\text{}}&
\multicolumn{4}{|c||}{\text{}}&
\multicolumn{1}{|c|}{\text{}}\\

\cline{3-7}

\multicolumn{1}{|c|}{\text{}}&
\multicolumn{1}{|c||}{\text{}}&
\multicolumn{1}{|c|}{\lambda_2}&
\multicolumn{1}{|c|}{\lambda_3}&
\multicolumn{1}{|c||}{\lambda_4}&
\multicolumn{1}{|c||}{\overset{g}{\underset{j=1}{\sum}} \lambda_j}&
\multicolumn{1}{|c|}{\overset{g}{\underset{j=1}{\sum}} \lambda_j}\\
[-\halfbls] &&&&&&\\
\hline &&&&&& \\ [-\halfbls]
(6)      & hyp   & 0.7375  & 0.4284  & 0.1198 & 2.2857 & 16/7  \\
[-\halfbls] &&&&&&\\
\hline &&&&&& \\ [-\halfbls]
(6)      &even & 0.5965  & 0.2924  & 0.1107 & 1.9996 & 14/7  \\
[-\halfbls] &&&&&&\\
\hline &&&&&& \\ [-\halfbls]
(6)      & odd & 0.4733  & 0.2755  & 0.1084 & 1.8572 & 13/7  \\
[-\halfbls] &&&&&&\\
\hline &&&&&& \\ [-\halfbls]
(1,5)    & -   & 0.5459  & 0.3246  & 0.1297 & 2.0002 & 2     \\
[-\halfbls] &&&&&&\\
\hline &&&&&& \\ [-\halfbls]
(2,4)    &even & 0.6310  & 0.3496  & 0.1527 & 2.1333 & 32/15 \\
[-\halfbls] &&&&&&\\
\hline &&&&&& \\ [-\halfbls]
(2,4)    & odd & 0.4789  & 0.3134  & 0.1412 & 1.9335 & 29/15 \\
[-\halfbls] &&&&&&\\
\hline &&&&&& \\ [-\halfbls]
(3,3)    & hyp   & 0.7726  & 0.5182  & 0.2097 & 2.5005 & 5/2   \\
[-\halfbls] &&&&&&\\
\hline &&&&&& \\ [-\halfbls]
(3,3)    & -   & 0.5380  & 0.3124  & 0.1500 & 2.0004 & 2     \\
[-\halfbls] &&&&&&\\
\hline &&&&&& \\ [-\halfbls]
%
(1,2,3)  & -   & 0.5558  & 0.3557  & 0.1718 & 2.0833 & 25/12 \\
[-\halfbls] &&&&&&\\
\hline &&&&&& \\ [-\halfbls]
(1,1,4)  & -   & 0.55419 & 0.35858 & 0.15450& 2.06727 & 1137/550     \\
[-\halfbls] &&&&&&\\
\hline &&&&&& \\ [-\halfbls]
(2,2,2)  &even & 0.6420  & 0.3785  & 0.1928 & 2.2133 & 737/333\\
[-\halfbls] &&&&&&\\
\hline &&&&&& \\ [-\halfbls]
(2,2,2)  & odd & 0.4826  & 0.3423  & 0.1749 & 1.9998 & 2     \\
[-\halfbls] &&&&&&\\
\hline &&&&&& \\ [-\halfbls]
(1,1,1,3)& -   & 0.5600  & 0.3843  & 0.1849 & 2.1292 & 66/31 \\
[-\halfbls] &&&&&&\\
\hline &&&&&& \\ [-\halfbls]
(1,1,2,2)& -   & 0.5604  & 0.3809  & 0.1982 & 2.1395 & 5045/2358   \\
[-\halfbls] &&&&&&\\
\hline &&&&&& \\ [-\halfbls]
(1,1,1,1,2)& - & 0.5632  & 0.4032  & 0.2168 &  2.1832 & 131/60    \\
[-\halfbls] &&&&&&\\
\hline &&&&&& \\ [-\halfbls]
(1,1,1,1,1,1)&-& 0.5652  & 0.4198  & 0.2403 &  2.2253 & 839/377   \\
[-\halfbls] &&&&&&\\ \hline
\end{array}
$$
\end{table}

\clearpage   

\begin{table}
\small
$$
\begin{array}{|c|c||c|c|c|c||c||c|}

\hline
&&\multicolumn{6}{|c|}{}\\

\multicolumn{1}{|c|}{\text{Degrees}}&
\multicolumn{1}{|c||}{\text{Con-}}&
\multicolumn{6}{|c|}{\text{Lyapunov exponents}}\\

\multicolumn{1}{|c|}{\text{of }}&
\multicolumn{1}{|c||}{\text{nected}}&
\multicolumn{6}{|c|}{\text{}}\\

\cline{3-8}

\multicolumn{1}{|c|}{\text{zeros}}&
\multicolumn{1}{|c||}{\text{compo-}}&
\multicolumn{5}{|c||}{\text{}}&
\multicolumn{1}{|c|}{\text{}}\\

\multicolumn{1}{|c|}{}&
\multicolumn{1}{|c||}{\text{nent}}&
\multicolumn{5}{|c||}{\text{Experimental}}&
\multicolumn{1}{|c|}{\text{Exact}}\\

\multicolumn{1}{|c|}{\text{}}&
\multicolumn{1}{|c||}{\text{}}&
\multicolumn{5}{|c||}{\text{}}&
\multicolumn{1}{|c|}{\text{}}\\

\cline{3-8}

\multicolumn{1}{|c|}{\text{}}&
\multicolumn{1}{|c||}{\text{}}&
\multicolumn{1}{|c|}{\lambda_2}&
\multicolumn{1}{|c|}{\lambda_3}&
\multicolumn{1}{|c|}{\lambda_4}&
\multicolumn{1}{|c||}{\lambda_5}&
\multicolumn{1}{|c||}{\overset{g}{\underset{j=1}{\sum}} \lambda_j}&
\multicolumn{1}{|c|}{\overset{g}{\underset{j=1}{\sum}} \lambda_j}\\
[-\halfbls] &&&&&&&\\
\hline &&&&&&& \\ [-\halfbls] (8) & hyp &
0.798774 & 0.586441 & 0.305803 & 0.086761 & 2.777779  &  \frac{25}{9}\\
[-\halfbls] &&&&&&&\\
\hline &&&&&&& \\ [-\halfbls]
(8) & even &
0.597167 & 0.362944 & 0.189205 & 0.072900 & 2.222217  & \frac{20}{9}\\
[-\halfbls] &&&&&&&\\
\hline &&&&&&& \\ [-\halfbls]
(8) & odd &
0.515258 & 0.343220 & 0.181402 & 0.071107 & 2.110987  & \frac{19}{9} \\
[-\halfbls] &&&&&&&\\
\hline &&&&&&& \\ [-\halfbls]
(7, 1) & - &
0.560205 & 0.378184 & 0.206919 & 0.081789 & 2.227098  &  \frac{2423}{1088} \\
[-\halfbls] &&&&&&&\\
\hline &&&&&&& \\ [-\halfbls]
(6, 2) & even &
0.603895 & 0.385796 & 0.220548 & 0.091624 & 2.301862   &  \frac{178429}{77511} \\
[-\halfbls] &&&&&&&\\
\hline &&&&&&& \\ [-\halfbls]
(6, 2) & odd &
0.521181 & 0.368690 & 0.211988 & 0.088735 & 2.190594   &  \frac{46}{21} \\
[-\halfbls] &&&&&&&\\
\hline &&&&&&& \\ [-\halfbls]
(6, 1, 1) & - &
0.563306 & 0.398655 & 0.229768 & 0.093637 & 2.285367   &  \frac{59332837}{25961866} \\
[-\halfbls] &&&&&&&\\
\hline &&&&&&& \\ [-\halfbls]
(5, 3) & - &
0.561989 & 0.376073 & 0.216214 & 0.095789 & 2.250066   &  \frac{9}{4} \\
[-\halfbls] &&&&&&&\\
\hline &&&&&&& \\ [-\halfbls]
(5, 2, 1) & - &
0.564138 & 0.396293 & 0.236968 & 0.103124 & 2.300523   &  \frac{4493}{1953} \\
[-\halfbls] &&&&&&&\\
\hline &&&&&&& \\ [-\halfbls]
(5, 1, 1, 1) & - &
0.565422 & 0.414702 & 0.252838 & 0.107906 & 2.340868    &  \frac{103}{44} \\
[-\halfbls] &&&&&&&\\ \hline
\end{array}
$$
\end{table}

\pagebreak

\section{Square-tiled surfaces and permutations}
\label{a:pairs:of:permutations}

\subsection{Alternative interpretation of Siegel--Veech constant
for arithmetic Teichm\"uller discs}

Consider  an  $N$-square-tiled  surface  and enumerate its squares in
some way. The structure of the square tiling can be encoded by a pair
of  permutations $(\pihor,\pivert)$, indicating for each square (say,
for  a  square  number  $k$)  the number $\pi_{hor}(k)$ of its direct
neighbor  to  the  right, and the number $\pi_{hor}(k)$ of its direct
neighbor  on  top.  Reciprocally,  any  ordered  pair of permutations
$(\pihor,\pivert)$  from $\mathfrak{S}_N$, such that $\pihor,\pivert$
do  not  have nontrivial common invariant subsets in $\{1,\dots,N\}$,
defines a connected square-tiled surface.

Applying a simultaneous conjugation
\begin{equation}
\label{eq:conjugation}
(\pi\circ\pihor\circ\pi^{-1}\ ,\ \pi\circ\pivert\circ\pi^{-1})
\end{equation}
by the same permutation $\pi$ to both permutations $(\pihor,\pivert)$
we  do  not change the square-tiled surface, but only the enumeration
of  the  squares. Thus, $N$-square-tiled surfaces are in a one-to-one
correspondence  with  the  resulting  equivalence  classes of ordered
pairs of permutations.

Let $S(\pihor,\pivert)\in\cH(m_1,\dots,m_\noz)$. The degrees $m_i$ of
zeroes  can  be  reconstructed  from  $(\pihor,\pivert)$  as follows.
Consider a decomposition of the commutator
$$
[\pihor,\pivert]:=\pihor\circ\pivert\circ\pihor^{-1}\circ\pivert^{-1}
$$
into   cycles.   Then   the   following   two   unordered  sets  with
multiplicities coincide:
\begin{multline*}
\{m_1+1,\dots,m_\noz+1\}
\ =\\=\
\{\text{Lengths of cycles of }[\pihor,\pivert]
  \text{, which are longer than 1}\}\,.
\end{multline*}

Consider   the  following  generators  $T,S$  of  the  group  $\SLZ$:
\begin{align*}
T     &:=\begin{pmatrix}1&-1\\0&1\end{pmatrix}\\
R     &:=\begin{pmatrix}0&1\\-1&0\end{pmatrix}\,.
\end{align*}
In  terms  of  pairs  of  permutations  the  action of $T$ and $R$ on
square-tiled surfaces can be represented as
\begin{align*}
T(\pihor,\pivert)&=(\pihor\,,\,\pivert\circ\pihor^{-1})\\
R(\pihor,\pivert)&=(\pivert^{-1}\,,\,\pihor)\,.
\end{align*}
Thus,   an   $\SLZ$-orbit   $\cO(S)$   of   a   square-tiled  surface
$S(\pihor,\pivert)$  can  be  obtained as an orbit of the equivalence
class $\widetilde{(\pihor,\pivert)}$ under the transformations $T, S$
as  above  in  the  set  of  equivalence  classes of ordered pairs of
permutations.

We  can rewrite now expression~\eqref{eq:SVconstant:for:square:tiled}
for  the  Siegel--Veech  constant of an arithmetic Teichm\"uller disc
$\cM_1$   as  follows.  Let  $\cO(S)$  be  the  $\SLZ$-orbit  of  the
square-tiled  surface $S(\pi_h,\pi_v)$. Let $\cO(\pi_h,\pi_v)$ be the
corresponding  orbit  in  the  set  of equivalence classes of ordered
permutations. Then

\begin{multline*}
   %
c_{\mathit{area}}(\cM_1)=
\cfrac{3}{\pi^2}\cdot
\cfrac{1}{\card \cO(S)}\
\sum_{S_i\in\cO(S)}\ \
\sum_{\substack{
\mathit{horizontal}\\
\mathit{cylinders\ cyl}_{ij}\\
such\ that\\S_i=\sqcup\mathit{cyl}_{ij}}}\
\cfrac{h_{ij}}{w_{ij}}
\ = \\ =\
\cfrac{3}{\pi^2}\cdot
\frac{1}{\card \cO(\pi_h,\pi_v)}
\sum_{\substack{\widetilde{(\pihor,\pivert)}\\
\text{in }\cO(\pi_h,\pi_v)}}
\sum_{\substack{\text{cycles }c_i\\\text{ in }\pihor}}
\frac{1}{\text{length of }c_i}
\end{multline*}

Note  that  the subset of \textit{noncommuting} pairs of permutations
$(\pihor,\pivert)$    in    $\mathfrak{S}_N\times\mathfrak{S}_N$   is
invariant      under     the     action~\eqref{eq:conjugation}     of
$\mathfrak{S}_N$,  and this action does not have fixed points in this
subset.  Hence,  when the surface $S(\pi_h,\pi_v)$ has genus at least
two,   the  projection  of  the  $T,R$-orbit  of  $(\pi_h,\pi_v)$  in
$\mathfrak{S}_N\times\mathfrak{S}_N$  to the orbit $\cO(\pi_h,\pi_v)$
in  the  set of equivalence classes is a $(N!)$-to-one map. Since the
collection  of lengths of the cycles of a permutation does not change
under  the  conjugation,  we  can  rewrite  the  expression  for  the
Siegel--Veech constant in terms of the $T,S$-orbit:

\begin{equation}
\label{eq:SVconstant:for:pair:of:permutations}
c_{\mathit{area}}(\cM_1)=
\cfrac{3}{\pi^2}\cdot
\frac{1}{\card\big(T,R\text{-orbit of }(\pi_h,\pi_v)\big)}
\sum_{\substack{
  (\pihor,\pivert)\text{}\\
  \text{in the}\\
  \rule{0pt}{7pt}T,R\text{-orbit}
}}
\ \sum_{\substack{\text{cycles }c_i\\\text{ in }\pihor}}
\frac{1}{\text{length of }c_i}
\end{equation}
   %

\subsection{Non varying phenomenon}

By  Corollaries~\ref{cor:hyp:connected:comps}  and~\ref{cor:g2},  the
Siegel--Veech  constant  of  any  arithmetic  Teichm\"uller disc in a
hyperelliptic   locus  depends  only  on  the  ambient  locus.  Being
formulated                 in                 terms                of
equation~\eqref{eq:SVconstant:for:pair:of:permutations}          this
statement    becomes   by   far   more   intriguing.   For   example,
Corollary~\ref{cor:g2}  implies  the following statements about pairs
of permutations.

\textbf{Corollary     \ref{cor:g2}$\pmb{}'$.}    \textit{    Consider
permutations  $\pi_h,\pi_v\in\mathfrak{S}_N$  such that $\pi_h,\pi_v$
do not have nontrivial common invariant subsets in $\{1,\dots,N\}$.
}

\textit{
If  the commutator $[\pi_h,\pi_v]$ has a single cycle of length three
and all other cycles have lengths one, than
}
$$
\frac{1}{\card\big(T,R\text{-orbit of }(\pi_h,\pi_v)\big)}
\sum_{\substack{
  (\pihor,\pivert)\text{}\\
  \text{in the}\\
  \rule{0pt}{7pt}T,R\text{-orbit}
}}
\ \sum_{\substack{\text{cycles }c_i\\\text{ in }\pihor}}
\frac{1}{\text{length of }c_i}
=
\frac{10}{9}
$$
\textit{  Here  by  a  ``$T,R$-orbit of $(\pi_h,\pi_v)$'' we mean the
minimal  subset  in  $\mathfrak{S}_N\times\mathfrak{S}_N$  containing
$(\pi_h,\pi_v)$ and invariant under the operations $T$ and $R$.
}

\textit{  If the commutator $[\pi_h,\pi_v]$ has exactly two cycles of
length two and all other cycles have lengths one, than
}
$$
\frac{1}{\card\big(T,R\text{-orbit of }(\pi_h,\pi_v)\big)}
\sum_{\substack{
  (\pihor,\pivert)\text{}\\
  \text{in the}\\
  \rule{0pt}{7pt}T,R\text{-orbit}
}}
\ \sum_{\substack{\text{cycles }c_i\\\text{ in }\pihor}}
\frac{1}{\text{length of }c_i}
=
\frac{5}{4}
$$
\medskip

In  other words, when the commutator has a single nontrivial cycle of
length  three,  or  only  two  nontrivial  cycles  of length two, the
average inverse length of a cycle over all cycles of all permutations
in a $T,R$-orbit does not depend neither on $N$ nor on a specific $T,
R$-orbit for a given $N$.

Experimenting  with orbits of square-tiled surfaces, the authors have
observed  the  same phenomenon in further strata in small genera. For
example,  in  genus  three  the  Siegel--Veech constant of arithmetic
Teichm\"uller  discs  did not vary for discs in all strata except the
principal one, $\cH(1,1,1,1)$.

Of  course, this non-varying phenomenon was initially checked only for
orbits  of  size  sufficiently  small to be treated by a computer (of
cardinality below $10^6$). However, we have conjectured that it would
be  valid for all orbits in a certain list of connected components of
the strata in genera $3, 4, 5$.

An  explanation and a proof of this non-varying phenomenon was finally
recently  found by D.~Chen and M.~M\"oller~\cite{Chen:Moeller} almost
a decade after it was conjectured.

\subsection{Global average}

Finally,             one            can            use            the
interpretation~\eqref{eq:SVconstant:for:pair:of:permutations}  of the
Siegel--Veech  constant  of an arithmetic Teichm\"uller disc to state
the  following  statement,  where  the operations $T$ and $R$ are not
present anymore.

\begin{Definition}
A  pair  $(\pi_h,\pi_v)$  of  permutations  in  $\mathfrak{S}_N$  has
\textit{type}  $(m_1,\dots,m_\noz)$  if  $\pi_h,\pi_v$  do  not  have
nontrivial  common  invariant  subsets  in $\{1,\dots,N\}$ and if the
length  spectrum  of  decomposition  into  cycles  of  the commutator
$[\pihor,\pivert]$ satisfies
\begin{multline*}
\{m_1+1,\dots,m_\noz+1\}
\ =\\=\
\{\text{Lengths of cycles of }[\pihor,\pivert]
  \text{, which are longer than 1}\}\,.
\end{multline*}
\end{Definition}

\begin{Proposition}
For  any  connected  stratum  $\cH(m_1,\dots,m_\noz)$ the limit below
exists and is equal to the normalized Siegel--Veech constant:
$$
\lim_{N\to\infty}\
\sum_{k=1}^N\
\sum_{\substack{
  (\pihor,\pivert)\\
  \text{\rm of type}\\
  (m_1,\dots,m_\noz)\\
  \text{\rm in }\mathfrak{S}_k\times\mathfrak{S}_k
}}
\ \sum_{\substack{\text{\rm cycles }c_i\\\text{\rm in }\pihor}}
\frac{1}{\text{\rm length of }c_i}
=
\frac{\pi^2}{3}\cdot c_{area}\big(\cH(m_1,\dots,m_\noz)\big)
$$
\end{Proposition}
\begin{proof}
This is essentially the content of \cite[Appendix A]{Chen:slope}.
\end{proof}

\subsection*{Acknowledgments}
We  are extremely grateful to A.~Kokotov for clarifying to us several
confusions  related  to  normalization  of  the flat Laplacian and to
K.~Rafi  who  has  carefully read a preliminary version of this text,
has made several important remarks and has given some helpful advice.
We  thank  D.~Korotkin  for  suggesting us an idea of the alternative
proof    of   Theorem~\ref{theorem:main:local:formula}.   We   highly
appreciate  important  discussions with J.~Grivaux, with H.~Masur, with
M.~M\"oller  and  with  A.~Wright.

We  express our special thanks to G.~Forni and to P.~Hubert for their
careful  reading  of  the paper. Their numerous and valuable comments
and suggestions have helped us to radically improve the presentation.

The  authors  are  grateful to HIM, IHES, MPIM and to the Universities of
Chicago  and  Rennes  1 for their hospitality during preparation of this
paper.


\end{document}